\documentclass[11pt,twoside]{article}
\usepackage{simplewick}
\usepackage{times}
\usepackage{amsmath,amssymb,amsthm}
\usepackage{color}
\usepackage{hep}
\usepackage{mathrsfs}
\usepackage{leftidx}
\usepackage{graphicx}
\usepackage{float}
\usepackage{enumerate}
\usepackage{titletoc}
\usepackage{epstopdf}
\usepackage{ulem}

\allowdisplaybreaks
\def \no{\nonumber}

\def\R {\Bbb R}

\def\p{\partial}
\def\ve{\varepsilon}
\def\f{\frac}
\def\na{\nabla}

\def\al{\alpha}
\def\t{\tilde}
\def\q{\quad}
\def\vp{\varphi}
\def\O{\Omega}
\def\th{\theta}
\def\g{\gamma}
\def\G{\Gamma}

\def\dl{\delta}
\def\p{\partial}
\def\ve{\varepsilon}
\def\f{\frac}

\def\na{\nabla}

\def\al{\alpha}

\def\t{\tilde}
\def\o{\omega}
\def\O{\Omega}
\def\vp{\varphi}
\def\th{\theta}

\def\g{\gamma}
\def\G{\Gamma}

\def\dl{\delta}

\def\m{\sqrt{\mu_m}}
\def\q{\qquad}

\def\ds{\displaystyle}
\allowdisplaybreaks

\begin{document}
	\footskip=0pt
	\footnotesep=2pt
	\let\oldsection\section
	\renewcommand\section{\setcounter{equation}{0}\oldsection}
	\renewcommand\thesection{\arabic{section}}
	\renewcommand\theequation{\thesection.\arabic{equation}}
	\newtheorem{claim}{\noindent Claim}[section]
	\newtheorem{theorem}{\noindent Theorem}[section]
	\newtheorem{lemma}{\noindent Lemma}[section]
	\newtheorem{proposition}{\noindent Proposition}[section]
	\newtheorem{definition}{\noindent Definition}[section]
	\newtheorem{remark}{\noindent Remark}[section]
	\newtheorem{corollary}{\noindent Corollary}[section]
	\newtheorem{example}{\noindent Example}[section]

\title{Global smooth solutions of 2D quasilinear wave equations
with higher order null conditions
and short pulse initial data}
\author{Ding
Bingbing$^{1,*}$, \quad Xin Zhouping$^{2,*}$, \quad Yin
Huicheng$^{1,}$\footnote{Ding Bingbing (13851929236@163.com, bbding@njnu.edu.cn) and Yin Huicheng (huicheng@nju.edu.cn, 05407@njnu.edu.cn) were
supported by the NSFC (No.12331007, No.12071223). Xin Zhouping(zpxin@ims.cuhk.edu.hk)
is partially supported by the Zheng Ge Ru Foundation, Hong Kong RGC Earmarked Research Grants
CUHK-14301421, CUHK-14300819, CUHK-14302819, CUHK-14301023, Basic and Applied Basic Research Foundations of Guangdong Province 2020131515310002, and the Key Projects of National Nature Science Foundation of China (No.12131010, No.11931013).}\vspace{0.5cm}\\
\small 1.  School of Mathematical Sciences and Mathematical Institute, Nanjing Normal University,\\
\small  Nanjing, 210023, China.
\\
\vspace{0.5cm}
\small 2. Institute of Mathematical Sciences, The Chinese University of Hong Kong, Shatin, NT, Hong Kong.}
\date{}
\maketitle
% \vskip 0.2in

\centerline {\bf Abstract} \vskip 0.3 true cm

For the short pulse initial data
with a first order outgoing constraint condition and optimal orders of smallness, we establish the global existence of smooth solutions
to 2D quasilinear wave equations with higher order null conditions. Such kinds of wave equations
include 2D relativistic membrane equations, 2D membrane equations,
and some 2D quasilinear equations which come
from the nonlinear Maxwell equations in electromagnetic theory or
from the corresponding Lagrangian functionals as perturbations of the Lagrangian densities of linear wave operators. The main ingredients of the analysis here include looking for a new good unknown, finding some key identities based on the higher order null conditions and the resulting null frames, as well as overcoming the difficulties due to the slow decay of solutions to the 2-D wave equation, so that the solutions can be estimated precisely.

\vskip 0.2 true cm
{\bf Keywords:} Quasilinear wave equations, short pulse initial data, higher order null condition,

\qquad \qquad\quad  inverse foliation density, good unknown, null frame

\vskip 0.2 true cm {\bf Mathematical Subject Classification:} 35L05, 35L72

\vskip 0.4 true cm
%\head
\tableofcontents

\section{Introduction}\label{in}
%\endhead
\subsection{The problem and main results}

Consider the 2D quasilinear wave equation
\begin{equation}\label{quasi}
\sum_{\al,\beta=0}^2g^{\al\beta}(\p\phi)\p_{\al\beta}^2\phi=0
\end{equation}
with the short pulse initial data
\begin{align}\label{id}
\phi|_{t=1}=\delta^{2-\varepsilon_0}\phi_0(\f{r-1}{\delta},\omega),\
\p_t\phi|_{t=1}=\delta^{1-\varepsilon_0}\phi_1(\f{r-1}{\delta},\omega),
\end{align}
where $(x^0, x)=(t, x^1, x^2)\in [1,\infty)\times\R^2$,
$\p=(\p_0,\p_1,\p_2)=(\p_{x^0}, \p_{x^1}, \p_{x^2})$,
$g^{\al\beta}(\p\phi) =g^{\beta\al}(\p\phi)$ are smooth
functions of their arguments, $\dl>0$ is small, $0<\ve_0<1$ is a fixed constant, $r=|x|=\sqrt{(x^1)^2+(x^2)^2}$,
$\o=(\o_1,\o_2)=\f{x}{r}\in\mathbb S^1$, $(\phi_0,\phi_1)(s,\o)$ are smooth functions defined
in $\mathbb R\times \mathbb S^1$
with compact support in $(-1,0)$ for the variable $s$. It is pointed out that when $\ve_0=\f12$ in \eqref{id},
such a class of initial data are first introduced by D. Christodoulou \cite{C3}
for studying the formation of trapped surfaces in the Einstein vacuum spacetime
(see also \cite{K-R}).

In this paper, we assume that for fixed integer $k\geq 2$, it holds that
\begin{equation}\label{g}
g^{\al\beta}(\p\phi)=m^{\al\beta}+\sum_{0\leq\gamma_1, \cdots, \gamma_k\leq2}g^{\al\beta,
	\gamma_1,\cdots,\gamma_k}\p_{\gamma_1}\phi\cdots\p_{\gamma_k}\phi+h^{\al\beta}(\p\phi)
\end{equation}
for small $\p\phi$,
where $m^{00}=-1$, $m^{ii}=1$ for $1\le i\le 2$, $m^{\al\beta}=0$ for $\al\neq\beta$,
the constants $g^{\al\beta,
\gamma_1,\cdots,\gamma_k}$ are equal for all the permutations of $(\al,\beta)$ and $(\gamma_1,\cdots,\gamma_k)$ respectively,
and $h^{\al\beta}(\p\phi)=O(|\p\phi|^{k+1})$.

A lot of physical or geometric models admit the form \eqref{g},
which include:

(1) The 2D relativistic membrane equation (corresponding to $k=2$ in \eqref{g})
\begin{equation}\label{HCC-1-0}
\p_t\big(\ds\f{\p_t\phi}{\sqrt{1-(\p_t\phi)^2+|\na_x\phi|^2}}\big)
-div\big(\ds\f{\nabla_x\phi}{\sqrt{1-(\p_t\phi)^2+|\na_x\phi|^2}}\big)=0\quad\text{with $\na_x\phi=(\p_1\phi,\p_2\phi)$}.
\end{equation}
Note that \eqref{HCC-1-0} is  the Euler-Lagrange equation of the area functional
$\int_{\mathbb R\times\mathbb R^2}\sqrt{1+|\nabla_x \phi|^2-(\p_t\phi)^2}dtdx$
for the embedding of $(t,x)\to (t,x,\phi(t,x))$ in the Minkowski spacetime.

(2) The 2D nonlinear membrane equation (corresponding to $k=2$ in \eqref{g})
\begin{equation}\label{HCCF-2-0}
\p_t^2\phi-div\big(\ds\f{\nabla_x\phi}{\sqrt{1+|\na_x\phi|^2}}\big)=0,
\end{equation}
where $\phi(t,x)$ stands for the position of  membrane at point $(t,x)$.

(3) The nonlinear wave equation (corresponding to $k=p$ in \eqref{g}, $p\in\mathbb N$)
\begin{equation}\label{Son-1}
-(1+(\p_t\phi)^{p+1})\p_t^2\phi+\Delta\phi=0,
\end{equation}
which comes from the variation of the Lagrangian functional
\begin{equation*}\label{Son-0}
L(\phi)=-\ds\f12(\p_t\phi)^2+\f12|\nabla_x\phi|^2
-\f{(\p_t\phi)^{p+3}}{(p+3)(p+2)}.
\end{equation*}
Especially, \eqref{Son-1} with $p=1$ may be regarded as a model equation from the nonlinear
version of Maxwell equations in the electromagnetic theory (see \cite[Section 1.3]{MY});

(4) The nonlinear wave equation (corresponding to $k=2p$ in \eqref{g}, $p\in\mathbb N$)
\begin{equation}\label{Son-00}
-\p_t^2\phi+\Delta\phi-\p_t\bigg(\big(-\ds\f12(\p_t\phi)^2+\f12|\nabla_x\phi|^2\big)^{p}\p_t\phi\bigg)
+div\bigg(\big(-\ds\f12(\p_t\phi)^2+\f12|\nabla_x\phi|^2\big)^{p}\nabla_x\phi\bigg)=0,
\end{equation}
which is the Euler-Lagrangian equation of
$L(\phi)=-\ds\f12(\p_t\phi)^2+\f12|\nabla_x\phi|^2
+\f{1}{p+1}\big(-\ds\f12(\p_t\phi)^2+\f12|\nabla_x\phi|^2\big)^{p+1}$, which can be thought as a
suitable perturbation of the Lagrangian density of the linear wave operator.

(5) As a generalization of \eqref{Son-1} and \eqref{Son-00}, the nonlinear wave
equation  (corresponding to $k=2p$ in \eqref{g}, $p\in\mathbb N$)
\begin{equation}\label{Son-1-0}
\begin{split}
&-(1+(\p_t\phi)^{2p+1})\p_t^2\phi+\Delta\phi-\p_t\bigg(\big(-\ds\f12(\p_t\phi)^2+\f12|\nabla_x\phi|^2\big)^{p}\p_t\phi\bigg)\\
&+div\bigg(\big(-\ds\f12(\p_t\phi)^2+\f12|\nabla_x\phi|^2\big)^{p}\nabla_x\phi\bigg)=0
\end{split}
\end{equation}
satisfies the $(2p)^{\text{th}}$ null condition but violates the $(2p+1)^{\text{th}}$ (see the definition of the $k^{\text{th}}$
null condition in \eqref{null}). Note that
\eqref{Son-1-0} corresponds to the Euler-Lagrangian equation of
\begin{equation*}
L(\phi)=-\ds\f12(\p_t\phi)^2+\f12|\nabla_x\phi|^2
+\f{1}{p+1}\big(-\ds\f12(\p_t\phi)^2+\f12|\nabla_x\phi|^2\big)^{p+1}-\f{(\p_t\phi)^{2p+3}}{(2p+3)(2p+2)}\quad
(p\in\mathbb N).
\end{equation*}

For the study of the global existence of smooth solutions  to \eqref{quasi}
with \eqref{g}, it is often crucial to introduce the following definition of the $k^{\text{th}}$
null condition $(k\ge 2, k\in\mathbb N)$
for \eqref{quasi}:
\begin{equation}\label{null}
\sum_{0\le \al,\beta,\gamma_1, ..., \gamma_k\le 2}g^{\al\beta,\gamma_1,\cdots,\gamma_k}
\xi_\al\xi_\beta\xi_{\gamma_1}\cdots{\xi_{\gamma_k}}\equiv 0\quad \text{holds for $\xi_0=-1$ and
	$(\xi_1, \xi_2)\in \mathbb S^1$}.
\end{equation}
From now on, the $k^{\text{th}}$ null condition with $k\ge 2$ will be called higher order  null condition in this paper.

We emphasize that for small $\p\phi$ and general $g^{\al\beta}(\p\phi)$ with
\begin{equation}\label{H-01X}
\ds g^{\al\beta}(\p\phi)=m^{\al\beta}+\sum_{0\le \gamma\le 2} g^{\al\beta,\gamma}\p_\gamma\phi+\sum_{0\le\gamma_1,\g_2\le 2}g^{\al\beta,\gamma_1\g_2}\p_{\gamma_1}\phi\p_{\g_2}\phi+O(|\p\phi|^3),
\end{equation}
the  first and second null conditions have been introduced for \eqref{quasi}
(see \cite{C1, K1, A, A2}), namely,
there hold that for $\xi_0=-1$ and $(\xi_1,\xi_2)\in\mathbb S^1$,
\begin{equation}\label{Euler-0}
\ds\sum_{0\le\al,\beta,\gamma\le 2} g^{\al\beta,\gamma}\xi_\al\xi_\beta\xi_\gamma\equiv0\quad\text{and}
\sum_{0\le\al,\beta,\gamma_1,\gamma_2\le 2} g^{\al\beta,\gamma_1\gamma_2}\xi_\al\xi_\beta\xi_{\gamma_1}\xi_{\gamma_2}\equiv0, \quad\text{respectively.}
\end{equation}
Moreover, it has been shown that for small smooth  initial data with compact supports
\begin{align}\label{id-1}
\phi|_{t=1}=\delta \psi_0(x),\
\p_t\phi|_{t=1}=\delta\psi_1(x),\ (\psi_0(x),\psi_1(x))\not\equiv 0,
\end{align}
\eqref{quasi} with \eqref{H-01X} and \eqref{id-1} admits
a global smooth solution $\phi$ if and only if \eqref{Euler-0} holds true,
otherwise, the solution $\phi$ may blow up in finite
time and further the shock is formed (see \cite{A1,A,A2,C1,C2,CM,0-Speck,H,J2, K1, LS, J, S2}).

For 2D potential equation of irrotational isentropic
Chaplygin gases
\begin{equation}\label{quasi-Chap}
\begin{split}
&\ds\sum_{\al,\beta=0}^2g^{\al\beta}(\p \phi)\p_{\al\beta}^2\phi= -\p_t^2\phi+\triangle\phi-2\sum_{i=1}^2\p_i\phi\p_t\p_i\phi
+2\p_t\phi\triangle\phi-\sum_{i,j=1}^2\p_i\phi\p_j\phi\p_{ij}^2\phi
+|\nabla\phi|^2\triangle\phi=0,
\end{split}
\end{equation}
whose coefficients fulfill \eqref{Euler-0}.
Under the short pulse initial data \eqref{id} and the first order outgoing constraint condition
\begin{align}\label{Y-0}
(\p_t+\p_r)\phi|_{t=1}=O(\delta^{2-\varepsilon_0}),
\end{align}
it is shown in \cite{Ding3} that when $0<\ve_0<\f12$, \eqref{quasi-Chap} has a global smooth solution $\phi$
with $|\p\phi|\le C\delta^{1-\varepsilon_0}t^{-1/2}$ for all $t\ge 1$.
As illustrated in \cite{Ding3}, in order to keep the strict hyperbolicity of \eqref{quasi} or \eqref{quasi-Chap},
the smallness of $\p\phi$ is certainly required
and can be derived from \eqref{id} and \eqref{Y-0}.

Motivated by the results of \cite{MY, Ding6, Lu1} in 3D case, we expect to find the optimal order $\varepsilon_k^*$
of smallness in \eqref{id} such that if \eqref{g} satisfies the $k^{\text{th}}$ null condition
but does not fulfill the $(k+1)^{\text{th}}$ null condition, then the smooth solution of \eqref{quasi} with \eqref{id} and \eqref{Y-0}
exists globally when $\varepsilon_0\in(0,\varepsilon_{k}^*)$ and may blow up in finite time when $\varepsilon_0\in[\varepsilon_{k}^*,1)$.

Note that it follows easily from the forms of \eqref{id} and \eqref{Y-0} that for $0\leq p\leq 1$,
\begin{align}\label{Y-0-a}
(\p_t+\p_r)^p\O^j\p^\al\phi|_{t=1}=O(\delta^{2-\varepsilon_0-|\al|}),
\end{align}
where $\O=x^1\p_2-x^2\p_1$. Without loss of generality and for simplicity, it is assumed that in \eqref{quasi},
\begin{equation}\label{g00}
g^{00}(\p\phi)=-1.
\end{equation}

Our main result in the paper is:
\begin{theorem}\label{main}
	Let $\ve_k^*=\f{k}{k+1}$ with $k\ge2$.
Under the conditions \eqref{g}, \eqref{null},  \eqref{Y-0} and \eqref{g00},
it holds that for small $\delta>0$,
when $\ve_0\in(0,\ve_k^*)$, \eqref{quasi} with \eqref{id}
admits a global smooth solution $\phi$ which satisfies
$$\phi\in C^\infty([1,+\infty)\times\mathbb R^2), \quad\text{$|\p\phi|\le C\delta^{1-\varepsilon_0}t^{-1/2}$
and $|\phi|\le C\delta^{1-\varepsilon_0}t^{1/2}$ for $t\ge 1$,}$$
where $C>0$ is a uniform constant independent of $\dl$ and $\ve_0$.
\end{theorem}

\begin{remark}\label{blowup}
	As a direct application of Theorem \ref{main}, \eqref{Son-1} with $p=k\ge 2$ has a global solution when the initial data are identical to those in Theorem \ref{main} and $\varepsilon_0 \in (0, \varepsilon_{k}^*)$. Conversely,
for certain initial short pulse data,	
it follows from \cite{Lu1} that the solution of \eqref{Son-1} with $p=k\ge 1$  may blow up
in finite time when $\varepsilon_0\in[\varepsilon_{k}^*,1)$.
As previously mentioned, \eqref{Son-1}  with $p=k$ is a special equation satisfying the $k^{\text{th}}$
null condition but violates the ${(k+1)}^{\text{th}}$ null condition. Therefore, we conjecture that, for $\varepsilon_0\in[\varepsilon_{k}^*,1)$ with $k\geq2$, if
	$h^{\al\beta}(\p\phi)=\sum_{0\leq\gamma_1, \cdots, \gamma_{k+1}\leq2}h^{\al\beta,\g_1,\cdots,\g_{k+1}}(\p_{\g_1}\phi)\cdots(\p_{\g_{k+1}}\phi)+O(|\p\phi|^{k+2})$ in \eqref{g}
and $\sum_{0\leq\gamma_1, ..., \gamma_{k+1}\leq2}h^{\al\beta,\gamma_1,\cdots,\gamma_{k+1}}
	\xi_\al\xi_\beta\xi_{\gamma_1}\cdots{\xi_{\gamma_{k+1}}}\not\equiv0$ for $\xi_0=-1$ and $(\xi_1, \xi_2)\in \mathbb S^1$, then there exist a class of initial data specified by \eqref{id} such that the smooth solution of \eqref{quasi} will undergo finite-time blowup.
This intriguing phenomenon will be one of the topics we intend to explore in our future research.
\end{remark}

On the other hand, it is easy to verify that \eqref{HCC-1-0} satisfies the second order null condition but \eqref{HCCF-2-0} does not,
\eqref{Son-00} satisfies the $(2p)^{\text{th}}$ null condition, \eqref{Son-1-0} fulfills the $(2p)^{\text{th}}$ null condition
but does not fulfills the $(2p+1)^{\text{th}}$ null condition.
Collecting Theorem \ref{main} and the analogous arguments in  \cite{MY, Ding6, Lu1}
as for \eqref{Son-1} with \eqref{id} and \eqref{Y-0}, we have

\begin{corollary}

Under the conditions \eqref{id} and \eqref{Y-0-a},

(1) \eqref{HCC-1-0} has a global smooth solution $\phi$ when $0<\ve_0<\ve_2^*=\f23$.

(2)  the smooth solution $\phi$ of \eqref{Son-00} exists globally when $0<\ve_0<\ve_{2p}^*=\f{2p}{2p+1}$.

(3) assume that there exists a point $(s_0,\o_0)\in(-1,0)\times\mathbb{S}^1$ such that
\begin{equation}\label{shock-formation assumption of data-0}
\begin{split}
&\p_s\phi_0(s_0,\o_0)\p_s^2\phi_0(s_0,\o_0)>1\qquad \text{for $\ve_1^*=\f12<\ve_0<1$},\\
&\p_s\phi_0(s_0,\o_0)\p_s^2\phi_0(s_0,\o_0)>2\qquad \text{for $\ve_0=\ve_1^*$},
\end{split}
\end{equation}
the smooth solution $\phi$ of \eqref{HCCF-2-0} will blow up in finite time and further the shock is formed.

(4) when $0<\ve_0<\ve_{2p}^*=\f{2p}{2p+1}$, the smooth solution $\phi$ of \eqref{Son-1-0} exists globally;
when $\ve_{2p}^*\le\ve_0<1$, if there exists a point $(s_0,\o_0)\in(-1,0)\times\mathbb{S}^1$ such that
\begin{equation}\label{shock-00}
\begin{split}
&\phi_1^{2p}(s_0,\o_0)\p_s\phi_1(s_0,\o_0)>\f{2}{2p+1}\qquad \text{for $\ve_{2p}^*<\ve_0<1$},\\
&\phi_1^{2p}(s_0,\o_0)\p_s\phi_1(s_0,\o_0)>\f{(p+1)2^{2p+1}}{(2^{2p}-1)(2p+1)}\qquad \text{for $\ve_0=\ve_{2p}^*$},
\end{split}
\end{equation}
then the smooth solution $\phi$ of \eqref{Son-1-0} can blow up in finite time and further the shock is formed.
\end{corollary}

\subsection{Some remarks}
\begin{remark}\label{R1.1}
Note that
\eqref{Y-0-a} can be easily fulfilled for any given smooth function $\phi_0$
and the choice of $\phi_1(s,\o)=-\p_s\phi_0(s,\o)$.
\end{remark}

\begin{remark}\label{R1.3}
Under the outgoing constraint conditions \eqref{Y-0-a} with large enough $p$ and some number $\ve_0$, for the short pulse initial data,
the authors in \cite{MPY,Wang} have
established  the existence of global smooth solutions
for the 3D semi-linear wave system satisfying the first null condition
and for $n$-dimensional ($n=2,3$) relativistic membrane equations, respectively.	
Note that, as illustrated in \cite{Ding3}, the largeness of  $p$ and some suitable $\ve_0$ for \eqref{Y-0-a}
play the  key roles in the arguments of \cite{MPY,Wang}.
Note that the equation of the form \eqref{quasi} studied in \cite{Wang} is a special case with $k=2$
and $h^{\al\beta}(\p\phi)\equiv0$ in \eqref{null}. In addition, the particular divergence structure in \cite{Wang} makes it possible to introduce an elaborate vector field which approximates $\p_t+\p_r$ effectively and to obtain the global existence of the solution with the help of the method in \cite{MPY}. However, it seems difficult to adopt such an approach in \cite{MPY, Wang} to treat the general \eqref{quasi} with the short pulse initial data \eqref{id} under the weaker outgoing constraint condition \eqref{Y-0} and the optimal scope of $\varepsilon_0$. On the other hand, some recent progress on the global large data solutions of \eqref{HCC-1-0}
can be referred to \cite{A-W} and \cite{LOS}.
\end{remark}

\begin{remark}\label{R1.5}
 Our analysis is strongly motivated by the geometric approach pioneered by D. Christodoulou \cite{C2}, whose original purpose was to study the formation of shocks in multi-dimensional hyperbolic systems with genuinely nonlinear conditions, and subsequently for second order wave equations that do not satisfy the corresponding null conditions. See also \cite{CM,0-Speck,LS,MY,J,S2} for further details.
\end{remark}

\begin{remark}\label{R1.5-1}
 Compared to \cite{Ding4} and \cite{Ding3}, the main new difficulty in this paper is how to get the precise estimates of related quantities which makes possible to get global solution when $\ve_0$ approaches $\ve_k^*$. To this end, near the outgoing light cone surface, by identifying new suitable unknown variables (e.g. \eqref{DA}, \eqref{mC}, \eqref{mD}) and utilizing some fundamental identities (e.g. \eqref{gLLL2}, \eqref{LG}) that arise under the higher-order null conditions and the corresponding null frames, we are able to estimate accurately the solution and determine the optimal smallness exponent for the short pulse initial data (for example, the higher order null conditions help us rewrite $(\p_{\varphi_\gamma}g^{\al\beta})\mathring L_\al\mathring L_\beta\mathring L_\gamma$ in \eqref{GT} as \eqref{gLLL2}, and the resulting good unknown $\mathscr A$ is defined in \eqref{DA}. The dedicate estimate of $\mathscr A$ in \eqref{bZmA} yields \eqref{Zmu} with the aid of \eqref{lmu}, which is crucial for closing energy estimates in \eqref{D22} and \eqref{D22L} when $\ve_0$ approaches $\ve_k^*$). Within the interior of the outgoing light cone, based on precise boundary estimates of the solution on  $\tilde C_{2\delta}=\{(t,x):t\geq t_0,t-r=2\delta\}$ (see Proposition \ref{11.1}) due to the condition $0<\ve_0<\ve_k^*$, we derive global spacetime weighted energy estimates for the related Goursat problem.
\end{remark}

\subsection{Sketch of the proof and some notations}

We now sketch the proof of Theorem \ref{main}.
As in \cite{Ding4, Ding3}, the proof of Theorem 1.1 relies on some a priori uniform energy estimates and the continuous induction
argument. Let $A_{2\delta}=\{(t,x):t\geq 1+2\delta,0\leq t-r\leq 2\delta\}$
and $B_{2\dl}=\{(t, x): t\ge 1+2\dl, t-r\ge 2\dl\}$, which are near the outermost outgoing conic surface  $C_0=\{(t, x): t\ge 1+2\dl, t=r\}$
and inside the light cone $\{(t,x): t\ge 1+2\dl, t\ge r\}$, respectively.
Based on the local existence of solution $\phi$ with some desired properties
for $t\in[1,1+2\delta]$, both the global uniform energy estimates
in $A_{2\delta}$ and $B_{2\dl}$ will be established.

We start with estimates of the solution $\phi$ of \eqref{quasi} in $A_{2\delta}$.
Using the geometric framework of D. Christodoulou in \cite{C2},
we shall prove that the outgoing characteristic conic surfaces never intersect as in  \cite{Ding4, Ding3}.
To this end,
as in \cite{C2, J}, the {\it ``optical function"} $u$ of \eqref{quasi} can be defined as
\begin{equation}\label{H0-3}
\left\{
\begin{aligned}
&\ds\sum_{\al,\beta=0}^2g^{\al\beta}(\p
\phi)\p_\al{u}\p_\beta{u}=0,\\
&u(1+2\delta,x)=1+2\delta-r,
\end{aligned}
\right.
\end{equation}
and subsequently  the inverse foliation density is
\begin{equation}\label{H0-4}
\mu=
-\ds(\sum_{\al=0}^2g^{\al0}\p_\al u)^{-1},
\end{equation}
where $g^{\al\beta}=g^{\al\beta}(\p\phi)$. Set the metric $g=\ds\sum_{\al,\beta=0}^2g_{\al\beta}dx^\al dx^\beta$
with $(g_{\al\beta})$ being the inverse matrix of $(g^{\al\beta})$ and
\begin{equation}\label{H0-5}
\mathring L=-\mu \ds\sum_{\al,\beta=0}^2g^{\al\beta}\p_\al u\p_\beta.
\end{equation}
Extend the local coordinate $\th$ on the standard circle $\mathbb S^1$ by
\begin{equation}\label{H0-6}
\left\{
\begin{aligned}
&\mathring L\vartheta=0,\\
&\vartheta|_{t=1+2\dl}=\th.
\end{aligned}
\right.
\end{equation}
As in \cite{J,MY}, perform the change of coordinates:
$(t, x^1, x^2)\longrightarrow (s, u, \vartheta)$ near $C_0$ with
\begin{equation}\label{H0-7}
\begin{array}{l}
s=t,\quad u=u(t,x),\quad \vartheta=\vartheta(t,x).\\
\end{array}
\end{equation}
Then $X:=\f{\p}{\p\vartheta}$ is a tangent vector on the curved circle $S_{s,u}:=\{(s',u',\vartheta):s'=s,u'=u\}$.
Under the suitable bootstrap assumptions on the smallness and time decay rate of  $\p\phi$ (see $(\star)$
in Section \ref{BA}), we will show that the inverse foliation density
$\mu$ satisfies $|\mathring L\mu|\lesssim \delta^{k(1-\varepsilon_0)}s^{-(k+2)/2}\mu+\delta^{k-(k+1)\varepsilon_0}s^{-(k+1)/2}$.
By $\mu(1+2\delta,x)\sim 1$ and through the integration along integral curves of $\mathring L$, $\mu\sim 1$ can
be derived for small $\dl>0$ and $\ve_0\in (0,\ve_k^*)$. The positivity of $\mu$ implies that the outgoing characteristic conic surfaces never
intersect as long as the smooth solution $\phi$ to \eqref{main} with \eqref{id}
exists. Set $\varphi=(\varphi_0, \varphi_1, \varphi_2):=(\p_t\phi, \p_1\phi, \p_2\phi)$.
Then $\varphi$ solves a quasilinear wave system:
\begin{equation}\label{Y-1}
\mu\Box_{g}\varphi_\gamma=F_\gamma(\varphi, \p\varphi),\quad\gamma=0,1,2,
\end{equation}
where $\Box_g=\f{1}{\sqrt{\det g}}\p_\al(\sqrt{\det g}g^{\al\beta}\p_\beta)$, and $F_\gamma$ are smooth functions
in their arguments. Let $\Psi^{m+1}_\g=Z^{m+1}\varphi_\g$,
where $Z$ stands for one of some chosen first order vector fields. It follows from \eqref{Y-1}
and direct computation on the commutator $[\mu\Box_{g}, Z^{m+1}]$ that
\begin{equation}\label{0Y-1}
\mu\Box_{g}\Psi_\g^{m+1}=\Phi_\g^{m+1}
\end{equation}
with $\Phi_\g^{m+1}$ containing the $(m+2)$-th order derivatives of $\varphi$. In fact,
both $\slashed\nabla Z^{m}\textrm{tr}\chi$ and $\slashed\nabla^2Z^{m}\mu$ appear in the expression of $\Phi_\g^{m+1}$,
where $\chi_{XX}=g(\mathscr D_{X} \mathring L, X)$ is the second fundamental form of $S_{s,u}$
with $\mathscr D$ being the Levi-Civita connection of $g$, and $\textrm{tr}\chi$ is the trace of $\chi$ on $S_{s,u}$.
Following the analogous procedures in Section 3-Section 11 of \cite{Ding3} or
in Section 2-Section 9 of \cite{Ding4}, we can derive the energy estimates for \eqref{0Y-1}
to obtain the appropriate smallness orders of $\dl$ and time decay rates of $\Psi_\g^{m+1}$. However, in order
to look for the optimal exponent $\ve_k^*$ related to the $k-$th null condition, it is necessary to derive
the optimal smallness order of $\dl$ and precise time decay rate for all related quantities.
This gives rise to certain essential difficulties.
Some of these key difficulties and our new strategies to overcome them are sketched as follows (which
are different from those in \cite{Ding3}):

$\bullet$  Control $\|Z^{m+1}\mu\|_{L^2}$ precisely by the corresponding  energies and fluxes so that the
highest $(m+1)-$order  energies of $\varphi$
admit the required optimal smallness orders of $\dl$ and further the bootstrap assumptions can be closed.
In this process, note that $\|Z^{m+1}\mu\|_{L^2}$
can be estimated from  $\|Z^{m+1}\mathring L\mu\|_{L^2}$ by integrating along integral curves of $\mathring L$.
In addition, due to the $k-$th null condition, $\mathring L\mu$ can be composed of $\mathring L\phi$
and $\slashed d\phi$ with  $\slashed d$ being a differential on $S_{s,u}$. Thus,
in order to estimate $\|Z^{m+1}\mathring L\mu\|_{L^2}$, one needs to treat both $\|Z^{m+1}\mathring L\phi\|_{L^2}$
and $\|\slashed dZ^{m+1}\phi\|_{L^2}$. Unfortunately, although it has been found that $\|Z^{m+1}\mathring L\phi\|_{L^2}$
admits a better time decay rate than $\|\slashed dZ^{m+1}\phi\|_{L^2}$ (see Lemma \ref{Ldphi}), the optimal smallness order
of $\dl$ cannot be obtained. Thus more careful analysis on the structure of $\mathring L\mu$ is needed. A key observation here is that
$\mathscr A
=g^{\al\beta,\g a\g_3\cdots\g_k}\o_\al\o_\beta\o_\g(m_{ab}\slashed d^Xx^b\slashed d_X\phi)\o_{\g_3}\cdots\o_{\g_k}$ is a good unknown
since $\|Z^{m+1}\mathscr A\|_{L^2}$ admits the higher order smallness  of $\dl$ and better time decay rate than $\|\slashed dZ^{m+1}\phi\|_{L^2}$.
Note that $Z^{m+1}\mathring L\mu$ consists of mainly $Z^{m+1}\mathscr A$ and $Z^{m+1}\mathring L\phi$,
moreover, $\|Z^{m+1}\mathscr A\|_{L^2}$ and $\|Z^{m+1}\mathring L\phi\|_{L^2}$ are bounded by $\|R^m\slashed\triangle\mu\|_{L^2}$, where $R$ is a specific vectorfield of $Z$ which is projected by $\O$ on $S_{s,u}$ and $\slashed\triangle$ is a Laplace operator of $S_{s,u}$.
Thus the remaining task is to treat $\|R^m\slashed\triangle\mu\|_{L^2}$.

$\bullet$ Deal with $\|R^m\slashed\triangle\mu\|_{L^2}$.
Note that $\Phi_{\g}^{m+1}$ encompasses both $\slashed dZ^m\textrm{tr}\chi$ and $\slashed\nabla^2Z^m\mu$. Consequently, as demonstrated in \cite{Ding4, Ding3}, we can estimate the $L^2$ norms of $\slashed dZ^m\textrm{tr}\chi$ and $Z^m\slashed\triangle\mu$ simultaneously, not solely $\|R^m\slashed\triangle\mu\|_{L^2}$.
By deriving the transport equations for $\slashed\triangle\mu$ and $\textrm{tr}\chi$ along $\mathring L$ and making full use of the $k-$th null condition, the estimates of $\|\slashed dZ^m\textrm{tr}\chi\|_{L^2}$
and $\|Z^m\slashed\triangle\mu\|_{L^2}$ can be precisely controlled by the energies $E_{i,\leq m+2}$ ($i=1,2$) and the fluxes  $F_{1,\leq m+2}$
(see \eqref{eil}-\eqref{fil} and \eqref{Zmu}-\eqref{dnchi}).
In the process, one of the main observations is that $\sum_{j=1}^3\slashed d\varphi_j\cdot\slashed dx^j$ may be represented by $\mathring L\vp$ (see \eqref{dxdp}). This leads to that the estimate on the basic terms $\|\slashed dZ^{m+1}\varphi_j\|_{L^2}$ in
$\|\slashed dZ^m\text{tr}\chi\|_{L^2}$ can be bounded by  $\int_0^uF_{1,m+2}(s,u')du'$ (see \eqref{D22}-\eqref{D22L}).
Then all required quantities are eventually estimated through the optimal smallness orders of $\dl$
and better time decay rates, which are the keys to treat the case when $\ve_0$ is near  $\ve_k^*$ and further to
close the arguments of bootstrap assumptions on $\vp$.

On the other hand, in order to study the global Goursat problem for \eqref{quasi} and derive the a priori
uniform estimates of $\phi$ in $B_{2\dl}$ by energy methods,
the properties for $|\p^\al L^q\phi|\lesssim\delta^{2-\varepsilon_0}t^{-1/2-q}$ ($L=\p_t+\p_r$) on $\t C_{2\delta}$
with the higher order smallness $O(\delta^{2-\varepsilon_0})$ and the time decay rate of $t^{-1/2-q}$
are crucial since the related boundary integrals on $\t C_{2\dl}$ in the energy estimates require the
the smallness order of $O(\delta^{2-\varepsilon_0})$.  However, all the estimates on $\phi$ in  $A_{2\dl}$
are obtained under the modified frame $\{\mathring L, T, R\}$,
where $\mathring L$, $T$, $R$ are some suitable vectors approximating $L$, $-\p_r=\f{\underline L-L}{2}$ $(\underline L=\p_t-\p_r)$,
$\O=x^1\p_2-x^2\p_1$ respectively on the time $t=1+2\delta$. To get the $L^\infty$ estimates on $\phi$ under the
derivatives $\{L,\underline L, \O\}$
and further obtain the required smallness and time decay rate of $L\phi$, we need to express the vector fields $\{L,\underline L, \O\}$
by $\{\mathring L, T, R\}$. As in \cite{Ding4}, the most noteworthy expression is $L=c_{L}\mathring L+c_TT+c_{\O}\O$,
where $c_L$, $c_T$ and $c_{\O}$ are
smooth functions.  One expects that the desired derivative $L\phi$ and its derivatives can admit the smallness of higher
orders with respect to $\dl$
and faster time decay rate. Unfortunately, $T\phi$ and its derivatives have not enough smallness and time decay rate as needed,
it is hoped as in \cite{Ding4} that the coefficient $c_T$ can make up for these deficiencies. The generality of equation \eqref{quasi} and
the optimal scope of $\ve_0\in (0, \ve_k^*)$
make it more difficult to analyze the smallness orders and the time decay rates of the related coefficients $c_T$, $c_\O$ and so on.
Thanks to the $k-$th null condition \eqref{null}, we can find
their governing differential equations along $\mathring L$ and subsequently
get the desired estimates of  $c_L$, $c_T$ and $c_{\O}$.

The rest of the paper is organized as follows. In Section \ref{p-0}, first, some preliminaries on
the related differential geometry and definitions
are given. Based on the studies of the local solution to \eqref{quasi}, the crucial bootstrap assumptions
are given and further some important quantities are estimated in $A_{2\dl}$.
In Section \ref{EE}, we first carry out the energy estimates for the linearized equation $\mu\Box_g\Psi=\Phi$
and define some suitable higher order weighted energies and fluxes as in \cite{J}. Subsequently,
the higher order $L^2$ estimates of some related quantities with the precise smallness orders of $\dl$
and time decay rates are obtained.
In Section \ref{L2chimu}, the top order $L^2$ estimates on the derivatives of $\chi$ and $\mu$ are established.
On the other hand, the estimates for the error terms are derived in Section \ref{ert}.
In Section \ref{YY} and Section \ref{inside}, the global uniform estimates of the solution to \eqref{quasi} in $A_{2\dl}$
and $B_{2\dl}$ are obtained, respectively. Therefore, the proof of Theorem \ref{main}  can be completed
by the continuous induction argument.
In addition, some computations on the deformation tensor and commutator relations are given in Appendix A.
The proof on the  local solution properties of \eqref{quasi} for $1\le t\le 1+2\dl$ is given in Appendix B.

Through the whole paper, unless stated otherwise, Greek indices $\{\al, \beta, \gamma, \cdots\}$ corresponding to the
time-space coordinates are chosen in $\{0, 1, 2\}$, Latin indices $\{i, j, a, b, c, \cdots\}$ corresponding to the
spatial coordinates are $\{1, 2\}$, and the Einstein summation convention to sum over repeated upper
and lower indices is used. In addition, the convention $f_1\lesssim f_2$ means that there exists a generic positive constant
$C$ independent of the parameter $\dl>0$ in \eqref{id} and the variables $(t,x)$ such that $f_1\leq Cf_2$. The coefficients $g^{\al\beta}(\p\phi)$ are denoted as $g^{\al\beta}$ by convenience.
Denote $g=g_{\al\beta}dx^\al dx^\beta$ by a Lorentz metric, where $(g_{\al\beta})$ is the inverse matrix of $(g^{\al\beta})$.

Finally, the following conventional notations will be used too:
\begin{align*}
&t_0:=1+2\delta,\\
&\o^i:=\o_i=\f{x^i}{r},\ i=1,2,\quad\o_\perp:=(-\o^2, \o^1),\\
&\p_x=\nabla_x=(\p_1, \p_2,),\quad L:=\p_t+\p_r,\quad\underline L:=\p_t-\p_r,\\
&\O:=\epsilon_i^jx^i\p_j,\quad S:=t\p_t+r\p_r=\f{t-r}{2}\underline L+\f{t+r}{2}L,\\
&H_i:=t\p_i+x^i\p_t=\o^i\big(\f{r-t}{2}\underline L+\f{t+r}{2}L\big)+\f{t\o^i_\perp}{r}\O,\\
&\Sigma_{t}:=\{(s,x): s=t, x\in\mathbb R^2\},\\
&\varphi=(\varphi_0, \varphi_1, \varphi_2)=(\p_0\phi, \p_1\phi, \p_2\phi),
\end{align*}
where $\epsilon_{1}^2=1$, $\epsilon_{2}^1=-1$, $\epsilon_i^i=0$ for $i=1,2$.

\section{Preliminaries and some estimates under bootstrap assumptions}\label{p-0}

\subsection{Geometry and basic equalities under null frames}\label{p}

We give some preliminaries on the
related geometry and definitions, which will be
utilized as basic tools to establish the a priori global estimates on the
smooth solution $\phi$ to \eqref{quasi} in $A_{2\dl}$. This part is completely similar to \cite{Ding4, Ding3, J},
so we just list the notations and state the facts without going into the details.

The definitions of the optical function $u$, the inverse foliation density $\mu$ and
the definition of $\vartheta$ have been
given in \eqref{H0-3}, \eqref{H0-4} and \eqref{H0-6}, respectively. Note that
$$
\tilde{L}=-g^{\al\beta}\p_\al u\p_\beta
$$
is a tangent vector field of the outgoing light cone $\{ u=C\}$ and $\tilde{L}t=\mu^{-1}$ holds.

Set
$$
\mathring{L}=\mu\tilde{L},\quad \mathring{\underline L}=\mu\mathring L+2T\quad\text{with $T=\mu\tilde T$
and $\tilde T=-g^{\al 0}\p_\al-\mathring L$}.
$$
Then
$\mathring{L}$, $\mathring{\underline L}$ and  $X=\f{\p}{\p\vartheta}$
form a  null frame with respect to the metric $g$.

In the new coordinate $(s, u, \vartheta)$ of \eqref{H0-7}, the following subsets can be defined (see Figure \ref{pic:p1}):

\begin{definition}
\begin{equation}\label{error}
\begin{split}
&\Sigma_s^{u}:=\{(s',u',\vartheta): s'=s,0\leq u'\leq  u\},\quad u\in [0, 4\delta],\\
&C_{u}:=\{(s',u',\vartheta): s'\geq 1+2\delta, u'=u\},\\
&C_{u}^s:=\{(s',u',\vartheta): 1+2\delta\leq s'\leq s, u'=u\},\\
&S_{s, u}:=\Sigma_s\cap C_{u},\\
&D^{s, u}:=\{(s', u',\vartheta): 1+2\delta\leq s'<s, 0\leq u'\leq u\}.
\end{split}
\end{equation}
\end{definition}

	\begin{figure}[htbp]
	\centering
	\includegraphics[scale=0.35]{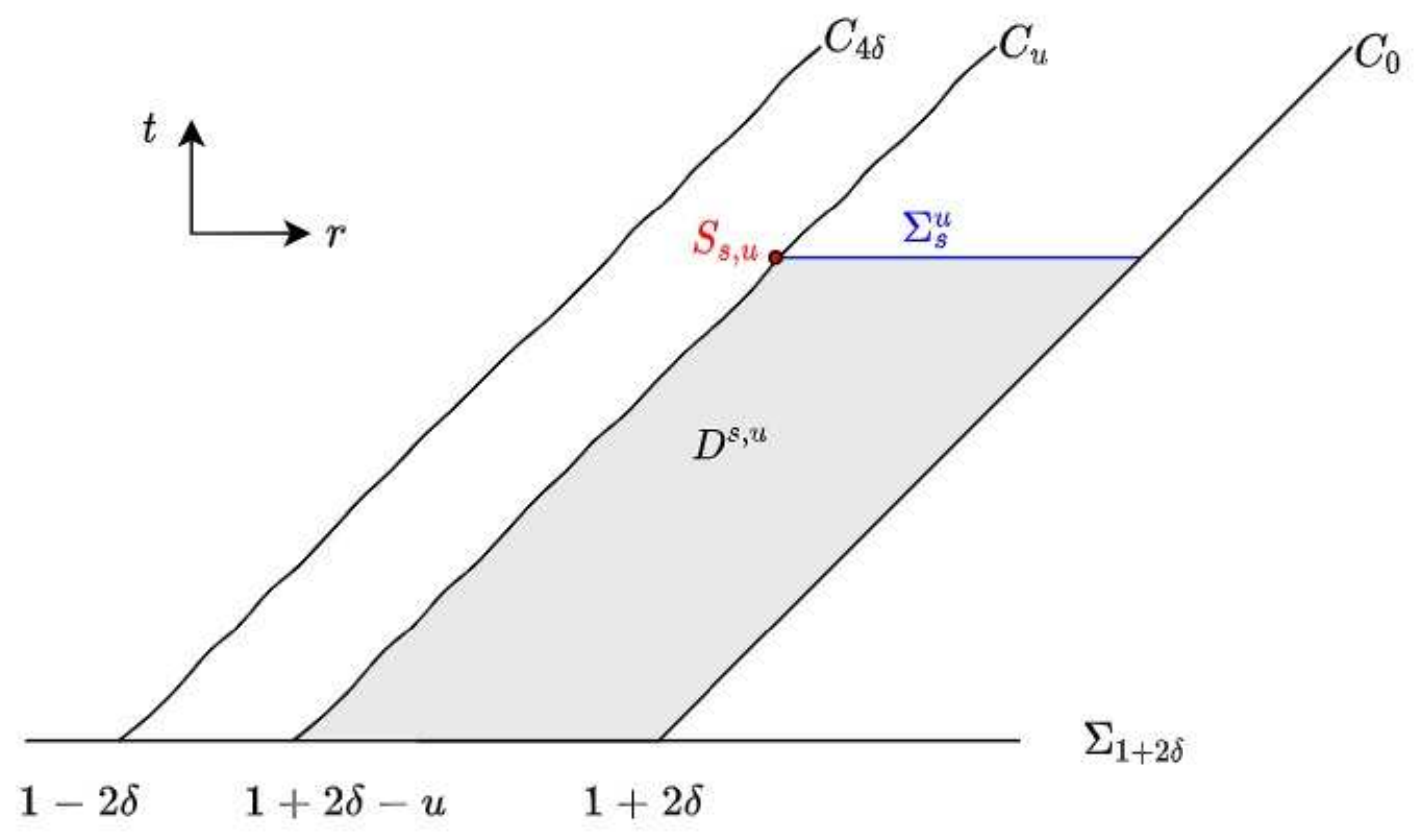}
	\caption{The indications of some domains}\label{pic:p1}
\end{figure}

The following geometric notations are needed.
\begin{definition} For the metric $g$ on the spacetime,
	\begin{itemize}
		\item $\underline g=(g_{ij})$ is defined as the induced metric of $g$ on $\Sigma_t$, i.e., $\underline g(U,V)=g(U,V)$ for any tangent vectors $U$ and $V$ of $\Sigma_t$;
		\item $\slashed \Pi_\al^\beta:=\delta_\al^\beta-\delta_\al^0\mathring L^\beta+\mathring L_\al\tilde{T}^\beta$ is the
projection tensor field on $S_{s,u}$ of type $(1,1)$, where $\delta_\al^\beta$ is Kronecker delta;
		\item define $\slashed\xi:=\slashed\Pi\xi$ as the tensor field on $S_{s,u}$ for the $(m,n)$-type spacetime tensor field
$\xi$, whose components are $$\slashed\xi^{\al_1\cdots\al_m}_{\beta_1\cdots\beta_n}:=(\slashed\Pi\xi)^{\al_1\cdots\al_m}_{\beta_1\cdots\beta_n}
=\slashed\Pi_{\beta_1}^{\beta_1'}\cdots\slashed\Pi_{\beta_n}^{\beta_n'}\slashed\Pi_{\al_1'}^{\al_1}\cdots
\slashed\Pi_{\al_m'}^{\al_m}\xi^{\al_1'\cdots\al_m'}_{\beta_1'\cdots\beta_n'}.$$
		Specially, $\slashed g=(\slashed g_{\al\beta})$ is the induced metric of $g$ on $S_{s,u}$;
		\item $\slashed g^{XX}$ is defined as the reciprocal of $\slashed g_{XX}$ with $\slashed g_{XX}=g(X, X)$;
		\item $\mathscr D$ and $\slashed\nabla$ denote the Levi-Civita connection of $g$ and $\slashed g$, respectively;
		\item $\Box_g:=g^{\al\beta}\mathscr{D}^2_{\al\beta}$, $\slashed\triangle:=\slashed g^{XX}\slashed\nabla^2_{X}$;
		\item $\mathcal L_V\xi$ denotes the Lie derivative of $\xi$ with respect to $V$ and $\slashed{\mathcal L}_V\xi:=\slashed\Pi(\mathcal L_V\xi)$ for any tensor field $\xi$ and vector $V$;
		\item when $\xi$ is a $(m,n)$-type spacetime tensor field, the square of its norm is defined as
		$$
		|\xi|^2:=g_{\al_1\al_1'}\cdots g_{\al_m\al_m'}g^{\beta_1\beta_1'}\cdots g^{\beta_n\beta_n'}\xi_{\beta_1\cdots\beta_n}^{\al_1\cdots\al_m}\xi_{\beta_1'\cdots\beta_n'}^{\al_1'\cdots\al_m'}.
		$$
	\end{itemize}
\end{definition}

In the null frame $\{\mathring L, \mathring{\underline L}, X\}$,
{\it{the second fundamental forms}} $\chi$ and $\sigma$ are defined as
\begin{equation}\label{chith}
\chi_{XX}:=g(\mathscr D_X \mathring L, X),\quad\sigma_{XX}:=g(\mathscr D_X\tilde T,X).
\end{equation}
Meanwhile, define the {\it one-form tensors} $\zeta$ and $\eta$ as
\begin{equation}\label{zetaeta}
\zeta_X:=g(\mathscr D_X \mathring L,\tilde T),\quad\eta_X:=-g(\mathscr D_X T, \mathring L).
\end{equation}
For any vector field $V$, denote its associate {\it{deformation tensor}} by
\begin{equation}\label{dt}
\leftidx{^{(V)}}\pi_{\al\beta}:=g(\mathscr{D}_{\al}V,\p_{\beta})+g(\mathscr{D}_{\beta}V,\p_{\al}).
\end{equation}
For $t\ge t_0$, the following {\it{error terms}} are defined as
\begin{equation}\label{errorv}
\begin{split}
&\check{L}^0:=0,\ \check{L}^i:=\mathring L^i-\f{x^i}{\varrho},\ \check{T}^i:=\tilde T^i+\f{x^i}{\varrho},\
\check{\chi}_{XX}:=\chi_{XX}-\f{1}{\varrho}\slashed g_{XX},
\end{split}
\end{equation}
here and below $\varrho=t-u$.
Note that in the new coordinate system $(s, u, \vartheta)$, it holds that
$\mathring L=\f{\p}{\p s}$ and  $T=\f{\p}{\p u}-\Xi$ with $\Xi=\Xi^XX$ for some smooth function $\Xi^X$.
In addition, \cite[Lemma 3.66]{J} gives

\begin{lemma}
In domain $D^{s, u}$, the Jacobian determinant of the map $(s, u, \vartheta)\rightarrow
(x^0, x^1, x^2)$ is
\begin{equation}\label{Jacobian}
\det\f{\p(x^0, x^1, x^2)}{\p(s, u, \vartheta)}=\mu(\det\underline g)^{-1/2}\sqrt{\slashed g_{XX}}.
\end{equation}
\end{lemma}

On $\mathbb{S}^1$, the vector field $\Omega=x^1\p_2-x^2\p_1$
is taken as the tangent derivative.
In order to project $\Omega$ on $S_{s, u}$, as in (3.39b) of \cite{J}, we define
\begin{equation*}
R:=\slashed\Pi\Omega,\quad \slashed d:=\slashed\Pi d
\end{equation*}
to be the rotation vectorfield and the differential of $S_{s, u}$, respectively. Note that the explicit expression of $R$ is
\begin{equation}\label{R}
R=(\delta_j^i+g_{\al j}\mathring L^\al\tilde T^i){\Omega}^j\p_i
=\Omega-g_{aj}\tilde{T}^a{\Omega}^j\tilde{T}.
\end{equation}
For brevity, one writes
\begin{equation}\label{omega}
\upsilon:=g_{ab}\tilde{T}^a{\Omega}^b=g_{aj}\check{T}^a\epsilon_i^jx^i-(g_{aj}-m_{aj})\f{\epsilon_i^j x^ix^a}\varrho.
\end{equation}
Then
$$R=\Omega-\upsilon\tilde{T}.$$

\begin{definition}\label{2.3}
For any continuous function $f$ and tensor field $\xi$, define
\begin{align*}
&\int_{S_{s, u}}f:=\int_{S_{s, u}}fd\nu_{\slashed g}=\int_{\mathbb{S}}f(s, u,\vartheta)\sqrt{\slashed g_{XX}(s, u,\vartheta)}d\vartheta,\\
&\int_{ C^s_{ u}}f:=\int_{t_0}^s\int_{S_{\tau, u}}f(\tau, u,\vartheta) d\nu_{\slashed g}d\tau,\quad
\int_{\Sigma_s^{ u}}f:=\int_0^{ u}\int_{S_{s, u'}}f(s, u',\vartheta) d\nu_{\slashed g}d u',\\
&\int_{D^{s, u}}f:=\int_{t_0}^s\int_0^{ u}\int_{S_{\tau, u'}}f(\tau, u',\vartheta)d\nu_{\slashed g}d u'd\tau,\quad
\|\xi\|_{s,u}:=\sqrt{\int_{\Sigma_s^{ u}}|\xi|^2}.
\end{align*}
\end{definition}

Finally, for reader's convenience, we recall the  notations of contraction and trace as follows.
\begin{definition}[]\label{def}
(1) If $\xi$ is a $(0,2)$-type spacetime tensor, $\Lambda$ is a 1-form, $U$ and $V$ are vector fields, then the contraction
of $\xi$ with respect to $U$ and $V$ is defined as
$$
\xi_{UV}:=\xi_{\al\beta}U^{\al}V^{\beta},
$$
and the contraction of $\Lambda$ with respect to $U$ is
\[
\Lambda_U:=\Lambda_{\al}U^{\al}.
\]
(2) If $\xi$ is a $(0,2)$-type tensor on $S_{s,u}$, the trace of $\xi$ is defined as
	\[
	\text{tr}\xi:=\slashed g^{XX}\xi_{XX}.
	\]
\end{definition}

Note that $(g_{\al\beta})$ is the inverse matrix of $(g^{\al\beta})$.
By $g^{\al\beta}g^{\lambda\kappa}g_{\kappa\beta}=g^{\al\lambda}$, it holds that
\begin{equation}\label{ggg}
 g^{\al\beta}g^{\lambda\kappa}(\p_{\varphi_\gamma} g_{\kappa\beta})=-\p_{\varphi_\gamma} g^{\al\lambda}.
\end{equation}

With the help of ${\chi}$, $\sigma$, ${\zeta}$ and ${\eta}$
defined as before, one can derive some basic equalities in the frame $\{\mathring L, \mathring{\underline L}, X\}$
or the frame $\{T, \mathring{\underline L}, X\}$.

Let
\begin{equation}\label{FG}
\begin{split}
&\ds G_{\al\beta}^\gamma:=\p_{\varphi_\gamma} g_{\al\beta}.
\end{split}
\end{equation}
For any vector fields $\ds U=U^\al\p_\al$ and $\ds V=V^\al\p_\al$, set $G_{UV}^\gamma=G_{\al\beta}^\gamma U^\al V^\beta$. One has the following transport equation of $\mu$ along $\mathring L$, whose proof is exact similar to \cite[Lemma 4.3]{Ding4}.

\begin{lemma}\label{mu}
$\mu$ satisfies
 \begin{equation}\label{lmu}
 \begin{split}
 \mathring L\mu=-\f12\mu G_{\mathring L\mathring L}^\gamma\mathring L\varphi_\gamma-\mu G_{\tilde T\mathring L}^\gamma\mathring L\varphi_\gamma+\f12G_{\mathring L\mathring L}^\gamma T\varphi_\gamma.
 \end{split}
 \end{equation}
\end{lemma}

\begin{remark}
In the analysis later, special attention is needed for the terms containing $T\varphi_\gamma$ since  $T\varphi_\gamma$ has the worse
smallness and slower time decay rate than the ones of $\bar Z\varphi_\g$, where $\bar Z\in\{\mathring L, R\}$.
\end{remark}

As in \cite{Ding3}, the following explicit expressions hold:
\begin{align}
&\zeta_X=-\f12\Big\{G_{\tilde T\mathring L}^\gamma\slashed d_X\varphi_\gamma+G_{\tilde T\tilde T}^\gamma \slashed d_X\varphi_\gamma-G_{X\tilde T}^\gamma{\mathring L}\varphi_\gamma+G^\gamma_{X\mathring L}{\tilde T}\varphi_\gamma\Big\},\label{zeta}\\
&\sigma_{XX}=-G_{X\mathring L}^\gamma\slashed d_X\varphi_\gamma-G_{X\tilde T}^\gamma\slashed d_X\varphi_\gamma+\f12G_{XX}^\gamma{\mathring L}\varphi_\gamma+\f12G_{XX}^\g{\tilde T}\varphi_\gamma-\chi_{XX}.\label{theta}
\end{align}

Note that the quantity ``deformation tensor" defined in \eqref{dt}
 will occur in the forthcoming energy estimates. It is necessary  to compute the components
 of $\leftidx{^{(V)}}\pi$ in the null frame $\{\mathring L,\underline{\mathring L},X\}$.
 This will be given in Appendix A. In addition, as in \cite{Ding3},
it follows from \eqref{lmu} and \eqref{Lpi}-\eqref{Rpi} that
the vector fields $\mathring L$ and $T$ appear frequently. Based on this, we need the equations for $\mathring L^i$ and $\check L^i$ under the actions of the derivatives of null frame $\{T,\mathring L, X\}$ as well as the connection coefficients of the frames.

\begin{lemma}\label{4.2}
	It holds that
	\begin{align}
	&\mathring L\mathring L^i=\f12G^\gamma_{\mathring L\mathring L}\mathring L\varphi_{\gamma}\tilde T^i
-\big\{G^\gamma_{X\mathring L}\mathring L\varphi_{\gamma}-\f12G^\gamma_{\mathring L\mathring L}\slashed d_X\varphi_{\gamma}\big\}\slashed d^Xx^i,\label{LL}\\
	&\mathring L\big(\varrho\check L^i\big)=\varrho\mathring L\mathring L^i,\label{LeL}\\
	&\slashed d_X\mathring L^i=\textrm{tr}\chi{\slashed d_X}x^i-\{G^\gamma_{\mathring L\tilde T}\slashed d_X\varphi_{\gamma}
+\f12G^\gamma_{\tilde T\tilde T}\slashed d_X\varphi_{\gamma}\}\tilde T^i-\f12G_{XX}^\gamma\mathring L\varphi_{\gamma}\slashed d^Xx^i,\label{dL}\\
	&\slashed d_X\check{L}^i=\textrm{tr}\check\chi{\slashed d_X}x^i-\{G^\gamma_{\mathring L\tilde T}{\slashed d_X}\varphi_{\gamma}
+\f12G^\gamma_{\tilde T\tilde T}{\slashed d_X}\varphi_{\gamma}\}\tilde T^i-\f12G_{XX}^\gamma\mathring L\varphi_{\gamma}\slashed d^Xx^i,\label{deL}\\
	&T\mathring{L}^i=\{\slashed d_X\mu-\f12\mu G^\gamma_{\tilde T\tilde T}\slashed d_X\varphi_{\gamma}
-G^\gamma_{X\mathring L}T\varphi_{\gamma}\}\slashed d^Xx^i+\big\{\f12\mu G^\gamma_{\mathring L\mathring L}\mathring L\varphi_\gamma\no\\
	&\qquad\quad+\mu G^\gamma_{\tilde T\mathring L}\mathring L\varphi_\gamma+\f12\mu G^\gamma_{\tilde T\tilde T}\mathring L\varphi_\gamma\big\}\mathring L^i+\f12G^\gamma_{\mathring L\mathring L}T\varphi_\gamma\tilde T^i\label{TL}
	\end{align}
and
\begin{equation}\label{cdf}
\begin{split}
&\mathscr D_{\mathring L}\mathring L=(\mu^{-1}\mathring L\mu)\mathring L,\quad\mathscr D_{T}\mathring L=-\mathring L\mu\mathring L+\eta^XX,\quad\mathscr D_X\mathring L=-\zeta_X\mathring L+\textrm{tr}{\chi}X,\\
&\mathscr D_{\mathring L}T=-\mathring L\mu\mathring L-\mu\zeta^XX,\quad\mathscr D_TT=\mu\mathring L\mu\mathring L
+(\mu^{-1}T\mu+\mathring L\mu)T-\mu(\slashed d^X\mu) X,\\
&\mathscr D_XT=\mu\zeta_X\mathring L+\mu^{-1}\eta_XT+\mu\textrm{tr}{\sigma}X,\\
&\mathscr D_XX=\slashed{\nabla}_XX+(\sigma_{XX}+\chi_{XX})\mathring L+\mu^{-1}\chi_{XX}T.
\end{split}
\end{equation}
In addition, the covariant derivatives of the frame $\{\mathring{\underline L}, \mathring L, X\}$ are
\begin{equation}\label{LuL}
\begin{split}
&\mathscr D_{\mathring {\underline L}}{\mathring L}=-\mathring L\mu\mathring L+2\eta^XX,\quad\mathscr D_{\mathring L}\mathring{\underline L}=-2\mu\zeta^XX,\\
&\mathscr D_{\mathring{\underline L}}\mathring{\underline L}=(\mu^{-1}\mathring{\underline L}\mu+\mathring
 L\mu)\mathring{\underline L}-(2\mu \slashed d^X\mu) X.
\end{split}
\end{equation}
\end{lemma}
\begin{proof}
The proofs are similar to these in Lemma 3.4-Lemma 3.6 in \cite{Ding3} so are omitted.
\end{proof}

As in \cite{Ding3}, to estimate $\varphi$, one needs to
derive the equations for $\varphi_\gamma$ under the action of the covariant wave operator $\Box_g=g^{\al\beta}\mathscr{D}^2_{\al\beta}=g^{\al\beta}\p_{\al\beta}^2-g^{\al\beta}\Gamma_{\al\beta}^\lambda\p_\lambda$, where $\Gamma_{\al\beta}^\gamma$ are Christoffel symbols with
\begin{equation*}
\begin{split}
\Gamma_{\al\beta}^\gamma=\f12g^{\gamma\kappa}\big(G_{\kappa\beta}^\nu\p_\al\varphi_\nu
+G_{\al\kappa}^\nu\p_\beta\varphi_\nu-G_{\al\beta}^\nu\p_\kappa\varphi_\nu\big).
\end{split}
\end{equation*}

Note that
\begin{equation}\label{boxg}
\begin{split}
\Box_g\varphi_\gamma=&\p_{\varphi_\nu}g^{\al\beta}\big(-\p_\nu\varphi_\gamma\p_\beta\varphi_\al+\p_\nu\varphi_\al\p_\beta\varphi_\gamma\big)
+\f12(g^{\al\beta}\p_{\varphi_\nu}g_{\al\beta})g^{\lambda\kappa}\p_\kappa\varphi_\nu\p_\lambda\varphi_\gamma.
\end{split}
\end{equation}
In terms of
\begin{align}
&g^{\al\beta}=-\mathring L^\al\mathring L^\beta-\tilde T^\al\mathring L^\beta-\mathring L^\al\tilde T^\beta+(\slashed d_Xx^\al)(\slashed d^Xx^\beta),\label{gab}\\
&\p_\al=\delta_\al^0\mathring L-\mu^{-1}\mathring L_\al T+g_{\al i}(\slashed d^Xx^i)X,\label{pal}
\end{align}
then in the frame $\{T, \mathring L, X\}$, \eqref{boxg} can be rewritten as
\begin{equation}\label{ge}
\begin{split}
\mu\Box_g \varphi_\gamma=&\f12(g^{\al\beta}G_{\al\beta}^
\nu)(\mu\slashed d^X\varphi_\nu)\slashed d_X\varphi_\gamma+f_1(\varphi,\mathring L^1, \mathring L^2)\varphi^{k-1}\left(
\begin{array}{ccc}
\mathring L\varphi\\
(\slashed d_Xx)\slashed d^X\varphi\\
\end{array}
\right)T\varphi_\gamma\\
&+f_2(\varphi, \mathring L^1, \mathring L^2)\varphi^{k-1}\left(
\begin{array}{ccc}
T\varphi\\
\mu\mathring L\varphi\\
\mu(\slashed d_Xx)\slashed d^X\varphi\\
\end{array}
\right)
\left(
\begin{array}{ccc}
\mathring L\varphi_\gamma\\
(\slashed d_Xx)\slashed d^X\varphi_\gamma\\
\end{array}
\right),
\end{split}
\end{equation}
where $f_1$ and $f_2$ are the generic smooth functions with respect to their arguments, and $\left(
\begin{array}{ccc}
A_1\\
\vdots\\
A_n\\
\end{array}
\right)
\left(
\begin{array}{ccc}
B_1\\
\vdots\\
B_m\\
\end{array}
\right)$ stands for the terms which conclude  $A_iB_j$ $(1\leq i\leq n,1\leq j\leq m)$.
It can be checked that in the right hand side of \eqref{ge}, the worst factor $T\varphi$
is always accompanied by the ``good" multipliers
such as $\mathring L\varphi$ and $\slashed d\varphi$.

\subsection{The bootstrap assumptions and some related estimates}\label{BA}

As in \cite{Ding3}, the existence and properties of local in time solution to \eqref{quasi} with \eqref{id}, \eqref{Y-0} and
\eqref{Y-0-a}-\eqref{g00} can be established easily in the following Theorem,
whose proof is given in Appendix B.

\begin{theorem}\label{Th2.1}
Under the assumptions \eqref{Y-0} and
\eqref{Y-0-a}-\eqref{g00}, when $\delta>0$ is small, equation \eqref{quasi} with \eqref{id}
admits a local smooth solution $\phi\in C^\infty([1, t_0]\times\mathbb R^2)$. Moreover,
for $m,n,p\in\mathbb N_0$, $\kappa\in\mathbb N_0^3$, it holds that
\begin{enumerate}[(i)]
	\item
	\begin{align}
	&|L^m\p^\kappa\O^p\phi(t_0,x)|\lesssim\delta^{2-|\kappa|-\varepsilon_0},\quad r \in[1-2\delta, 1+2\delta],\label{local1-2}\\
	&|\underline L^n\p^\kappa\O^p\phi(t_0,x)|\lesssim\delta^{2-|\kappa|-\varepsilon_0},\quad r\in [1-3\delta, 1+\delta].\label{local2-2}
	\end{align}
	\item
	\begin{equation}\label{local3-2}
	|\p^\kappa\O^p\phi(t_0,x)|\lesssim
	\delta^{3-|\kappa|-\varepsilon_0},
	\quad r\in [1-3\delta, 1+\delta].
	\end{equation}
	\item
	\begin{equation}\label{local3-3}
	|\underline L^mL^n\O^p\phi(t_0,x)|\lesssim\delta^{2-\varepsilon_0},\quad r\in[1-2\delta,1+\delta].
	\end{equation}
\end{enumerate}
\end{theorem}

It should be pointed out that the special null condition structures of 2D Chaplygin gases are extensively
applied in \cite{Ding3} and simplify many identities compared to the general null conditions in the paper.
Therefore, although the remainder of this subsection is analogous to Section 3 in \cite{Ding3},
the differential equation of $\mu$ and  the last term on the right hand side of
\eqref{lmu}  require  to be analyzed more delicately.

To show the global existence of solution $\phi$ to \eqref{quasi} with \eqref{id} near $C_0$,
we will utilize the bootstrap argument and make the following bootstrap assumptions in $D^{s,u}$:
\begin{equation*}
\begin{split}
&\delta^l\|\mathring LZ^\jmath\varphi_\gamma\|_{L^\infty(\Sigma_s^{u})}+\delta^l\|\slashed dZ^\jmath\varphi_\gamma\|_{L^\infty(\Sigma_s^{u})}
\leq  M\delta^{1-\varepsilon_0} s^{-3/2},\\
&\delta^l\|\mathring{\underline L}Z^\jmath\varphi_\gamma\|_{L^\infty(\Sigma_s^{u})}
+\delta^{l-1}\|Z^\jmath\varphi_\gamma\|_{L^\infty(\Sigma_s^{u})}\leq  M\delta^{-\varepsilon_0} s^{-1/2},\\
&\dl\|\slashed\nabla^2\varphi_\gamma\|_{L^\infty(\Sigma_s^{u})}+
\|\slashed\nabla^2\phi\|_{L^\infty(\Sigma_s^{u})}\leq  M\delta^{2-\varepsilon_0} s^{-5/2},\\
&\delta^l\|\mathring LZ^\jmath\phi\|_{L^\infty(\Sigma_s^{u})}+\delta^l\|\slashed dZ^\jmath\phi\|_{L^\infty(\Sigma_s^{u})}\leq  M\delta^{2-\varepsilon_0} s^{-3/2},\\
&\delta^l\|Z^\jmath\phi\|_{L^\infty(\Sigma_s^{u})}\leq M\delta^{2-\varepsilon_0}s^{-1/2},\\
\end{split}\tag{$\star$}
\end{equation*}
where $|\jmath|\leq N$, $N$ is a fixed large positive integer, $M$ is some positive constant to be suitably chosen
(at least double bounds of the corresponding quantities on time $t_0$ in Theorem \ref{Th2.1}),
$Z\in\{\varrho\mathring L,T,R\}$, and $l$ is the number of $T$ included in $Z^\jmath$.

One can now derive a rough estimate of $\mu$ under the assumptions ($\star$) as follows. By
$$1=g_{ij}\tilde T^i\tilde T^j=\big(1+O(M^k\delta^{k(1-\varepsilon_0)} s^{-k/2})\big)\ds\sum_{i=1}^2|\tilde T^i|^2,$$
then
\begin{equation}\label{L}
|\tilde T^i|, |\mathring L^i|\leq 1+O( M^k\delta^{k(1-\varepsilon_0)} s^{-k/2}).
\end{equation}
Moreover, by $|\slashed dx^i|^2=\slashed g^{ab}\slashed d_ax^i\slashed d_bx^i=g^{ii}+(g^{0i})^2-(\tilde T^i)^2$,
it follows from $(\star)$ and \eqref{L} that
\begin{equation}\label{dx}
|\slashed dx^i|\lesssim 1.
\end{equation}
In addition, $|\mathring L(\varrho\check L^i)|\lesssim M^k\delta^{k(1-\varepsilon_0)}s^{-k/2}$ holds
by \eqref{LeL} and $(\star)$, which further gives that through integrating $\mathring L(\varrho\check L^i)$
along integral curves of $\mathring L$ and
using $\check T^i=-g^{0i}-\check L^i$,
 \begin{equation}\label{chL}
 |\check L^i|+|\check T^i|\lesssim M^k\delta^{k(1-\varepsilon_0)}s^{-1}\ln s.
 \end{equation}
Due to $\ds g_{ij}(\check T^i-\f{x^i}{\varrho})(\check T^j-\f{x^j}{\varrho})=1$, then
 \[
 (g_{ij}\o^i\o^j)\f{r^2}{\varrho^2}-(2g_{ij}\check T^i\o^j)\f r\varrho+g_{ij}\check T^i\check T^j-1=0.
 \]
 Thus
 \begin{equation}\label{rrho}
 \check\varrho:=\f{r}{\varrho}-1=\f{1-g_{ij}\o^i\o^j-g_{ij}\check T^i\check T^j+2g_{ij}\check T^i\o^j}{\sqrt{g_{ij}\o^i\o^j-(g_{ij}\o^i\o^j)(g_{ab}\check T^a\check T^b)+(g_{ij}\check T^i\o^j)^2}
 +g_{ij}\o^i\o^j-g_{ij}\check T^i\o^j}.
 \end{equation}
This yields
 \begin{equation}\label{cr}
 |\check\varrho|\lesssim M^k\delta^{k(1-\varepsilon_0)}s^{-1}\ln s,
 \end{equation}
 and hence  by the definition of $\upsilon$ in \eqref{omega},
 \begin{equation}\label{Ri}
 |\upsilon|\lesssim M^k\delta^{k(1-\varepsilon_0)}\ln s.
 \end{equation}

To estimate $\mu$, one needs to handle the last term in \eqref{lmu}.
It follows from \eqref{pal} and \eqref{ggg} that
\begin{equation}\label{GT}
\begin{split}
G_{\mathring L\mathring L}^\gamma T\varphi_\gamma=&G_{\mathring L\mathring L}^0T^i(\mathring L\varphi_i)
-G^\gamma_{\mathring L\mathring L}\mathring L_\gamma\tilde T^i T\varphi_i+\mu g_{\gamma i}G_{\mathring L\mathring L}^\gamma (\slashed d^Xx^i)\tilde T^j\slashed d_X\varphi_j\\
=&G_{\mathring L\mathring L}^0T^i(\mathring L\varphi_i)
+\underline{(\p_{\varphi_\gamma}g^{\al\beta})\mathring L_\al\mathring L_\beta\mathring L_\gamma}\tilde T^i T\varphi_i
+\mu g_{\gamma i}G_{\mathring L\mathring L}^\gamma (\slashed d^Xx^i)\tilde T^j\slashed d_X\varphi_j.
\end{split}
\end{equation}
The null condition \eqref{null} means that
the underline factor in \eqref{GT} can be written as
\begin{equation}\label{gLLL}
\begin{split}
&-G_{\mathring L\mathring L}^\g\mathring L_\g=(\p_{\varphi_\gamma}g^{\al\beta})\mathring L_\al\mathring L_\beta\mathring L_\gamma=kg^{\al\beta,\gamma\gamma_2\cdots\gamma_k}\varphi_{\gamma_2}\cdots\varphi_{\gamma_k}\mathring L_\al\mathring L_\beta\mathring L_\gamma+(\p_{\varphi_\gamma}h^{\al\beta})\mathring L_\al\mathring L_\beta\mathring L_\gamma\\
=&kg^{\al\beta,\gamma\gamma_2\cdots\gamma_k}\varphi_{\gamma_2}\cdots\varphi_{\gamma_k}m_{\al\al'}m_{\beta\beta'}m_{\gamma\gamma'}\f{\tilde x^{\al'}}\varrho\f{\tilde x^{\beta'}}\varrho\f{\tilde x^{\gamma'}}\varrho+f(\varphi, \check L^1, \check L^2,\f{x}\varrho)\left(
\begin{array}{ccc}
\varphi^k\\
\check L^{1}\varphi^{k-1}\\
\check L^{2}\varphi^{k-1}
\end{array}
\right)\\
=&kg^{\al\beta,\gamma\gamma_2\cdots\gamma_k}\varphi_{\gamma_2}\cdots\varphi_{\gamma_k}\o_\al\o_\beta\o_\gamma+f(\varphi, \check L^1, \check L^2,\f{x}\varrho,\check\varrho)\left(
\begin{array}{ccc}
\varphi^k\\
\check L^{1}\varphi^{k-1}\\
\check L^{2}\varphi^{k-1}\\
\check\varrho\varphi^{k-1}
\end{array}
\right)\\
=&f(\varphi, \check L^1,\check L^2,\f{x}\varrho,\check\varrho)\left(
\begin{array}{ccc}
\varphi^k\\
\check L^{1}\varphi^{k-1}\\
\check L^{2}\varphi^{k-1}\\
\check\varrho\varphi^{k-1}\\
(\mathring L\phi)\varphi^{k-2}\\
(\slashed d^Xx)(\slashed d_X\phi)\varphi^{k-2}
\end{array}
\right),
\end{split}
\end{equation}
where $\f{x^i}\varrho=\check\varrho\o_i+\o_i$, $\varphi_{\gamma_j}=\delta_{\gamma_j}^0\mathring L\phi-\mathring L_{\gamma_j}\tilde T^a\varphi_a+g_{\gamma_ja}\slashed d^Xx^a\slashed d_X\phi$, \eqref{pal} and \eqref{null} are used,
$(\tilde x^0, \tilde x^1, \tilde x^2)=(\varrho, x^1, x^2)$
and $(\o_0, \o_1, \o_2)=(-1, \o^1, \o^2)$. Therefore,
$|(\p_{\varphi_\gamma}g^{\al\beta})\mathring L_\al\mathring L_\beta\mathring L_\gamma|\lesssim M^k\delta^{k(1-\varepsilon_0)}s^{-k/2}$
holds by \eqref{L}-\eqref{chL}, \eqref{cr} and $(\star)$.
Subsequently, it follows from the expression \eqref{GT} that
\begin{equation}\label{GTe}
|G_{\mathring L\mathring L}^\gamma T\varphi_\gamma|\lesssim \mu M^k\delta^{k(1-\varepsilon_0)}s^{-(k+2)/2}+M^{k+1}\delta^{k-(k+1)\varepsilon_0}s^{-(k+1)/2}.
\end{equation}
Together with $(\star)$ and \eqref{lmu},  it is easy to
obtain $|\mathring L\mu|\lesssim M^k\delta^{k(1-\varepsilon_0)}s^{-(k+2)/2}\mu+M^{k+1}\delta^{k-(k+1)\varepsilon_0}s^{-(k+1)/2}$. For $\delta>0$  suitably small, by
integrating $\mathring L\mu$ along integral curves of $\mathring L$ and noting $\ds\mu=\f{1}{\sqrt{(g^{0i}\o_i)^2+g^{ij}\o_i\o_j}}$
$=1+O(\delta^{k(1-\varepsilon_0)})$ on $\Sigma_{t_0}$,
one directly has that due to $\ve_0<\ve_k^*$,
\begin{equation}\label{phimu}
\mu=1+O( M^{k+1}\delta^{k-(k+1)\varepsilon_0})
\end{equation}
which has a positive lower bound as long as $(\star)$ holds and $\dl$ is small enough.

To improve the assumptions ($\star$), one can rewrite equation \eqref{ge} in the new frame
$\{\mathring{\underline L}, \mathring L, X\}$ as
\begin{equation}\label{fequation}
\mathring L\mathring{\underline L}\varphi_\gamma+\f{1}{2\varrho}\mathring{\underline L}\varphi_\gamma
=\mu\slashed\triangle\varphi_\gamma+H_\gamma,
\end{equation}
here one has used the fact $\mu\Box_g\varphi_\g=-\mathscr D_{\mathring L\mathring {\underline L}}^2\varphi_\g
+\mu\slashed g^{XX}\mathscr D_{X}^2\varphi_\g=-\mathring L\mathring{\underline L}\varphi_\g-2\mu\zeta^X(\slashed d_X\varphi_\g)+\mu\slashed\triangle\varphi_\g-\mu(\textrm{tr}\sigma+\textrm{tr}\chi)\mathring L\varphi_\g
-\textrm{tr}\chi T\varphi_\g$
by \eqref{cdf} and \eqref{LuL}. In addition, by \eqref{zeta}-\eqref{theta},
\begin{equation}\label{H}
\begin{split}
H_\gamma=&-(\textrm{tr}\check\chi)T\varphi_\gamma+\f{1}{2\varrho}\mu\mathring L\varphi_\gamma+f_1(\varphi,\mathring L^1, \mathring L^2)\varphi^{k-1}\left(
\begin{array}{ccc}
\mathring L\varphi\\
(\slashed d_Xx)\slashed d^X\varphi\\
\end{array}
\right)T\varphi_\gamma\\
&+f_2(\varphi, \mathring L^1, \mathring L^2, \slashed dx, \slashed g)\varphi^{k-1}\left(
\begin{array}{ccc}
T\varphi\\
\mu\mathring L\varphi\\
\mu\slashed d\varphi\\
\end{array}
\right)
\left(
\begin{array}{ccc}
\mathring L\varphi_\gamma\\
\slashed d\varphi_\gamma\\
\end{array}
\right).
\end{split}
\end{equation}

As in \cite{Ding4, Ding3}, it is noted from the expression of $H_\g$ that if there are the terms including the factor $T\varphi_\al$
which admits the ``bad" smallness or slow time decay rate,
then there always appear another product factor equipped with the ``good"
smallness and fast time decay rate, e.g.
$\textrm{tr}\check\chi, \mathring L\varphi$ and so on.

Unless stated otherwise, from now on, in this subsection, the pointwise estimates for the related quantities
are all carried out inside domain $D^{s,u}$.

It follows from the expression \eqref{fequation} that the elaborate
estimate of $\mathring{\underline L}\varphi_\gamma$ can be achieved by integrating  \eqref{fequation}
along integral curves of $\mathring L$. To this end, one should estimate $\check{\chi}$ and
other terms in $H_{\g}$.

To estimate $\check\chi$, one needs the structure equation for $\chi$.
Analogously to \cite[Lemma 5.2]{Ding4}, one has

\begin{lemma}\label{LTchi}
	The second fundamental form $\chi$ and its ``error" form $\check\chi$, defined in \eqref{chith} and \eqref{errorv}
respectively, satisfy the following structure equations:
\begin{equation}\label{Lchi}
\begin{split}
\mathring L\chi_{XX}=&-G_{X\mathring L}^\gamma(\slashed d_X\mathring L\varphi_\gamma)+\f12G_{\mathring L\mathring L}^\gamma\slashed\nabla_{X}^2\varphi_\gamma+\f12G_{XX}^\gamma(\mathring L^2\varphi_{\gamma})
-\big\{\f12G_{\mathring L\mathring L}^\gamma+G_{\tilde T\mathring L}^\gamma\}\mathring L\varphi_\gamma\chi_{XX}\\
&+(G_{X\mathring L}^\gamma\slashed d_X\varphi_\gamma)\textrm{tr}\chi+(\textrm{tr}\chi)\chi_{XX}+f(\varphi, \slashed dx,\mathring L^1, \mathring L^2)\varphi^{k-2}\left(
\begin{array}{ccc}
\mathring L\varphi\\
\slashed d\varphi\\
\end{array}
\right)^2,
\end{split}
\end{equation}
\begin{equation}\label{Tchi}
\begin{split}
\slashed{\mathcal L}_T\chi_{XX}=&\slashed\nabla_{X}^2\mu-\f12\big\{\mu G_{\tilde T\tilde T}^\gamma\slashed\nabla_{X}^2\varphi_\gamma+2G_{X\mathring L}^\gamma\slashed d_XT\varphi_\gamma-G_{XX}^\gamma\mathring LT\varphi_\gamma\big\}-\mu(\textrm{tr}\chi)^2\slashed g_{XX}\\
&-\{G_{\tilde T\mathring L}^\gamma\slashed d_X\varphi_\gamma+G_{\tilde T\tilde T}^\gamma\slashed d_X\varphi_\gamma-G_{X\tilde T}^\gamma\mathring L\varphi_\gamma\}d_X\mu+\Upsilon\cdot\chi_{XX}\\
&+f_1(\varphi, \slashed dx, \mathring L^1, \mathring L^2)\varphi^{k-2}
\left(
\begin{array}{ccc}
\varphi\slashed d\mathring L^1\\
\varphi\slashed d\mathring L^2\\
\mathring L\varphi\\
\slashed d\varphi\\
\end{array}
\right)\left(
\begin{array}{ccc}
\mu\mathring L\varphi\\
\mu\slashed d\varphi\\
T{\varphi}\\
\end{array}
\right)
+f_2(\varphi, \mathring L^1, \mathring L^2)\varphi^{k-1}\slashed\nabla^2x\left(
\begin{array}{ccc}
\mu\mathring L\varphi\\
T{\varphi}\\
\end{array}
\right).
\end{split}
\end{equation}

And hence,
\begin{equation}\label{Lchi'}
\begin{split}
\mathring L\check\chi_{XX}=&-G_{X\mathring L}^\gamma(\slashed d_X\mathring L\varphi_\gamma)+\f12G_{\mathring L\mathring L}^\gamma\slashed\nabla_{X}^2\varphi_\gamma+\f12G_{XX}^\gamma(\mathring L^2\varphi_{\gamma})-\big\{\f12G_{\mathring L\mathring L}^\gamma+G_{\tilde T\mathring L}^\gamma\}\mathring L\varphi_\gamma\check\chi_{XX}\\
&-\f{\slashed g_{XX}}\varrho\big\{\f12G_{\mathring L\mathring L}^\gamma+G_{\tilde T\mathring L}^\gamma\}\mathring L\varphi_\gamma+(G_{X\mathring L}^\gamma\slashed d_X\varphi_\gamma)\textrm{tr}\check\chi+\f{1}{\varrho}(G_{X\mathring L}^\gamma\slashed d_X\varphi_\gamma)+(\textrm{tr}\check\chi)\check\chi_{XX}\\
&+f(\varphi, \slashed dx,\mathring L^1,\mathring L^2)\varphi^{k-2}\left(
\begin{array}{ccc}
\mathring L\varphi\\
\slashed d\varphi\\
\end{array}
\right)^2,
\end{split}
\end{equation}
\begin{equation}\label{Tchi'}
\begin{split}
\slashed{\mathcal L}_T\check{\chi}_{XX}=&\slashed\nabla_{X}^2\mu-\f12\big\{\mu G_{\tilde T\tilde T}^\gamma\slashed\nabla_{X}^2\varphi_\gamma+2G_{X\mathring L}^\gamma\slashed d_XT\varphi_\gamma-G_{XX}^\gamma\mathring LT\varphi_\gamma\big\}-\mu(\textrm{tr}\check\chi)^2\slashed g_{XX}\\
&-\{G_{\tilde T\mathring L}^\gamma\slashed d_X\varphi_\gamma+G_{\tilde T\tilde T}^\gamma\slashed d_X\varphi_\gamma-G_{X\tilde T}^\gamma\mathring L\varphi_\gamma\}d_X\mu+\Upsilon\cdot\check\chi_{XX}+\f{1}{2\varrho}\Theta\cdot\slashed g_{XX}\\
&+\f{\mu-1}{\varrho^2}\slashed g_{XX}+f_1(\varphi, \slashed dx, \mathring L^1,\mathring L^2)\varphi^{k-2}
\left(
\begin{array}{ccc}
\varphi\slashed d\mathring L^1\\
\varphi\slashed d\mathring L^2\\
\mathring L\varphi\\
\slashed d\varphi\\
\end{array}
\right)\left(
\begin{array}{ccc}
\mu\mathring L\varphi\\
\mu\slashed d\varphi\\
T{\varphi}\\
\end{array}
\right)\\
&
+f_2(\varphi, \mathring L^1,\mathring L^2)\varphi^{k-1}\slashed\nabla^2x\left(
\begin{array}{ccc}
\mu\mathring L\varphi\\
T{\varphi}\\
\end{array}
\right),
\end{split}
\end{equation}
where
\begin{align}
\Upsilon=&\f12\mu G_{\mathring L\mathring L}^\gamma\mathring L\varphi_\gamma+\mu G_{\tilde T\mathring L}^\gamma\mathring L\varphi_\gamma-\f12 G_{\mathring L\mathring L}^\gamma T\varphi_\gamma-\f32\mu G_{X\mathring L}^\gamma\slashed d^X\varphi_\gamma-\f12\mu G_{X\tilde T}^\gamma\slashed d^X\varphi_\gamma\no\\
&+\f12\slashed g^{XX}G_{XX}^\gamma T\varphi_\gamma-\f12G_{\mathring L\tilde T}^\gamma T\varphi_\gamma+\f12\mu\slashed g^{XX}G_{XX}^\gamma \mathring L\varphi_\gamma,\label{Upsion}\\
\Theta=&\mu G_{\mathring L\mathring L}^\gamma\mathring L\varphi_\gamma+2\mu G_{\tilde T\mathring L}^\gamma\mathring L\varphi_\gamma-G_{\mathring L\mathring L}^\gamma T\varphi_\gamma+\mu G_{X\mathring L}^\gamma\slashed d^X\varphi_\gamma+3G_{X\tilde T}^\gamma\slashed d^X\varphi_\gamma\no\\
&-\mu\slashed g^{XX}G_{XX}^\gamma\mathring L\varphi_\gamma-\slashed g^{XX}G_{XX}^\gamma T\varphi_\gamma-G_{\mathring L\tilde T}^\gamma T\varphi_\gamma,\label{Theta}
\end{align}
and $f, f_1, f_2$ are the generic smooth functions of their arguments.
\end{lemma}

Based on \eqref{Lchi'}, $\check\chi$ can be estimated by integrating along integral curves of $\mathring L$.

\begin{proposition}\label{chi'}
	Under the assumptions $(\star)$, when $\delta>0$ is small, it holds that
	\begin{equation}\label{echi'}
	|\check{\chi}|=|\textrm{tr}\check\chi|\lesssim M^k\delta^{k(1-\varepsilon_0)} s^{-2}\ln s.
	\end{equation}
	Meanwhile,
	\begin{equation}\label{chi}
	|\chi|=\f{1}{\varrho}+O(M^k\delta^{k(1-\varepsilon_0)} s^{-2}\ln s).
	\end{equation}
\end{proposition}

\begin{proof}
	By $\textrm{tr}\check\chi=\slashed g^{XX}\check{\chi}_{XX}$, then
	\begin{equation}\label{Ltrhchi}
	\mathring L\big(\textrm{tr}\check\chi\big)=-2(\textrm{tr}\check\chi)^2-\f{2}{\varrho}\textrm{tr}\check\chi
	+\slashed g^{XX}\mathring L\check\chi_{XX}.
	\end{equation}
	Substituting \eqref{Lchi'} into \eqref{Ltrhchi}, and using assumptions $(\star)$, \eqref{L} and \eqref{dx} to estimate
	the right hand side of \eqref{Ltrhchi} except $\check\chi$ itself, one can get
	\[
	|\mathring L\big(\varrho^2\textrm{tr}\check\chi\big)|\lesssim
	M^k\delta^{k(1-\varepsilon_0)}s^{-k/2}+M^k\delta^{k(1-\varepsilon_0)}s^{-(k+2)/2}|\varrho^2\textrm{tr}\check\chi|
	+\varrho^{-2}|\varrho^2\textrm{tr}\check\chi|^2.
	\]
	Thus, for small $\delta>0$, one has
	\[
	|\check{\chi}|\lesssim M^k\delta^{k(1-\varepsilon_0)} s^{-2}\ln s.
	\]
\end{proof}

Additionally, it is noted that the operator $R$ is just only the scaling operator $r\slashed\nabla$ under
the assumptions $(\star)$ and the estimate \eqref{echi'}, which are similar to the ones in \cite[Lemma 12.22]{J}.

\begin{corollary}\label{12form}
 Under the assumptions $(\star)$, when $\delta>0$ is small,
\begin{enumerate}
	\item
  if $\xi$ is a 1-form on $S_{s, u}$, then
  \begin{align}
  &(\xi_a R^a)^2\sim r^2|\xi|^2\label{1-f},\\
  &|\slashed{\mathcal L}_{R}\xi|^2\sim r^2|\slashed\nabla\xi|^2+O(M^k\delta^{k(1-\varepsilon_0)}s^{-1}\ln s)|\xi|^2;\label{1f}
  \end{align}
\item
  if $\xi$ is a 2-form on $S_{s, u}$, then
  \begin{equation}\label{2-f}
  \begin{split}
  |\slashed{\mathcal L}_{R}\xi|^2\sim r^2|\slashed\nabla\xi|^2+O(M^k\delta^{k(1-\varepsilon_0)}s^{-1}\ln s)|\xi|^2
    \end{split}
  \end{equation}
  and
  \begin{equation}\label{2f}
  |\slashed\nabla^2\xi|\lesssim \varrho^{-2}|\slashed{\mathcal L}_R^{\leq 2}\xi|.
  \end{equation}
\end{enumerate}
\end{corollary}

%In the above, the equation of $\varphi$ in the frame $\{\mathring L,\mathring {\underline L}\}$ has been obtained in \eqref{fequation}.
To close the bootstrap
assumptions $(\star)$, one can improve the estimate of $\varphi$ by \eqref{fequation} to obtain the following proposition,
whose proofs are same as those in \cite[Proposition 6.3 and Corollary 6.3]{Ding4}.

\begin{proposition}\label{LTRh}
	Under the assumptions $(\star)$, for any operate $Z\in\{\varrho\mathring L, T, R\}$, when $\delta>0$ is small,
it holds that for $m\leq N-1$,
\begin{equation}\label{z}
\begin{split}
&|\slashed{\mathcal{L}}_{ Z}^{m;l,p}\check{\chi}|\lesssim M^k\delta^{k(1-\varepsilon_0)-l}s^{-2}\ln s,\quad|\slashed{\mathcal{L}}_{Z}^{m+1;l,p}\slashed dx^j|\lesssim \delta^{-l},\quad\\
&|Z^{m+1;l,p}\check L^j|+|\slashed{\mathcal{L}}_{Z}^{m;l,p}\leftidx{^{(R)}}{\slashed\pi}|+|\slashed{\mathcal{L}}_{Z}^{m;l,p}\leftidx{^{(R)}}{\slashed\pi}_{\mathring L}|+|Z^{m+1;l,p}\check\varrho|\lesssim M^k\delta^{k(1-\varepsilon_0)-l}s^{-1}\ln s,\\
& |\slashed{\mathcal{L}}_{Z}^{m;l,p}R|+|Z^{m+1;l,p}\upsilon|\lesssim M^k\delta^{k(1-\varepsilon_0)-l}\ln s,\quad|Z^{m+1;l,p}\mu|\lesssim M^{k+1}\delta^{k(1-\varepsilon_0)-\varepsilon_0-l},\\
& |\slashed{\mathcal{L}}_{Z}^{m;l,p}\leftidx{^{(T)}}{\slashed\pi}|\lesssim M^k\delta^{k(1-\varepsilon_0)-1-l}s^{-1}+s^{-l-1},\quad |\slashed{\mathcal{L}}_{Z}^{m;l,p}\leftidx{^{(T)}}{\slashed\pi}_{\mathring L}|\lesssim M^k\delta^{k(1-\varepsilon_0)-1-l}s^{-1},\\
&|\slashed{\mathcal{L}}_{Z}^{m;l,p}\leftidx{^{(R)}}{\slashed\pi}_{T}|\lesssim(M^k\delta^{k(1-\ve_0)-l}+M^{2k}\delta^{2k(1-\ve_0)-1-l})s^{-1}\ln s,
\end{split}
\end{equation}
and
\begin{equation}\label{ztl}
\begin{split}	
&\dl|Z^{m;l,p}\varphi_\gamma(s, u, \vartheta)|+|Z^{m;l,p}\phi(s, u, \vartheta)|\lesssim\delta^{2-l-\varepsilon_0}s^{-1/2},
\end{split}
\end{equation}
 where $(m;l,p)$ means the numbers of $Z$, $T$ and $\varrho\mathring L$ are $m$, $l$ and $p$ respectively.
\end{proposition}

Noted that as stated in the end of \cite[Section 5]{Ding3}, all the related constants
in Proposition \ref{LTRh} can be independent of $M$. Therefore,
from now on, we may apply these estimates independent of $M$.

\section{Energy estimates and some higher order $L^2$ estimates}\label{EE}

As in \cite{Ding4}, to close the bootstrap assumptions $(\star)$, one needs further refined estimates than those derived in Subsection \ref{BA}.  Note that $\varphi_\gamma$ satisfies the nonlinear equation \eqref{ge}, and each
derivative of $\varphi_\g$ also fulfills similar equation with the same metric. Thus, one may focus on
the energy estimates for any smooth
function $\Psi$ solving the following linear equation
\begin{equation}\label{gel}
\mu\Box_g\Psi=\Phi,
\end{equation}
for a given function $\Phi$, where $\Psi$ and its derivatives vanish on $C_0^{s}$.

As in \cite{Ding4}, we choose two multipliers $V_1\Psi:=\varrho^{2\iota}\mathring L\Psi$ $(0<\iota<\f12)$
and $V_2\Psi:=\mathring{\underline L}\Psi$, the energies $E_i[\Psi](s, u)$ and fluxes $F_i[\Psi](s, u)$
are defined after performing integration by part to $\mu(\Box_g\Psi)(V_i\Psi)$ over $D^{s,u}$ $(i=1,2)$,
\begin{align}
&E_1[\Psi](s, u):=\f12\ds\int_{\Sigma_s^{u}}\mu\varrho^{2\iota}\{(\mathring L\Psi)^2+|\slashed d\Psi|^2\},\label{E1}\\
&E_2[\Psi](s, u):=\f12\ds\int_{\Sigma_s^{u}}\{(\mathring{\underline L}\Psi)^2+\mu^2|\slashed d\Psi|^2\},\label{E2}\\
&F_1[\Psi](s, u):=\ds\int_{C_{u}^s}\varrho^{2\iota}(\mathring L\Psi)^2,\label{F1}\\
&F_2[\Psi](s,u):=\ds\int_{C_{u}^s}\mu |\slashed d\Psi|^2.\label{F2}
\end{align}
Thus,
\begin{equation}\label{EI}
E_i[\Psi](s, u)-E_i[\Psi](t_0, u)+F_i[\Psi](s, u)=-\ds\int_{D^{s, u}}\Phi\cdot V_i\Psi-\int_{D^{s, u}}\f12\mu Q_{\al\beta}[\Psi]\leftidx{^{(V_i)}}\pi^{\al\beta},
\end{equation}
where
$Q_{\al\beta}[\Psi]:=(\p_\al\Psi)(\p_\beta\Psi)-\f12 g_{\al\beta}g^{\nu\lambda}(\p_\nu\Psi)(\p_\lambda\Psi)$
and $\leftidx{^{(V_i)}}\pi$ is the deformation tensor with respect to the vector field $V_i$ as defined in \eqref{dt}. Since $\f12\mu Q_{\al\beta}[\Psi]\leftidx{^{(V_i)}}\pi^{\al\beta}$ $(i=1,2)$ have the same expressions as (7.7) and (7.8) in \cite{Ding4},
one then can derive that
 \begin{equation}\label{QV1}
\begin{split}
&\int_{D^{s, u}}-\f12 \mu Q^{\al\beta}[\Psi]\leftidx{^{(V_1)}}\pi_{\al\beta}\\
\lesssim&\delta^{k(1-\varepsilon_0)-\varepsilon_0}\int_{t_0}^s\tau^{-3/2} E_1[\Psi](\tau,u)d\tau+\delta^{-1}\int_0^uF_1[\Psi](s,u')du'+\delta\int_{t_0}^s\tau^{2\iota-2}E_2[\Psi](\tau,u)d\tau
\end{split}
\end{equation}
and
\begin{equation}\label{QV2}
\begin{split}
&\int_{D^{s, u}}-\f12 \mu Q^{\al\beta}[\Psi]\leftidx{^{(V_2)}}\pi_{\al\beta}\\
\lesssim&\delta^{-1}\int_0^uF_2[\Psi](s,u')du'+\delta^{-1}\int_0^uF_1[\Psi](s,u')du'
+\int_{t_0}^s\tau^{-2}E_2[\Psi](\tau,u)d\tau.
\end{split}
\end{equation}
And hence,
\begin{equation}\label{e}
\begin{split}
&\delta E_2[\Psi](s, u)+\delta F_2[\Psi](s, u)+E_1[\Psi](s, u)+F_1[\Psi](s, u)\\
\lesssim& \delta E_2[\Psi](t_0, u)+E_1[\Psi](t_0, u)+\delta\int_{D^{s, u}}|\Phi\cdot \mathring{\underline L}\Psi|+\int_{D^{s, u}}\varrho^{2\iota}|\Phi\cdot \mathring L\Psi|.
\end{split}
\end{equation}
Choose $\Psi=\Psi_\gamma^{m+1}=Z^{m+1}\varphi_\gamma$ and then
$\Phi=\Phi_\gamma^{m+1}=\mu\Box_g\Psi_\gamma^{m+1}$ ({$m\leq2N-6$})
so that \eqref{gel} holds. Note that
\begin{equation}\label{Psi}
\begin{split}
\Phi_\gamma^{m+1}&=\mu[\Box_g,Z]\Psi_\gamma^{m}+Z\big(\mu\Box_g\Psi_\gamma^{m}\big)-(Z\mu)\Box_g\Psi_\gamma^{m}\\
&=\mu\mathscr D^\al({\leftidx{^{(Z)}}C_{\gamma}^{m}}_{,\al})+(Z+\leftidx{^{(Z)}}\lambda)\Phi_\gamma^{m},
\end{split}
\end{equation}
where
\begin{equation}\label{ClP}
\begin{split} &{\leftidx{^{(Z)}}C_{\gamma}^{m}}_{,\al}=\big(\leftidx{^{(Z)}}\pi_{\nu\al}
-\f12g_{\nu\al}(g_{\kappa\lambda}\leftidx{^{(Z)}}\pi^{\kappa\lambda})\big)g^{\nu\beta}\p_\beta\Psi_\gamma^{m},\\
&\leftidx{^{(Z)}}\lambda=-\mu^{-1}\leftidx{^{(Z)}}\pi_{\mathring LT}+\f12\textrm{tr}\leftidx{^{(Z)}}{\slashed{\pi}}-\mu^{-1}Z\mu,\\
&\Psi_\gamma^0=\varphi_\gamma,\quad\Phi_\g^0=\mu\Box_g\varphi_\g,
\end{split}
\end{equation}
 and $\Phi_\g^0$ equals
the right hand side of \eqref{ge}.
Consequently, for $\Psi_\g^{m+1}=Z_{m+1}Z_{m}\cdots Z_{1}\varphi_\gamma$ with $Z_j\in\{\varrho\mathring L, T, R\}$,
then by \eqref{Psi}, the induction argument gives
\begin{equation}\label{Phik}
\begin{split}
\Phi_\gamma^{m+1}=&\sum_{j=1}^{m}\big(Z_{m+1}+\leftidx{^{(Z_{m+1})}}\lambda\big)\dots\big(Z_{m+2-j}
+\leftidx{^{(Z_{m+2-j})}}\lambda\big)\big(\mu\mathscr D^\al{\leftidx{^{(Z_{m+1-j})}}C_{\gamma}^{m-j}}_{,\al}\big)\\
&+\mu\mathscr D^\al{\leftidx{^{(Z_{m+1})}}C_\gamma^{m}}_{,\al}+\big(Z_{m+1}+\leftidx{^{(Z_{m+1})}}\lambda\big)\dots\big(Z_{1}
+\leftidx{^{(Z_1)}}\lambda\big)\Phi_\gamma^0\\
&=:J_1^{m+1}+J_2^{m+1},\qquad m\geq 1,\\
\Phi_\gamma^{1}=&\big(Z_{1}+\leftidx{^{(Z_1)}}\lambda\big)\Phi_\gamma^0+\mu\mathscr D^\al{\leftidx{^{(Z_{1})}}C_{\gamma}^{0}}_{,\al},
\end{split}
\end{equation}
where $J_1^{m+1}$ and $J_2^{m+1}$ stand for the first and the second line in the right hand side of \eqref{Phik}, respectively.

By \eqref{Lpi}-\eqref{Rpi} in Appendix A, one has
\begin{equation}\label{lamda}
\begin{split}
&\leftidx{^{(T)}}\lambda=\f12\textrm{tr}\leftidx{^{(T)}}{\slashed\pi},\quad
\leftidx{^{(\varrho\mathring L)}}\lambda=\varrho\textrm{tr}_{\slashed g}\check\chi+2,\quad
\leftidx{^{(R)}}\lambda=\f12\text{tr}\leftidx{^{(R)}}{\slashed\pi}.
\end{split}
\end{equation}
In addition, notice that in the null frame $\{\mathring{\underline L},\mathring L,X\}$, the
term $\mu\mathscr D^\al{\leftidx{^{(Z)}}C_{\gamma}^{m}}_{,\al}$ can be written as
\begin{equation}\label{muC}
\begin{split}
\mu\mathscr D^\al{\leftidx{^{(Z)}}C_{\gamma}^{m}}_{,\al}=&-\f12\mathring L\big({\leftidx{^{(Z)}}C_\gamma^{m}}_{,\mathring{\underline L}}\big)-\f12\mathring{\underline L}\big({\leftidx{^{(Z)}}C_\gamma^{m}}_{,\mathring L}\big)
+\slashed\nabla^X\big(\mu {\leftidx{^{(Z)}}{\slashed C}_{\g}^{m}}_{,X}\big)\\
&-\f12\big(\mathring L\mu+\mu \textrm{tr}{\chi}+\textrm{tr}\leftidx{^{(T)}}{\slashed\pi}\big){\leftidx{^{(Z)}}C_\g^{m}}_{,\mathring L}-\f12\textrm{tr}\chi{\leftidx{^{(Z)}}C_\g^{m}}_{,\mathring{\underline L}},
\end{split}
\end{equation}
where
\begin{equation}\label{C}
\begin{split}
&{\leftidx{^{(Z)}}C_\g^{m}}_{,\mathring L}=\leftidx{^{(Z)}}{\slashed\pi}_{\mathring LX}(\slashed d^X\Psi_\g^{m})-\f12(\textrm{tr}\leftidx{^{(Z)}}{\slashed\pi})\mathring L\Psi_\g^{m},\\
&{\leftidx{^{(Z)}}C_\gamma^{m}}_{,\mathring{\underline L}}=-2(\leftidx{^{(Z)}}\pi_{LT}+\mu^{-1}\leftidx{^{(Z)}}\pi_{TT})(\mathring L\Psi_\gamma^{m})+\leftidx{^{(Z)}}{\slashed\pi}_{\mathring{\underline L}X}(\slashed d^X\Psi_\gamma^{m})-\f12(\textrm{tr}\leftidx{^{(Z)}}{\slashed\pi})\mathring{\underline L}\Psi_\gamma^{m},\\
&\mu{\leftidx{^{(Z)}}{\slashed C}_\gamma^{m}}_{,X}=-\f12\leftidx{^{(Z)}}{\slashed\pi}_{\mathring{\underline L}X}(\mathring L\Psi_\gamma^{m})-\f12\leftidx{^{(Z)}}{\slashed\pi}_{\mathring LX}(\mathring{\underline L}\Psi_\gamma^{m})
+\leftidx{^{(Z)}}{{\pi}}_{\mathring LT}(\slashed d_X\Psi_\gamma^{m})+\f12\mu\textrm{tr}\leftidx{^{(Z)}}{\slashed\pi}\slashed d_X\Psi_\gamma^{m}.
\end{split}
\end{equation}
Substituting \eqref{C} into \eqref{muC} directly would lead to a lengthy and tedious equation for $\mu\mathscr D^\al{\leftidx{^{(Z)}}C_{\gamma}^{m}}_{,\al}$. To overcome this default and handle the resulting terms more easily,
one can decompose $\mu\mathscr D^\al{\leftidx{^{(Z)}}C_{\gamma}^{m}}_{,\al}$ into the following three parts as in \cite{MY}:
\begin{equation}\label{muZC}
\mu\mathscr D^\al{\leftidx{^{(Z)}}C_{\gamma}^{m}}_{,\al}=\leftidx{^{(Z)}}D_{\gamma,1}^m+\leftidx{^{(Z)}}D_{\gamma,2}^m+\leftidx{^{(Z)}}D_{\gamma,3}^m,
\end{equation}
where
\begin{equation}\label{D1}
\begin{split}
\leftidx{^{(Z)}}D_{\gamma,1}^m=&\f12\textrm{tr}\leftidx{^{(Z)}}{\slashed{\pi}}\big(\mathring L\mathring{\underline L}\Psi_\gamma^m+\f12\textrm{tr}{\chi}\mathring{\underline L}\Psi_\gamma^m\big)-\leftidx{^{(Z)}}{\slashed\pi}_{\mathring{\underline L}X}(\slashed d^X\mathring L\Psi_\gamma^{m})-\leftidx{^{(Z)}}{\slashed\pi}_{\mathring LX}(\slashed d^X \mathring{\underline L}\Psi_\gamma^{m})\\
&+(\leftidx{^{(Z)}}\pi_{\mathring LT}+\leftidx{^{(Z)}}\pi_{\tilde TT})(\mathring L^2\Psi_\gamma^m)+\f12\mu\textrm{tr}\leftidx{^{(Z)}}{{\slashed\pi}}\slashed\triangle\Psi_\gamma^m+\leftidx{^{(Z)}}\pi_{\mathring LT}\slashed\triangle\Psi_\gamma^m,
\end{split}
\end{equation}
\begin{equation}\label{D2}
\begin{split}
\leftidx{^{(Z)}}D_{\gamma,2}^m=&\big\{\mathring L(\leftidx{^{(Z)}}\pi_{\mathring LT}+\leftidx{^{(Z)}}\pi_{\tilde TT})-\f12\slashed\nabla^X\leftidx{^{(Z)}}{\slashed\pi}_{\mathring{\underline L}X}+\f14\mathring{\underline L}(\textrm{tr}\leftidx{^{(Z)}}{\slashed\pi})\big\}\mathring L\Psi_\gamma^m+\f14\mathring L(\textrm{tr}\leftidx{^{(Z)}}{\slashed{\pi}})\mathring{\underline L}\Psi_\gamma^m\\
&-\big\{\f12\slashed{\mathcal L}_{\mathring{\underline L}}\leftidx{^{(Z)}}{\slashed\pi}_{\mathring LX}-\slashed d_X\leftidx{^{(Z)}}{\pi}_{\mathring LT}-\f12\slashed d_X(\mu\textrm{tr}\leftidx{^{(Z)}}{\slashed\pi})+\f12(\slashed{\mathcal L}_{\mathring L}\leftidx{^{(Z)}}{\slashed\pi}_{\mathring{\underline L}X})\big\}\slashed d^X\Psi_\gamma^m\\
&-\f12(\slashed\nabla^X\leftidx{^{(Z)}}{\slashed\pi}_{\mathring LX})\mathring{\underline L}\Psi_\gamma^m,
\end{split}
\end{equation}
\begin{equation}\label{D3}
\begin{split}
\leftidx{^{(Z)}}D_{\gamma,3}^m=&\f12\Big\{(\textrm{tr}\leftidx{^{(Z)}}{\slashed{\pi}})(\slashed d_X\mu+2\mu\zeta_X)-(\mathring L\mu-\mu\textrm{tr}\chi-\textrm{tr}\leftidx{^{(T)}}{\slashed\pi})\leftidx{^{(Z)}}{\slashed\pi}_{\mathring LX}+\textrm{tr}{\chi}\leftidx{^{(Z)}}{\slashed\pi}_{\mathring{\underline L}X}\Big\}\slashed d^X\Psi_\gamma^m\\
&+\Big\{\textrm{tr}{\chi}(\leftidx{^{(Z)}}\pi_{\mathring LT}+\leftidx{^{(Z)}}\pi_{\tilde TT})+\f14(\mu\textrm{tr}{\chi}+\textrm{tr}\leftidx{^{(T)}}{\slashed{\pi}})\textrm{tr}\leftidx{^{(Z)}}{\slashed{\pi}}+\f12\slashed d^X\mu\leftidx{^{(Z)}}{\slashed\pi}_{\mathring LX}\Big\}\mathring L\Psi_\gamma^m.
\end{split}
\end{equation}
 Note that all the terms in $\leftidx{^{(Z)}}D_{\gamma,1}^m$ are the products of the deformation tensor and the second
order derivatives of $\Psi_\g^m$, except the first term containing the factor of the form $\mathring L\mathring{\underline L}\Psi_\gamma^m+\f12\textrm{tr}{\chi}\mathring{\underline L}\Psi_\gamma^m$ (see \eqref{D1}). It should be emphasized here that such a structure is crucial in our analysis since $\Psi_\gamma^m$ is the derivative of $\varphi_\gamma$ and by \eqref{fequation}, $\mathring L\mathring{\underline L}\varphi_\gamma+\f12\textrm{tr}{\chi}\mathring{\underline L}\varphi_\gamma=H_\g+\f12\textrm{tr}\check\chi\mathring{\underline L}\varphi_\g$ admits the better smallness and the faster time decay than those for $\mathring L\mathring{\underline L}\varphi_\gamma$ and $\f12\textrm{tr}{\chi}\mathring{\underline L}\varphi_\gamma$ separately. $\leftidx{^{(Z)}}D_{\gamma,2}^m$
collects all the products of the first order derivatives of the deformation
 tensor and the first order derivatives of $\Psi_\g^m$, and $\leftidx{^{(Z)}}D_{\gamma,3}^m$ is the rest.

The explicit expressions of $\Phi_\g^{m+1}$ given in \eqref{Phik}-\eqref{lamda} and \eqref{muZC}-\eqref{D3} are important to estimate the last two integrals of \eqref{e}.
From the left side of \eqref{e}, it is natural to define the corresponding weighted energies and fluxes as in \cite{MY}:
\begin{align}
E_{i,m+1}(s , u)&=\sum_{\gamma=0}^4\sum_{Z\in\{\varrho\mathring L,T,R\}}\delta^{2l}E_i[Z^{m}\varphi_\gamma](s, u),\quad i=1,2,\label{ei}\\
F_{i,m+1}(s, u)&=\sum_{\gamma=0}^4\sum_{Z\in\{\varrho\mathring L,T,R\}}\delta^{2l}F_i[Z^{m}\varphi_\gamma](s, u),\quad i=1,2,\label{fi}\\
E_{i,\leq m+1}(s, u)&=\sum_{0\leq n\leq m}E_{i,n+1}(s, u),\quad i=1,2,\label{eil}\\
F_{i,\leq m+1}(s, u)&=\sum_{0\leq n\leq m}F_{i,n+1}(s, u),\quad i=1,2,\label{fil}
\end{align}
where $l$ is the number of $T$ in $Z^m$. These weighted energies will be estimated in the subsequent sections.

Next, one derives the higher order $L^2$ estimates for some related quantities
so that the last two terms of \eqref{e} can be absorbed by the left hand side, and hence the higher order energy estimates on \eqref{ge} can
be obtained. To this end, on needs two lemmas in \cite[Section 6]{Ding4}, where analogous results in 3D case
can be found in \cite[Lemma 7.3]{MY} and \cite[Lemma 12.57]{J} respectively.

\begin{lemma}\label{L2T}
 For any function $\psi\in C^1(D^{s, u})$ vanishing on $C_0$, one has that  for small $\delta>0$,
 \begin{align}
  \int_{S_{s, u}} \psi^2&\lesssim\delta\int_{\Sigma_s^u}|T\psi|^2\lesssim\delta\int_{\Sigma_{s}^{u}}\big(|\mathring{\underline L}\psi|^2+\mu^2|\mathring L\psi|^2\big)\label{SSi},\\
  \int_{\Sigma_{s}^{u}} \psi^2&\lesssim\delta^2\int_{\Sigma_s^u}|T\psi|^2\lesssim\delta^2\int_{\Sigma_{s}^{u}}\big(|\mathring{\underline L}\psi|^2+\mu^2|\mathring L\psi|^2\big).\label{SiSi}
 \end{align}
Therefore,
 \begin{align}
 \int_{S_{s, u}} \psi^2&\lesssim\delta\big(E_2[\psi](s, u)+\varrho^{-2\iota}E_1[\psi](s, u)\big),\label{SE}\\
 \int_{\Sigma_{s}^{u}} \psi^2&\lesssim\delta^2\big(E_2[\psi](s, u)+\varrho^{-2\iota}E_1[\psi](s, u)\big).\label{SiE}
 \end{align}
\end{lemma}

\begin{lemma}\label{L2L}
	For any function $f\in C(D^{s, u})$, set
	\[
	F(s,u,\vartheta):=\int_{t_0}^sf(\tau,u,\vartheta)d\tau.
	\]
	Under the assumptions $(\star)$, then for small $\delta>0$,
	\begin{equation}\label{Ff}
	\|F\|_{s,u}\lesssim\varrho(s,u)^{1/2}\int_{t_0}^s\f{1}{\varrho(\tau,u)^{1/2}}\|f\|_{\tau,u}d\tau,
	\end{equation}
	where $\|F\|_{s,u}$ denotes the $L^2$ norm of $F$ over $\Sigma_s^u$  which is defined in Definition \ref{2.3}.
\end{lemma}

One can now derive $L^2$ estimates for the quantities such as
${\check\chi}$, $\leftidx{^{(Z)}}{\slashed\pi}_{\mathring L}$, $\leftidx{^{(R)}}{\slashed\pi}$, $\check{L}$, $\upsilon$, $x^j$
 $(j=1,2)$ and $R$ similarly as in \cite[Proposition 8.1]{Ding4}. However, here one has to pay more attentions to the
precise orders of smallness for $\dl$ and the time decay or increasing rates in each involved
quantity, one can get

\begin{proposition}\label{L2chi}
 Under the assumptions $(\star)$, when $\delta>0$ is small, it holds that for ${m\leq 2N-6}$,
 \begin{align*}
  &\delta^l\|Z^{m+1}\mu\|_{s,u}
 \lesssim\delta^{k(1-\varepsilon_0)-1/2}s^{1/2}\ln s+\delta^{(k-1)(1-\varepsilon_0)} s^{1/2}\Big(\sqrt{\tilde E_{1,\leq m+2}(s,u)}\\
 &\qquad\qquad\qquad\quad+\ln s\sqrt{\tilde E_{2,\leq m+2}(s,u)}+\dl^{k(1-\ve_0)-1}\sqrt{\int_0^u \tilde F_1(s,u')du'}\Big),\\
 &\delta^l\|\slashed{\mathcal L}_{Z}^m\check\chi\|_{s,u}\lesssim\delta^{k(1-\varepsilon_0)+1/2}s^{-3/2}\ln s
 +s^{-1}\mathscr E_{m+2}(s,u),\\
 &\delta^l\|\slashed{\mathcal L}_Z^{m+1}\slashed g\|_{s,u}+s^{-1}\delta^l\|{Z}^{m+2}x\|_{s,u}\lesssim\delta^{1/2}s^{1/2}
 +\mathscr E_{m+2}(s,u),\\
   &\delta^l\|\slashed{\mathcal L}_Z^{m+1}R\|_{s,u}+\delta^l\|Z^{m+1}\upsilon\|_{s,u}\lesssim\delta^{k(1-\varepsilon_0)+1/2}s^{1/2}\ln s
   +s\mathscr E_{m+2}(s,u),\\
   &\delta^l\|\slashed{\mathcal L}_{Z}^m\leftidx{^{(R)}}{\slashed\pi}_T\|_{s,u}\lesssim(1+\dl^{k(1-\ve_0)-1}\ln s)\Big(\delta^{k(1-\varepsilon_0)+1/2}s^{-1/2}\ln s
   +\mathscr E_{m+2}(s,u)\Big),\\
      &\delta^l\|Z^{m+1}\check L^j\|_{s,u}+\delta^l\|Z^{m+1}\check\varrho\|_{s,u}+\delta^l\|\slashed{\mathcal L}_{Z}^m\leftidx{^{(R)}}{\slashed\pi}\|_{s,u}+\delta^l\|\slashed{\mathcal L}_{Z}^m\leftidx{^{(R)}}{\slashed\pi}_{\mathring L}\|_{s,u}\\
   &\qquad\qquad\qquad\qquad\qquad\qquad\lesssim\delta^{k(1-\varepsilon_0)+1/2}s^{-1/2}\ln s+\mathscr E_{m+2}(s,u),\quad j=1,2,\\
   &\delta^l\|\slashed{\mathcal L}_{Z}^m\leftidx{^{(T)}}{\slashed\pi}_{\mathring L}\|_{s,u}\lesssim\delta^{k(1-\varepsilon_0)-1/2}s^{-1/2}\ln s
   +\delta^{(k-1)(1-\varepsilon_0)}s^{-1/2}\Big(\sqrt{\tilde E_{1,\leq m+2}(s,u)}\\
   &\qquad\qquad\qquad\qquad+\ln s\sqrt{\tilde E_{2,\leq m+2}(s,u)}+\dl^{k(1-\ve_0)-1}\sqrt{\int_0^u \tilde F_1(s,u')du'}\Big),\\
       &\delta^l\|\slashed{\mathcal L}_{Z}^m\leftidx{^{(T)}}{\slashed\pi}\|_{s,u}\lesssim(\delta^{1/2}+\dl^{k(1-\ve_0)-1/2}\ln s)s^{-1/2}
   +\delta^{(k-1)(1-\varepsilon_0)}s^{-1/2}\Big(\sqrt{\tilde E_{1,\leq m+2}(s,u)}\\
   &\qquad\qquad\qquad\qquad+\ln s\sqrt{\tilde E_{2,\leq m+2}(s,u)}+(s^{-1/2}+\dl^{k(1-\ve_0)-1})\sqrt{\int_0^u \tilde F_1(s,u')du'}\Big),
 \end{align*}
 where $\mathscr E_{m+2}(s,u)=\dl^{k(1-\ve_0)}s^{-\iota}\sqrt{\tilde E_{1,\leq m+2}(s,u)}+\dl^{k(1-\ve_0)+\ve_0}(1+\dl^{k(1-\ve_0)-1})s^{-1/2}\ln^2 s\sqrt{\tilde E_{2,\leq m+2}(s,u)}+\dl^{(k-1)(1-\ve_0)}\sqrt{\int_0^u \tilde F_1(s,u')du'}$ with $\tilde E_{i,\leq m+2}(s,u)=\sup_{t_0\leq\tau\leq s}E_{i,\leq m+2}(\tau,u)$ $(i=1,2)$ and $\tilde F_{1,\leq m+2}(s,u)$ $=\sup_{t_0\leq\tau\leq s}F_{1,\leq m+2}(\tau,u)$.
\end{proposition}

\begin{proof}
	By Lemma \ref{4.2}, the explicit expressions of $\varrho\mathring L\check L^i$, $T\check L^i$ and $R\check L^i$
	can be given since $\varrho\mathring L\check L^i=\varrho\mathring L\mathring L^i-\check L^i$, $T\check L^i=T\mathring L^i+\f{\mu}{\varrho}g^{0i}+\f{\mu}{\varrho}\check L^i+\f{\mu-1}{\varrho^2}x^i$ and $R\check L^i={R^X}\slashed d_X\check L^i$
	hold.
	One then gets
	\begin{equation}\label{ZkL}
	\begin{split}
	\delta^l\|Z^{m+1}\check L^i\|_{s,u}\lesssim&\delta^{l_1}\|Z^{m_1}\check L^i\|_{s,u}+s\delta^{l_1}\|\slashed{\mathcal L}_Z^{m_1}\check\chi\|_{s,u}+\delta s^{-1}\delta^{l_{0}}\|Z^{m_{0}}\mu\|_{s,u}\\
	&+\delta^{k(1-\varepsilon_0)}s^{-2}\ln s\delta^{l_0}\|Z^{m_0}x\|_{s,u}+\delta^{k(1-\varepsilon_0)}s^{-1}\ln s\delta^{l_1}\|\slashed{\mathcal L}_Z^{m_1}\slashed g\|_{s,u}\\
	&+\delta^{(k-1)(1-\varepsilon_0)}s^{-(k-1)/2}\delta^{l_0}\|Z^{m_0}\varphi_\gamma\|_{s,u},
	\end{split}
	\end{equation}
	here $l$ and $l_p$ ($p=0,1$) are the numbers of $T$ in $Z^{m+1}$ and $Z^{m_p}$ respectively with $1\leq m_p\leq m+1-p$.
	Note that $|\check L^i|\lesssim\delta^{k(1-\varepsilon_0)}s^{-1}\ln s$ by \eqref{chL}.
	Together with $\|1\|_{s,u}\lesssim\delta^{1/2}\varrho(s,u)^{1/2}$, this yields
	$\|\check L^i\|_{s,u}\lesssim\delta^{k(1-\varepsilon_0)+1/2}s^{-1/2}\ln s$. Thus, with the help of \eqref{SiE},
	\begin{equation}\label{ZkcL}
	\begin{split}
	\delta^l\|Z^{m+1}\check L^i\|_{s,u}\lesssim&\delta^{k(1-\varepsilon_0)+1/2}s^{-1/2}\ln s+s\delta^{l_1}\|\slashed{\mathcal L}_Z^{m_1}\check\chi\|_{s,u}+\delta s^{-1}\delta^{l_{0}}\|Z^{m_{0}}\mu\|_{s,u}\\
	&+\delta^{k(1-\varepsilon_0)}s^{-2}\ln s\delta^{l_0}\|Z^{m_0}x\|_{s,u}+\delta^{k(1-\varepsilon_0)}s^{-1}\ln s\delta^{l_1}\|\slashed{\mathcal L}_Z^{m_1}\slashed g\|_{s,u}\\
	&+\delta^{k(1-\varepsilon_0)+\varepsilon_0}s^{-(k-1)/2}\Big\{s^{-\iota}\sqrt{E_{1,\leq m+2}(s,u)}+\sqrt{E_{2,\leq m+2}(s,u)}\Big\}.
	\end{split}
	\end{equation}
	Similarly, due to $\varrho Lx^i=\varrho\check L^i+x^i$, $Tx^i=\mu(-g^{0i}-\check L^i-\f{x^i}\varrho)$ and $Rx^a=\O^a+\upsilon(g^{0a}+\check L^a+\f{x^a}\varrho)$, then
	\begin{equation}\label{Zkx}
	\begin{split}
	\delta^{l}\|Z^{m+2}x&\|_{s,u}\lesssim\delta^{\f12}s^{\f32}+\delta^{l_0}\|Z^{m_0}\upsilon\|_{s,u}
	+s\sum_{j=1}^2\delta^{l_0}\|Z^{m_0}\check L^j\|_{s,u}+\delta\delta^{l_{0}}\|Z^{m_{0}}\mu\|_{s,u}\\
	&+\delta^{k(1-\varepsilon_0)+\varepsilon_0+1}s^{-(k-1)/2}\ln s\Big\{s^{-\iota}\sqrt{E_{1,\leq m+2}(s,u)}+\sqrt{E_{2,\leq m+2}(s,u)}\Big\}.
	\end{split}
	\end{equation}
	And it follows from $\slashed{\mathcal L}_{\varrho\mathring L}R=\varrho\leftidx{^{(R)}}{\slashed\pi}_{\mathring L}$, $\slashed{\mathcal L}_TR=\leftidx{^{(R)}}{\slashed\pi}_T$ and \eqref{Rpi}  in Appendix A that
	\begin{equation}\label{ZR}
	\begin{split}
	\delta^l\|\slashed{\mathcal L}_Z^{m+1}R&\|_{s,u}\lesssim\delta^{k(1-\varepsilon_0)+1/2}s^{1/2}\ln s+\delta^{k(1-\varepsilon_0)}s^{-1}\ln s\delta^{l_0}\|Z^{m_0}x\|_{s,u}
	\\
	&+\delta^{k(1-\varepsilon_0)}s^{-1}\delta^{l_1}\|Z^{m_1}\upsilon\|_{s,u}+s^2\delta^{l_1}\|\slashed{\mathcal L}_Z^{m_1}\check\chi\|_{s,u}+\sum_{j=1}^2s\delta^{l_0}\|Z^{m_1}\check L^j\|_{s,u}\\
	&+\delta^{k(1-\varepsilon_0)}\ln s\delta^{l_1}\|\slashed{\mathcal L}_Z^{m_1}\slashed g\|_{s,u}+\delta^{k(1-\varepsilon_0)+1}s^{-1}\ln s\delta^{l_{0}}\|Z^{m_{0}}\mu\|_{s,u}\\
	&+\delta^{k(1-\varepsilon_0)+\varepsilon_0}s^{-(k-3)/2} \Big\{s^{-\iota}\sqrt{E_{1,\leq m+2}(s,u)}+\sqrt{E_{2,\leq m+2}(s,u)}\Big\}.
	\end{split}
	\end{equation}
	Additionally, by \eqref{omega}, \eqref{rrho} and the facts that $\slashed{\mathcal L}_{\varrho\mathring L}\slashed g=\varrho\check\chi+\slashed g$, $\slashed{\mathcal L}_T\slashed g=\leftidx{^{(T)}}{\slashed\pi}$ and $\slashed{\mathcal L}_R\slashed g=\leftidx{^{(R)}}{\slashed\pi}$, it holds that
	\begin{equation}\label{Zku}
	\begin{split}
	\delta^{l}\|Z^{m+1}\upsilon\|_{s,u}\lesssim& \delta^{k(1-\varepsilon_0)}s^{-1}\ln s\delta^{l_0}\|Z^{k_0}x\|_{s,u}+s\delta^{l_0}\|Z^{m_0}\check L^1\|_{s,u}+s\delta^{l_0}\|Z^{m_0}\check L^2\|_{s,u}\\
	&+\delta^{k(1-\varepsilon_0)+\varepsilon_0}s^{-(k-3)/2}\Big\{ s^{-\iota}\sqrt{E_{1,\leq m+2}(s,u)}+\sqrt{E_{2,\leq m+2}(s,u)}\Big\},\qquad
	\end{split}
	\end{equation}
	\begin{equation}\label{Zkr}
	\begin{split}
	\delta^l\|Z^{m+1}\check\varrho\|_{s,u}\lesssim&\delta^{2k(1-\varepsilon_0)+1/2}s^{-3/2}\ln^2s+\delta^{k(1-\varepsilon_0)}s^{-2}\ln s\delta^{l_0}\|Z^{k_0}x\|_{s,u}
	+\delta^{l_0}\|Z^{m_0}\check L^1\|_{s,u}\\
	+\delta^{l_0}\|Z^{m_0}&\check L^2\|_{s,u}+\delta^{k(1-\varepsilon_0)+\varepsilon_0}s^{-(k-1)/2}\Big\{ s^{-\iota}\sqrt{E_{1,\leq m+2}(s,u)}+\sqrt{E_{2,\leq m+2}(s,u)}\Big\},
	\end{split}
	\end{equation}
	\begin{equation}\label{Zkg}
	\begin{split}
	\delta^l\|\slashed{\mathcal L}_Z^{m+1}\slashed g\|_{s,u}\lesssim&\delta^{1/2}s^{1/2}+s^{-1}\delta^{l_1}\|Z^{m_1}\upsilon\|_{s,u}+\delta s^{-1}\delta^{l_{0}}\|Z^{m_{0}}\mu\|_{s,u}\\
	&+\delta^{k(1-\varepsilon_0)}s^{-(k+2)/2}\delta^{l_0}\|Z^{k_0}x\|_{s,u}+s\delta^{l_1}\|\slashed{\mathcal L}_Z^{m_1}\check\chi\|_{s,u}\\
	&+(\delta^{k(1-\varepsilon_0)+1}s^{-2}+\delta^{2k(1-\ve_0)}s^{-2}\ln s)\delta^{l_0}\big\{\|Z^{m_1}\check L^1\|_{s,u}+\|Z^{m_1}\check L^2\|_{s,u}\big\}\\
	&+\delta^{k(1-\varepsilon_0)+\varepsilon_0}s^{-(k-1)/2}\Big\{s^{-\iota}\sqrt{E_{1,\leq m+2}(s,u)}+\sqrt{E_{2,\leq m+2}(s,u)}\Big\}.\qquad
	\end{split}
	\end{equation}
	
	Collecting \eqref{ZkcL}-\eqref{Zkg} yields
	\begin{align}
	\delta^l\|Z^{m+1}\check L^i\|_{s,u}\lesssim&\delta^{k(1-\varepsilon_0)+1/2}s^{-1/2}\ln s+s\delta^{l_1}\|\slashed{\mathcal L}_Z^{m_1}\check\chi\|_{s,u}+\delta s^{-1}\delta^{l_{0}}\|Z^{m_{0}}\mu\|_{s,u}\no\\
	&+\delta^{k(1-\varepsilon_0)+\varepsilon_0}s^{-(k-1)/2}\Big\{s^{-\iota}\sqrt{E_{1,\leq m+2}(s,u)}+\sqrt{E_{2,\leq m+2}(s,u)}\Big\},\label{Zk1L}\\
	\delta^l\|Z^{m+2}x^i\|_{s,u}\lesssim&\delta^{1/2}s^{3/2}+s^2\delta^{l_1}\|\slashed{\mathcal L}_Z^{m_1}\check\chi\|_{s,u}+\delta\delta^{l_{0}}\|Z^{m_{0}}\mu\|_{s,u}\no\\
	&+\delta^{k(1-\varepsilon_0)+\varepsilon_0}s^{-(k-3)/2}\Big\{s^{-\iota}\sqrt{E_{1,\leq m+2}(s,u)}+\sqrt{E_{2,\leq m+2}(s,u)}\Big\},\label{Zk1x}\\
	\delta^l\|\slashed{\mathcal L}_Z^{m+1}\slashed g\|_{s,u}\lesssim&\delta^{1/2}s^{1/2}+s\delta^{l_1}\|\slashed{\mathcal L}_Z^{m_1}\check\chi\|_{s,u}+\delta s^{-1}\delta^{l_{0}}\|Z^{m_{0}}\mu\|_{s,u}\no\\
	&+\delta^{k(1-\varepsilon_0)+\varepsilon_0}s^{-(k-1)/2}\Big\{ s^{-\iota}\sqrt{E_{1,\leq m+2}(s,u)}+\sqrt{E_{2,\leq m+2}(s,u)}\Big\},\label{Zk1g}\\
	\delta^{l}\|Z^{m+1}\upsilon\|_{s,u}\lesssim&\delta^{k(1-\varepsilon_0)+1/2}s^{1/2}\ln s+s^2\delta^{l_1}\|\slashed{\mathcal L}_Z^{m_1}\check\chi\|_{s,u}+\delta\delta^{l_{0}}\|Z^{m_{0}}\mu\|_{s,u}\no\\
	&+\delta^{k(1-\varepsilon_0)+\varepsilon_0}s^{-(k-3)/2}\Big\{ s^{-\iota}\sqrt{E_{1,\leq m+2}(s,u)}+\sqrt{E_{2,\leq m+2}(s,u)}\Big\},\label{Zk1u}\\
	\delta^l\|Z^{m+1}\check\varrho\|_{s,u}\lesssim&\delta^{k(1-\varepsilon_0)+1/2}s^{-1/2}\ln s+s\delta^{l_1}\|\slashed{\mathcal L}_Z^{m_1}\check\chi\|_{s,u}+\delta s^{-1}\delta^{l_{0}}\|Z^{m_{0}}\mu\|_{s,u}\no\\
	&+\delta^{k(1-\varepsilon_0)+\varepsilon_0}s^{-(k-1)/2}\Big\{s^{-\iota}\sqrt{E_{1,\leq m+2}(s,u)}+\sqrt{E_{2,\leq m+2}(s,u)}\Big\},\label{Zk1Rr}\\
	\delta^l\|\slashed{\mathcal L}_Z^{m+1}R\|_{s,u}\lesssim& \delta^{k(1-\varepsilon_0)+1/2}s^{1/2}\ln s+s^2\delta^{l_1}\|\slashed{\mathcal L}_Z^{m_1}\check\chi\|_{s,u}+\delta\delta^{l_{0}}\|Z^{m_{0}}\mu\|_{s,u}\no\\
	&+\delta^{k(1-\varepsilon_0)+\varepsilon_0}s^{-(k-3)/2}\Big\{s^{-\iota}\sqrt{E_{1,\leq m+2}(s,u)}+\sqrt{E_{2,\leq m+2}(s,u)}\Big\}.\label{Zk1R}
	\end{align}
	Thanks to \eqref{Lpi}-\eqref{Rpi} in Appendix A, one can deduce
	\begin{align}
	\delta^l\|\slashed{\mathcal L}_Z^m\leftidx{^{(T)}}{\slashed\pi}\|_{s,u}\lesssim&\delta^{1/2}s^{-1/2}
	+\delta^{l_1}\|\slashed{\mathcal L}_Z^{m_1}\check\chi\|_{s,u}+s^{-1}\delta^{l_{0}}\|Z^{m_{0}}\mu\|_{s,u}\no\\
	&+\delta^{(k-1)(1-\varepsilon_0)}s^{-(k-1)/2}\Big\{s^{-\iota}\sqrt{E_{1,\leq m+2}(s,u)}+\sqrt{E_{2,\leq m+2}(s,u)}\Big\},\label{ZkTpi}\\
	\delta^l\|\slashed{\mathcal L}_Z^m\leftidx{^{(T)}}{\slashed\pi}_{\mathring L}\|_{s,u}\lesssim&\delta^{k(1-\varepsilon_0)-1/2}s^{-1/2}+\delta^{k(1-\varepsilon_0)-1}s^{-(k-2)/2}\delta^{l_1}\|\slashed{\mathcal L}_Z^{m_1}\check\chi\|_{s,u}+ s^{-1}\delta^{l_{0}}\|Z^{m_{0}}\mu\|_{s,u}\no\\
	&+\delta^{(k-1)(1-\varepsilon_0)}s^{-(k-1)/2}\Big\{s^{-\iota}\sqrt{E_{1,\leq m+2}(s,u)}+\sqrt{E_{2,\leq m+2}(s,u)}\Big\},\label{ZkTLpi}\\
	\delta^{l}\|\slashed{\mathcal L}_Z^{m}\leftidx{^{(R)}}{\slashed\pi}_T\|_{s,u}\lesssim&\delta^{k(1-\varepsilon_0)+1/2}s^{-1/2}\ln s
	+s\delta^{l_1}\|\slashed{\mathcal L}_Z^{m_1}\check\chi\|_{s,u}+\delta s^{-1}\ln s\delta^{l_{0}}\|Z^{m_{0}}\mu\|_{s,u}\no\\
	+(\delta^{k(1-\varepsilon_0)+\varepsilon_0}&+\delta^{(2k-1)(1-\ve_0)})s^{-(k-1)/2}\Big\{s^{-\iota}\sqrt{E_{1,\leq m+2}(s,u)}+\sqrt{E_{2,\leq m+2}(s,u)}\Big\},\label{ZkRLpi}\\
	\delta^{l}\|\slashed{\mathcal L}_Z^{m}\leftidx{^{(R)}}{\slashed\pi}_{\mathring L}\|_{s,u}\lesssim&\delta^{k(1-\varepsilon_0)+1/2}s^{-1/2}\ln s+s\delta^{l_1}\|\slashed{\mathcal L}_Z^{m_1}\check\chi\|_{s,u}
	+\delta s^{-1}\delta^{l_{0}}\|Z^{m_{0}}\mu\|_{s,u}\no\\
	&+\delta^{k(1-\varepsilon_0)+\varepsilon_0}s^{-(k-1)/2} \Big\{s^{-\iota}\sqrt{E_{1,\leq m+2}(s,u)}+\sqrt{E_{2,\leq m+2}(s,u)}\Big\},\label{ZkRLp}\\
	\delta^l\|\slashed{\mathcal L}_Z^{m}\leftidx{^{(R)}}{\slashed\pi}\|_{s,u}\lesssim&\delta^{k(1-\varepsilon_0)+1/2}s^{-1/2}\ln s+s\delta^{l_1}\|\slashed{\mathcal L}_Z^{m_1}\check\chi\|_{s,u}+\delta s^{-1}\delta^{l_{0}}\|Z^{m_{0}}\mu\|_{s,u}\no\\
	&+\delta^{k(1-\varepsilon_0)+\varepsilon_0}s^{-(k-1)/2}\Big\{s^{-\iota}\sqrt{E_{1,\leq m+2}(s,u)}+\sqrt{E_{2,\leq m+2}(s,u)}\Big\}.\label{ZkRp}
	\end{align}
	On the other hand, for the solution $\phi$ to \eqref{quasi},
	\begin{equation}\label{Zm+1p}
	\begin{split}
	&\delta^l\|Z^{m+1}\phi\|_{s,u}\lesssim\delta^{1+l}\|[T,Z^{m+1}]\phi\|_{s,u}+\delta^{1+l}\|Z^{m+1}(\mu\tilde T^i\varphi_i)\|_{s,u}\\
	\lesssim&\delta^{3-\ve_0}s^{-3/2}\delta^{l_1}\|\slashed{\mathcal L}_Z^{m_1}\leftidx{^{(R)}}{\slashed\pi}_T\|_{s,u}+\delta^{3-\ve_0}s^{-1/2}\delta^{l_1}\|\slashed{\mathcal L}_Z^{m_1}\leftidx{^{(T)}}{\slashed\pi}_{\mathring L}\|_{s,u}+\delta^{2-\ve_0}s^{-1/2}\delta^{l'}\|Z^{m'+1}\mu\|_{s,u}\\
	&+\delta^{(k+1)(1-\ve_0)+1}s^{-3/2}\delta^{l_1}\|\slashed{\mathcal L}_Z^{m_1}\slashed g\|_{s,u}+(\delta+\delta^{k(1-\ve_0)})s^{-1}\delta^{l_0}\|Z^{m_0}\phi\|_{s,u}+\delta\delta^{l_0}\|Z^{m_0}\varphi\|_{s,u}\\
	&+\delta^{2-\ve_0}s^{-1/2}\delta^{l_0}\big(\|\tilde Z^{m_0}\check L^1\|_{s,u}+\|\tilde Z^{m_0}\check L^2\|_{s,u}\big)+\delta^{2-\ve_0}s^{-3/2}\delta^{l_0}\|Z^{m_0}x\|_{s,u},
	\end{split}
	\end{equation}
	which means
	\begin{equation}\label{Zmp1Y}
	\begin{split}
	\delta^l\|Z^{m+1}\phi\|_{s,u}\lesssim&\delta^{5/2-\ve_0}+\dl^{2-\ve_0}s^{1/2}\delta^{l_1}\|\slashed{\mathcal L}_Z^{m_1}\check\chi\|_{s,u}+\delta^{2-\ve_0} s^{-1/2}\delta^{l_{0}}\|Z^{m_{0}}\mu\|_{s,u}\\
	&+\delta^{2}\Big\{s^{-\iota}\sqrt{E_{1,\leq m+2}(s,u)}+\sqrt{E_{2,\leq m+2}(s,u)}\Big\}.
	\end{split}
	\end{equation}

	Note that all the terms in the left hand side of \eqref{Zk1L}-\eqref{Zmp1Y} are controlled by the $L^2-$norms of
	the derivatives of $\check\chi$ and $\mu$.
	Next we deal with $\|\slashed{\mathcal L}_Z^{m_1}\check\chi\|_{s,u}$ and $\|Z^{m_{0}}\mu\|_{s,u}$.

	First, we treat the cases that all the vectorfields $Z'$s in the left hand side of \eqref{Zk1L}-\eqref{Zmp1Y}
	are $R'$s. Following the processes to obtain \eqref{Zk1L}-\eqref{Zmp1Y}, one can show that the following inequalities
	hold when all $Z'$s are replaced by $R'$s,
	\begin{align}
	\|R^{m+1}\check L^i\|_{s,u}+&\|R^{m+1}\check\varrho\|_{s,u}+\|\slashed{\mathcal L}_R^{m}\leftidx{^{(R)}}{\slashed\pi}\|_{s,u}+\|\slashed{\mathcal L}_R^{m}\leftidx{^{(R)}}{\slashed\pi}_{\mathring L}\|_{s,u}\no\\
	\lesssim&\delta^{k(1-\varepsilon_0)+1/2}s^{-1/2}\ln s+s\|\slashed{\mathcal L}_R^{m_1}\check\chi\|_{s,u}+\delta s^{-1}\|R^{m_{0}}\mu\|_{s,u}\no\\
	&+\delta^{k(1-\varepsilon_0)+\varepsilon_0}s^{-(k-1)/2}\Big\{s^{-\iota}\sqrt{E_{1,\leq m+2}(s,u)}+\sqrt{E_{2,\leq m+2}(s,u)}\Big\},\no\\
	\|R^{m+2}x\|_{s,u}\lesssim&\delta^{1/2}s^{3/2}+s^2\|\slashed{\mathcal L}_R^{m_1}\check\chi\|_{s,u}+\delta\|R^{m_{0}}\mu\|_{s,u}\no\\
	&+\delta^{k(1-\varepsilon_0)+\varepsilon_0}s^{-(k-3)/2}\Big\{s^{-\iota}\sqrt{E_{1,\leq m+2}(s,u)}+\sqrt{E_{2,\leq m+2}(s,u)}\Big\},\no\\
	\|\slashed{\mathcal L}_R^{m+1}\slashed g\|_{s,u}\lesssim&\delta^{1/2}s^{1/2}+s\|\slashed{\mathcal L}_R^{m_1}\check\chi\|_{s,u}+\delta s^{-1}\|R^{m_{0}}\mu\|_{s,u}\no\\
	&+\delta^{k(1-\varepsilon_0)+\varepsilon_0}s^{-(k-1)/2}\Big\{ s^{-\iota}\sqrt{E_{1,\leq m+2}(s,u)}+\sqrt{E_{2,\leq m+2}(s,u)}\Big\},\no\\
	\|R^{m+1}\upsilon\|_{s,u}\lesssim&\delta^{k(1-\varepsilon_0)+1/2}s^{1/2}\ln s+s^2\|\slashed{\mathcal L}_R^{m_1}\check\chi\|_{s,u}+\delta\|R^{m_{0}}\mu\|_{s,u}\label{RL2}\\
	&+\delta^{k(1-\varepsilon_0)+\varepsilon_0}s^{-(k-3)/2}\Big\{ s^{-\iota}\sqrt{E_{1,\leq m+2}(s,u)}+\sqrt{E_{2,\leq m+2}(s,u)}\Big\},\no\\
	\|\slashed{\mathcal L}_R^m\leftidx{^{(T)}}{\slashed\pi}\|_{s,u}\lesssim&\delta^{1/2}s^{-1/2}
	+\|\slashed{\mathcal L}_R^{m_1}\check\chi\|_{s,u}+s^{-1}\|R^{m_{0}}\mu\|_{s,u}\no\\
	&+\delta^{(k-1)(1-\varepsilon_0)}s^{-(k-1)/2}\Big\{s^{-\iota}\sqrt{E_{1,\leq m+2}(s,u)}+\sqrt{E_{2,\leq m+2}(s,u)}\Big\},\no\\
	\|\slashed{\mathcal L}_R^m\leftidx{^{(T)}}{\slashed\pi}_{\mathring L}\|_{s,u}\lesssim&\delta^{k(1-\varepsilon_0)-1/2}s^{-1/2}+\delta^{k(1-\varepsilon_0)-1}s^{-{(k-2)}/2}\|\slashed{\mathcal L}_R^{m_1}\check\chi\|_{s,u}+ s^{-1}\|R^{m_{0}}\mu\|_{s,u}\no\\
	&+\delta^{(k-1)(1-\varepsilon_0)}s^{-(k-1)/2}\Big\{s^{-\iota}\sqrt{E_{1,\leq m+2}(s,u)}+\sqrt{E_{2,\leq m+2}(s,u)}\Big\},\no\\
	\|\slashed{\mathcal L}_R^{m}\leftidx{^{(R)}}{\slashed\pi}_T\|_{s,u}\lesssim&\delta^{k(1-\varepsilon_0)+\f12}s^{-\f12}\ln s
	+s\|\slashed{\mathcal L}_R^{m_1}\check\chi\|_{s,u}+\delta s^{-1}\ln s\|R^{m_{0}}\mu\|_{s,u}\no\\
	+(\delta^{k(1-\varepsilon_0)+\varepsilon_0}&+\delta^{(2k-1)(1-\ve_0)})s^{-(k-1)/2}\Big\{s^{-\iota}\sqrt{E_{1,\leq m+2}(s,u)}+\sqrt{E_{2,\leq m+2}(s,u)}\Big\},\no\\
	\|R^{m+1}\phi\|_{s,u}\lesssim&\delta^{5/2-\ve_0}+\dl^{2-\ve_0}s^{1/2}\|\slashed{\mathcal L}_R^{m_1}\check\chi\|_{s,u}+\delta^{2-\ve_0} s^{-1/2}\|R^{m_{0}}\mu\|_{s,u}\no\\
	&+\delta^{2}\Big\{s^{-\iota}\sqrt{E_{1,\leq m+2}(s,u)}+\sqrt{E_{2,\leq m+2}(s,u)}\Big\},\no\\
	\|R^{m}\mathring L\phi\|_{s,u}\lesssim&\delta^{5/2-\ve_0}s^{-1}+(\dl^{2-\ve_0}+\dl^{(k+1)(1-\ve_0)})s^{-1/2}\|\slashed{\mathcal L}_R^{m_1}\check\chi\|_{s,u}+\delta^{2-\ve_0} s^{-3/2}\|R^{m_{0}}\mu\|_{s,u}\no\\
	&+\delta^{2}\Big\{s^{-1-\iota}\sqrt{E_{1,\leq m+2}(s,u)}+\sqrt{E_{2,\leq m+2}(s,u)}\Big\}.\no
	\end{align}

We now estimate the $L^2-$norm of $\mathring LR^{m+1}\mu$.
Note that it follows from \eqref{Lchi'}, \eqref{fequation} and \eqref{pal} that
	\begin{equation}\label{Lrchi}
	   \begin{split}
		\mathring L(\varrho^2\text{tr}\check\chi)=&-\varrho^2G_{X\mathring L}^\gamma(\slashed d^X\mathring L\vp_\g)+\f12\varrho^2G_{\mathring L\mathring L}^\g\slashed\triangle\vp_\g+\f12\varrho^2\slashed g^{XX}G_{XX}^\g\mathring L^2\vp_\g-\varrho^2(\text{tr}\check\chi)^2\\
		&+\Big(-\f12G_{\mathring L\mathring L}^\g\mathring L\vp_\g-G_{\t T\mathring L}^\g\mathring L\vp_\g+G_{X\mathring L}^\g\slashed d^X\vp_\g\Big)(\varrho^2\text{tr}\check\chi+\varrho)+\cdots,
			\end{split}
	\end{equation}
	where
	\begin{align*}
	&G_{\mathring L\mathring L}^\g\slashed\triangle\vp_\g=\mu^{-1}G_{\mathring L\mathring L}^\g\Big(L(\mu\mathring L\vp_\g)+\f1{2\varrho}\mu\mathring L\vp_\g+2LT\vp_\g+\f1\varrho T\vp_\g-H_\g\Big)\\
	=&\mu^{-1}G_{\mathring L\mathring L}^\g\Big(L(\mu\mathring L\vp_\g)+\f1{2\varrho}\mu\mathring L\vp_\g-H_\g\Big)+2\mu^{-1}\big\{G_{\mathring L\mathring L}^0L(\mu\tilde T^i\mathring L\vp_i)+G_{\mathring L\mathring L}^\g L(\mu\tilde T^ig_{\g j}\slashed d^Xx^j\slashed d_X\vp_i)\\
	&-G_{\mathring L\mathring L}^\g(\mathring L\mathring L_{\g})\tilde T^iT\vp_i-(G_{\mathring L\mathring L}^\g\mathring L_\g)(L\tilde T^i)T\vp_i\underline{-(G_{\mathring L\mathring L}^\g\mathring L_\g)\tilde T^i(\mathring LT\vp_i)}+\f1{2\varrho}\big(\mu G_{\mathring L\mathring L}^0\tilde T\mathring L\vp_i\\
	&+\mu G_{\mathring L\mathring L}^\g\tilde T^ig_{\g j}\slashed d^Xx^j\slashed d_X\vp_i\underline{-(G_{\mathring L\mathring L}^\g\mathring L_\g)\tilde T^iT\vp_i}\big)
	\big\}.
	\end{align*}
	Applying \eqref{fequation} again for the above two underline terms yields
	\begin{align*}
	G_{\mathring L\mathring L}^\g\slashed\triangle\vp_\g
	=&-(G_{\mathring L\mathring L}^\g\mathring L_\g)\tilde T^i\slashed\triangle\vp_i+G_{\mathring L\mathring L}^\g\mathring L^2\vp_\g+2G_{\mathring L\mathring L}^0\tilde T^i(\mathring L^2\vp_i)\no\\
	&+2G_{\mathring L\mathring L}^\g\tilde{T}^ig_{\g j}\slashed d^Xx^j(\slashed d_X\mathring L\vp_i)-(G_{\mathring L\mathring L}^\g\mathring L_\g)\tilde T^i(\mathring L^2\vp_i)\\
	&+\mu^{-1}f_1(\varphi, \slashed dx, \mathring L^1, \mathring L^2)\left(
	\begin{array}{ccc}
	G_{\mathring L\mathring L}^\g\mathring L_\g\\
	\vp^{k-1}\\
	\end{array}
	\right)
	\left(
	\begin{array}{ccc}
	\mathring L\vp\\
	\slashed d\vp\\
	\end{array}
	\right)
	\left(
	\begin{array}{ccc}
	\vp^{k-1}T\vp\\
	\mu\mathring L\vp\\
	\mu\slashed d\vp\\
	\mu\varrho^{-1}
	\end{array}
	\right)\no\\
	&+\mu^{-1}(\text{tr}\check\chi) f_2(\varphi, \slashed dx, \mathring L^1, \mathring L^2)\left(
	\begin{array}{ccc}
	(G_{\mathring L\mathring L}^\g\mathring L_\g)T\vp\\
	\mu\vp^{k-1}\mathring L\vp\\
	\mu\vp^{k-1}\slashed d\vp
	\end{array}
	\right).\no
	\end{align*}
Together with \eqref{gLLL} and Proposition \ref{LTRh}, by the analogous proof of (8.43) in \cite{Ding4}, this deduces
	\begin{align}\label{G3}
	&\|R^m(\varrho^2\mu G_{\mathring L\mathring L}^\g\slashed\triangle\vp_\g)\|_{s,u}+\|R^m(\varrho^2\mu G_{X\mathring L}^\gamma\slashed d^X\mathring L\vp_\g)\|_{s,u}+\|R^m(\varrho^2\mu\slashed g^{XX}G_{XX}^\g\mathring L^2\vp_\g)\|_{s,u}\no\\
	\lesssim&\dl^{k(1-\ve_0)}\|\slashed dR^{\leq m+1}\vp\|_{s,u}+\dl^{(k-1)(1-\ve_0)}s^{1/2}\big(s\|\mathring LR^{\leq m}\mathring L\vp\|_{s,u}+\|\mathring LR^{\leq m+1}\vp\|_{s,u}\big)\no\\
	&+\dl^{k(1-\ve_0)}s^{-1}\big(\|R^{\leq m}\mu\|_{s,u}+\|\slashed{\mathcal L}_{R}^{\leq m}\leftidx{^{(R)}}{\slashed\pi}_{\mathring L}\|_{s,u}\big)+\dl^{k(1-\ve_0)-\ve_0}s^{1/2}\|R^{\leq m}\text{tr}\check\chi\|_{s,u}\\
	&+(1+\dl^{k(1-\ve_0)-1}\ln s)\Big\{\dl^{k(1-\ve_0)}s^{-1}\big(\sum_{i=1}^2\|R^{\leq m}\check L^i\|_{s,u}+\|R^{\leq m}\check\varrho\|_{s,u}+\|\slashed{\mathcal L}_{R}^{\leq m-1}\leftidx{^{(R)}}{\slashed\pi}\|_{s,u}\no\\
	&+s^{-1}\|R^{\leq m}\upsilon\|_{s,u}\big)+\dl^{(k-1)(1-\ve_0)}s^{-3/2}\big(s\|R^{\leq m}\mathring L\phi\|_{s,u}+\|R^{\leq m+1}\phi\|_{s,u}\big)	\Big\}\no\\
	&+\dl^{k(1-\ve_0)+\ve_0}(1+\dl^{k(1-\ve_0)-1})s^{-1/2}(s^{-\iota}\sqrt{E_{1,\leq m+2}(s,u)}+\sqrt{E_{2,\leq m+2}(s,u)}).\no
	\end{align}
Collecting \eqref{Lrchi}, \eqref{G3} and \eqref{RL2}, we have
	\begin{equation}\label{RLc}
		\begin{split}
	&\|LR^m(\mu\varrho^2\text{tr}\check\chi)\|_{s,u}\\
	\lesssim&\dl^{k(1-\ve_0)+1/2}s^{-1/2}+\dl^{(k-1)(1-\ve_0)}s^{1/2}\big(s\|\mathring LR^{\leq m}\mathring L\vp\|_{s,u}+\|\mathring LR^{\leq m+1}\vp\|_{s,u})\\
	&+\dl^{k(1-\ve_0)}s^{-\iota}\sqrt{E_{1,\leq m+2}(s,u)}+\dl^{k(1-\ve_0)+\ve_0}(1+\dl^{k(1-\ve_0)-1})s^{-1/2}\sqrt{E_{2,\leq m+2}(s,u)}\\
	&+\dl^{k(1-\ve_0)-\ve_0}s^{1/2}\|R^{\leq m}(\mu\text{tr}\check\chi)\|_{s,u}+\dl^{k(1-\ve_0)}s^{-1}\|R^{\leq m}\mu\|_{s,u}\\
	&+\dl^{k(1-\ve_0)+1}(1+\dl^{k(1-\ve_0)-1})s^{-2}\ln s\|R^{\leq m+1}\mu\|_{s,u}.
	\end{split}
	\end{equation}
On the other hand, by virtue of \eqref{lmu} and the commutation of vector fields, one has
	\begin{align}\label{LbZmu}
	&\delta^l\|\mathring LR^{m+1}\mu\|_{s,u}\no\\
	\lesssim&\dl^{k(1-\ve_0)-1}s^{-1}\big(\|\slashed{\mathcal L}_R^{\leq m}\leftidx{^{(R)}}{\slashed\pi}_{\mathring L}\|_{s,u}+\|R^{m+1}\check L^1\|_{s,u}+\|R^{m+1}\check L^2\|_{s,u}+s^{-1}\|R^{\leq m+1}\upsilon\|_{s,u}\big)\no\\
	&+\dl^{2k(1-\ve_0)-\ve_0}s^{-2}\ln s\|\slashed{\mathcal L}_R^{\leq m-1}\leftidx{^{(R)}}{\slashed\pi}\|_{s,u}+\dl^{k(1-\ve_0)}s^{-2}\ln s\|R^{\leq m+1}\mu\|_{s,u}\no\\
	&+\dl^{(k-1)(1-\ve_0)-1}s^{-1/2}\big(\|R^{\leq m+1}\vp\|_{s,u}+\dl\|TR^{\leq m+1}\vp\|_{s,u}+\dl\|\slashed dR^{\leq m}\mathring L\vp\|_{s,u}\big)\no\\
	&+\dl^{k(1-\ve_0)}s^{-2}\|\slashed{\mathcal L}_R^{\leq m}\leftidx{^{(R)}}{\slashed\pi}_{T}\|_{s,u}\no\\
	\lesssim&\delta^{2k(1-\varepsilon_0)-1/2}s^{-3/2}\ln s+\delta^{k(1-\varepsilon_0)-1}\|R^{\leq m}(\mu\text{tr}\check\chi)\|_{s,u}\no\\
	&+\delta^{(k-1)(1-\varepsilon_0)}s^{-1/2}\Big\{s^{-\iota}\sqrt{E_{1,\leq m+2}(s,u)}+\sqrt{E_{2,\leq m+2}(s,u)}\Big\}\\
	&+\delta^{k(1-\varepsilon_0)}s^{-2}\ln s\|R^{\leq m+1}\mu\|_{s,u}+\delta^{2k(1-\varepsilon_0)-1}s^{-2}\ln s\|R^{\leq m}\mu\|_{s,u}.\no
	\end{align}

	Set $F(s,u,\vartheta)=\varrho(s,u)^2({R}^{m}(\mu\textrm{tr}\check\chi))(s,u,\vartheta)-\varrho(t_0,u)^2({R}^{m}(\mu\textrm{tr}\check\chi))(t_0,u,\vartheta)$ in \eqref{Ff}.
Since for any 2-form $\xi$ on $S_{s,u}$, $|\mathring L(\varrho^2\textrm{tr}\xi)|\lesssim\varrho^2|\slashed{\mathcal L}_{\mathring L}\xi|+\delta^{k(1-\varepsilon_0)}s^{-2}\ln s(\varrho^2|\textrm{tr}\xi|)$ holds. Then
	applying \eqref{RLc} and Grownwall's equality simultaneously yields
	\begin{align}\label{tchi}
	&s^{3/2}\|{R}^{m}(\mu\textrm{tr}\check\chi)\|_{s,u}\no\\
	\lesssim&\delta^{k(1-\varepsilon_0)+1/2}\ln s
	+\delta^{k(1-\varepsilon_0)+1}(1+\dl^{k(1-\ve_0)-1})\int_{t_0}^s\tau^{-5/2}\ln\tau\|R^{\leq m+1}\mu\|_{\tau,u}d\tau\no\\
	&+\delta^{k(1-\varepsilon_0)}s^{1/2-\iota}\sqrt{\tilde E_{1,\leq m+2}(s,u)}+\dl^{k(1-\ve_0)+\ve_0}(1+\dl^{k(1-\ve_0)-1})\ln s\sqrt{\tilde E_{2,\leq m+2}(s,u)}\\
	&+\dl^{k(1-\ve_0)}\int_{t_0}^s\tau^{-3/2}\|R^{\leq m}\mu\|_{\tau,u}d\tau+\dl^{(k-1)(1-\ve_0)}s^{1/2}\sqrt{\int_0^uF_{1,\leq m+2}(s,u')du'}.\no
	\end{align}

	Analogously, \eqref{LbZmu} shows that
	\begin{equation}\label{tRmu}
	\begin{split}
	s^{-1/2}\|R^{m+1}\mu\|_{s,u}\lesssim&\delta^{k(1-\varepsilon_0)-\ve_0+1/2}
	+\delta^{k(1-\varepsilon_0)-1}\int_{t_0}^s\tau^{-1/2}\|{R}^{\leq m}(\mu\textrm{tr}\check\chi)\|_{\tau,u}d\tau\\
	&+\delta^{2k(1-\varepsilon_0)-1}\int_{t_0}^s\tau^{-5/2}\ln\tau\|R^{\leq m}\mu\|_{\tau,u}d\tau\\
	&+\delta^{(k-1)(1-\varepsilon_0)}\sqrt{\tilde E_{1,\leq m+2}(s,u)}+\dl^{(k-1)(1-\ve_0)}\ln s\sqrt{\tilde E_{2,\leq m+2}(s,u)}.
	\end{split}
	\end{equation}
	When $m\leq N-1$, it follows from Proposition \ref{LTRh} that $|R^{m+1}\mu|\leq \dl^{k(1-\ve_0)-\ve_0}$, which implies $\|R^{m+1}\mu\|_{s,u}\leq\dl^{k(1-\ve_0)-\ve_0+1/2}s^{1/2}$. For $N\leq m\leq 2N-5$, one can make the following induction:
	\begin{equation}\label{induction}
	\begin{split}
	s^{-1/2}\|R^{m}\mu\|_{s,u}\lesssim&\delta^{k(1-\varepsilon_0)-\ve_0+1/2}+\delta^{(k-1)(1-\varepsilon_0)}\sqrt{\tilde E_{1,\leq m+2}(s,u)}\\
	+\dl^{(k-1)(1-\ve_0)}&\ln s\sqrt{\tilde E_{2,\leq m+2}(s,u)}+\dl^{(2k-1)(1-\ve_0)-1}\sqrt{\int_0^uF_{1,\leq m+2}(s,u')du'}.
	\end{split}
	\end{equation}
	Substitute \eqref{induction} into the right hands of \eqref{tchi} and \eqref{tRmu} to obtain
	\begin{align*}
	&s^{3/2}\|{R}^{m}(\mu\textrm{tr}\check\chi)\|_{s,u}\\
	\lesssim&\delta^{k(1-\varepsilon_0)+1/2}\ln s
	+\delta^{k(1-\varepsilon_0)+1}(1+\dl^{k(1-\ve_0)-1})\int_{t_0}^s\tau^{-5/2}\ln\tau\|R^{\leq m+1}\mu\|_{\tau,u}d\tau\\
	&+\delta^{k(1-\varepsilon_0)}s^{1/2-\iota}\sqrt{\tilde E_{1,\leq m+2}(s,u)}+\dl^{k(1-\ve_0)+\ve_0}(1+\dl^{k(1-\ve_0)-1})\ln^2 s\sqrt{\tilde E_{2,\leq m+2}(s,u)}\\
	&+\dl^{(k-1)(1-\ve_0)}s^{1/2}\sqrt{\int_0^u\tilde F_{1,\leq m+2}(s,u')du'}
	\end{align*}
	and
	\begin{align*}
	&s^{-1/2}\|R^{m+1}\mu\|_{s,u}\\
	\lesssim&\delta^{k(1-\varepsilon_0)-\ve_0+1/2}
	+\delta^{k(1-\varepsilon_0)-1}\int_{t_0}^s\tau^{-1/2}\|{R}^{\leq m}(\mu\textrm{tr}\check\chi)\|_{\tau,u}d\tau\\
	&+\delta^{(k-1)(1-\varepsilon_0)}\Big(\sqrt{\tilde E_{1,\leq m+2}}+\ln s\sqrt{\tilde E_{2,\leq m+2}}\Big)+\dl^{(4k-1)(1-\ve_0)-2}\sqrt{\int_0^u\tilde F_{1,\leq m+2}(s,u')du'}.
	\end{align*}
	And hence,
	\begin{align}
	\|{R}^{m}(\mu\textrm{tr}\check\chi)\|_{s,u}&\lesssim \delta^{k(1-\varepsilon_0)+1/2}s^{-3/2}\ln s+\delta^{k(1-\varepsilon_0)}s^{-1-\iota}\sqrt{\tilde E_{1,\leq m+2}(s,u)}\no\\
	&+\dl^{k(1-\ve_0)+\ve_0}(1+\dl^{k(1-\ve_0)-1})s^{-3/2}\ln^2 s\sqrt{\tilde E_{2,\leq m+2}(s,u)}\no\\
	&+\dl^{(k-1)(1-\ve_0)}s^{-1}\sqrt{\int_0^u\tilde F_{1,\leq m+2}(s,u')du'},\label{LRc1}\\
	\|R^{m+1}\mu\|_{s,u}&\lesssim \delta^{k(1-\varepsilon_0)-\ve_0+1/2}s^{1/2}+\delta^{(k-1)(1-\varepsilon_0)}s^{1/2}\Big\{\sqrt{\tilde E_{1,\leq m+2}}+\ln s\sqrt{\tilde E_{2,\leq m+2}}\Big\}\no\\
	&+\dl^{(2k-1)(1-\ve_0)-1}s^{1/2}\sqrt{\int_0^u\tilde F_{1,\leq m+2}(s,u')du'}.\label{Rkmu}
	\end{align}
	Thus,
	\begin{equation}\label{LRc}
	\|\slashed{\mathcal L}_{R}^m\check\chi\|_{s,u}\lesssim\delta^{k(1-\varepsilon_0)+1/2}s^{-3/2}\ln s
	+s^{-1}\mathscr E_{m+2}(s,u)
	\end{equation}
	and then \eqref{induction} hold.
	
	If $\slashed{\mathcal L}_Z^m\check\chi=\slashed{\mathcal L}_Z^{m_1}\slashed{\mathcal L}_T\slashed{\mathcal L}_R^{m_2}\check\chi$ with $m_1+m_2=m-1$, then $\slashed{\mathcal L}_Z^m\check\chi=\slashed{\mathcal L}_Z^{m_1}[\slashed{\mathcal L}_T,\slashed{\mathcal L}_R^{m_2}]\check\chi+\slashed{\mathcal L}_Z^{m_1}\slashed{\mathcal L}_R^{m_2}\slashed{\mathcal L}_T\check\chi$ holds. Hence
	it follows from Lemma \ref{com}, \eqref{Tchi'} and  \eqref{Zk1L}-\eqref{ZkRp} that
	\begin{equation}\label{TRchi}
	\begin{split}
	&\delta^l\|\slashed{\mathcal L}_Z^m\check\chi\|_{s,u}\\
	\lesssim&\delta^{k(1-\varepsilon_0)+1/2}s^{-3/2}+\delta^{k(1-\varepsilon_0)}s^{-1}\delta^{l_1}\|\slashed{\mathcal L}_Z^{m_1}\check\chi\|_{s,u}+\delta s^{-2}\delta^{l_0}\|Z^{m_0}\mu\|_{s,u}\\
	&+\delta^{k(1-\varepsilon_0)+\varepsilon_0}s^{-(k+1)/2}\Big\{s^{-\iota}\sqrt{E_{1,\leq m+2}}+\sqrt{E_{2,\leq m+2}}\Big\}.
	\end{split}
	\end{equation}

	If $\slashed{\mathcal L}_Z^m\check\chi=\slashed{\mathcal L}_Z^{m_1}\slashed{\mathcal L}_{\varrho\mathring L}\slashed{\mathcal L}_R^{m_2}\check\chi$ with $m_1+m_2=m-1$, then $\slashed{\mathcal L}_Z^m\check\chi=\slashed{\mathcal L}_Z^{m_1}[\slashed{\mathcal L}_{\varrho\mathring L},\slashed{\mathcal L}_R^{m_2}]\check\chi+\slashed{\mathcal L}_Z^{m_1}\slashed{\mathcal L}_R^{m_2}(\varrho\slashed{\mathcal L}_{\mathring L}\check\chi)$, and it follows from \eqref{Lchi'} that
	\begin{equation}\label{rLRchi}
	\begin{split}
	&\delta^l\|\slashed{\mathcal L}_Z^m\check\chi\|_{s,u}\\
	\lesssim&\delta^{k(1-\varepsilon_0)+1/2}s^{-3/2}+\delta^{k(1-\varepsilon_0)}s^{-1}\ln s\delta^{l_1}\|\slashed{\mathcal L}_Z^{m_1}\check\chi\|_{s,u}+\delta^{k(1-\varepsilon_0)+1} s^{-3}\ln s\delta^{l_0}\|Z^{m_0}\mu\|_{s,u}\\
	&+\delta^{k(1-\varepsilon_0)+\varepsilon_0}s^{-(k+1)/2}\Big\{s^{-\iota}\sqrt{E_{1,\leq m+2}}+\sqrt{E_{2,\leq m+2}}\Big\}.
	\end{split}
	\end{equation}

	Therefore, for any $Z\in\{\varrho\mathring L,T,R\}$, by \eqref{LRc}, \eqref{TRchi} and \eqref{rLRchi}, we obtain
	\begin{equation}\label{zchi}
	\begin{split}
	\delta^l\|\slashed{\mathcal L}_Z^m\check\chi\|_{s,u}\lesssim& \delta^{k(1-\varepsilon_0)+1/2}s^{-3/2}\ln s+\delta s^{-2}\delta^{l_0}\|Z^{m_0}\mu\|_{s,u}+s^{-1}\mathscr E_{m+2}(s,u).
	\end{split}
	\end{equation}

\item For $Z^{m+1}\mu=\bar Z^{m+1}\mu$ with $\bar Z\in\{T, R\}$,
taking $F(s,u,\vartheta)=\delta^l\bar Z^{m+1}\mu(s,u,\vartheta)-\delta^l\bar Z^{m+1}\mu(t_0,u,\vartheta)$ in \eqref{Ff},
meanwhile using
\begin{equation*}
\begin{split}
&\delta^l\|\mathring L\bar Z^{m+1}\mu\|_{s,u}\\
\lesssim&\delta^{k(1-\varepsilon_0)-\varepsilon_0}s^{-1}\delta^{l_1}\|\slashed{\mathcal L}_{\bar Z}^{m_1}\leftidx{^{(R)}}{\slashed\pi}_{\mathring L}\|_{s,u}+\delta^{k(1-\varepsilon_0)}s^{-2}\ln s\delta^{l_0}\|\bar Z^{m_0}\mu\|_{s,u}\\
&+\delta^{(k-2)(1-\varepsilon_0)-\ve_0}s^{-1/2}\delta^{l_0}\|\bar Z^{ m_0}\varphi\|_{s,u}+\delta^{(k+1)(1-\varepsilon_0)}s^{-1}\delta^{l_1}\|\slashed{\mathcal L}_{\bar Z}^{k_1}\leftidx{^{(T)}}{\slashed\pi}_{\mathring L}\|_{s,u}\\
&+\delta^{k(1-\varepsilon_0)-1}s^{-k/2}\delta^{l_0}\|\bar Z^{m_0}\check L^i\|_{s,u}+\delta^{k(1-\varepsilon_0)}s^{-2}\delta^{l_1}\|\slashed{\mathcal L}_{\bar Z}^{m_1}\leftidx{^{(R)}}{\slashed\pi}_T\|_{s,u}\\
&+\delta^{k(1-\varepsilon_0)-1}s^{-k/2}\delta^{l_0}\|\slashed{\mathcal L}_{\bar Z}^{m_0}\slashed dx^i\|_{s,u}+\delta^{(k-1)(1-\varepsilon_0)}s^{-(k-1)/2}\delta^{l_0}\|\mathring L\bar Z^{m_0}\varphi\|_{s,u}\\
&+\delta^{(k-1)(1-\varepsilon_0)}s^{-1/2}\delta^{l_0}\|T\bar Z^{m_0}\varphi\|_{s,u}+\delta^{2k(1-\varepsilon_0)-\varepsilon_0}s^{-2}\ln s\delta^{l_1}\|\slashed{\mathcal L}_{\bar Z}^{m_1}\slashed g\|_{s,u}\\
\lesssim&\delta^{k(1-\varepsilon_0)-1/2}s^{-1/2}+\delta^{k(1-\varepsilon_0)-1}\delta^{l_1}\|\slashed{\mathcal L}_{\bar Z}^{m_1}\check\chi\|_{s,u}+\delta^{k(1-\varepsilon_0)}s^{-2}\ln s\delta^{l_0}\|\bar Z^{m_0}\mu\|_{s,u}\\
&+\delta^{(k-1)(1-\varepsilon_0)}s^{-1/2}\Big\{s^{-\iota}\sqrt{E_{1,\leq m+2}(s,u)}+\sqrt{E_{2,\leq m+2}(s,u)}\Big\}
\end{split}
\end{equation*}

and \eqref{zchi}, one has
	\begin{equation}\label{bZmu}
	\begin{split}
	\delta^l\|\bar Z^{m+1}\mu\|_{s,u}
	\lesssim&\delta^{k(1-\varepsilon_0)-1/2}s^{1/2}\ln s+\delta^{(k-1)(1-\varepsilon_0)}s^{1/2}\Big\{\sqrt{\tilde E_{1,\leq m+2}}+\ln s\sqrt{\tilde E_{2,\leq m+2}}\Big\}\\
	&+\dl^{(2k-1)(1-\ve_0)-1}s^{1/2}\sqrt{\int_0^u\tilde F_{1,\leq m+2}(s,u')du'}.
	\end{split}
	\end{equation}

	If $Z^{m+1}\mu=Z^{m_1}(\varrho\mathring L)\bar Z^{m_2}\mu$ with $m_1+m_2=m$, with the help of \eqref{lmu},
	then it holds that
	\begin{equation}\label{ZLmu}
	\begin{split}
	&\delta^l\|Z^{m+1}\mu\|_{s,u}=\delta^l\|Z^{m_1}[\varrho\mathring L,\bar Z^{m_2}]\mu+Z^{m_1}\bar Z^{m_2}(\varrho\mathring L)\mu\|_{s,u}\\
	\lesssim&\delta^{k(1-\varepsilon_0)-1/2}s^{1/2}+\delta^{k(1-\varepsilon_0)-1}s\delta^{l_1}\|\slashed{\mathcal L}_Z^{m_1}\check\chi\|_{s,u}+\delta^{k(1-\varepsilon_0)} s^{-1}\ln s\delta^{l_0}\|Z^{m_0}\mu\|_{s,u}\\
	&+\delta^{(k-1)(1-\varepsilon_0)} \Big\{s^{-\iota}\sqrt{E_{1,\leq m+2}(s,u)}+\sqrt{E_{2,\leq m+2}(s,u)}\Big\}\\
	\lesssim&\delta^{k(1-\varepsilon_0)-1/2}s^{1/2}+\delta^{k(1-\varepsilon_0)} s^{-1}\ln s\delta^{l_0}\|Z^{m_0}\mu\|_{s,u}\\
	&+\delta^{(k-1)(1-\varepsilon_0)} \Big\{s^{-\iota}\sqrt{E_{1,\leq m+2}(s,u)}+\sqrt{E_{2,\leq m+2}(s,u)}\Big\},
	\end{split}
	\end{equation}
	where \eqref{zchi} is  used in the last inequality.
	
	Thus, for any $Z\in\{\varrho\mathring L,T,R\}$, it follows from \eqref{bZmu} and \eqref{ZLmu} that
	\begin{equation}\label{Fmu}
	\begin{split}
	\delta^l\|Z^{m+1}\mu\|_{s,u}
	\lesssim&\delta^{k(1-\varepsilon_0)-1/2}s^{1/2}\ln s+\delta^{(k-1)(1-\varepsilon_0)}s^{1/2}\Big\{\sqrt{\tilde E_{1,\leq m+2}}+\ln s\sqrt{\tilde E_{2,\leq m+2}}\Big\}\\
	&+\dl^{(2k-1)(1-\ve_0)-1}s^{1/2}\sqrt{\int_0^u\tilde F_{1,\leq m+2}(s,u')du'},
	\end{split}
	\end{equation}
	which, together with \eqref{zchi}, implies that
	\begin{equation}\label{Fchi}
	\begin{split}
	\delta^l\|\slashed{\mathcal L}_Z^m\check\chi\|_{s,u}\lesssim&
	\delta^{k(1-\varepsilon_0)+1/2}s^{-3/2}\ln s+s^{-1}\mathscr E_{m+2}(s,u).
	\end{split}
	\end{equation}
	
	Substituting \eqref{Fmu}-\eqref{Fchi} into \eqref{Zk1L}-\eqref{ZkRp}
	yields the proof of Proposition \ref{L2chi}.
\end{proof}

It follows from the $L^\infty$ estimates in Subsection \ref{BA} that $|Z^{m+1}\mu|\lesssim\delta^{k(1-\ve_0)-\ve_0-l}$ holds
when $m\leq N-3$, which implies $\delta^l\|Z^{m+1}\mu\|_{s,u}\lesssim\delta^{k(1-\ve_0)-\ve_0+1/2}s^{1/2}$ $(m\leq N-3)$ due to $\|1\|_{s,u}\lesssim\dl^{1/2}s^{1/2}$.
Note that the upper bound $\delta^{k(1-\ve_0)-1/2-l}s^{1/2}\ln s$ of the $L^2$ norm of $Z^{m+1}\mu$ for $m\le 2N-6$
(see Proposition \ref{L2chi})
is much worse than $\delta^{k(1-\ve_0)-\ve_0+1/2-l}s^{1/2}$.
This leads to some essential difficulties for obtaining the global uniform estimates of $\p\phi$ in $A_{2\dl}$ when $\ve_0$ is near $\ve_k^*$.
Therefore, one needs to improve the $L^2$ norm of $Z^{m+1}\mu$ for $N-2\leq m\leq 2N-6$,
which will be achieved by studying the transport equation \eqref{lmu} carefully with special attention to the term $\f12G_{\mathring L\mathring L}^\gamma T\varphi_\gamma$ in  \eqref{lmu}. According to \eqref{GT} and \eqref{gLLL}, one may estimate  $\|Z^{m+1}\mathring L\phi\|_{s,u}$ and $\|\slashed{\mathcal L}_Z^{m+1}\slashed d\phi\|_{s,u}$, and even further obtain better estimates on the $L^2$ norm
of $Z^{m+1}\mathring L\mu$. It is emphasized here that one cannot treat $\mathring L\phi$ simply as $\mathring L^\al\varphi_\al$,
otherwise, the better smallness and time-decay rate of $\mathring L\phi$ will be lost. The same caution applies to $\slashed d\phi$.

We now apply \eqref{SiSi} to establish the estimates of $\|Z^{m+1}\mathring L\phi\|_{s,u}$ and $\|\slashed{\mathcal L}_Z^{m+1}\slashed d\phi\|_{s,u}$.

\begin{lemma}\label{Ldphi}
	Under the assumption $(\star)$, when $\delta>0$ is small, it holds that for $m\leq 2N-6$,
	\begin{equation}\label{ZLphi}
	\begin{split}
	\delta^l\|\bar Z^{m+1}\mathring L\phi\|_{s,u}
	&\lesssim\delta^{5/2-\ve_0}s^{-1}+\delta s^{-\iota}\sqrt{\tilde E_{1,\leq m+2}(s,u)}+\delta^2s^{-1}\sqrt{\tilde E_{2,\leq m+2}(s,u)}\\
	&+(\delta^{3-\ve_0}+\delta^{(k+1)(1-\ve_0)+1})s^{-1/2}\|R^m\slashed\triangle\mu\|_{s,u}+\delta^{2-\ve_0}s^{-3/2}\delta^{l'}\|\bar Z^{m'+1}\mu\|_{s,u}\\
	&+\dl^{k(1-\ve_0)}(\dl+\dl^{k(1-\ve_0)})s^{-3/2}\sqrt{\int_0^u \tilde F_1(s,u')du'},
	\end{split}
	\end{equation}
	\begin{equation}\label{Zdphi}
	\begin{split}
	\delta^{l}\|\slashed{\mathcal L}_{\bar Z}^{m+1}\slashed d\phi\|_{s,u}\lesssim&\delta^{5/2-\ve_0}s^{-1}+\delta s^{-\iota}\sqrt{\tilde E_{1,\leq m+2}(s,u)}+\delta^2s^{-1}\sqrt{\tilde E_{2,\leq m+2}(s,u)}\\
	&+\delta^{2-\ve_0}s^{1/2}\|R^m\slashed\triangle\mu\|_{s,u}+\delta^{2-\ve_0}s^{-3/2}\delta^{l'}\|\bar Z^{m'+1}\mu\|_{s,u},\qquad\qquad\qquad\qquad\\
	&+\dl^{k(1-\ve_0)+1}s^{-3/2}\sqrt{\int_0^u \tilde F_1(s,u')du'}
	\end{split}
	\end{equation}
		\begin{equation}\label{ZL2phi}
	\begin{split}
	\delta^l\|\bar Z^{m}\mathring L^2\phi\|_{s,u}
	&\lesssim\delta^{5/2-\ve_0}s^{-2}+\delta s^{-1-\iota}\sqrt{\tilde E_{1,\leq m+2}(s,u)}+\delta^2s^{-2}\sqrt{\tilde E_{2,\leq m+2}(s,u)}\\
	+&\delta^{(k+1)(1-\ve_0)+1}(\delta+\delta^{k(1-\ve_0)})s^{-5/2}\|R^m\slashed\triangle\mu\|_{s,u}+\delta^{2-\ve_0}s^{-5/2}\delta^{l'}\|\bar Z^{m'+1}\mu\|_{s,u}\\
	&+\dl^{k(1-\ve_0)}(\dl+\dl^{k(1-\ve_0)})s^{-5/2}\sqrt{\int_0^u \tilde F_1(s,u')du'},
	\end{split}
	\end{equation}
	where $\bar Z\in\{T,R\}$, the numbers of $T$ in $\bar Z^{m+1}$ and $\bar Z^{m'+1}$ are $l$ and $l'$ respectively, $m'\leq m$.
\end{lemma}

\begin{proof}
	For $\bar Z^{m+1}=R^{m+1}$, due to \eqref{SiSi}, one has
	\begin{equation}\label{ZLp}
	\begin{split}
	&\|R^{m+1}\mathring L\phi\|_{s,u}\lesssim\delta\|TR^{m+1}\mathring L\phi\|_{s,u}\\
	\lesssim&\delta\|R^{m+1}T(\mathring L^\al\varphi_\al)\|_{s,u}+\delta^{3-\ve_0}s^{-5/2}\|\slashed{\mathcal L}_{R}^{m_1}\leftidx{^{(R)}}{\slashed\pi}_T\|_{s,u}+\delta s^{-1}\|R^{m_0}\mathring L\phi\|_{s,u},
	\end{split}
	\end{equation}
	where $T(\mathring L^\al \varphi_\al)=(T\mathring L^i)\varphi_i+\mu\tilde T^i(\mathring L\varphi_i)$.
Recall that $T\mathring L^i$ satisfies \eqref{TL}. Then
	\begin{equation}\label{TLp}
	T(\mathring L^\al \varphi_\al)=\mu\tilde T^i(\mathring L\varphi_i)+(\slashed d_X\mu)\slashed d^X\phi-(G_{X\mathring L}^\gamma T\varphi_\gamma)\slashed d^X\phi+\f12(G_{\mathring L\mathring L}^\gamma T\varphi_\gamma)\tilde T^i\varphi_i+\dots.
	\end{equation}
	With the help of \eqref{GT} and \eqref{gLLL}, after substituting \eqref{TLp} into \eqref{ZLp} and absorbing all $\|R^{m_0}\mathring L\phi\|_{s,u}$ in the right hand side by $\sum_{m\leq 2N-6}\|R^{m+1}\mathring L\phi\|_{s,u}$ in \eqref{ZLp}, one obtains
	\begin{align}
	&\|R^{m+1}\mathring L\phi\|_{s,u}\no\\
	\lesssim&(\delta^{2-\ve_0}+\delta^{(k+1)(1-\ve_0)})s^{-3/2}\big(\|R^{m_0}\check L^1\|_{s,u}+\|R^{m_0}\check L^2\|_{s,u}\big)+\delta^{(k+1)(1-\ve_0)}s^{-3/2}\|R^{m_0}\check\varrho\|_{s,u}\no\\
	&+\delta^{2-\ve_0}s^{-5/2}\|R^{m_{-1}}x\|_{s,u}+\delta^{(k+1)(1-\ve_0)+1}s^{-3/2}\|TR^{m_0}\varphi\|_{s,u}\no\\
	&+\delta^{k(1-\ve_0)+2}s^{-2}\|\slashed dR^{m_0}\varphi\|_{s,u}+\delta^{2-\ve_0}s^{-3/2}\|\slashed{\mathcal L}_{R}^{m_1}\leftidx{^{(R)}}{\slashed\pi}_{\mathring L}\|_{s,u}+\delta s^{-1}\|R^{m_0}\varphi\|_{s,u}\label{RLphi}\\
	&+\delta\|\mathring LR^{m_0}\varphi\|_{s,u}+\delta^{3-\ve_0}s^{-5/2}\|\slashed{\mathcal L}_{R}^{m_1}\leftidx{^{(R)}}{\slashed\pi}_T\|_{s,u}+\delta^{k(1-\ve_0)}s^{-1}\|\slashed dR^{m_0}\phi\|_{s,u}\no\\
	&+\delta^{3-\ve_0}s^{-3/2}\|\slashed dR^{m+1}\mu\|_{s,u}+\delta^{2-\ve_0}s^{-3/2}\|R^{m'+1}\mu\|_{s,u}.\no
	\end{align}
	where $1\leq m_p\leq m+1-p$ with $p=-1,0,1$. Using Proposition \ref{L2chi} to estimate the second to fourth lines in \eqref{RLphi}
yields
	\begin{equation}\label{RLp}
	\begin{split}
	\|R^{m+1}\mathring L&\phi\|_{s,u}\lesssim\delta^{5/2-\ve_0}s^{-1}+\delta s^{-\iota}\sqrt{\tilde E_{1,\leq m+2}(s,u)}
+\delta^2s^{-1}\sqrt{\tilde E_{2,\leq m+2}(s,u)}\\
	&+\delta^{k(1-\ve_0)}s^{-1}\|\slashed dR^{m_0}\phi\|_{s,u}+\delta^{3-\ve_0}s^{-3/2}\|\slashed dR^{m+1}\mu\|_{s,u}\\
	&+\delta^{2-\ve_0}s^{-3/2}\|R^{m'+1}\mu\|_{s,u}+\dl^{k(1-\ve_0)}(\dl+\dl^{k(1-\ve_0)})s^{-3/2}\sqrt{\int_0^u \tilde F_1(s,u')du'}.
	\end{split}
	\end{equation}
	
	If $\bar Z^{m+1}=\bar Z^{p_1}TR^{p_2}$ with $p_1+p_2=m$, then $\bar Z^{m+1}\mathring L\phi=\bar Z^{p_1}R^{p_2}T(\mathring L^\al\varphi_\al)+\bar Z^{p_1}[T,R^{p_2}]\mathring L\phi$. Using \eqref{TLp} again, one can get similarly that
	\begin{equation}\label{ZTp}
	\begin{split}
	\delta^l\|\bar Z^{m+1}\mathring L\phi\|_{s,u}\lesssim&\delta^{5/2-\ve_0}s^{-1}+\delta s^{-\iota}\sqrt{\tilde E_{1,\leq m+2}(s,u)}+\delta^2s^{-1}\sqrt{\tilde E_{2,\leq m+2}(s,u)}\\
	+\delta^{k(1-\ve_0)}&s^{-1}\delta^{l_0}\|\slashed d\bar Z^{m_0}\phi\|_{s,u}+\delta^{2-\ve_0}s^{-3/2}\delta^{l'}\|\bar Z^{m'+1}\mu\|_{s,u}\\
	&+\dl^{k(1-\ve_0)}(\dl+\dl^{k(1-\ve_0)})s^{-3/2}\sqrt{\int_0^u \tilde F_1(s,u')du'}.
	\end{split}
	\end{equation}
	Thus, for any $\bar Z^{m+1}$ which contains $l$ vector fields $T$, it follows from \eqref{RLp} and \eqref{ZTp} that
		\begin{equation}\label{bZLp}
	\begin{split}
	\delta^l\|\bar Z&^{m+1}\mathring L\phi\|_{s,u}
	\lesssim\delta^{5/2-\ve_0}s^{-1}+\delta s^{-\iota}\sqrt{\tilde E_{1,\leq m+2}(s,u)}+\delta^2s^{-1}\sqrt{\tilde E_{2,\leq m+2}(s,u)}\\
	&+\delta^{k(1-\ve_0)}s^{-1}\delta^{l_0}\|\slashed d\bar Z^{m_0}\phi\|_{s,u}+\delta^{3-\ve_0}s^{-1/2}\| R^{m}\slashed\triangle\mu\|_{s,u}+\delta^{2-\ve_0}s^{-3/2}\delta^{l'}\|\bar Z^{m'+1}\mu\|_{s,u}\\
	&+\dl^{k(1-\ve_0)}(\dl+\dl^{k(1-\ve_0)})s^{-3/2}\sqrt{\int_0^u \tilde F_1(s,u')du'}.
	\end{split}
	\end{equation}
	
Due to $|\slashed{\mathcal L}_{\bar Z}^{m+1}\slashed d\phi|\lesssim s^{-1}|R\bar Z^{m+1}\phi|$,
one needs to estimate the $L^2$ norm of $R\bar Z^{m+1}\phi$. As before, for $\bar Z^{m+1}=R^{m+1}$, there holds
	\begin{equation}\label{Rp}
	\begin{split}
	&\|R^{m+2}\phi\|_{s,u}
	\lesssim\delta\|R^{m+1}TR\phi\|_{s,u}+\delta^{3-\ve_0}s^{-3/2}\|\slashed{\mathcal L}_{R}^{m_1}\leftidx{^{(R)}}{\slashed\pi}_T\|_{s,u}+\delta s^{-1}\|R^{m_0}R\phi\|_{s,u}.
	\end{split}
	\end{equation}
	Note that $TR\phi=\leftidx{^{(R)}}{\slashed\pi}_{TX}\slashed d^X\phi+(R\mu)\tilde T^i\varphi_i-\mu(Rg^{0i})\varphi_i-\mu(R\mathring L^i)\varphi_i+\mu \tilde T^iR\varphi_i$. In addition, it follows from \eqref{Rpi} in Appendix A and \eqref{dL} that
	\begin{equation*}
	\begin{split}
	&\leftidx{^{(R)}}{\slashed\pi}_{TX}\slashed d^X\phi-\mu(R\mathring L^i)\varphi_i\\
	=&\upsilon\slashed d^X\mu\slashed d_X\phi+\mu(G_{X\mathring L}^\g R\varphi_\g)\slashed d^X\phi+\mu(G_{X\tilde T}^\g R\varphi_\g)\slashed d^X\phi+\mu\epsilon_i^a\check T^ig_{ab}(\slashed d_Xx^b\slashed d^X\phi)-\f{\mu}{\varrho}R\phi\\
	&+\upsilon(G_{X\tilde T}^\g T\varphi_\g)\slashed d^X\phi-\f12\upsilon\mu(G_{\tilde T\tilde T}^\g\slashed d^X\varphi_\g)\slashed d_X\phi+\mu(G_{\mathring L\tilde T}^\g R\varphi_\g)\tilde T^i\varphi_i+\f12\mu(G_{\tilde T\tilde T}^\g R\varphi_\g)\tilde T^i\varphi_i.
	\end{split}
	\end{equation*}
	Thus,
	\begin{align}
	\|R^{m+2}\phi\|_{s,u}
	\lesssim&\delta^{2-\ve_0}s^{-1/2}\big\{s\|\slashed dR^{m+1}\mu\|_{s,u}+\|R^{m'+1}\mu\|_{s,u}\big\}+\delta^{3-\ve_0}s^{-3/2}\|\slashed{\mathcal L}_{R}^{m_1}\leftidx{^{(R)}}{\slashed\pi}_T\|_{s,u}\no\\
	&+\delta\|Z^{m_0}\varphi\|_{s,u}+\delta s\|\slashed dR^{m+1}\varphi\|_{s,u}+\delta^{2k(1-\ve_0)+2}s^{-2}\ln s\|TR^{m+1}\varphi\|_{s,u}\no\\
	&+\delta^{(k+1)(1-\ve_0)+1}s^{-5/2}\|R^{m_0}\upsilon\|_{s,u}+\delta^{(k+1)(1-\ve_0)+1}s^{-3/2}\ln s\|\slashed{\mathcal L}_R^{m_0}\slashed g\|_{s,u}\no\\
	&+\delta^{2-\ve_0}s^{-1/2}(\|R^{m_0}\check L^1\|_{s,u}+\|R^{m_0}\check L^2\|_{s,u})+\delta^{2-\ve_0}s^{-3/2}\|R^{m_{-1}}x\|_{s,u}\no\\
	\lesssim\delta^{5/2-\ve_0}&+\delta s^{1-\iota}\sqrt{\tilde E_{1,\leq m+2}(s,u)}+\delta^2\sqrt{\tilde E_{2,\leq m+2}(s,u)}+\delta^{2-\ve_0}s^{1/2}\|\slashed dR^{m+1}\mu\|_{s,u}\label{Rmp}\\
	+\delta^{2-\ve_0}&s^{-1/2}\|R^{m'+1}\mu\|_{s,u}+\dl^{k(1-\ve_0)+1}s^{-1/2}\sqrt{\int_0^u \tilde F_1(s,u')du'}\no.
	\end{align}
	For $\bar Z^{m+1}=\bar Z^{p_1}TR^{p_2}$ with $p_1+p_2=m$ and the number of $T$ in $\bar Z^{p_1}$ being $q_1$,
due to $R\bar Z^{m+1}\mathring \phi=R\bar Z^{p_1}R^{p_2}(\mu \tilde T^i\varphi_i)+R\bar Z^{p_1}[T,R^{p_2}]\phi$, then it holds that
	\begin{equation*}
	\begin{split}
	\delta^l\|R\bar Z^{m+1}\phi\|_{s,u}\lesssim&\delta^{5/2-\ve_0}+\delta^2 s^{-\iota}\sqrt{\tilde E_{1,\leq m+2}(s,u)}+\delta^2\sqrt{\tilde E_{2,\leq m+2}(s,u)}+\delta^{2-\ve_0}s^{-1/2}\delta^{l'}\|\bar Z^{m'+1}\mu\|_{s,u}\\
	&+\delta s^{-2}\ln s\delta^{l_1}\|R\bar Z^{m_1}\phi\|_{s,u}+\dl^{k(1-\ve_0)+1}s^{-1/2}\sqrt{\int_0^u \tilde F_1(s,u')du'},
	\end{split}
	\end{equation*}
	where $1\leq m_1\leq m$ and the number of $T$ in $\bar Z^{m_1}$ is at most $q_1$. Then it follows from induction that
	\begin{equation}\label{RZTp}
	\begin{split}
	\delta^l\|R\bar Z^{m+1}\phi\|_{s,u}\lesssim&\delta^{5/2-\ve_0}+\delta^2 s^{-\iota}\sqrt{\tilde E_{1,\leq m+2}(s,u)}
+\delta^2\sqrt{\tilde E_{2,\leq m+2}(s,u)}\\
	&+\delta^{2-\ve_0}s^{-1/2}\delta^{l'}\|\bar Z^{m'+1}\mu\|_{s,u}+\delta s^{-2}\ln s\|R^{m_1+1}\phi\|_{s,u}\\
	&+\dl^{k(1-\ve_0)+1}s^{-1/2}\sqrt{\int_0^u \tilde F_1(s,u')du'}\\
	\lesssim&\delta^{5/2-\ve_0}+\delta^2 s^{-\iota}\sqrt{\tilde E_{1,\leq m+2}(s,u)}+\delta^2\sqrt{\tilde E_{2,\leq m+2}(s,u)}\\
	&+\delta^{2-\ve_0}s^{-1/2}\delta^{l'}\|\bar Z^{m'+1}\mu\|_{s,u}+\dl^{k(1-\ve_0)+1}s^{-1/2}\sqrt{\int_0^u \tilde F_1(s,u')du'}.
	\end{split}
	\end{equation}
Note that the last inequality in \eqref{RZTp} holds when the number $m+2$ is replaced by $m_1+1$ in \eqref{Rmp}.
Therefore, \eqref{Zdphi} follows from \eqref{Rmp} and \eqref{RZTp}, and furthermore, \eqref{ZLphi} is obtained
by inserting \eqref{Zdphi} into \eqref{bZLp}.
	
Analogously, due to
	$$
	\|R^m\mathring L^2\phi\|_{s,u}\lesssim\delta\|[T,R^m\mathring L]\mathring L\phi\|_{s,u}
+\dl\|R^m\mathring LT(\mathring L^\al\varphi_\al)\|_{s,u},
	$$
	 then it follows from \eqref{TLp} that for  $\tilde Z\in\{\varrho\mathring L, R\}$,
	\begin{equation*}
	\begin{split}
	&\|R^m\mathring L^2\phi\|_{s,u}\\
	\lesssim&(\delta^{2-\ve_0}+\delta^{(k+1)(1-\ve_0)})s^{-5/2}\big(\|\tilde Z^{m_0}\check L^1\|_{s,u}+\|\tilde Z^{m_0}\check L^2\|_{s,u}\big)+\delta^{(k+1)(1-\ve_0)}s^{-5/2}\|\tilde Z^{m_0}\check\varrho\|_{s,u}\\
	&+\delta^{2-\ve_0}s^{-7/2}\|\tilde Z^{m_{-1}}x\|_{s,u}+\delta^{3-\ve_0}s^{-5/2}\|\slashed{\mathcal L}_{R}^{m_1}\leftidx{^{(T)}}{\slashed\pi}_{\mathring L}\|_{s,u}+\delta^{3-\ve_0}s^{-7/2}\|\slashed{\mathcal L}_{R}^{m_1}\leftidx{^{(R)}}{\slashed\pi}_T\|_{s,u}\\
	&+\delta^{k(1-\ve_0)}s^{-2}\|R^{m_0}\mathring L\phi\|_{s,u}+\delta^{k(1-\ve_0)}s^{-3}\|R^{m_0}\phi\|_{s,u}+\delta^{3-\ve_0}s^{-3/2}\|\slashed dR^{m}\mathring L\mu\|_{s,u}\\
	&+\delta^{2-\ve_0}s^{-3/2}\|R^{m_0}\mathring L\mu\|_{s,u}+\delta^{2-\ve_0}s^{-5/2}\|R^{m_0}\mu\|_{s,u}+\dl s^{-1-\iota}\sqrt{E_{1,\leq m+2}}+\dl^2 s^{-2}\sqrt{E_{2,\leq m+2}}.
	\end{split}
	\end{equation*}
	Utilizing \eqref{lmu}, \eqref{GT} and \eqref{gLLL} leads to
	\begin{equation*}
	\begin{split}
	&\|R^m\mathring L^2\phi\|_{s,u}\\
	\lesssim&(\delta^{2-\ve_0}+\delta^{(k+1)(1-\ve_0)})s^{-5/2}\big(\|\tilde Z^{m_0}\check L^1\|_{s,u}+\|\tilde Z^{m_0}\check L^2\|_{s,u}\big)+\delta^{(k+1)(1-\ve_0)}s^{-5/2}\|\tilde Z^{m_0}\check\varrho\|_{s,u}\\
	&+\delta^{2-\ve_0}s^{-7/2}\|\tilde Z^{m_{-1}}x\|_{s,u}+\delta^{3-\ve_0}s^{-5/2}\|\slashed{\mathcal L}_{R}^{m_1}\leftidx{^{(T)}}{\slashed\pi}_{\mathring L}\|_{s,u}+\delta^{3-\ve_0}s^{-7/2}\|\slashed{\mathcal L}_{R}^{m_1}\leftidx{^{(R)}}{\slashed\pi}_T\|_{s,u}\\
&+\delta^{k(1-\ve_0)}s^{-2}\|R^{m_0}\mathring L\phi\|_{s,u}+\delta^{k(1-\ve_0)+1}s^{-3}\|\slashed dR^{m+1}\phi\|_{s,u}+\delta^{k(1-\ve_0)}s^{-3}\|R^{m_0}\phi\|_{s,u}\\
	&+\delta^{2-\ve_0}s^{-5/2}\|R^{m_0}\mu\|_{s,u}+\dl s^{-1-\iota}\sqrt{E_{1,\leq m+2}}+\dl^2 s^{-2}\sqrt{E_{2,\leq m+2}}\\
	\lesssim&\delta^{5/2-\ve_0}s^{-2}+\delta s^{-1-\iota}\sqrt{\tilde E_{1,\leq m+2}(s,u)}+\delta^2s^{-2}\sqrt{\tilde E_{2,\leq m+2}(s,u)}\\
&+\delta^{(k+1)(1-\ve_0)+1}(\delta+\delta^{k(1-\ve_0)})s^{-5/2}\|R^m\slashed\triangle\mu\|_{s,u}
+\delta^{2-\ve_0}s^{-5/2}\delta^{l'}\|R^{m'+1}\mu\|_{s,u}\\
&+\dl^{k(1-\ve_0)}(\dl+\dl^{k(1-\ve_0)})s^{-5/2}\sqrt{\int_0^u \tilde F_1(s,u')du'},
	\end{split}
	\end{equation*}
	where \eqref{ZLphi} and \eqref{Zdphi} have been used. For the case of $\bar Z^{p_1}TR^{p_2}\mathring L^2\phi$ with $p_1+p_2=m-1$, the desired result can be obtained by using $\bar Z^{p_1}TR^{p_2}\mathring L^2\phi=\bar Z^{p_1}R^{p_2}[T,\mathring L]\mathring L\phi+\bar Z^{p_1}[T,R^{p_2}]\mathring L^2\phi+\bar Z^{p_1}R^{p_2}\mathring LT(\mathring L^\al\varphi_\al)$.
\end{proof}

\eqref{RZTp} implies that when there is at least one $T$ in $\bar Z^{m+1}$, it then holds that
\begin{equation}\label{ZTdphi}
\begin{split}
\delta^{l}\|\slashed{\mathcal L}_{\bar Z}^{m+1}\slashed d\phi\|_{s,u}\lesssim&\delta^{5/2-\ve_0}s^{-1}+\delta^2 s^{-1-\iota}\sqrt{\tilde E_{1,\leq m+2}(s,u)}+\delta^2s^{-1}\sqrt{\tilde E_{2,\leq m+2}(s,u)}\\
&+\delta^{2-\ve_0}s^{-3/2}\delta^{l'}\|\bar Z^{m'+1}\mu\|_{s,u}+\dl^{k(1-\ve_0)+1}s^{-3/2}\sqrt{\int_0^u \tilde F_1(s,u')du'}.
\end{split}
\end{equation}

In addition, it follows from \eqref{Zm+1p} that
\begin{equation}\label{Zmp}
\begin{split}
\delta^l\|Z^{m+1}\phi\|_{s,u}\lesssim&\delta^{5/2-\ve_0}+\delta^2s^{-\iota}\sqrt{\tilde E_{1,\leq m+2}(s,u)}+\delta^2\sqrt{\tilde E_{2,\leq m+2}(s,u)}\\
&+\delta^{2-\ve_0}s^{-1/2}\delta^{l'}\|\bar Z^{m'+1}\mu\|_{s,u}+\dl^{k(1-\ve_0)+1}s^{-1/2}\sqrt{\int_0^u \tilde F_1(s,u')du'}.
\end{split}
\end{equation}

One can check easily from \eqref{ZLphi} and \eqref{Zdphi} that $\|\bar Z^{m+1}\mathring L\phi\|_{s,u}$
has admitted the better smallness order and time decay rate than $\|\slashed{\mathcal L}
_{\bar Z}^{m+1}\slashed d\phi\|_{s,u}$ (see the coefficients of $\|R^m\slashed\triangle\mu\|_{s,u}$).
In addition, it should be pointed out that $\delta^{2-\ve_0}s^{1/2}\|R^m\slashed\triangle\mu\|_{s,u}$ in \eqref{Zdphi} is not enough to
close the energy estimates derived in Section \ref{ert} below  when $\ve_0$ is near $\ve_k^*$.
To improve $L^2$ norm estimate of $\bar Z^{m+1}\mu$ by the transport equation for $\mathring L\bar Z^{m+1}\mu$
 derived from \eqref{lmu}, our main strategy is to explore the possibility that the related term $\f12G_{\mathring L\mathring L}^\g T\varphi_\g$ in \eqref{lmu} may admit
more precise  smallness and better time decay rate information by some suitable combinations of $\slashed d\phi$.
This will be the next main task.

It follows from \eqref{pal} that $\varphi_{\g_j}=\delta_{\g_j}^0\mathring L\phi-\mathring L_{\g_j}\tilde T^{l_j}\varphi_{l_j}
+g_{\g_ja}\slashed d^Xx^a\slashed d_X\phi$. Then
\begin{align}
&g^{\al\beta,\g\g_2\cdots\g_k}\varphi_{\g_2}\cdots\varphi_{\g_k}\no\\
=&(-1)^{k-1}g^{\al\beta,\g\g_2\cdots\g_k}\mathring L_{\g_2}\cdots\mathring L_{\g_k}(\tilde T^{l_2}\cdots \tilde T^{l_k})(\varphi_{l_2}\cdots\varphi_{l_k})\no\\
&+(-1)^k(k-1)g^{\al\beta,\g\g_2\cdots\g_k}(g_{\g_2a}\slashed d^Xx^a\slashed d_X\phi)\underbrace{\mathring L_{\g_3}\cdots\mathring L_{\g_k}(\tilde T^{l_3}\cdots \tilde T^{l_k})(\varphi_{l_3}\cdots\varphi_{l_k})}_{\text{vanish when }k=2}\no\\
&+f(\varphi, \check L^1, \check L^2,\f{x}\varrho)\left(
\begin{array}{ccc}
(\mathring L\phi)\varphi^{k-2}\\
(\slashed d^Xx)(\slashed d_X\phi)(\slashed d^Xx)(\slashed d_X\phi)\varphi^{k-3}\\
\end{array}
\right)\no\\
=&(-1)^{k-1}g^{\al\beta,\g\g_2\cdots\g_k}\o_{\g_2}\cdots\o_{\g_k}(\tilde T^{l_2}\cdots \tilde T^{l_k})(\varphi_{l_2}\cdots\varphi_{l_k})\no\\
&+(-1)^k(k-1)g^{\al\beta,\g a\g_3\cdots\g_k}(m_{ab}\slashed d^Xx^b\slashed d_X\phi)\underbrace{\o_{\g_3}\cdots\o_{\g_k}(\tilde T^{l_3}\cdots \tilde T^{l_k})(\varphi_{l_3}\cdots\varphi_{l_k})}_{\text{vanish when }k=2}\label{gp}\\
&+f(\varphi, \check L^1, \check L^2,\f{x}\varrho,\check\varrho)\left(
\begin{array}{ccc}
\check L^{1}\varphi^{k-1}\\
\check L^{2}\varphi^{k-1}\\
\check\varrho\varphi^{k-1}\\
(\mathring L\phi)\varphi^{k-2}\\
(\slashed d^Xx)(\slashed d_X\phi)\varphi^{2k-2}\\
(\slashed d^Xx)(\slashed d_X\phi)(\slashed d^Xx)(\slashed d_X\phi)\varphi^{k-3}\\
\end{array}
\right).\no
\end{align}
Substituting \eqref{gp} into the third line of \eqref{gLLL} and using the null condition \eqref{null} yield
\begin{align}
&(\p_{\varphi_\gamma}g^{\al\beta})\mathring L_\al\mathring L_\beta\mathring L_\gamma\no\\
=&(-1)^kk(k-1)g^{\al\beta,\g a\g_3\cdots\g_k}\o_\al\o_\beta\o_\g(m_{ab}\slashed d^Xx^b\slashed d_X\phi)\underbrace{\o_{\g_3}\cdots\o_{\g_k}(\tilde T^{l_3}\cdots \tilde T^{l_k})(\varphi_{l_3}\cdots\varphi_{l_k})}_{\text{vanish when }k=2}\no\\
&+f(\varphi, \check L^1, \check L^2,\f{x}\varrho,\check\varrho)\left(
\begin{array}{ccc}
\varphi^k\\
\check L^{1}\varphi^{k-1}\\
\check L^{2}\varphi^{k-1}\\
\check\varrho\varphi^{k-1}\\
(\mathring L\phi)\varphi^{k-2}\\
(\slashed d^Xx)(\slashed d_X\phi)\varphi^{2k-2}\\
(\slashed d^Xx)(\slashed d_X\phi)(\slashed d^Xx)(\slashed d_X\phi)\varphi^{k-3}\\
\end{array}
\right).\label{gLLL2}
\end{align}
Based on \eqref{gLLL2}, we can introduce a good unknown as
\begin{equation}\label{DA}
\mathscr A
=g^{\al\beta,\g a\g_3\cdots\g_k}\o_\al\o_\beta\o_\g(m_{ab}\slashed d^Xx^b\slashed d_X\phi)\o_{\g_3}\cdots\o_{\g_k},
\end{equation}
which is a combination of  $\slashed d\phi$.

\begin{lemma}\label{8.4}
	Under the assumption $(\star)$, when $\delta>0$ is small, it holds that for $m\leq 2N-6$,
	\begin{equation}\label{A}
	\begin{split}
	\delta^l\|&\bar Z^{m+1}\mathscr A\|_{s,u}
	\lesssim\delta^{3/2-\ve_0}(\delta+\delta^{k(1-\ve_0)}\ln s)s^{-1}+\delta s^{-\iota}\sqrt{\tilde E_{1,\leq m+2}(s,u)}\\
	&+\big(\delta^2+\delta^{k(1-\ve_0)}(\dl+\dl^{k(1-\ve_0)})\ln^2 s\big) s^{-1}\sqrt{\tilde E_{2,\leq m+2}(s,u)}+\delta^{2-\ve_0}s^{-3/2}\delta^{l'}\|\bar Z^{m'+1}\mu\|_{s,u}\\
	&+(\delta^{3-\ve_0}+\delta^{(k+1)(1-\ve_0)+1})s^{-1/2}\|R^m\slashed\triangle\mu\|_{s,u}+\dl^{k(1-\ve_0)}s^{-1/2}\sqrt{\int_0^u \tilde F_1(s,u')du'},
	\end{split}
	\end{equation}
	where $\bar Z\in\{T,R\}$, the numbers of $T$ in $\bar Z^{m+1}$ and $\bar Z^{m'+1}$ are $l$ and $l'$ respectively, $m'\leq m$.
\end{lemma}

\begin{remark}
	Note that the coefficients of $\|R^m\slashed\triangle\mu\|_{s,u}$
in \eqref{ZLphi} and \eqref{A} are the same.
We will use \eqref{A} instead of \eqref{ZLphi} to estimate $\|L\bar Z^{m+1}\mu\|_{s,u}$ later.
\end{remark}

\begin{proof}
	Due to $\slashed d^Xx^b\slashed d_X\phi=(g^{bj}-\mathring L^bg^{0j}-\mathring L^jg^{0b}-\mathring L^b\mathring L^j)\varphi_j$ by \eqref{gab} and $m_{ab}\mathring L^b=\o_a+m_{ab}\check L^b+\check\varrho\o_a$, then it holds that
	\begin{equation*}
	\begin{split}
	\mathscr A=&g^{\al\beta,\g a\g_3\cdots\g_k}\o_\al\o_\beta\o_\g m_{ab}(g^{bj}-\mathring L^bg^{0j}-\mathring L^jg^{0b})\varphi_j\o_{\g_3}\cdots\o_{\g_k}\\
	&-g^{\al\beta,\g a\g_3\cdots\g_k}\o_\al\o_\beta\o_\g\o_a\o_{\g_3}\cdots\o_{\g_k}(\mathring L^j\varphi_j)+f(\o)\left(
	\begin{array}{ccc}
	\check L^1\\
	\check L^2\\
	\check\varrho\\
	\end{array}
	\right)(\mathring L^j\varphi_j)\\
	=&g^{\al\beta,\g a\g_3\cdots\g_k}\o_\al\o_\beta\o_\g m_{ab}(g^{bj}-\mathring L^bg^{0j}-\mathring L^jg^{0b})\varphi_j\o_{\g_3}\cdots\o_{\g_k}\\
	&-g^{\al\beta,\g 0\g_3\cdots\g_k}\o_\al\o_\beta\o_\g\o_{\g_3}\cdots\o_{\g_k}(\mathring L^j\varphi_j)+f(\o)\left(
	\begin{array}{ccc}
	\check L^1\\
	\check L^2\\
	\check\varrho\\
	\end{array}
	\right)(\mathring L^j\varphi_j),
	\end{split}
	\end{equation*}
	where the last identity comes from \eqref{null}. This, together with \eqref{SiSi}, yields
	\begin{equation}\label{RA}
	\begin{split}
	&\|R^{m+1}\mathscr A\|_{s,u}\\
	\lesssim&\|R^{m+1}\Big\{f(\o)\left(
	\begin{array}{ccc}
	\check L^1\\
	\check L^2\\
	\check\varrho\\
	\end{array}
	\right)(\mathring L^j\varphi_j)\Big\}\|_{s,u}\\
	&+\delta\|[R^{m+1},T]\big\{g^{\al\beta,\g a\g_3\cdots\g_k}\o_\al\o_\beta\o_\g m_{ab}(g^{bj}-\mathring L^bg^{0j}-\mathring L^jg^{0b})\varphi_j\o_{\g_3}\cdots\o_{\g_k}\\
	&\qquad\qquad\qquad\quad-g^{\al\beta,\g 0\g_3\cdots\g_k}\o_\al\o_\beta\o_\g\o_{\g_3}\cdots\o_{\g_k}(\mathring L^j\varphi_j)\big\}\|_{s,u}\\
	&+\delta\|R^{m+1}\big\{T\big(g^{\al\beta,\g a\g_3\cdots\g_k}\o_\al\o_\beta\o_\g\o_{\g_3}\cdots\o_{\g_k} m_{ab}(g^{bj}-\mathring L^bg^{0j}-\mathring L^jg^{0b})\big)\varphi_j\big\}\|_{s,u}\\
	&+\delta\|R^{m+1}\big\{T\big(g^{\al\beta,\g 0\g_3\cdots\g_k}\o_\al\o_\beta\o_\g\o_{\g_3}\cdots\o_{\g_k}\big)\mathring L^j\varphi_j\big\}\|_{s,u}\\
	&+\delta\|R^{m+1}\big\{g^{\al\beta,\g a\g_3\cdots\g_k}\o_\al\o_\beta\o_\g\o_{\g_3}\cdots\o_{\g_k} m_{ab}(g^{bj}-\mathring L^bg^{0j}-\mathring L^jg^{0b})T\varphi_j\\
	&\qquad\qquad\quad-g^{\al\beta,\g 0\g_3\cdots\g_k}\o_\al\o_\beta\o_\g\o_{\g_3}\cdots\o_{\g_k}(\mathring L^jT\varphi_j)\big\}\|_{s,u}\\
	&+\delta\|R^{m+1}\big\{g^{\al\beta,\g 0\g_3\cdots\g_k}\o_\al\o_\beta\o_\g\o_{\g_3}\cdots\o_{\g_k}(T\mathring L^j)\varphi_j\big\}\|_{s,u}.
	\end{split}
	\end{equation}
	Denote the first five lines in the right hand side of \eqref{RA} as $\mathcal I$, the sixth-seventh
lines as $\mathcal{II}$, and the last line as $\mathcal {III}$, which will be estimated separately as follows.
	\begin{enumerate}
		\item Since $\o^i=\f{1}{\check\varrho+1}\cdot\f{x^i}\varrho$, $T\o^i=\f{\mu\tilde T^i}{\varrho(\check\varrho+1)}-\sum_{j=1}^3\f{\mu x^ix^j\tilde T^j}{\varrho^3(\check\varrho+1)^3}$ and $T\mathring L^i$ satisfies \eqref{TL}, then
		\begin{equation}\label{I}
		\begin{split}
		\mathcal I\lesssim&\delta^{1-\ve_0}s^{-1/2}\big(\sum_{i=1}^2\|R^{m_0}\check L^i\|_{s,u}+\|R^{m_0}\check\varrho\|_{s,u}\big)+\delta^{k(1-\ve_0)}s^{-1}\ln s\|R^{m_0}\varphi\|_{s,u}\\
		&+(\delta^{(k+1)(1-\ve_0)}\ln s+\delta^{2-\ve_0})s^{-5/2}\|R^{m_{-1}}x\|_{s,u}+\delta^{2-\ve_0}s^{-3/2}\|\slashed{\mathcal L}_R^{m_1}\leftidx{^{(R)}}{\slashed\pi}_T\|_{s,u}\\
		&+\delta^{2-\ve_0}s^{-5/2}\ln s(\delta^{k(1-\ve_0)}+\delta^{2k(1-\ve_0)-1})\|\slashed{\mathcal L}_R^{m_1}\slashed g\|_{s,u}+\boxed{\delta^{2-\ve_0}s^{-3/2}\|R^{m'+1}\mu\|_{s,u}}\\
		&+\boxed{\delta^{(k+1)(1-\ve_0)+1}s^{-3/2}\|\slashed dR^{m+1}\mu\|_{s,u}}+\delta^{k(1-\ve_0)+1}s^{-1}\|TR^{m'+1}\varphi\|_{s,u}\\
		&+\delta^{2k(1-\ve_0)+1}s^{-2}\|\slashed dR^{m+1}\varphi\|_{s,u}+\delta^{2k(1-\ve_0)+1}s^{-2}\|R^{m'+1}\mathring L\varphi\|_{s,u},
		\end{split}
		\end{equation}
		where the second boxed term above arises from $\delta\|R^{m+1}\big(g^{0j}T\mathring L^b+g^{0b}T\mathring L^j\big)\varphi_j\|_{s,u}$.
		
		\item Making use of \eqref{null} again, that is, $g^{\al\beta,\g 0\g_3\cdots\g_k}\o_\al\o_\beta\o_\g\o_{\g_3}\cdots\o_{\g_k}
=g^{\al\beta,\g a\g_3\cdots\g_k}\o_\al\o_\beta\o_\g\o_a\o_{\g_3}\cdots\o_{\g_k}$, then one can  rewrite $\o_a=m_{ab}\mathring L^b+(\f{1}{\check\varrho+1}-1)m_{ab}\mathring L^b-\f{1}{\check\varrho+1}m_{ab}\check L^b$ to get
		\begin{equation*}
		\begin{split}
		\mathcal{II}=&\delta\|R^{m+1}\big\{g^{\al\beta,\g a\g_3\cdots\g_k}
\o_\al\o_\beta\o_\g\o_{\g_3}\cdots\o_{\g_k}(m_{ab}\slashed d^Xx^b\slashed d_X\varphi_j)\mu\tilde T^j\big\}\|_{s,u}\\
		&+\delta\|R^{m+1}\Big\{f(\o,\mathring L^1,\mathring L^2,\check\varrho)\left(
		\begin{array}{ccc}
		\check L^1\\
		\check L^2\\
		\check\varrho\\
		\end{array}
		\right)(\mathring L^jT\varphi_j)\Big\}\|_{s,u},
		\end{split}
		\end{equation*}
		which implies
		\begin{equation}\label{II}
		\begin{split}
		\mathcal{II}\lesssim&\delta\|\slashed dR^{m+1}\varphi\|_{s,u}+\delta^{2-\ve_0}s^{-5/2}\|R^{m_{-1}}x\|_{s,u}+\delta^{1-\ve_0}s^{-1/2}\|R^{m_0}\check\varrho\|_{s,u}\\
		&+\delta^{1-\ve_0}s^{-1/2}\big(\|R^{m_0}\check L^1\|_{s,u}+\|R^{m_0}\check L^2\|_{s,u}\big)+\boxed{\delta^{2-\ve_0}s^{-3/2}\|R^{m'+1}\mu\|_{s,u}}\\
		&+\delta s^{-1}\|R^{m_0}\varphi\|_{s,u}+\delta^{k(1-\ve_0)+1}s^{-1}\ln s\|TR^{m'+1}\varphi\|_{s,u}\\
		&+\delta^{(k+1)(1-\ve_0)+1}s^{-5/2}\ln s\|\slashed{\mathcal L}_R^{m_1}\leftidx{^{(R)}}{\slashed\pi}_T\|_{s,u}+\delta^{2-\ve_0}s^{-3/2}\|\slashed{\mathcal L}_R^{m_1}\slashed g\|_{s,u}.
		\end{split}
		\end{equation}
		
		\item It follows from \eqref{TL} and \eqref{gLLL} that
		\begin{equation}\label{III}
		\begin{split}
		\mathcal{III}\lesssim&\boxed{\delta^{3-\ve_0}s^{-3/2}\|\slashed dR^{m+1}\mu\|_{s,u}}+\boxed{(\delta^{(k+1)(1-\ve_0)+1}+\delta^{3-\ve_0})s^{-5/2}\|R^{m'+1}\mu\|_{s,u}}\\
		&+\delta^{(k+1)(1-\ve_0)+1}s^{-5/2}\|\slashed{\mathcal L}_R^{m_1}\slashed g\|_{s,u}+\delta^{k(1-\ve_0)}s^{-1}\|\slashed dR^{m'+1}\phi\|_{s,u}\\
		&+\delta^{k(1-\ve_0)}s^{-1}\|R^{m'+1}\mathring L\phi\|_{s,u}+\delta^{(k+2)(1-\ve_0)}s^{-5/2}\|R^{m_0+1}x\|_{s,u}\\
		&+\delta^{(k+1)(1-\ve_0)}s^{-3/2}\big(\|R^{m_0}\check L^1\|_{s,u}+\|R^{m_0}\check L^2\|_{s,u}+\|R^{m_0}\check\varrho\|_{s,u}\big)\\
		&+\delta^{(k+1)(1-\ve_0)+1}s^{-3/2}\|TR^{m'+1}\varphi\|_{s,u}+\delta^{(k+2)(1-\ve_0)+1}s^{-3}\|\slashed{\mathcal L}_R^{m_1}\leftidx{^{(R)}}{\slashed\pi}_T\|_{s,u}\\
		&+\delta^{k(1-\ve_0)+1}s^{-1}\|R^{m'+1}\mathring L\varphi\|_{s,u}+\delta^{k(1-\ve_0)+1}s^{-1}\|\slashed dR^{m+1}\varphi\|_{s,u}\\
		&+\delta^{(k+1)(1-\ve_0)}s^{-3/2}\|R^{m_0}\varphi\|_{s,u}.
		\end{split}
		\end{equation}
	\end{enumerate}

Substituting \eqref{I}-\eqref{III} to \eqref{RA}, using Proposition \ref{L2chi} and Lemma \ref{Ldphi} to estimate all terms except the boxed ones, and adopting \eqref{1f} and \eqref{LZf} in Appendix A to find the relationship between $\slashed dR^{m+1}\mu$ and $R^{m}\slashed\triangle\mu$,
then one can get eventually that
	\begin{equation}\label{RmA}
	\begin{split}
	\|R^{m+1}&\mathscr A\|_{s,u}
	\lesssim\delta^{3/2-\ve_0}(\delta+\delta^{k(1-\ve_0)}\ln s)s^{-1}+\delta s^{-\iota}\sqrt{\tilde E_{1,\leq m+2}(s,u)}\\
	&+\big(\delta^2+\delta^{k(1-\ve_0)}(\dl+\dl^{k(1-\ve_0)})\ln^2 s\big) s^{-1}\sqrt{\tilde E_{2,\leq m+2}(s,u)}+\delta^{2-\ve_0}s^{-3/2}\|R^{m'+1}\mu\|_{s,u}\\
	&+(\delta^{3-\ve_0}+\delta^{(k+1)(1-\ve_0)+1})s^{-1/2}\|R^m\slashed\triangle\mu\|_{s,u}+\dl^{k(1-\ve_0)}s^{-1/2}\sqrt{\int_0^u \tilde F_1(s,u')du'}.
\end{split}
\end{equation}

For $\bar Z^{m+1}=\bar Z^{p_1}TR^{p_2}$ with $p_1+p_2=m$, it follows from the expression of $\mathscr A$, Proposition \ref{L2chi}, \eqref{ZTdphi} and \eqref{Zmp} that
\begin{equation}\label{ZmA}
\begin{split}
&\delta^l\|\bar Z^{m+1}\mathscr A\|_{s,u}\\
\lesssim&\delta^{2-\ve_0}s^{-3/2}\big\{\delta^{l_0}\|\bar Z^{m_0}\check\varrho\|_{s,u}+s^{-1}\delta^{l_{-1}}\|\bar Z^{m_{-1}}x\|_{s,u}\big\}+\delta^l\|\slashed{\mathcal L}_{\bar Z}^{m+1}\slashed d\phi\|_{s,u}+s^{-1}\delta^{l_0}\|\bar Z^{m_0}\phi\|_{s,u}\\
\lesssim&\delta^{5/2-\ve_0}s^{-1}+\delta^2 s^{-1-\iota}\sqrt{\tilde E_{1,\leq m+2}}+\delta^2s^{-1}\sqrt{\tilde E_{2,\leq m+2}}+\delta^{2-\ve_0}s^{-3/2}\delta^{l'}\|\bar Z^{m'+1}\mu\|_{s,u}\\
&+\dl^{k(1-\ve_0)+1}s^{-3/2}\sqrt{\int_0^u \tilde F_1(s,u')du'}.
\end{split}
\end{equation}
Therefore, \eqref{A} follows from \eqref{RmA} and \eqref{ZmA}.
\end{proof}

Next we can refine the $L^2$ norm of $Z^{m+1}\mu$ despite the appearance of $R^m\slashed\triangle\mu$ which will be dealt with later.
\begin{proposition}\label{8.2}
	Under the assumption $(\star)$, when $\delta>0$ is small, it holds that for $m\leq 2N-6$,
	\begin{equation}\label{Zm}
	\begin{split}
	&\delta^l\|Z^{m+1}\mu\|_{s,u}\lesssim\delta^{k(1-\ve_0)-\ve_0+1/2}s^{1/2}+\delta^{(k-1)(1-\ve_0)} s^{1/2}\sqrt{\tilde E_{1,\leq m+2}(s,u)}\\
	&\qquad+\delta^{k(1-\ve_0)}s^{1/2}\sqrt{\tilde E_{2,\leq m+2}(s,u)}+\delta^{k(1-\ve_0)}(\delta+\delta^{k(1-\ve_0)})s^{1/2}\int_{t_0}^s\tau^{-3/2}\|R^m\slashed\triangle\mu\|_{\tau,u}d\tau\\
	&\qquad+\dl^{(2k-1)(1-\ve_0)-1}s^{1/2}\sqrt{\int_0^u \tilde F_1(s,u')du'}.
	\end{split}
	\end{equation}
\end{proposition}
\begin{proof}
	For any $\bar Z\in\{T, R\}$, \eqref{lmu} together with \eqref{GT} and \eqref{gLLL2} gives
	\begin{align*}
	&\delta^l\|\mathring L\bar Z^{m+1}\mu\|_{s,u}\\
	\lesssim&\delta^{k(1-\ve_0)}s^{-1}\delta^l\|T\bar Z^{m+1}\varphi\|_{s,u}+\delta^{(k+1)(1-\ve_0)}s^{-5/2}\delta^{l_1}\|\slashed{\mathcal L}_{\bar Z}^{m_1}\leftidx{^{(R)}}{\slashed\pi}_T\|_{s,u}\\
	&+\delta^{k(1-\ve_0)-1}s^{-1}\delta^{l_0}\big(\|\bar Z^{m_0}\check \varrho\|_{s,u}+\|\bar Z^{m_0}\check L^1\|_{s,u}+\|\bar Z^{m_0}\check L^2\|_{s,u}\big)\\
	&+\delta^{(k-1)(1-\ve_0)-1}s^{-1/2}\delta^{l_0}\|\bar Z^{m_0}\mathring L\phi\|_{s,u}+\delta^{k(1-\ve_0)-1}s^{-1}\delta^{l_0}\|\bar Z^{m_0}\varphi\|_{s,u}\\
	&+\delta^{(k-1)(1-\ve_0)-1}(\delta+\delta^{k(1-\ve_0)})s^{-3/2}\delta^{l_0}\|\slashed{\mathcal L}_{\bar Z}^{m_0}\slashed d\phi\|_{s,u}+\delta^{k(1-\ve_0)-\ve_0}s^{-1}\delta^{l_1}\|\slashed{\mathcal L}_{\bar Z}^{m_1}\leftidx{^{(R)}}{\slashed\pi}_{\mathring L}\|_{s,u}\\
	&+\delta^{k(1-\ve_0)-\ve_0}s^{-5/2}\delta^{l_{-1}}\|\bar Z^{m_{-1}}x\|_{s,u}+\delta^{(k-1)(1-\ve_0)}s^{-1/2}\delta^l\|\slashed d\bar Z^{m+1}\varphi\|_{s,u}\\
	&+\delta^{(k-1)(1-\ve_0)-1}s^{-1/2}\delta^{l_0}\|\bar Z^{m_0}\mathscr A\|_{s,u}+\delta^{(k-1)(1-\ve_0)}s^{-1/2}\delta^{l_0}\|\mathring L\bar Z^{m_0}\varphi\|_{s,u}\qquad\qquad\\
	&+\delta^{(k+1)(1-\ve_0)}s^{-1}\delta^{l_1}\|\slashed{\mathcal L}_{\bar Z}^{m_1}\leftidx{^{(T)}}{\slashed\pi}_{\mathring L}\|_{s,u}+\delta^{2k(1-\ve_0)-\ve_0}s^{-2}\ln s\delta^{l_1}\|\slashed{\mathcal L}_{\bar Z}^{m_1}\slashed g\|_{s,u}\\
	&+\delta^{k(1-\ve_0)}s^{-2}\delta^{l_0}\|\bar Z^{m_0}\mu\|_{s,u},
	\end{align*}
where $l_p$ ($p=-1,0,1$) are the numbers of $T$ in $Z^{m_p}$ and $1\leq m_p\leq m+1-p$.
It follows from Proposition \ref{L2chi}, Lemma \ref{Ldphi} and \eqref{A} that
	\begin{equation*}
	\begin{split}
	\delta^l\|\mathring L\bar Z^{m+1}\mu\|_{s,u}
	\lesssim&\delta^{k(1-\ve_0)}(\delta+\delta^{k(1-\ve_0)})s^{-1}\|R^m\slashed\triangle\mu\|_{s,u}+\delta^{k(1-\ve_0)}s^{-2}\delta^{l_0}\|\bar Z^{m_0}\mu\|_{s,u}\\
	&+\delta^{k(1-\ve_0)-\ve_0+1/2}s^{-1}+\delta^{(k-1)(1-\ve_0)}s^{-1/2-\iota}\sqrt{\tilde E_{1,\leq m+2}(s,u)}\\
	&+\delta^{k(1-\ve_0)}s^{-1}\sqrt{\tilde E_{2,\leq m+2}(s,u)}+\dl^{(2k-1)(1-\ve_0)-1}s^{-1}\sqrt{\int_0^u \tilde F_1(s,u')du'}.
	\end{split}
	\end{equation*}
	As in \eqref{tRmu}, by taking $F(s,u,\vartheta)=\delta^l\big(\bar Z^{m+1}\mu(s,u,\vartheta)-\bar Z^{m+1}\mu(t_0, u, \vartheta)\big)$ in \eqref{Ff}, one gets
	\begin{equation}\label{bZm}
	\begin{split}
	&\delta^l\|\bar Z^{m+1}\mu\|_{s,u}\lesssim\delta^{k(1-\ve_0)-\ve_0+1/2}s^{1/2}+\delta^{(k-1)(1-\ve_0)} s^{1/2}\sqrt{\tilde E_{1,\leq m+2}(s,u)}\\
	&\qquad+\delta^{k(1-\ve_0)}s^{1/2}\sqrt{\tilde E_{2,\leq m+2}(s,u)}+\delta^{k(1-\ve_0)}(\delta+\delta^{k(1-\ve_0)})s^{1/2}\int_{t_0}^s\tau^{-3/2}\|R^m\slashed\triangle\mu\|_{\tau,u}d\tau\\
	&\qquad+\dl^{(2k-1)(1-\ve_0)-1}s^{1/2}\sqrt{\int_0^u \tilde F_1(s,u')du'}.
	\end{split}
	\end{equation}
	
	For $Z^{m+1}=Z^{p_1}(\varrho\mathring L)\bar Z^{p_2}$ with $p_1+p_2=m$, due to \eqref{lmu}, \eqref{GT} and \eqref{gLLL}, one can get by Proposition \ref{L2chi} and \eqref{Zmp} that
	\begin{equation}\label{ZLZm}
	\begin{split}
	&\delta^l\|Z^{p_1}(\varrho\mathring L)\bar Z^{p_2}\mu\|_{s,u}
	\lesssim\delta^l\|Z^{p_1}(\varrho\bar Z^{p_2}\mathring L\mu)\|_{s,u}+\delta^l\|Z^{p_1}(\varrho[\mathring L,\bar Z^{p_2}]\mu)\|_{s,u}\\
	\lesssim&\delta^{k(1-\ve_0)-\ve_0}\delta^{l_1}\|\slashed{\mathcal L}_Z^{m_1}\leftidx{^{(R)}}{\slashed\pi}_{\mathring L}\|_{s,u}+\delta^{(k+1)(1-\ve_0)}\delta^{l_1}\|\slashed{\mathcal L}_Z^{m_1}\leftidx{^{(T)}}{\slashed\pi}_{\mathring L}\|_{s,u}+\delta^{k(1-\ve_0)}s^{-1}\ln s\delta^{l_1}\|Z^{m_1}\mu\|_{s,u}\\
	&+\delta^{2k(1-\ve_0)-\ve_0}s^{-1}\ln s\delta^{l_1}\|\slashed{\mathcal L}_Z^{m_1}\slashed g\|_{s,u}+\delta^{k(1-\ve_0)-1}\delta^{l_1}\big(\|Z^{m_1}\check L^1\|_{s,u}+\|Z^{m_1}\check L^2\|_{s,u}+\|Z^{m_1}\check\varrho\|_{s,u}\big)
	\\&+\delta^{k(1-\ve_0)-\ve_0}s^{-3/2}\delta^{l_0}\|Z^{m_0}x\|_{s,u}+\delta^{k(1-\ve_0)-1}\delta^{l_0}\|Z^{m_0}\varphi\|_{s,u}+\delta^{(k-1)(1-\ve_0)-1}s^{-1/2}\delta^{l_0}\|Z^{m_0}\phi\|_{s,u}\\
	\lesssim&\delta^{k(1-\ve_0)-\ve_0+1/2}+\delta^{k(1-\ve_0)}(1+\dl^{k(1-\ve_0)-1}) s^{-\iota}\sqrt{\tilde E_{1,\leq m+2}(s,u)}+\delta^{k(1-\ve_0)}\sqrt{\tilde E_{2,\leq m+2}(s,u)}\\
	&+\delta^{k(1-\ve_0)}s^{-1}\ln s\delta^{l_1}\|Z^{m_1}\mu\|_{s,u}+\dl^{(2k-1)(1-\ve_0)-1}\sqrt{\int_0^u \tilde F_1(s,u')du'}.
	\end{split}
	\end{equation}
	
	Thus, \eqref{Zm} follows from \eqref{bZm} and \eqref{ZLZm}.
\end{proof}

It follows from \eqref{Zm}, Lemma \ref{Ldphi}, \eqref{Zmp} and \eqref{A} that
\begin{corollary}\label{Zphi}
	Under the assumption $(\star)$, for small $\delta>0$, it holds that for $m\leq 2N-6$,
	\begin{align}
	\delta^l\|\bar Z^{m+1}\mathring L\phi&\|_{s,u}
	\lesssim\delta^{5/2-\ve_0}s^{-1}+\delta s^{-\iota}\sqrt{\tilde E_{1,\leq m+2}(s,u)}+\delta^2s^{-1}\sqrt{\tilde E_{2,\leq m+2}(s,u)}\no\\
+&\delta^{2-\ve_0}(\delta+\delta^{k(1-\ve_0)})\big\{s^{-1/2}\|R^m\slashed\triangle\mu\|_{s,u}
+\delta^{k(1-\ve_0)}s^{-1}\int_{t_0}^s\tau^{-3/2}\|R^m\slashed\triangle\mu\|_{\tau,u}d\tau\big\}\\
+&\dl^{k(1-\ve_0)}(\dl s^{-3/2}+\dl^{k(1-\ve_0)}s^{-1})\sqrt{\int_0^u \tilde F_1(s,u')du'},\no\\
	\delta^{l}\|\slashed{\mathcal L}_{\bar Z}^{m+1}\slashed d\phi&\|_{s,u}\lesssim\delta^{5/2-\ve_0}s^{-1}
+\delta s^{-\iota}\sqrt{\tilde E_{1,\leq m+2}(s,u)}+\delta^2s^{-1}\sqrt{\tilde E_{2,\leq m+2}(s,u)}\no\\
	+\delta^{2-\ve_0}&s^{1/2}\|R^m\slashed\triangle\mu\|_{s,u}+\delta^{(k+1)(1-\ve_0)+1}(\delta
+\delta^{k(1-\ve_0)})s^{-1}\int_{t_0}^s\tau^{-3/2}\|R^m\slashed\triangle\mu\|_{\tau,u}d\tau\\
+\dl^{k(1-\ve_0)}&(\dl s^{-3/2}+\dl^{k(1-\ve_0)}s^{-1})\sqrt{\int_0^u \tilde F_1(s,u')du'},\no\\
	\delta^l\|\bar Z^{m}\mathring L^2\phi&\|_{s,u}
	\lesssim\delta^{5/2-\ve_0}s^{-2}+\delta s^{-1-\iota}\sqrt{\tilde E_{1,\leq m+2}(s,u)}+\delta^2s^{-2}\sqrt{\tilde E_{2,\leq m+2}(s,u)}\no\\
	+&\delta^{(k+1)(1-\ve_0)+1}(\delta+\delta^{k(1-\ve_0)})s^{-2}\big\{s^{-1/2}\|R^m\slashed\triangle\mu\|_{s,u}
+\int_{t_0}^s\tau^{-\f32}\|R^m\slashed\triangle\mu\|_{\tau,u}d\tau\big\}\\
+&\dl^{k(1-\ve_0)}(\dl s^{-5/2}+\dl^{k(1-\ve_0)}s^{-2})\sqrt{\int_0^u \tilde F_1(s,u')du'},\no\\
	\delta^l\|Z^{m+1}\phi&\|_{s,u}\lesssim\delta^{5/2-\ve_0}+\delta(\delta s^{-\iota}+\delta^{k(1-\ve_0)})\sqrt{\tilde E_{1,\leq m+2}(s,u)}+\delta^2\sqrt{\tilde E_{2,\leq m+2}(s,u)}\no\\
	&\qquad\ \ +\delta^{(k+1)(1-\ve_0)+1}(\delta+\delta^{k(1-\ve_0)})\int_{t_0}^s\tau^{-3/2}\|R^m\slashed\triangle\mu\|_{\tau,u}d\tau\label{Zm1p}\\
	&\qquad\quad+\dl^{k(1-\ve_0)}(\dl s^{-1/2}+\dl^{k(1-\ve_0)})\sqrt{\int_0^u \tilde F_1(s,u')du'},\no
	\end{align}
	and
	\begin{equation}\label{bZmA}
	\begin{split}
	\delta^l\|\bar Z^{m+1}&\mathscr A\|_{s,u}
	\lesssim\delta^{3/2-\ve_0}(\delta+\delta^{k(1-\ve_0)}\ln s)s^{-1}+\delta s^{-\iota}\sqrt{\tilde E_{1,\leq m+2}(s,u)}\\
	&+\big(\delta^2+\delta^{k(1-\ve_0)}(\dl+\dl^{k(1-\ve_0)})\ln^2 s\big) s^{-1}\sqrt{\tilde E_{2,\leq m+2}(s,u)}\\
	&+(\delta+\delta^{k(1-\ve_0)})\delta^{2-\ve_0}\big\{s^{-1/2}\|R^m\slashed\triangle\mu\|_{s,u}
+\dl^{k(1-\ve_0)}s^{-1}\int_{t_0}^s\tau^{-3/2}\|R^m\slashed\triangle\mu\|_{\tau,u}d\tau\big\}\\
&+\dl^{k(1-\ve_0)}s^{-1/2}\sqrt{\int_0^u \tilde F_1(s,u')du'}.
	\end{split}
	\end{equation}
\end{corollary}

\section{Top order $L^2$ estimates on the derivatives of $\chi$ and $\mu$}\label{L2chimu}

	It follows from Proposition \ref{8.2} that to complete the estimate of $\|Z^{m+1}\mu\|_{s,u}$, it remains to bound $\|R^m\slashed\triangle\mu\|_{s,u}$. Furthermore, by the energy estimate \eqref{e} and \eqref{Psi}, one checks easily that the top order derivatives of $\varphi$
in the energy estimate for \eqref{Psi} are of $2N-4$. On the other hand, by the expression of $\leftidx{^{(Z)}}D_{\gamma,2}^m$
 in \eqref{D2}, the top order derivatives of $\chi$ and $\mu$ in \eqref{e} are of $2N-5$ and $2N-4$ respectively as in \cite{Ding4}. However,
 it follows from  Proposition \ref{L2chi} that the $L^2$ estimates on the $2N-5$ top order derivatives of ${\chi}$
and the $2N-4$ top order derivatives of
$\mu$ should be controlled by the energies of $\vp$ with the orders $2N-3$, which
is beyond the range of energies containing up to $2N-4$ order derivatives of $\vp$.
To overcome this difficulty, as in \cite{C2, J}, also motivated by \cite{Ding4, Ding3},
we will
treat the smallness orders and the time decay rates of $\textrm{tr}{\chi}$ and $\slashed\triangle\mu$
with the corresponding top order derivatives. It should be pointed out that although the main ideas here are similar to
\cite[Section 9]{Ding4}, yet it seems hard to adopt the analysis in \cite{Ding4} to close the energies in
Section \ref{YY} when $\ve_0$ is near $\ve_k^*$ in this paper. For instance, the quantity $e$ in equation \eqref{e} and the factor $(\mathring LG_{\al\beta}^\gamma)\mathring L^\al\mathring L^\beta\underline{\mathring L}\varphi_\g$ in equation \eqref{LFal} below should be handled more carefully by the null structure \eqref{null} and a new identity \eqref{dxdp}. In this case, the resulting new transport
equation \eqref{LhF} is much more suitable for estimating the $L^2$ norm of the top derivative of $\textrm{tr}{\chi}$ in the paper.

\subsection{Estimates on the derivatives of $\slashed d${tr}$\chi$ and $\slashed\nabla\chi$}\label{trchi}

{\bfseries 1.} If there is at least one vector field $\varrho\mathring L$ in $Z^m$, that is, $Z^m=Z^{p_1}(\varrho\mathring L)\bar Z^{p_2}$ and $p_1+p_2=m-1$ with $\bar Z\in\{T,R\}$, then
it follows from \eqref{Lchi'} and Proposition \ref{L2chi} that
\begin{align}
&\delta^l\|\slashed dZ^m(\textrm{tr}\check\chi)\|_{s,u}=\delta^l\|\slashed dZ^{p_1}\bar Z^{p_2}(\varrho\mathring L) (\textrm{tr}\check\chi)
+\slashed dZ^{p_1}[\varrho\mathring L,\bar Z^{p_2}](\textrm{tr}\check\chi)\|_{s,u}\no\\
\lesssim&s^{-1}\delta^{l_1}\|\slashed{\mathcal L}_Z^{m_1}\check\chi\|_{s,u}+\delta^{k(1-\varepsilon_0)}s^{-3}\ln s\delta^{l_1}(\|\slashed{\mathcal L}_Z^{m_1}\leftidx^{{(R)}}\slashed\pi_{\mathring L}\|_{s,u}+\delta\|\slashed{\mathcal L}_Z^{m_1}\leftidx^{{(T)}}\slashed\pi_{\mathring L}\|_{s,u})\no\\
&+\delta^{k(1-\varepsilon_0)}s^{-3}\big\{\delta^{l_1}\|Z^{m_1}\check L^1\|_{s,u}+\delta^{l_1}\|Z^{m_1}\check L^2\|_{s,u}+s^{-1}\delta^{l_0}\|Z^{k_0}x\|_{s,u}+\ln s\delta^{l_1}\|\slashed{\mathcal L}_Z^{m_1}\slashed g\|_{s,u}\big\}\no\\
&+\delta^{(k-1)(1-\varepsilon_0)}s^{-5/2}\delta^{l_0}\|Z^{m_0}\varphi\|_{s,u}+\delta^{(k-1)(1-\varepsilon_0)}s^{-3/2}\delta^{l_0}\|\slashed dZ^{m_0}\varphi\|_{s,u}\no\\
\lesssim&\delta^{k(1-\varepsilon_0)+1/2}s^{-5/2}\ln s+\delta^{(k-1)(1-\varepsilon_0)}s^{-3/2-\iota}\sqrt{\tilde E_{1,\leq m+2}}\no\\
&+\delta^{k(1-\varepsilon_0)+\varepsilon_0}(1+\dl^{k(1-\ve_0)-1}) s^{-5/2}\ln^2 s\sqrt{\tilde E_{2,\leq m+2}}+\dl^{(k-1)(1-\ve_0)}s^{-2}\sqrt{\int_0^u \tilde F_1(s,u')du'}.\label{nz}
\end{align}
Thus,
\begin{equation}\label{dch}
\begin{split}
\delta^l\|\slashed\nabla\slashed{\mathcal L}_{Z}^m\check\chi\|_{s,u}\lesssim&\delta^{k(1-\varepsilon_0)+1/2}s^{-5/2}\ln s+\delta^{(k-1)(1-\varepsilon_0)}s^{-3/2-\iota}\sqrt{\tilde E_{1,\leq m+2}}\\
&+\delta^{k(1-\varepsilon_0)+\varepsilon_0}(1+\dl^{k(1-\ve_0)-1}) s^{-5/2}\ln^2 s\sqrt{\tilde E_{2,\leq m+2}}\\
&+\dl^{(k-1)(1-\ve_0)}s^{-2}\sqrt{\int_0^u \tilde F_1(s,u')du'}.
\end{split}
\end{equation}

{\bfseries 2.} When all vectorfields $Z'$s are in $\{T, R\}$, then \eqref{Lchi} implies
\begin{equation}\label{Ltrchi}
\begin{split}
\mathring L(\textrm{tr}&{\chi})=-|{\check\chi}|^2-\f2\varrho\text{tr}\chi+\f{1}{\varrho^2}-G_{X\mathring L}^\gamma\slashed d^X\mathring L\varphi_\gamma
+\f12G_{\mathring L\mathring L}^\gamma\slashed\triangle\varphi_\gamma+\f12\slashed g^{XX}G_{XX}^\gamma\mathring L^2\varphi_{\gamma}\\
&-\big\{\f12G_{\mathring L\mathring L}^\gamma(\mathring L\varphi_\gamma)+G_{\tilde T\mathring L}^\gamma(\mathring L\varphi_\gamma)-G_{X\mathring L}^\gamma(\slashed d^X\varphi_{\gamma})\big\}\textrm{tr}\chi+f(\varphi, \slashed dx,\mathring L^1,\mathring L^2)\varphi^{k-2}\left(
\begin{array}{ccc}
\mathring L\varphi\\
\slashed d\varphi\\
\end{array}
\right)^2,
\end{split}
\end{equation}
where
\begin{equation}\label{triphi}
\begin{split}
&\slashed\triangle\varphi_\g=\mu^{-1}\big(\mathring L\mathring{\underline L}\varphi_\g+\f{1}{2\varrho}\mathring {\underline L}\varphi_\g-{H}_\g\big)
\end{split}
\end{equation}
by \eqref{fequation}.
Note that
\begin{equation*}
\begin{split}
G_{X\mathring L}^\gamma\slashed d^X\mathring L\varphi_\gamma=&\mathring L(G_{X\mathring L}^\gamma\slashed d^X\varphi_\gamma)-(\mathring LG_{\al\beta}^\g)\mathring L^\beta\slashed d_Xx^\al\slashed d^X\varphi-G_{\al\beta}(\slashed d_X\mathring L^\al)\mathring L^\beta\slashed d^X\varphi_\g\\
&-G_{\al\beta}\slashed d_Xx^\al(\mathring L\mathring L^\beta)\slashed d^X\varphi_\g+2\textrm{tr}\chi (G_{X\mathring L}^\g\slashed d^X\varphi_\g),
\end{split}
\end{equation*}
and the term $\mathring L(G_{X\mathring L}^\gamma\slashed d^X\varphi_\gamma)$ can be moved to the left hand of \eqref{Ltrchi}.
Analogous treatments can be carried out for $\f12\slashed g^{XX}G_{XX}^\gamma\mathring L^2\varphi_{\gamma}$ and $\f12\mu^{-1}G_{\mathring L\mathring L}^\gamma\mathring L\mathring{\underline L}\varphi_\g$.
Thus, with the help of \eqref{lmu}, \eqref{LL}, \eqref{dL} and \eqref{H}, \eqref{Ltrchi} can be written as a new form containing only the first order
or zero-th order derivatives of $\varphi_\g$ on the right hand side of the equality
\begin{equation}\label{Ltr}
\begin{split}
\mathring L\big(\textrm{tr}\chi-E)=\big(-\f{2}{\varrho}+\mathcal E\big)\textrm{tr}\chi+\f{1}{\varrho^2}-|\check\chi|^2-\f12\mu^{-1}\mathring L(G_{\al\beta}^\gamma)\mathring L^\al\mathring L^\beta\underline{\mathring L}\varphi_\g+e,
\end{split}
\end{equation}
where
\begin{equation}\label{Ee}
\begin{split}
E=&-G_{X\mathring L}^\gamma\slashed d^X\varphi_\gamma+\f12\mu^{-1}G_{\mathring L\mathring L}^\gamma\mathring{\underline L}\varphi_\gamma
+\f12\slashed g^{XX}G_{XX}^\gamma\mathring L\varphi_\gamma,\\
\mathcal E=&-\f12G_{\mathring L\mathring L}^\gamma\mathring L\varphi_\gamma-G_{\tilde T\mathring L}^\gamma\mathring L\varphi_\gamma+\f12\mu^{-1}G_{\mathring L\mathring L}^\gamma T\varphi_\gamma,
\end{split}
\end{equation}
and
\begin{equation}\label{e1}
\begin{split}
e=&\f12\mu^{-2}(G_{\mathring L\mathring L}^\gamma T\varphi_\gamma)^2+f_1(\varphi, \slashed dx^i,\mathring L^i,\slashed g)\varphi^{k-2}\left(
\begin{array}{ccc}
\mathring L\varphi\\
\slashed d\varphi\\
\end{array}
\right)^2+f_2(\varphi, \slashed dx^i,\mathring L^i,\slashed g)\mu^{-1}\varphi^{2k-2}\mathring L\varphi T\varphi\\
&+f_3(\varphi, \slashed dx^i,\mathring L^i,\slashed g)\mu^{-1}\varphi^{k-1}(G_{\mathring L\mathring L}^\g\slashed d_X\varphi_\g)T\varphi
+f_4(\varphi, \slashed dx^i,\mathring L^i,\slashed g)\mu^{-1}\varphi^{k-1}(G_{\mathring L\mathring L}^\g T\varphi_\g)\slashed d\varphi.
\end{split}
\end{equation}
Here note that for the global estimate on the solution $\phi$
to \eqref{quasi} with \eqref{id} for all $\ve_0<\ve_k^*$ especially for $\ve_0$ near the critical
exponent $\ve_k^*$, we will make full use of the null structure \eqref{null} to
write explicitly the factors $G_{\mathring L\mathring L}^\g\slashed d_X\varphi_\g$ and $G_{\mathring L\mathring L}^\g T\varphi_\g$
in the expression of $e$. In fact, according to \eqref{GT},
\begin{equation}\label{GTp}
G_{\mathring L\mathring L}^\g T\varphi_\g=-(G_{\mathring L\mathring L}^\g\mathring L_\g)\tilde T^iT\varphi_i
+\mu G_{\mathring L\mathring L}^0\tilde T^i(\mathring L\varphi_i)
+\mu g_{\gamma i}G_{\mathring L\mathring L}^\gamma (\slashed d^Xx^i)\tilde T^j\slashed d_X\varphi_j.
\end{equation}
In addition, inserting \eqref{gab} into \eqref{quasi} gives
\begin{equation}\label{dxdp}
\slashed d^Xx^i\slashed d_X\varphi_i=\mu^{-1}\underline{\mathring L}^\al\mathring L\varphi_\al
=-(\mathring L^\al+2g^{0\al})\mathring L\varphi_\al,
\end{equation}
which, together with \eqref{pal}, implies
\begin{equation}\label{Gdp}
\begin{split}
&G_{\mathring L\mathring L}^\g\slashed d_X\varphi_\g=G_{\mathring L\mathring L}^\g\slashed d_Xx^i(\p_\g\varphi_i)\\
=&\dl_\g^0G_{\mathring L\mathring L}^\g(\slashed d_Xx^i\mathring L\varphi_i)-(G_{\mathring L\mathring L}^\g\mathring L_\g)\tilde T^j\slashed d_X\varphi_j+g_{\g j}G_{\mathring L\mathring L}^\g\slashed d_Xx^j(\slashed d^Xx^i\slashed d_X\varphi_i)\\
=&\dl_\g^0G_{\mathring L\mathring L}^\g(\slashed d_Xx^i\mathring L\varphi_i)-(G_{\mathring L\mathring L}^\g\mathring L_\g)\tilde T^j\slashed d_X\varphi_j-g_{\g j}G_{\mathring L\mathring L}^\g\slashed d_Xx^j(\mathring L^\al+2g^{0\al})\mathring L\varphi_\al.
\end{split}
\end{equation}
Substituting \eqref{GTp} and \eqref{Gdp} into $\eqref{e1}$ yields
\begin{equation}\label{ee}
\begin{split}
e=&\f12\mu^{-2}(G_{\mathring L\mathring L}^\gamma T\varphi_\gamma)^2+f_1(\varphi, \slashed dx,\mathring L^1,\mathring L^2,\slashed g)\varphi^{k-2}\left(
\begin{array}{ccc}
\mathring L\varphi\\
\slashed d\varphi\\
\end{array}
\right)^2\\
&+f_2(\varphi, \slashed dx,\mathring L^1,\mathring L^2,\slashed g)\mu^{-1}\varphi^{2k-2}\mathring L\varphi T\varphi+f_3(\varphi, \slashed dx,\mathring L^1,\mathring L^2,\slashed g)\mu^{-1}\varphi^{k-1}(G_{\mathring L\mathring L}^\g\mathring L_\g)T\varphi \slashed d\varphi.
\end{split}
\end{equation}
Let $F^m=\slashed d\bar Z^m\textrm{tr}\chi-\slashed d\bar Z^m E$ with $\bar Z\in\{T,R\}$. Then by an induction
argument on \eqref{Ltr}, one can get
\begin{equation}\label{LFal}
\begin{split}
\slashed{\mathcal L}_{\mathring L}F^m=&(-\f{2}{\varrho}+\mathcal E)F^m+(-\f{2}{\varrho}+\mathcal E)\slashed d\bar Z^m E-\slashed d\bar Z^{m}(|\check\chi|^2)\\
&-\f12\slashed d\bar Z^{m}\big(\mu^{-1}\mathring L(G_{\al\beta}^\gamma)\mathring L^\al\mathring L^\beta\underline{\mathring L}\varphi_\g\big)+e^m,
\end{split}
\end{equation}
where
\begin{equation}\label{eal}
\begin{split}
e^m=&\slashed{\mathcal L}_{\bar Z}^m e^0+\underbrace{\sum_{p_1+p_2=m-1}{\tiny {\tiny {\tiny {\tiny }}}}\slashed{\mathcal L}_{\bar Z}^{p_1}\slashed{\mathcal L}_{[\mathring L,\bar Z]}F^{p_2}+\sum_{\tiny\begin{array}{c}p_1+p_2=m\\p_1\geq 1\end{array}}\bar Z^{p_1}\big(-\f{2}{\varrho}+\mathcal E)\slashed d\bar Z^{p_2}\textrm{tr}\chi}_{\text{vanish when $m=1$}}
\end{split}
\end{equation}
and
\begin{equation}\label{e0}
\begin{split}
e^0=&\slashed de+(\slashed d\mathcal E)\textrm{tr}\chi.
\end{split}
\end{equation}

Similarly,
the factor $(\mathring LG_{\al\beta}^\gamma)\mathring L^\al\mathring L^\beta\underline{\mathring L}\varphi_\g$ in \eqref{LFal}
can also be dealt with \eqref{null}. In fact, due to $$G_{\al\beta}^\g=-kg_{\al\al'}g_{\beta\beta'}g^{\al'\beta',\g\g_2\cdots\g_k}\varphi_{\g_2}\cdots\varphi_{\g_k}+f(\varphi)\varphi^k,$$
then by \eqref{pal} and $\mathring L\mathring L^i=-\f12(G_{\mathring L\mathring L}^\g\mathring L_\g)\tilde T^j\slashed d^Xx^i\slashed d_X\varphi_j+O(\varphi^{k-1}\mathring L\varphi)$ from \eqref{LL} and \eqref{Gdp}, one can obtain
\begin{equation*}
\begin{split}
\mathring LG_{\al\beta}^\g=&-k(k-1)g_{\al\al'}g_{\beta\beta'}g^{\al'\beta',\g\g_2\cdots\g_k}(\mathring L\p_{\g_2}\phi)\varphi_{\g_3}\cdots\varphi_{\g_k}+f(\varphi)\varphi^{k-1}\mathring L\varphi\\
=&k(k-1)g_{\al\al'}g_{\beta\beta'}g^{\al'\beta',\g\g_2\cdots\g_k}\mathring L_{\g_2}\tilde T^{i_2}(\mathring L\varphi_{i_2})\varphi_{\g_3}\cdots\varphi_{\g_k}\\
&-k(k-1)g_{\al\al'}g_{\beta\beta'}g^{\al'\beta',\g\g_2\cdots\g_k}\mathring L(\delta_{\g_2}^0\mathring L\phi+g_{\g_2i}\slashed d^Xx^i\slashed d_X\phi)\varphi_{\g_3}\cdots\varphi_{\g_k}\\
&+f(\varphi,\mathring L^1,\mathring L^2)\varphi^{k-1}\left(
\begin{array}{ccc}
\mathring L\varphi\\
(G_{\mathring L\mathring L}^\g\mathring L_\g)\slashed d^Xx\slashed d_X\varphi
\end{array}
\right).
\end{split}
\end{equation*}
Applying \eqref{pal} to $\varphi_{\g_p}=\p_{\g_p}\phi$ $(3\leq p\leq k)$ again leads to
\begin{equation*}
\begin{split}
LG_{\al\beta}^\g=&(-1)^kk(k-1)g_{\al\al'}g_{\beta\beta'}g^{\al'\beta',\g\g_2\cdots\g_k}\mathring L_{\g_2}\cdots\mathring L_{\g_k}\tilde T^{i_2}\cdots\tilde T^{i_k}(\mathring L\varphi_{i_2})\varphi_{i_3}\cdots\varphi_{i_k}\\
&-(-1)^{k}k(k-1)(k-2)g_{\al\al'}g_{\beta\beta'}g^{\al'\beta',\g\g_2\cdots\g_k}\mathring L_{\g_2}\tilde T^i(\mathring L\varphi_i)(g_{\g_3j}\slashed d^Xx^j\slashed d_X\phi)\\
&\qquad\qquad\qquad\qquad\qquad\cdot\mathring L_{\g_4}\cdots\mathring L_{\g_k}\tilde T^{i_4}\cdots\tilde T^{i_k}\varphi_{i_4}\cdots\varphi_{i_k}\\
&-k(k-1)g_{\al\al'}g_{\beta\beta'}g^{\al'\beta',\g\g_2\cdots\g_k}\mathring L(\delta_{\g_2}^0\mathring L\phi+g_{\g_2i}\slashed d^Xx^i\slashed d_X\phi)\varphi_{\g_3}\cdots\varphi_{\g_k}\\
&+f(\varphi,\mathring L^1, \mathring L^2)\left(
\begin{array}{ccc}
\varphi^{k-1}\mathring L\varphi\\
\varphi^{k-1}(G_{\mathring L\mathring L}^\g\mathring L_\g)\slashed d^Xx\slashed d_X\varphi\\
\mathring L\varphi\mathring L\phi\varphi^{k-3}\\
\mathring L\varphi(\slashed d_Xx\slashed d^X\phi)(\slashed d_Xx\slashed d^X\phi)\varphi^{k-4}
\end{array}
\right).
\end{split}
\end{equation*}
Thus,
\begin{equation*}
\begin{split}
&(\mathring LG_{\al\beta}^\gamma)\mathring L^\al\mathring L^\beta\underline{\mathring L}\varphi_\g=(\mathring LG_{\al\beta}^\gamma)\mathring L^\al\mathring L^\beta\underline{\mathring L}^\nu\p_\g\varphi_\nu\\
=&(-1)^{k-1}k(k-1)g^{\al\beta,\g\g_2\cdots\g_k}\mathring L_\al\mathring L_\beta\mathring L_\g\mathring L_{\g_2}\cdots\mathring L_{\g_k}\tilde T^{i_2}\cdots\tilde T^{i_k}(\mathring L\varphi_{i_2})\varphi_{i_3}\cdots\varphi_{i_k}(\mu^{-1}\underline{\mathring L}^\nu T\varphi_{\nu})\\
&+(-1)^{k}k(k-1)(k-2)g^{\al\beta,\g\g_2\cdots\g_k}\mathring L_\al\mathring L_\beta\mathring L_\g\mathring L_{\g_2}(g_{\g_3j}\slashed d^Xx^j\slashed d_X\phi)\mathring L_{\g_4}\cdots\mathring L_{\g_k}\\
&\qquad\qquad\qquad\qquad\qquad\cdot\tilde T^i(\mathring L\varphi_i)\tilde T^{i_4}\cdots\tilde T^{i_k}\varphi_{i_4}\cdots\varphi_{i_k}(\mu^{-1}\underline{\mathring L}^\nu T\varphi_{\nu})\\
&-k(k-1)g_{\al\al'}g_{\beta\beta'}g^{\al'\beta',\g\g_2\cdots\g_k}\mathring L(\delta_{\g_2}^0\mathring L\phi+g_{\g_2i}\slashed d^Xx^i\slashed d_X\phi)\varphi_{\g_3}\cdots\varphi_{\g_k}\mathring L^\al\mathring L^\beta\underline{\mathring L}\varphi_\g\\
&+f(\varphi,\mathring L^1, \mathring L^2,\f{x}\varrho)\left(
\begin{array}{ccc}
\varphi^{k-1}\mathring L\varphi(\underline{\mathring L}\varphi)\\
\varphi^{k-1}(G_{\mathring L\mathring L}^\g\mathring L_\g)\slashed d^Xx\slashed d_X\varphi(\underline{\mathring L}\varphi)\\
\mathring L\varphi\mathring L\phi\varphi^{k-3}(\underline{\mathring L}\varphi)\\
\mathring L\varphi(\slashed d_Xx\slashed d^X\phi)(\slashed d_Xx\slashed d^X\phi)\varphi^{k-4}(\underline{\mathring L}\varphi)\\
\mu(\mathring L\varphi)^2\varphi^{k-2}\\
\mu\mathring L\varphi(\slashed d^Xx\slashed d_X\varphi)\varphi^{k-2}
\end{array}
\right).\qquad\qquad\qquad\qquad\qquad
\end{split}
\end{equation*}
As for \eqref{gLLL}, the null condition \eqref{null} implies that
\begin{align}
&(\mathring LG_{\al\beta}^\gamma)\mathring L^\al\mathring L^\beta\underline{\mathring L}\varphi_\g\no\\
=&(-1)^{k}k(k-1)(k-2)\overbrace{g^{\al\beta,\g j'\g_3\cdots\g_k}\o_\al\o_\beta\o_\g(m_{jj'}\slashed d^Xx^j\slashed d_X\phi)\o_{\g_3}\cdots\o_{\g_k}}^{\mathscr A}\qquad\qquad\no\\
&\qquad\qquad\qquad\qquad\qquad\cdot\tilde T^i(\mathring L\varphi_i)\tilde T^{i_4}\cdots\tilde T^{i_k}\varphi_{i_4}\cdots\varphi_{i_k}(\mu^{-1}\underline{\mathring L}^\nu T\varphi_{\nu})\no\\
&-k(k-1)g_{\al\al'}g_{\beta\beta'}g^{\al'\beta',\g\g_2\cdots\g_k}\mathring L(\delta_{\g_2}^0\mathring L\phi+g_{\g_2i}\slashed d^Xx^i\slashed d_X\phi)\varphi_{\g_3}\cdots\varphi_{\g_k}\mathring L^\al\mathring L^\beta\underline{\mathring L}\varphi_\g\label{LG}\\
&+\underbrace{f(\varphi,\mathring L^1,\mathring L^2,\f{x}\varrho,\check\varrho)\left(
\begin{array}{ccc}
\varphi^{k-1}\mathring L\varphi(\underline{\mathring L}\varphi)\\
\varphi^{k-1}(G_{\mathring L\mathring L}^\g\mathring L_\g)\slashed d^Xx\slashed d_X\varphi(\underline{\mathring L}\varphi)\\
\mathring L\varphi\mathring L\phi\varphi^{k-3}(\underline{\mathring L}\varphi)\\
\mathring L\varphi(\slashed d_Xx\slashed d^X\phi)(\slashed d_Xx\slashed d^X\phi)\varphi^{k-4}(\underline{\mathring L}\varphi)\\
\mu(\mathring L\varphi)^2\varphi^{k-2}\\
\mu\mathring L\varphi(\slashed d^Xx\slashed d_X\varphi)\varphi^{k-2}\\
(\mathring L\varphi)T\varphi(\slashed d^Xx\slashed d_X\phi)\varphi^{2k-3}\\
\check L^1(\mathring L\varphi)\varphi^{k-2}(T\varphi)\\
\check L^2(\mathring L\varphi)\varphi^{k-2}(T\varphi)\\
\check\varrho(\mathring L\varphi)\varphi^{k-2}(T\varphi)\\
\check L^1(\slashed d^Xx\slashed d_X\phi)(\mathring L\varphi)\varphi^{k-3}(T\varphi)\\
\check L^2(\slashed d^Xx\slashed d_X\phi)(\mathring L\varphi)\varphi^{k-3}(T\varphi)\\
\check\varrho(\slashed d^Xx\slashed d_X\phi)(\mathring L\varphi)\varphi^{k-3}(T\varphi)\no
\end{array}
\right)}_{\mathscr B}.
\end{align}

Observe that the first term in the right hand side of \eqref{LG} contains
the good known $\mathscr A$ which has been estimated in \eqref{A}. Note also that
if one substitutes the second term of \eqref{LG} into \eqref{LFal} and estimates
the resulting term directly, then the final energy estimates cannot be closed since
the highest $m+3$  order derivatives of $\phi$ will  appear on the right hand side
of the related energy inequality (Corollary \ref{Zphi} only includes the estimates of at most
$m+2$  order derivatives of $\phi$). Fortunately, there exists
a crucial factor $\mathring L(\delta_{\g_2}^0\mathring L\phi+g_{\g_2i}\slashed d^Xx^i\slashed d_X\phi)$ in
the second term of \eqref{LG}, which helps us rewrite the corresponding term in \eqref{LFal} as
 $-\slashed{\mathcal L}_{\mathring L}\mathscr C^m+\cdots$, and $-\slashed{\mathcal L}_{\mathring L}\mathscr C^m$
 can be moved to the left hand side of \eqref{LFal} (see \eqref{Zd} below and $\mathscr C^m$
 admits some good properties). In addition, the last term $\mathscr B$ can be regarded
as the higher order error term, which can be treated easily. To deal with the term $-\f12\slashed{\mathcal L}_{\bar Z}^m\slashed d((\mathring LG_{\al\beta}^\gamma)\mathring L^\al\mathring L^\beta\underline{\mathring L}\varphi_\g)$ in \eqref{LFal}, as for \eqref{Ltr}, one can get from \eqref{LG} that
\begin{equation}\label{Zd}
\begin{split}
&\slashed d{\bar Z}^m\big(\mu^{-1}(\mathring LG_{\al\beta}^\gamma)\mathring L^\al\mathring L^\beta\underline{\mathring L}\varphi_\g\big)\\
=&\slashed{\mathcal L}_{\mathring L}\mathscr C^m+\mathscr D^m+\slashed d\bar Z^m(\mu^{-1}\mathscr B)\\
&+(-1)^kk(k-1)(k-2)\slashed d\bar Z^m\big\{\mathscr A\tilde T^i(\mathring L\varphi_i)\tilde T^{i_4}\cdots\tilde T^{i_k}\varphi_{i_4}\cdots\varphi_{i_k}(\mu^{-2}\underline{\mathring L}^\nu T\varphi_{\nu})\big\},
\end{split}
\end{equation}
where
\begin{equation}\label{mC}
\mathscr C^m={-}k(k-1)g^{\al\beta,\g\g_2\cdots\g_k}\slashed d\bar Z^m(\delta_{\g_2}^0\mathring L\phi+g_{\g_2i}\slashed d^Xx^i\slashed d_X\phi)\mu^{-1}\varphi_{\g_3}\cdots\varphi_{\g_k}\mathring L_\al\mathring L_\beta\underline{\mathring L}\varphi_\g
\end{equation}
and
\begin{equation}\label{mD}
\begin{split}
\mathscr D^m
=&-k(k-1)g^{\al\beta,\g\g_2\cdots\g_k}\Big\{\big([\slashed{\mathcal L}_{\bar Z}^m,\slashed{\mathcal L}_{\mathring L}]\slashed d(\delta_{\g_2}^0\mathring L\phi+g_{\g_2i}\slashed d^Xx^i\slashed d_X\phi)\big)\mu^{-1}\varphi_{\g_3}\cdots\varphi_{\g_k}\mathring L_\al\mathring L_\beta\underline{\mathring L}\varphi_\g\\
&\qquad-\slashed d\bar Z^m(\delta_{\g_2}^0\mathring L\phi+g_{\g_2i}\slashed d^Xx^i\slashed d_X\phi)\mathring L(\mu^{-1}\varphi_{\g_3}\cdots\varphi_{\g_k}\mathring L_\al\mathring L_\beta\underline{\mathring L}\varphi_\g)\\
&\qquad+\slashed{\mathcal L}_{\bar Z}^m\big\{\mathring L(\delta_{\g_2}^0\mathring L\phi+g_{\g_2i}\slashed d^Xx^i\slashed d_X\phi)\slashed d(\mu^{-1}\varphi_{\g_3}\cdots\varphi_{\g_k}\mathring L_\al\mathring L_\beta\underline{\mathring L}\varphi_\g)\big\}\\
&\qquad+\sum_{\tiny\begin{array}{c}p_1+p_2=m\\p_1\geq 1\end{array}}\slashed{\mathcal L}_{\bar Z}^{p_1}\slashed d\mathring L(\delta_{\g_2}^0\mathring L\phi+g_{\g_2i}\slashed d^Xx^i\slashed d_X\phi)\bar Z^{p_2}(\mu^{-1}\varphi_{\g_3}\cdots\varphi_{\g_k}\mathring L_\al\mathring L_\beta\underline{\mathring L}\varphi_\g)\Big\}.
\end{split}
\end{equation}
Substituting \eqref{Zd} into \eqref{LFal} and setting
$$
\hat F^m:=F^m+\f12\mathscr C^m
$$
and
\begin{equation}\label{hem}
\begin{split}
\hat e^m:=&e^m-\f12\mathscr D^m-\f12\slashed d\bar Z^m(\mu^{-1}\mathscr B)\\
&-\f12(-1)^kk(k-1)(k-2)\slashed d\bar Z^m\big\{\mathscr A\tilde T^i(\mathring L\varphi_i)\tilde T^{i_4}\cdots\tilde T^{i_k}\varphi_{i_4}\cdots\varphi_{i_k}(\mu^{-2}\underline{\mathring L}^\nu T\varphi_{\nu})\big\},
\end{split}
\end{equation}
then one can rewrite \eqref{LFal} as
\begin{equation}\label{LhF}
\slashed{\mathcal L}_{\mathring L}\hat F^m=(-\f{2}{\varrho}+\mathcal E)\hat F^m+(-\f{2}{\varrho}+\mathcal E)(\slashed d\bar Z^m E
-\f12\mathscr C^m)-\slashed d\bar Z^{m}(|\check\chi|^2)+\hat e^m.
\end{equation}
Note that for any one-form $\xi$ on $S_{s,u}$, it holds that
\begin{equation}\label{Lxi}
\mathring L(\varrho^2|\xi|^2)=-2\varrho^2\check{\chi}^{XX}\xi_X\xi_X
+2\varrho^2\slashed g^{XX}(\slashed{\mathcal L}_{\mathring L}\xi_X)\xi_X.
\end{equation}
By taking $\xi=\varrho^2\hat F^m$ in \eqref{Lxi} and using \eqref{LhF}, one can obtain
\begin{equation*}
\begin{split}
\mathring L\big(\varrho^6|\hat F^m|^2\big)
=2\varrho^6\big\{&-\textrm{tr}\check\chi |\hat F^m|^2+\mathcal E|\hat F^m|^2+\hat e^m\cdot\hat F^m\\
&+(-\f{2}{\varrho}+\mathcal E)(\slashed d\bar Z^m E-\f12\mathscr C^m)\cdot \hat F^m-\slashed d\bar Z^m(|\check\chi|^2)\cdot\hat F^m\big\}.
\end{split}
\end{equation*}
Then
\begin{equation}\label{LrhoF}
\begin{split}
|\mathring L(\varrho^3|\hat F^m|)|\lesssim&\delta^{k(1-\varepsilon_0)-\varepsilon_0}s^{-3/2}\varrho^3|\hat F^m|+\varrho^2|\slashed d\bar Z^m E|+\varrho^2|\mathscr C^m|+\varrho^3|\slashed d\bar Z^m(|\check\chi|^2)|+\varrho^3|\hat e^m|.
\end{split}
\end{equation}
It follows from \eqref{Ff} and \eqref{LrhoF} that
\begin{equation*}
\begin{split}
\delta^l\varrho^3\|\hat F^m&\|_{s,u}\lesssim\delta^l\|\hat F^m(t_0,\cdot,\cdot)\|_{L^2(\Sigma_{s}^{ u})}+\delta^l\varrho^{1/2}\int_{t_0}^s\Big\{\delta^{k(1-\varepsilon_0)
-\varepsilon_0}\tau^{-3/2}\varrho^{5/2}\|\hat F^m\|_{L^2(\Sigma_{\tau}^{ u})}\\
&+\tau^{3/2}\|\slashed d\bar Z^m E\|_{\tau,u}+\tau^{3/2}\|\mathscr C^m\|_{\tau,u}+\tau^{5/2}\|\slashed d\bar Z^m(|\check\chi|^2)\|_{L^2(\Sigma_{\tau}^u)}+\tau^{5/2}\|\hat e^m\|_{L^2(\Sigma_{\tau}^u)}\Big\}d\tau.
\end{split}
\end{equation*}
This, together with Gronwall's inequality, yields
\begin{equation}\label{Eal}
\begin{split}
\varrho^{5/2}\delta^l\|\hat F
^m\|_{s,u}
\lesssim&\delta^{k(1-\varepsilon_0)-\varepsilon_0+1/2}+\delta^l\int_{t_0}^s\Big\{\tau^{3/2}\|\slashed d\bar Z^m E\|_{\tau,u}\\
&+\tau^{3/2}\|\mathscr C^m\|_{\tau,u}+\tau^{5/2}\|\slashed d\bar Z^m(|\check\chi|^2)\|_{L^2(\Sigma_{\tau}^u)}+\tau^{5/2}\|\hat e^m\|_{L^2(\Sigma_{\tau}^u)}\Big\}{d}\tau.
\end{split}
\end{equation}
Each term in the integrand of \eqref{Eal} can be estimated as follows.

\vskip 0.1 true cm

{\bfseries (1) The estimate of $\|\slashed d\bar Z^{{m}} E\|_{L^2(\Sigma_{s}^{u})}$}
	
Due to $E=-G_{X\mathring L}^\gamma\slashed d^X\varphi_\gamma+\f12\mu^{-1}G_{\mathring L\mathring L}^\gamma\mathring{\underline L}\varphi_\gamma+\f12\slashed g^{XX}G_{XX}^\gamma\mathring L\varphi_\gamma$, by virtue of \eqref{GT} and \eqref{gLLL2},
then it holds that
	\begin{align}
	&\delta^l\|\slashed d\bar Z^m E\|_{s,u}\no\\
	\lesssim&\delta^{k(1-\varepsilon_0)-\varepsilon_0}s^{-7/2}\delta^{l_{-1}}\|\bar Z^{m_{-1}}x\|_{s,u}
+\delta^{k(1-\varepsilon_0)-1}s^{-2}\big\{\sum_{i=1}^2\delta^{l_0}\|\bar Z^{m_0}\check L^i\|_{s,u}+\delta^{l_0}\|\bar Z^{m_0}\check\varrho\|_{s,u}\big\}\no\\
	&+\delta^{(k-1)(1-\varepsilon_0)-1}s^{-3/2}\delta^{l_0}\|\bar Z^{m_0}\mathring L\phi\|_{s,u}+\delta^{(k-1)(1-\varepsilon_0)}(1+\dl^{k(1-\ve_0)-1})s^{-5/2}\delta^{l_0}\|\slashed d\bar Z^{m_0}\phi\|_{s,u}\no\\
	&+\delta^{k(1-\varepsilon_0)-1}s^{-2}\delta^{l_0}\|\bar Z^{m_0}\varphi\|_{s,u}
+\delta^{k(1-\varepsilon_0)}s^{-3}\delta^{l_0}\|\slashed{\mathcal L}_{\bar Z}^{m_0}\slashed g\|_{s,u}
+\dl^{k(1-\ve_0)}s^{-2}\delta^{l_0}\|T\bar Z^{m_0}\varphi\|_{s,u}\no\\
	&+\delta^{k(1-\varepsilon_0)-\varepsilon_0}s^{-5/2}\delta^{l_0}\|\bar Z^{m_0}\mu\|_{s,u}+\delta^{(k-1)(1-\varepsilon_0)}s^{-3/2}\delta^{l_0}\|\slashed d\bar Z^{m_0}\varphi\|_{s,u}\no\\
	&+\dl^{(k+1)(1-\ve_0)}s^{-7/2}\dl^{l_1}\|\slashed{\mathcal L}_{\bar Z}^{m_1}\leftidx{^{(R)}}{\slashed\pi}_T\|_{s,u}+\dl^{(k-1)(1-\ve_0)-1}s^{-3/2}\delta^{l_0}\|\bar Z^{m_0}\mathscr A\|_{s,u}\no\\
	\lesssim&\delta^{k(1-\ve_0)-\ve_0+1/2}s^{-2}+\delta^{(k-1)(1-\ve_0)}s^{-3/2-\iota}\sqrt{\tilde E_{1,\leq m+2}}+\delta^{k(1-\ve_0)}s^{-2}\sqrt{\tilde E_{2,\leq m+2}}\no\\
	&+\dl^{k(1-\ve_0)}(\delta+\dl^{k(1-\ve_0)})s^{-2}\big\{\|R^m\slashed\triangle\mu\|_{s,u}+\delta^{k(1-\ve_0)
-\ve_0}\int_{t_0}^s\tau^{-3/2}\|R^m\slashed\triangle\mu\|_{\tau,u}d\tau\big\}\label{sdE}\\
&+\dl^{(2k-1)(1-\ve_0)-1}s^{-2}\sqrt{\int_0^u \tilde F_1(s,u')du'},\no
	\end{align}
where one has used the related $L^\infty$ estimates in Subsection \ref{BA} and the $L^2$ estimates in
Propositions \ref{L2chi}, \ref{8.2} and Corollary \ref{Zphi},  $l_p$ is the number of $T$
in $Z^{m_p}$ ($p=-1,0,1$) and $1\leq m_p\leq m+1-p$.

\vskip 0.1 true cm

{\bfseries(2) The estimate of $\|\mathscr C^m\|_{L^2(\Sigma_{s}^{u})}$}

\vskip 0.1 true cm

\eqref{mC} implies that
\begin{equation}\label{cm}
\begin{split}
&\dl^l\|\mathscr C^m\|_{L^2(\Sigma_{s}^{u})}\\
\lesssim&\delta^{(k-1)(1-\ve_0)-1}s^{-3/2}\delta^l\|R\bar Z^m\mathring L\phi\|_{L^2(\Sigma_{s}^{u})}
+\dl^{(2k-1)(1-\ve_0)}s^{-7/2}\dl^{l_0}\|\bar Z^{m_0}\varphi\|_{L^2(\Sigma_{s}^{u})}\\
&+\dl^{2k(1-\ve_0)}s^{-5}\dl^{l_0}\|\bar Z^{m_0}x\|_{L^2(\Sigma_{s}^{u})}+\dl^{(2m-1)(1-\ve_0)-1}s^{-7/2}\dl^{l_0}\|\bar Z^{m_0}\phi\|_{L^2(\Sigma_{s}^{u})}\\
&+s^{-1}\dl^l\|g^{\al\beta,\g\g_2\cdots\g_k}R\bar Z^m(\slashed d^Xx^i\slashed d_X\phi)g_{\g_2i}\varphi_{\g_3}\cdots\varphi_{\g_k}\mathring L_\al\mathring L_\beta\underline{\mathring L}\varphi_\g\|_{L^2(\Sigma_{s}^{u})}.
\end{split}
\end{equation}
Note that the last term in \eqref{cm} can be estimated by the $L^2$ norm of $\mathscr A$.
Indeed, by $\varphi_{\g_p}=-\mathring L_{\g_p}\tilde T^i\varphi_i+O(\delta^{2-\ve_0}s^{-3/2})$, $\underline{\mathring L}\varphi_\g=-\mathring L_\g(\mathring L^\al+2\tilde T^\al)T\varphi_\al+O(\dl^{1-\ve_0}s^{-3/2})$, $\mathring L_\al=\o_\al+O(\dl^{k(1-\ve_0)}s^{-1}\ln s)$ and $\o_i=\f{x^i}{\varrho(1+\check\varrho)}$, one has
\begin{align}\label{gR}
&\dl^l\|g^{\al\beta,\g\g_2\cdots\g_k}R\bar Z^m(\slashed d^Xx^i\slashed d_X\phi)g_{\g_2i}\varphi_{\g_3}\cdots\varphi_{\g_k}\mathring L_\al\mathring L_\beta\underline{\mathring L}\varphi_\g\|_{L^2(\Sigma_{s}^{u})}\no\\
\lesssim&\dl^{(k-1)(1-\ve_0)-1}s^{-1/2}\dl^l\|g^{\al\beta,\g i'\g_3\cdots\g_k}R\bar Z^m(\slashed d^Xx^i\slashed d_X\phi)m_{ii'}\o_{\g_3}\cdots\o_{\g_k}\o_\al\o_\beta\o_\g\|_{L^2(\Sigma_{s}^{u})}\no\\
&+(\dl^{(2k-1)(1-\ve_0)-1}\ln s+\dl^{(k-1)(1-\ve_0)})s^{-3/2}\dl^l\|R\bar Z^m(\slashed d^Xx\slashed d_X\phi)\|_{L^2(\Sigma_{s}^{u})}\no\\
\lesssim&\dl^{(k-1)(1-\ve_0)-1}s^{-1/2}\dl^l\|R\bar Z^m\mathscr A\|_{L^2(\Sigma_{s}^{u})}+\delta^{k(1-\ve_0)}s^{-3}\dl^{l_{-1}}\|\bar Z^{m_{-1}}x\|_{L^2(\Sigma_{s}^{u})}\\
&+(\dl^{(2k-1)(1-\ve_0)-1}\ln s+\dl^{(k-1)(1-\ve_0)})s^{-3/2}\dl^l\|\slashed dR\bar Z^m\phi\|_{L^2(\Sigma_{s}^{u})}\no\\
&+\delta^{k(1-\ve_0)}s^{-2}\delta^{l_0}\|\bar Z^{m_0}\check\varrho\|_{L^2(\Sigma_{s}^{u})}+\dl^{(k-1)(1-\ve_0)-1}s^{-3/2}\delta^{l_0}\|\bar Z^{m_0}\phi\|_{L^2(\Sigma_{s}^{u})}.\no
\end{align}
Substituting \eqref{gR} into \eqref{cm} and using Proposition \ref{L2chi} and Corollary \ref{Zphi} lead to
\begin{align}\label{Cm}
\dl^l&\|\mathscr C^m\|_{L^2(\Sigma_{s}^{u})}\lesssim\delta^{k(1-\ve_0)-1/2}(\dl+\dl^{k(1-\ve_0)}\ln s)s^{-5/2}+\delta^{(k-1)(1-\ve_0)}s^{-3/2-\iota}\sqrt{\tilde E_{1,\leq m+2}}\no\\
	&+\big\{\dl^{(2k-1)(1-\ve_0)}(1+\dl^{k(1-\ve_0)-1})\ln^2 s+\delta^{k(1-\ve_0)+\ve_0}\big\}s^{-5/2}\sqrt{\tilde E_{2,\leq m+2}}\\
&+\dl^{2k(1-\ve_0)}(\delta+\dl^{k(1-\ve_0)})s^{-5/2}\int_{t_0}^s\tau^{-3/2}\|R^m\slashed\triangle\mu\|_{\tau,u}d\tau\no\\
&+\dl^{k(1-\ve_0)}(\dl+\dl^{k(1-\ve_0)}\ln s)s^{-2}\|R^m\slashed\triangle\mu\|_{s,u}+\dl^{(2k-1)(1-\ve_0)-1}s^{-2}\sqrt{\int_0^u \tilde F_1(s,u')du'}.\no
\end{align}

{\bfseries (3) The estimate of $\|\slashed d\bar Z^m(|\check\chi|^2)\|_{L^2(\Sigma_{s}^u)}$}

Due to $|\check\chi|^2=(\textrm{tr}\check\chi)^2$, then $\ds\slashed d\bar Z^m(|\check\chi|^2)
=2\textrm{tr}\check\chi(\hat F^m-\f12\mathscr C^m+\slashed d\bar Z^m E)+2\sum_{p\leq m-1}(\bar Z^{m-p}\textrm{tr}\check\chi)(\slashed d\bar Z^p\textrm{tr}\check\chi)$. This, together with Proposition \ref{L2chi}, \eqref{sdE} and \eqref{Cm}, yields
\begin{equation}\label{Y-19}
\begin{split}
\delta^l\|\slashed dZ^m(|\check\chi|^2)\|_{s,u}
	\lesssim&\delta^{k(1-\varepsilon_0)}s^{-2}\ln s\big(\delta^l\|\hat F^m\|_{s,u}+\dl^l\|\mathscr C^m\|_{s,u}
+\delta^l\|\slashed d\bar Z^m E\|_{s,u}\big)\\
	&+\delta^{k(1-\varepsilon_0)}s^{-3}\ln s\big\{\delta^{l_1}\|\slashed{\mathcal L}_Z^{m_1}\check\chi\|_{s,u}
+\delta^{k(1-\varepsilon_0)}s^{-2}\ln s\delta^{l_0}\|\slashed{\mathcal L}_Z^{m_0}\slashed g\|_{s,u}\big\}\\
	\lesssim&\delta^{k(1-\varepsilon_0)}s^{-2}\ln s\delta^{l}\|\hat F^m\|_{s,u}+\delta^{2k(1-\varepsilon_0)
+1/2-\ve_0}s^{-4}\ln s\\
	&+\delta^{(2k-1)(1-\varepsilon_0)}s^{-7/2-\iota}\ln s\sqrt{\tilde E_{1,\leq m+2}}
+\delta^{2k(1-\varepsilon_0)}s^{-4}\ln s\sqrt{\tilde E_{2,\leq m+2}}\\
	&+\dl^{2k(1-\ve_0)}(\delta+\dl^{k(1-\ve_0)}\ln s)s^{-4}\ln s\|R^m\slashed\triangle\mu\|_{s,u}\\
	&+\dl^{3k(1-\ve_0)-\ve_0}(\delta+\dl^{k(1-\ve_0)})s^{-4}\ln s\int_{t_0}^s\tau^{-3/2}\|R^m\slashed\triangle\mu\|_{\tau,u}d\tau\\
	&+\dl^{(2k-1)(1-\ve_0)}(1+\dl^{k(1-\ve_0)-1})s^{-4}\ln s\sqrt{\int_0^u \tilde F_1(s,u')du'}
\end{split}
\end{equation}

{\bfseries (4) The estimate of $\|\hat e^m\|_{L^2(\Sigma_{s}^u)}$}

\begin{itemize}
	\item

One starts with $\slashed{\mathcal L}_{\bar Z}^{p_1}\slashed{\mathcal L}_{[\mathring L,\bar Z]}F^{p_2}$
in \eqref{eal} for $p_1+p_2=m-1$. Due to
$[\mathring L, \bar Z]=\leftidx^{{(\bar Z)}}\slashed\pi_{\mathring L}^XX$, then
$\slashed{\mathcal L}_{[\mathring L,\bar Z]}F^{p_2}=\leftidx^{{(\bar Z)}}{\slashed\pi_{\mathring L}}^X\slashed\nabla_XF^{p_2}+F^{p_2}_X\slashed\nabla\leftidx^{{(\bar Z)}}{\slashed\pi_{\mathring L}}^X$.
This implies that
\begin{align}\label{LF}
&\delta^l\|\slashed{\mathcal L}_{\bar Z}^{m_1}\slashed{\mathcal L}_{[\mathring L, \bar Z]}F^{m_2}\|_{s,u}\no\\
\lesssim&\delta^{k(1-\varepsilon_0)} s^{-2}\ln s\Big\{\delta^l\|\hat F^m\|_{s,u}+\dl^l\|\mathscr C^m\|_{s,u}
+s^{-1}\delta^{l_1}\|\slashed{\mathcal L}_{ Z}^{m_1}\check\chi\|_{s,u}+\delta^{l_2}\|\slashed dZ^{m_2}E\|_{s,u}\Big\}\no\\
&+\delta^{k(1-\varepsilon_0)-\varepsilon_0}s^{-7/2}\delta^{l_1}\Big\{\|\slashed{\mathcal L}_{ Z}^{m_1}\leftidx^{{(R)}}\slashed\pi_{\mathring L}\|_{s,u}+\delta\|\slashed{\mathcal L}_{ Z}^{m_1}\leftidx^{{(T)}}\slashed\pi_{\mathring L}\|_{s,u}+\delta^{k(1-\varepsilon_0)}s^{-1}\ln s\|\slashed{\mathcal L}_{Z}^{m_1}\slashed g\|_{s,u}\Big\}\no\\
\lesssim&\delta^{2k(1-\varepsilon_0)-\varepsilon_0+1/2}s^{-4}\ln s+\delta^{2k(1-\varepsilon_0)}(\dl^{\ve_0-1}\ln s+\dl^{-\ve_0})s^{-7/2-\iota}\ln s\sqrt{\tilde{E}_{1,\leq m+2}}\no\\
&+\delta^{2k(1-\varepsilon_0)}(1+\dl^{k(1-\ve_0)-1})s^{-4}\ln^2 s\sqrt{\tilde{E}_{2,\leq m+2}}\\
&+\dl^{3k(1-\ve_0)}(\delta+\dl^{k(1-\ve_0)})s^{-4}\ln s\int_{t_0}^s\tau^{-\f32}\|R^m\slashed\triangle\mu\|_{\tau,u}d\tau+\delta^{k(1-\varepsilon_0)} s^{-2}\ln s\delta^l\|\hat F^m\|_{s,u}\no\\
&+\dl^{2k(1-\ve_0)}(\delta+\dl^{k(1-\ve_0)}\ln s)s^{-4}\ln s\|R^m\slashed\triangle\mu\|_{s,u}+\dl^{2k(1-\ve_0)-1}s^{-7/2}\sqrt{\int_0^u \tilde F_1(s,u')du'}.\no
\end{align}

\item
For the term $\bar Z^{p_1}(-\f2\varrho+\mathcal E)\slashed d\bar Z^{p_2}\textrm{tr}\chi$
in \eqref{eal} with $p_1+p_2=m$ and $p_1\geq 1$, by making use of \eqref{GT}, \eqref{gLLL}, Corollary \ref{Zphi}, Propositions \ref{L2chi} and \ref{8.2}, one gets
\begin{align}\label{Zrho}
&\delta^l\|\bar Z^{p_1}(-\f2\varrho+\mathcal E)\slashed d\bar Z^{p_2}\textrm{tr}\chi\|_{s,u}\no\\
\lesssim&s^{-5/2}\delta^{l_1}\|\slashed{\mathcal L}_{\bar Z}^{m_1}\check\chi\|_{s,u}+\delta^{2k(1-\varepsilon_0)-\ve_0}s^{-9/2}\ln s\Big\{\delta^{l_1}\|\bar Z^{k_1}\mu\|_{s,u}+s^{-1}\delta^{l_0}\|\bar Z^{m_0}x\|_{s,u}\Big\}\no\\
&+\delta^{2k(1-\varepsilon_0)-1}s^{-4}\ln s\Big\{\sum_{i=1}^2\delta^{l_1}\|\bar Z^{m_1}\check L^i\|_{s,u}+\delta^{l_1}\|\bar Z^{m_1}\check\varrho\|_{s,u}+\delta^{l_0}\|\bar Z^{m_0}\varphi\|_{s,u}\Big\}\no\\
&+\delta^{(2k-1)(1-\varepsilon_0)-1}s^{-9/2}\ln s\dl^{l_0}\|\bar Z^{m_0}\phi\|_{s,u}+\delta^{k(1-\varepsilon_0)}s^{-9/2}\ln s\delta^{l_1}\|\slashed{\mathcal L}_{\bar Z}^{m_1}\slashed g\|_{s,u}\no\\
\lesssim&\delta^{k(1-\varepsilon_0)}s^{-7/2-\iota}\ln s\sqrt{\tilde{E}_{1,\leq m+2}}+\delta^{k(1-\varepsilon_0)+\ve_0}(1+\dl^{k(1-\ve_0)-1})s^{-4}\ln^2 s\sqrt{\tilde E_{2,\leq m+2}}\no\\
&+\delta^{k(1-\varepsilon_0)+1/2}s^{-4}\ln s+\dl^{(k-1)(1-\ve_0)}s^{-7/2}\sqrt{\int_0^u \tilde F_1(s,u')du'}\\
&+\delta^{3k(1-\ve_0)-\ve_0}(\delta+\dl^{k(1-\ve_0)})s^{-4}\ln s\int_{t_0}^s\tau^{-3/2}\|R^m\slashed\triangle\mu\|_{\tau,u}d\tau.\no
\end{align}

\item
Analogously to the treatment on $\mathscr C^m$,
the second term in the right side hand of \eqref{mD} can be estimated as follows:
\begin{align*}
&\dl^l\|g^{\al\beta,\g\g_2\cdots\g_k}\slashed d\bar Z^m(\delta_{\g_2}^0\mathring L\phi+g_{\g_2i}\slashed d^Xx^i\slashed d_X\phi)\mathring L(\mu^{-1}\varphi_{\g_3}\cdots\varphi_{\g_k}\mathring L_\al\mathring L_\beta\underline{\mathring L}\varphi_\g)\|_{s,u}\\
\lesssim&\dl^{(k-1)(1-\ve_0)-1}s^{-5/2}\dl^l\|R\bar Z^m\mathring L\phi\|_{s,u}\\
&+s^{-1}\dl^l\|g^{\al\beta,\g\g_2\cdots\g_k}R\bar Z^m(g_{\g_2i}\slashed d^Xx^i\slashed d_X\phi)(\mathring L\varphi_{\g_3})\varphi_{\g_4}\cdots\varphi_{\g_k}\mathring L_\al\mathring L_\beta\underline{\mathring L}\varphi_\g\|_{s,u}\\
&+s^{-1}\dl^l\|g^{\al\beta,\g\g_2\cdots\g_k}R\bar Z^m(g_{\g_2i}\slashed d^Xx^i\slashed d_X\phi)\varphi_{\g_3}\cdots\varphi_{\g_k}\mathring L_\al\mathring L_\beta(\mathring L\underline{\mathring L}\varphi_\g)\|_{s,u}\\
&+s^{-1}\dl^l\|g^{\al\beta,\g\g_2\cdots\g_k}R\bar Z^m(g_{\g_2i}\slashed d^Xx^i\slashed d_X\phi)\varphi_{\g_3}\cdots\varphi_{\g_k}\mathring L(\mu^{-1}\mathring L_\al\mathring L_\beta)\underline{\mathring L}\varphi_\g\|_{s,u},
\end{align*}
which, together with
\begin{align*}
\mathring L\varphi_{\g_3}=&-\mathring L_{\g_3}\tilde T^{i_3}\mathring L\varphi_{i_3}+O(\dl^{1-\ve_0}(\dl+\dl^{k(1-\ve_0)})s^{-5/2}),\\
\mathring L\underline{\mathring L}\varphi_\g=&\mu^{-1}(\mathring L\mu)\underline{\mathring L}\varphi_\g
-\mathring L_\g(\mu^{-1}\underline{\mathring L}^\al T\mathring L\varphi_\al)+\mu^{-1}\mathring L_\g(T\mathring L^\al)\underline{\mathring L}\varphi_\al\\
&+O(\dl^{-\ve_0}(\dl+\dl^{k(1-\ve_0)})s^{-5/2})
\end{align*}
and an analogous analysis for \eqref{gR}, yields
\begin{align}\label{Dm2}
&\dl^l\|g^{\al\beta,\g\g_2\cdots\g_k}\slashed d\bar Z^m(\delta_{\g_2}^0\mathring L\phi+g_{\g_2i}\slashed d^Xx^i\slashed d_X\phi)\mathring L(\mu^{-1}\varphi_{\g_3}\cdots\varphi_{\g_k}\mathring L_\al\mathring L_\beta\underline{\mathring L}\varphi_\g)\|_{s,u}\no\\
\lesssim&\dl^{(k-1)(1-\ve_0)-1}s^{-5/2}\dl^l\|R\bar Z^m\mathring L\phi\|_{s,u}+\delta^{(k-1)(1-\ve_0)-1}(\dl+\dl^{k(1-\ve_0)-\ve_0})s^{-4}\dl^{l_0}\|\bar Z^{m_0}\phi\|_{s,u}\no\\
&+\dl^{1-\ve_0}s^{-5/2}\dl^l\|g^{\al\beta,\g\g_2\cdots\g_k}R\bar Z^m(g_{\g_2i}\slashed d^Xx^i\slashed d_X\phi)\mathring L_{\g_3}\varphi_{\g_4}\cdots\varphi_{\g_k}\mathring L_\al\mathring L_\beta\underline{\mathring L}\varphi_\g\|_{s,u}\no\\
&+\dl^{k(1-\ve_0)-\ve_0}s^{-5/2}\dl^l\|g^{\al\beta,\g\g_2\cdots\g_k}R\bar Z^m(g_{\g_2i}\slashed d^Xx^i\slashed d_X\phi)\varphi_{\g_3}\cdots\varphi_{\g_k}\mathring L_\al\mathring L_\beta\underline{\mathring L}\varphi_\g\|_{s,u}\no\\
&+\dl^{-\ve_0}s^{-5/2}\dl^l\|g^{\al\beta,\g\g_2\cdots\g_k}R\bar Z^m(g_{\g_2i}\slashed d^Xx^i\slashed d_X\phi)\varphi_{\g_3}\cdots\varphi_{\g_k}\mathring L_\al\mathring L_\beta\mathring L_\g\|_{s,u}\no\\
&+\delta^{(k-1)(1-\ve_0)-1}(\dl+\dl^{k(1-\ve_0)})s^{-\f72}\dl^l\|\slashed d\bar Z^{m+1}\phi\|_{s,u}\no\\
&+\delta^{k(1-\ve_0)}(\dl+\dl^{k(1-\ve_0)-\ve_0})s^{-11/2}\dl^{l_{-1}}\|\bar Z^{m_{-1}}x\|_{s,u}\no\\
\lesssim&\delta^{k(1-\ve_0)-1/2}(\dl+\dl^{k(1-\ve_0)}\ln s)s^{-\f72}+\delta^{(k-1)(1-\ve_0)}s^{-5/2-\iota}\sqrt{\tilde E_{1,\leq m+2}}\no\\
&+\delta^{(k-1)(1-\ve_0)}\big\{\delta^{k(1-\ve_0)}(1+\dl^{k(1-\ve_0)-1})\ln^2 s+\dl\big\} s^{-7/2}\sqrt{\tilde E_{2,\leq m+2}}\\
&+\dl^{(2k-1)(1-\ve_0)-1}s^{-3}\sqrt{\int_0^u \tilde F_1(s,u')du'}+\dl^{k(1-\ve_0)}(\dl+\dl^{k(1-\ve_0)}\ln s)s^{-3}\|R^m\slashed\triangle\mu\|_{s,u}\no\\
&+\dl^{2k(1-\ve_0)}(\delta+\dl^{k(1-\ve_0)})s^{-7/2}\int_{t_0}^s\tau^{-3/2}\|R^m\slashed\triangle\mu\|_{\tau,u}d\tau.\no
\end{align}
Similarly, one can derive that
\begin{align}\label{Dm1}
&\dl^l\|g^{\al\beta,\g\g_2\cdots\g_k}\big([\slashed{\mathcal L}_{\bar Z}^m,\slashed{\mathcal L}_{\mathring L}]\slashed d(\delta_{\g_2}^0\mathring L\phi+g_{\g_2i}\slashed d^Xx^i\slashed d_X\phi)\big)\mu^{-1}\varphi_{\g_3}\cdots\varphi_{\g_k}\mathring L_\al\mathring L_\beta\underline{\mathring L}\varphi_\g\|_{s,u}\no\\
\lesssim&\dl^{k(1-\ve_0)}s^{-3}\ln s\dl^l\|g^{\al\beta,\g\g_2\cdots\g_k}R\bar Z^m(g_{\g_2i}\slashed d^Xx^i\slashed d_X\phi)\varphi_{\g_3}\cdots\varphi_{\g_k}\mathring L_\al\mathring L_\beta\underline{\mathring L}\varphi_\g\|_{s,u}\no\\
&+\dl^{k(1-\ve_0)}s^{-4}\delta^{l_1}\|\slashed{\mathcal L}_{\bar Z}^{m_1}\leftidx{^{(R)}}{\slashed\pi}_{\mathring L}\|_{s,u}+\dl^{k(1-\ve_0)+1}s^{-4}\delta^{l_1}\|\slashed{\mathcal L}_{\bar Z}^{m_1}\leftidx{^{(T)}}{\slashed\pi}_{\mathring L}\|_{s,u}\no\\
&+\dl^{(2k-1)(1-\ve_0)-1}s^{-7/2}\ln s\dl^{l_0}\|\bar Z^{m_0}\mathring L\phi\|_{s,u}+\dl^{(2k-1)(1-\ve_0)-1}s^{-9/2}\ln s\dl^{l_0}\|\bar Z^{m_0}\phi\|_{s,u}\no\\
&+\dl^{2k(1-\ve_0)}s^{-6}\ln s\dl^{l_0}\|\bar Z^{m_0}x\|_{s,u}+\dl^{2k(1-\ve_0)}s^{-5}\ln s\dl^{l_0}\|\slashed{\mathcal L}_{\bar Z}^{m_0}\slashed g\|_{s,u}\no\\
\lesssim&\dl^{2k(1-\ve_0)-1/2}(\dl+\dl^{k(1-\ve_0)}\ln s)s^{-9/2}\ln s+\dl^{(2k-1)(1-\ve_0)}s^{-7/2-\iota}\ln s\sqrt{\tilde E_{1,\leq m+2}}\no\\
&+\delta^{(2k-1)(1-\ve_0)}\big\{\delta^{k(1-\ve_0)}(1+\dl^{k(1-\ve_0)-1})\ln^2 s+\dl\big\} s^{-9/2}\ln s\sqrt{\tilde E_{2,\leq m+2}}\\\
&+\dl^{2k(1-\ve_0)}(\dl+\dl^{k(1-\ve_0)}\ln s)s^{-4}\ln s\|R^m\slashed\triangle\mu\|_{s,u}\no\\
&+\dl^{3k(1-\ve_0)}(\dl+\dl^{k(1-\ve_0)})s^{-9/2}\ln s\int_{t_0}^s\tau^{-3/2}\|R^m\slashed\triangle\mu\|_{\tau,u}d\tau\no\\
&+\dl^{(2k-1)(1-\ve_0)}(1+\dl^{k(1-\ve_0)-1}\ln s)s^{-4}\sqrt{\int_0^u \tilde F_1(s,u')du'}.\no
\end{align}
To estimate the last two terms in \eqref{mD}, one can get by direct computations that
\begin{align}\label{Dm34}
&\dl^l\|\slashed{\mathcal L}_{\bar Z}^m\big\{\mathring L(\delta_{\g_2}^0\mathring L\phi+g_{\g_2i}\slashed d^Xx^i\slashed d_X\phi)\slashed d(\mu^{-1}\varphi_{\g_3}\cdots\varphi_{\g_k}\mathring L_\al\mathring L_\beta\underline{\mathring L}\varphi_\g)\big\}\no\\
&+\sum_{\tiny\begin{array}{c}p_1+p_2=m\\p_1\geq 1\end{array}}\slashed{\mathcal L}_{\bar Z}^{p_1}\slashed d\mathring L(\delta_{\g_2}^0\mathring L\phi+g_{\g_2i}\slashed d^Xx^i\slashed d_X\phi)\bar Z^{p_2}(\mu^{-1}\varphi_{\g_3}\cdots\varphi_{\g_k}\mathring L_\al\mathring L_\beta\underline{\mathring L}\varphi_\g)\|_{s,u}\no\\
\lesssim&\dl^{(k-1)(1-\ve_0)-1}s^{-3/2}\big(\dl^{l_1}\|\bar Z^{m_1}\mathring L^2\phi\|_{s,u}+s^{-1}\dl^{l_0}\|\bar Z^{m_0}\mathring L\phi\|_{s,u}+s^{-2}\dl^{l_0}\|\bar Z^{m_0}\phi\|_{s,u}\big)\no\\
&+\dl^{k(1-\ve_0)}s^{-5}\big(\dl^{l_{-1}}\|Z^{m_{-1}}x\|_{s,u}+s\dl^{l_0}\|\bar Z^{m_0}\mu\|_{s,u}+s\dl^{l_0}\|\bar Z^{m_0}\check L^1\|_{s,u}+s\dl^{l_0}\|\bar Z^{m_0}\check L^2\|_{s,u}\big)\no\\
&+\dl^{(k-1)(1-\ve_0)}s^{-7/2}\dl^{l_0}\|\bar Z^{m_0}\varphi\|_{s,u}+\dl^{k(1-\ve_0)+\ve_0}s^{-7/2}\dl^{l_0}\big(\|T\bar Z^{m_0}\varphi\|_{s,u}+\|\slashed d\bar Z^{m_0}\varphi\|_{s,u}\big)\no\\
&+\dl^{k(1-\ve_0)+1}s^{-5}\dl^{l_1}\|\slashed{\mathcal L}_{\bar Z}^{m_1}\leftidx{^{(R)}}{\slashed\pi}_T\|_{s,u}+\dl^{2k(1-\ve_0)}(\dl+\dl^{k(1-\ve_0)})s^{-6}\ln s\dl^{l_0}\|\slashed{\mathcal L}_{\bar Z}^{m_0}\slashed g\|_{s,u}\no\\
\lesssim&\delta^{k(1-\ve_0)+1/2}s^{-7/2}+\delta^{(k-1)(1-\ve_0)}s^{-5/2-\iota}\sqrt{\tilde E_{1,\leq m+2}}+\delta^{k(1-\ve_0)+\ve_0}s^{-7/2}\sqrt{\tilde E_{2,\leq m+2}}\no\\
&+\dl^{k(1-\ve_0)}(\dl+\dl^{k(1-\ve_0)}\ln s)s^{-3}\|R^m\slashed\triangle\mu\|_{s,u}\no\\
&+\dl^{2k(1-\ve_0)}(\delta+\dl^{k(1-\ve_0)})s^{-7/2}\int_{t_0}^s\tau^{-3/2}\|R^m\slashed\triangle\mu\|_{\tau,u}d\tau\\
&+\dl^{(2k-1)(1-\ve_0)-1}(\dl s^{-1/2}+\dl^{k(1-\ve_0)})s^{-7/2}\sqrt{\int_0^u \tilde F_1(s,u')du'},\no
\end{align}
where one has used Proposition \ref{L2chi} and Corollary \ref{Zphi} in the last inequality.

It follows from \eqref{Dm2}-\eqref{Dm34} that
\begin{align}\label{Dm}
&\dl^l\|\mathscr D^m\|_{s,u}\no\\
\lesssim&\delta^{k(1-\ve_0)-1/2}(\dl+\dl^{k(1-\ve_0)}\ln s)s^{-7/2}+\delta^{(k-1)(1-\ve_0)}s^{-5/2-\iota}\sqrt{\tilde E_{1,\leq m+2}}\no\\
&+\delta^{(k-1)(1-\ve_0)}\big\{\delta^{k(1-\ve_0)}(1+\dl^{k(1-\ve_0)-1})\ln^2 s+\dl\big\} s^{-7/2}\sqrt{\tilde E_{2,\leq m+2}}\no\\
&+\dl^{(2k-1)(1-\ve_0)-1}s^{-3}\sqrt{\int_0^u \tilde F_1(s,u')du'}+\dl^{k(1-\ve_0)}(\dl+\dl^{k(1-\ve_0)}\ln s)s^{-3}\|R^m\slashed\triangle\mu\|_{s,u}\no\\
&+\dl^{2k(1-\ve_0)}(\delta+\dl^{k(1-\ve_0)})s^{-7/2}\int_{t_0}^s\tau^{-3/2}\|R^m\slashed\triangle\mu\|_{\tau,u}d\tau.
\end{align}

\item Note that \eqref{gLLL} implies $|G_{\mathring L\mathring L}^\g\mathring L_\g|\lesssim\dl^{k(1-\ve_0)}s^{-1}$.
Then it follows from the definitions of $e$ and $\mathscr B$ in \eqref{ee} and \eqref{LG} respectively that
\begin{align}\label{B}
&\dl^l\|\slashed {\mathcal L}_{\bar Z}^m\slashed de\|_{s,u}+\dl^l\|\slashed d\bar Z^m(\mu^{-1}\mathscr B)\|_{s,u}\no\\
\lesssim&\dl^{k(1-\ve_0)-\ve_0}s^{-7/2}\dl^{l_0}\|\bar Z^{m_0}\mu\|_{s,u}+\dl^{k(1-\ve_0)}s^{-3}\dl^{l}\|T\bar Z^{m+1}\varphi\|_{s,u}\no\\
&+\dl^{k(1-\ve_0)-1}s^{-2}\dl^{l}\|\mathring L\bar Z^{m+1}\varphi\|_{s,u}+\dl^{k(1-\ve_0)-1}s^{-3}\dl^{l_0}\| Z^{m_0}\varphi\|_{s,u}\no\\
&+\dl^{(k-1)(1-\ve_0)}s^{-5/2}\dl^l\|\slashed d\bar Z^{m+1}\varphi\|_{s,u}+\dl^{(k+1)(1-\ve_0)}s^{-9/2}\dl^{l_1}\|\slashed{\mathcal L}_{\bar Z}^{m_1}\leftidx{^{(R)}}{\slashed\pi}_T\|_{s,u}\no\\
&+\dl^{k(1-\ve_0)-\ve_0}s^{-7/2}\dl^{l_1}\|\slashed{\mathcal L}_{\bar Z}^{m_1}\leftidx{^{(R)}}{\slashed\pi}_{\mathring L}\|_{s,u}+\dl^{(k+1)(1-\ve_0)}s^{-7/2}\dl^{l_1}\|\slashed{\mathcal L}_{\bar Z}^{m_1}\leftidx{^{(T)}}{\slashed\pi}_{\mathring L}\|_{s,u}\no\\
&+\dl^{k(1-\ve_0)-1}s^{-3}\dl^{l_0}\|\bar Z^{m_0}\check\varrho\|_{s,u}+\dl^{k(1-\ve_0)-1}s^{-3}\dl^{l_0}\big(\|\bar Z^{m_0}\check L^1\|_{s,u}+\|\bar Z^{m_0}\check L^2\|_{s,u}\big)\no\\
&+\dl^{k(1-\ve_0)-\ve_0}s^{-9/2}\dl^{l_{-1}}\|\bar Z^{m_{-1}}x\|_{s,u}+\dl^{2k(1-\ve_0)-2}s^{-3}\dl^{l}\|\bar Z^{m+1}\mathscr A\|_{s,u}\no\\
&+\dl^{(k-1)(1-\ve_0)-1}s^{-5/2}\dl^{l}\|\bar Z^{m+1}\mathring L\phi\|_{s,u}+\dl^{(k-1)(1-\ve_0)-1}s^{-7/2}\dl^{l}\|Z^{m+1}\phi\|_{s,u}\no\\
&+\dl^{(k-1)(1-\ve_0)-1}(\dl+\dl^{k(1-\ve_0)}\ln s)s^{-7/2}\dl^l\|\slashed{\mathcal L}_{\bar Z}^{m+1}\slashed d\phi\|_{s,u}+\dl^{k(1-\ve_0)}s^{-4}\dl^{l_0}\|\slashed{\mathcal L}_{\bar Z}^{m_0}\slashed g\|_{s,u}\no\\
\lesssim&\delta^{k(1-\ve_0)+1/2-\ve_0}s^{-3}+\delta^{(k-1)(1-\ve_0)}s^{-5/2-\iota}\sqrt{\tilde E_{1,\leq m+2}}+\delta^{k(1-\ve_0)}s^{-3}\sqrt{\tilde E_{2,\leq m+2}}\no\\
&+\dl^{k(1-\ve_0)-1}s^{-2}\dl^{l}\|\mathring L\bar Z^{m+1}\varphi\|_{s,u}+\delta^{k(1-\ve_0)}(\dl+\dl^{k(1-\ve_0)}\ln s)s^{-3}\|R^m\slashed\triangle\mu\|_{s,u}\no\\
&+\dl^{2k(1-\ve_0)-\ve_0}(\delta+\dl^{k(1-\ve_0)})s^{-3}\int_{t_0}^s\tau^{-3/2}\|R^m\slashed\triangle\mu\|_{\tau,u}d\tau\\
&+\dl^{(2k-1)(1-\ve_0)-1}s^{-3}\sqrt{\int_0^u \tilde F_1(s,u')du'},\no
\end{align}
where one has used \eqref{GT} and \eqref{gLLL2} to estimate $G_{\mathring L\mathring L}^\g T\varphi_\g$.

\item
It follows from direct computations that
\begin{align}\label{ATL}
&\dl^l\|\slashed d\bar Z^m\big\{\mathscr A\tilde T^i(\mathring L\varphi_i)\tilde T^{i_4}\cdots\tilde T^{i_k}\varphi_{i_4}\cdots\varphi_{i_k}(\mu^{-2}\underline{\mathring L}^\nu T\varphi_{\nu})\big\}\|_{s,u}\no\\
\lesssim&\dl^{(k-1)(1-\ve_0)-1}s^{-5/2}\dl^l\|\bar Z^{m+1}\mathscr A\|_{s,u}
+\dl^{(k-1)(1-\ve_0)-1}s^{-7/2}\dl^{l_0}\|\bar Z^{m_0}\phi\|_{s,u}\no\\
&+\dl^{k(1-\ve_0)}s^{-5}\big(\dl^{l_{-1}}\|\bar Z^{m_{-1}}x\|_{s,u}+s\dl^{l_0}\|\bar Z^{m_0}\check L^1\|_{s,u}+s\dl^{l_0}\|\bar Z^{m_0}\check L^2\|_{s,u}+s\dl^{l_0}\|\bar Z^{m_0}\check\varrho\|_{s,u}\big)\no\\
&+\dl^{k(1-\ve_0)+1}s^{-5}\dl^{l_1}\|\slashed{\mathcal L}_{\bar Z}^{m_1}\leftidx{^{(R)}}{\slashed\pi}_T\|_{s,u}+\dl^{k(1-\ve_0)}s^{-4}\dl^{l_0}\|\bar Z^{m_0}\mu\|_{s,u}\no\\
&+\dl^{(k-1)(1-\ve_0)}s^{-5/2-\iota}\sqrt{\tilde E_{1,\leq m+2}}+\dl^{k(1-\ve_0)+\ve_0}s^{-7/2}\sqrt{\tilde E_{2,\leq m+2}}\no\\
\lesssim&\delta^{k(1-\ve_0)-1/2}(\dl+\dl^{k(1-\ve_0)}\ln s)s^{-7/2}+\delta^{(k-1)(1-\ve_0)}s^{-5/2-\iota}\sqrt{\tilde E_{1,\leq m+2}}\no\\
&+\delta^{(k-1)(1-\ve_0)}\big\{\delta^{k(1-\ve_0)}(1+\dl^{k(1-\ve_0)-1})\ln^2 s+\dl\big\} s^{-7/2}\sqrt{\tilde E_{2,\leq m+2}}\\
&+\dl^{(2k-1)(1-\ve_0)-1}s^{-3}\sqrt{\int_0^u \tilde F_1(s,u')du'}+(\dl+\dl^{k(1-\ve_0)})\delta^{k(1-\ve_0)}s^{-3}\|R^m\slashed\triangle\mu\|_{s,u}\no\\
&+(\dl+\dl^{k(1-\ve_0)})\dl^{2k(1-\ve_0)}s^{-7/2}\int_{t_0}^s\tau^{-3/2}\|R^m\slashed\triangle\mu\|_{\tau,u}d\tau.\no
\end{align}

\item
It remains to estimate $\slashed{\mathcal L}_{\bar Z}^m(\textrm{tr}\chi\slashed d\mathcal E)$ in $\slashed{\mathcal L}_{\bar Z}^m e^0$
of \eqref{eal}.
Since $\mathcal E$ has an analogous expression as $E$ in \eqref{Ee}, it follows from \eqref{sdE} that
\begin{align}\label{Le0}
&\delta^l\|\slashed{\mathcal L}_{\bar Z}^m (\textrm{tr}\chi\slashed d\mathcal E)\|_{s,u}\no\\
\lesssim&\dl^{k(1-\ve_0)-\ve_0}s^{-5/2}\dl^{l_0}\|\slashed{\mathcal L}_{\bar Z}^{m_1}\check\chi\|_{s,u}+s^{-1}\dl^{l_1}\|\slashed d\bar Z^{m_1}\mathcal E\|_{s,u}+\dl^{2k(1-\ve_0)-\ve_0}s^{-9/2}\ln s\dl^{l_1}\|\slashed{\mathcal L}_{\bar Z}^{m_1}\slashed g\|_{s,u}\no\\
	\lesssim&\delta^{k(1-\ve_0)-\ve_0+1/2}s^{-3}+\delta^{(k-1)(1-\ve_0)}s^{-5/2-\iota}\sqrt{\tilde E_{1,\leq m+2}}+\delta^{k(1-\ve_0)}s^{-3}\sqrt{\tilde E_{2,\leq m+2}}\no\\
&+\dl^{k(1-\ve_0)}(\delta+\dl^{k(1-\ve_0)})s^{-3}\big\{\|R^m\slashed\triangle\mu\|_{s,u}+\delta^{k(1-\ve_0)
-\ve_0}\int_{t_0}^s\tau^{-3/2}\|R^m\slashed\triangle\mu\|_{\tau,u}d\tau\big\}\\
&+\dl^{(2k-1)(1-\ve_0)-1}s^{-3}\sqrt{\int_0^u \tilde F_1(s,u')du'}.\no
\end{align}
\end{itemize}

Combining the estimates \eqref{LF}-\eqref{Le0} with \eqref{hem} yields
\begin{align}\label{Y-21}
&\delta^l\|\hat e^m\|_{s,u}\no\\
\lesssim&\delta^{k(1-\ve_0)+1/2-\ve_0}s^{-3}+\delta^{(k-1)(1-\ve_0)}s^{-5/2-\iota}\sqrt{\tilde E_{1,\leq m+2}}
+\delta^{k(1-\ve_0)}s^{-3}\sqrt{\tilde E_{2,\leq m+2}}\no\\
&+\dl^{k(1-\ve_0)-1}s^{-2}\dl^{l}\|\mathring L\bar Z^{m+1}\varphi\|_{s,u}+\delta^{k(1-\ve_0)}(\dl+\dl^{k(1-\ve_0)}\ln s)s^{-3}\|R^m\slashed\triangle\mu\|_{s,u}\\
&+\dl^{2k(1-\ve_0)-\ve_0}(\delta+\dl^{k(1-\ve_0)})s^{-3}\int_{t_0}^s\tau^{-3/2}\|R^m\slashed\triangle\mu\|_{\tau,u}d\tau
+\delta^{k(1-\varepsilon_0)} s^{-2}\ln s\delta^l\|\hat F^m\|_{s,u}\no\\
&+\dl^{(k-1)(1-\ve_0)}(s^{-1/2}+\dl^{k(1-\ve_0)-1})s^{-3}\sqrt{\int_0^u \tilde F_1(s,u')du'}.\no
\end{align}

Inserting \eqref{sdE}, \eqref{Cm}, \eqref{Y-19} and \eqref{Y-21} into \eqref{Eal} gives
\begin{equation}\label{Fal}
\begin{split}
\delta^l\|&\hat F^m\|_{s,u}\lesssim\delta^{k(1-\varepsilon_0)+1/2-\ve_0}s^{-2}
+\delta^{(k-1)(1-\varepsilon_0)}s^{-3/2-\iota}\sqrt{\tilde E_{1,\leq m+2}}\\
&+\delta^{k(1-\varepsilon_0)} s^{-2}\sqrt{\tilde E_{2,\leq m+2}}+\dl^{k(1-\ve_0)}s^{-2}\int_{t_0}^s(\dl+\dl^{k(1-\ve_0)}\ln\tau)\tau^{-1/2}\|R^m\slashed\triangle\mu\|_{\tau,u}d\tau\\
&+\dl^{k(1-\ve_0)-1}s^{-3/2-\iota}\sqrt{\int_0^u\tilde F_{1,m+2}(s,u')du'}.
\end{split}\end{equation}
In addition, $\slashed d\bar Z^m\textrm{tr}\chi=\hat F^m+\slashed d \bar Z^m E-\f12\mathscr C^m$
holds due to the definition of $\hat F^m$. Hence, one has by inequalities \eqref{Fal}, \eqref{sdE} and \eqref{Cm} that
\begin{align}\label{d}
&\delta^l\|\slashed d\bar Z^m\textrm{tr}\chi\|_{s,u}
+\delta^l\|\slashed\nabla{\slashed{\mathcal L}}_{\bar Z}^m\check{\chi}\|_{s,u}\no\\
\lesssim&\delta^{k(1-\varepsilon_0)+1/2-\ve_0}s^{-2}
+\delta^{(k-1)(1-\varepsilon_0)}s^{-3/2-\iota}\sqrt{\tilde E_{1,\leq m+2}}
+\delta^{k(1-\varepsilon_0)} s^{-2}\sqrt{\tilde E_{2,\leq m+2}}\no\\
&+\dl^{k(1-\ve_0)-1}s^{-3/2-\iota}\sqrt{\int_0^u\tilde F_{1,m+2}(s,u')du'}+\dl^{k(1-\ve_0)}(\dl+\dl^{k(1-\ve_0)}\ln s)s^{-2}\|R^m\slashed\triangle\mu\|_{(s,u)}\no\\
&+\dl^{k(1-\ve_0)}s^{-2}\int_{t_0}^s(\dl+\dl^{k(1-\ve_0)}\ln\tau)\tau^{-1/2}\|R^m\slashed\triangle\mu\|_{\tau,u}d\tau.
\end{align}

In summary, for any $Z\in\{\varrho\mathring L, T, R\}$, it follows from  \eqref{nz}, \eqref{dch} and \eqref{d} that
\begin{align}\label{dnnchi}
&\delta^l\|\slashed dZ^m\textrm{tr}\chi\|_{s,u}
+\delta^l\|\slashed\nabla{\slashed{\mathcal L}}_{Z}^m\check{\chi}\|_{s,u}\no\\
\lesssim&\delta^{k(1-\varepsilon_0)+1/2-\ve_0}s^{-2}
+\delta^{(k-1)(1-\varepsilon_0)}s^{-3/2-\iota}\sqrt{\tilde E_{1,\leq m+2}}+\delta^{k(1-\varepsilon_0)} s^{-2}\sqrt{\tilde E_{2,\leq m+2}}\no\\
&+\dl^{k(1-\ve_0)-1}s^{-3/2-\iota}\sqrt{\int_0^u\tilde F_{1,m+2}(s,u')du'}+\dl^{k(1-\ve_0)}(\dl+\dl^{k(1-\ve_0)}\ln s)s^{-2}\|R^m\slashed\triangle\mu\|_{(s,u)}\no\\
&+\dl^{k(1-\ve_0)}s^{-2}\int_{t_0}^s(\dl+\dl^{k(1-\ve_0)}\ln\tau)\tau^{-1/2}\|R^m\slashed\triangle\mu\|_{\tau,u}d\tau.
\end{align}

\begin{remark}
	Here we emphasize that due to $\slashed dZ^m\textrm{tr}\check\chi=\slashed dZ^m\textrm{tr}\chi$
by \eqref{errorv}, \eqref{dnnchi} gives also the $L^2$ estimate of $\slashed dZ^m\textrm{tr}\check\chi$.
\end{remark}

\subsection{Estimates on the derivatives of $\slashed\triangle\mu$}\label{trimu}

Although the main idea in this subsection parallels that in \cite[Section 9.2]{Ding4}, in order to
close the energy estimate at the beginning of Section \ref{YY} as $\varepsilon_0$ approaches $\varepsilon_k^*$,
we need to obtain more refined estimates for each term in the transport equation under
the general higher null condition (for instance, see \eqref{ZmL} below).

Similarly to the analysis in \cite[Section 9.2]{Ding4}, one can set
\begin{equation}\label{EEF}
\begin{split}
\tilde E=&-\f12\mu (G_{\mathring L\mathring L}^\gamma+2G_{\tilde T\mathring L}^\gamma)\slashed\triangle\varphi_\gamma
+\f12\mu^{-1}G_{\mathring L\mathring L}^\gamma T\underline{\mathring L}\varphi_\gamma,\\
\tilde{\mathcal E}=&-\f12G^\gamma_{\mathring L\mathring L}\mathring L\varphi_\gamma-G^\gamma_{\tilde T\mathring L}\mathring L\varphi_\gamma,\\
\tilde F=&\slashed\triangle\mu-\tilde E.
\end{split}
\end{equation}
Then it holds that
\begin{equation*}
\begin{split}
\mathring L\tilde F=&-2\textrm{tr}\check\chi\tilde F+(-\f2\varrho+\tilde{\mathcal E})\tilde F
+(\slashed d^X\textrm{tr}\chi)\mathcal Q_X+\tilde e
\end{split}
\end{equation*}
with
\begin{align*}
\mathcal Q_X=&-\slashed d_X\mu-\mu G^\gamma_{X\tilde T}\mathring L\varphi_\gamma
-\mu G^{\gamma}_{\tilde T\mathring L}\slashed d_X\varphi_\gamma
-\mu G^\gamma_{\mathring L\mathring L}\slashed d_X\varphi_\gamma+ G^\gamma_{X\mathring L}T\varphi_\gamma
\end{align*}
and
\begin{align}\label{te}
\tilde e=&2(\slashed d^X\mu)\slashed d_X\tilde{\mathcal E}+(-2\textrm{tr}\check\chi-\f2\varrho+\tilde{\mathcal E})\tilde E-\f12\mathring L(\mu^{-1}G_{\mathring L\mathring L}^\gamma)T\underline{\mathring L}\varphi_\gamma\no\\
&-\f12\mu^{-1}G_{\mathring L\mathring L}^\gamma\big\{T\mu(\slashed\triangle\varphi_\gamma)+\leftidx^{{(T)}}{\slashed\pi}_{\mathring L}^X\slashed d_X\underline{\mathring L}\varphi_\gamma-\f{1}{2\varrho^2}\underline{\mathring L}\varphi_\gamma-\f{1}{2\varrho}T\underline{\mathring L}\varphi_\gamma+T H_\gamma\big\}\no\\
&+f_1(\varphi, \mathring L^1, \mathring L^2,\slashed dx,\varrho\slashed\triangle x)\left(
\begin{array}{ccc}
\mu\mathring L\varphi\\
\mu\slashed d\varphi\\
T\varphi\\
\mu\varphi\textrm{tr}\chi
\end{array}
\right)
\left(\begin{array}{ccc}
\varphi^{k-2}\slashed\triangle\varphi\\
\varphi^{k-1}\varrho^{-1}\slashed\triangle x\\
\varphi^{k-1}\textrm{tr}\check\chi\slashed\triangle x
\end{array}
\right)\\
&+f_2(\varphi,\mathring L^1,\mathring L^2,\varrho\slashed d\mathring L^1,\varrho\slashed d\mathring L^2,\varrho\slashed\triangle x)\varphi^{k-3}\left(
\begin{array}{ccc}
\varrho^{-1}\varphi^{k+1}\\
\slashed d\varphi\\
\varphi\slashed d\mathring L^1\\
\varphi\slashed d\mathring L^2
\end{array}
\right)
\left(
\begin{array}{ccc}
\slashed d\varphi\\
\mathring L\varphi\\
\varphi\slashed d\mathring L^1\\
\varphi\slashed d\mathring L^2
\end{array}
\right)
\left(
\begin{array}{ccc}
\mu\mathring L\varphi\\
\mu\slashed d\varphi\\
T\varphi
\end{array}
\right)\no\\
&+f_3(\varphi,\mathring L^1,\mathring L^2,\slashed dx)\slashed d^X\mu\slashed d_X\varphi\left(
\begin{array}{ccc}
\varphi^{2(k-1)}\slashed d\varphi\\
\varphi^{2(k-1)}\mathring L\varphi\\
\varphi^{k-1}\textrm{tr}\chi
\end{array}
\right)
+f_4(\varphi,\mathring L^1,\mathring L^2,\slashed dx)\left(
\begin{array}{ccc}
\slashed dT\varphi\\
\mu\slashed d\mathring L\varphi
\end{array}
\right)
\left(
\begin{array}{ccc}
\varphi^{k-2}\slashed d\varphi\\
\varphi^{k-1}\slashed d\mathring L^1\\
\varphi^{k-1}\slashed d\mathring L^2
\end{array}
\right).\no
\end{align}

Analogously to the introduction of $F^k$ in Subsection \ref{trchi}, one can set $\tilde F^m=\bar Z^m\slashed\triangle\mu-\bar Z^m\tilde E$
with $\bar Z$ being any vector field in $\{T,R\}$.  Then an induction argument yields that for $m\geq 1$,
\begin{equation}\label{LbarF}
\begin{split}
\mathring L\tilde F^m=&(-2\textrm{tr}\check\chi-\f2\varrho+\tilde{\mathcal E})\tilde F^m+(\slashed d^X{\bar Z}^k\textrm{tr}\chi)\mathcal Q_X+\leftidx^{{(\bar Z)}}{\slashed\pi}_{\mathring L}^X(\slashed d_X\tilde F^{m-1})+\tilde e^m,
\end{split}
\end{equation}
where
\begin{equation}\label{bare}
\begin{split}
\tilde e^m=&\underbrace{\sum_{{\mbox{\tiny$\begin{array}{c}m_1+m_2=m-1\\m_1\geq 1\end{array}$}}}(\slashed{\mathcal L}_{\bar Z}^{m_1}\leftidx^{{(\bar Z)}}{\slashed\pi}_{\mathring L}^X)(\slashed d_X\tilde F^{m_2})}_{\textrm{vanish when}\ m=1}+\bar Z^{m}\t e\\
&+\sum_{{\mbox{\tiny$\begin{array}{c}m_1+m_2=m\\m_1\geq 1\end{array}$}}}\Big\{\bar Z^{m_1}(\textrm{tr}\check\chi-\f2\varrho+\tilde{\mathcal E})\tilde F^{m_2}+(\slashed d^X{\bar Z}^{m_2}\textrm{tr}\chi)\slashed{\mathcal L}_{\bar Z}^{m_1}\mathcal Q_X\Big\}.
\end{split}
\end{equation}
As \eqref{gLLL}, it follows from \eqref{null} and $T^2\varphi_\g=T(\mu\dl_\g^0\tilde T^i\mathring L\varphi_i
-\mathring L_\g\tilde T^iT\varphi_i+\mu\tilde T^ig_{\g j}\slashed d^Xx^j\slashed d_X\varphi_i)$ that
\begin{equation}\label{GTT}
	G_{\mathring L\mathring L}^\g T^2\varphi_\g=f(\varphi, \mathring L^1,\mathring L^2,\f{x}\varrho,\check\varrho)\left(
	\begin{array}{ccc}
	\varphi^kT^2\varphi\\
	\varphi^{k-1}T(\mu\tilde T^a\mathring L\varphi_a)\\
	\varphi^{k-1}T(\mathring L_\g\tilde T^a)T\varphi_a\\
	\varphi^{k-1}T(\mu\tilde T^a g_{\g j}\slashed d^Xx^j\slashed d_X\varphi_a)\\
	\check L^{1}\varphi^{k-1}T^2\varphi\\
    \check L^{2}\varphi^{k-1}T^2\varphi\\
	\check\varrho\varphi^{k-1}T^2\varphi\\
	\varphi^{k-2}(\mathring L\phi)T^2\varphi\\
	\varphi^{k-2}(\slashed d^Xx)(\slashed d_X\phi)T^2\varphi
	\end{array}
	\right).
\end{equation}
Thus, $\dl^l|\tilde F^m(t_0,u,\vartheta)|\lesssim\delta^{k(1-\varepsilon_0)-1-\ve_0}$ holds by \eqref{GTT}.
In addition, by virtue of Propositions \ref{L2chi}, \ref{8.2} and Corollary \ref{Zphi}, one can apply \eqref{GTT} again to obtain
\begin{align}\label{ZmGT}
&\dl^l\|\bar Z^m(\mu^{-1}G_{\mathring L\mathring L}^\g T^2\varphi_\g)\|_{s,u}\no\\
\lesssim&\dl^{k(1-\ve_0)-1}s^{-1}\dl^l\|TZ^{m+1}\varphi\|_{s,u}+\dl^{k(1-\ve_0)-\ve_0}s^{-5/2}\dl^{l_1}\|\slashed{\mathcal L}_{\bar Z}^{m_1}\leftidx{^{(R)}}{\slashed\pi}_T\|_{s,u}\no\\
&+\dl^{k(1-\ve_0)-2}s^{-1}\dl^{l_0}\big(\|\bar Z^{m_0}\check L^1\|_{s,u}+\|\bar Z^{m_0}\check L^2\|_{s,u}+\|\bar Z^{m_0}\check\varrho\|_{s,u}+\|Z^{m_0}\varphi\|_{s,u}\big)\no\\
&+\dl^{(k-1)(1-\ve_0)-2}s^{-3/2}\dl^{l_0}\|Z^{m_0}\phi\|_{s,u}+\dl^{(k+1)(1-\ve_0)-2}s^{-5/2}\dl^{l_{-1}}\|\bar Z^{m_{-1}}x\|_{s,u}\no\\
&+\dl^{k(1-\ve_0)-1}s^{-2}\dl^{l_0}\|\slashed{\mathcal L}_{\bar Z}^{m_0}\slashed g\|_{s,u}+\dl^{k(1-\ve_0)-1-\ve_0}s^{-3/2}\dl^{l_0}\|\bar Z^{m_0}\mu\|_{s,u}\no\\
\lesssim&\dl^{k(1-\ve_0)-1/2-\ve_0}s^{-1}+\dl^{k(1-\ve_0)-1}s^{-1}\big\{(1+\dl^{k(1-\ve_0)-1})s^{-\iota}\sqrt{\tilde E_{1,\leq m+2}}+\sqrt{\tilde E_{2,\leq m+2}}\big\}\no\\
&+\dl^{2k(1-\ve_0)-1-\ve_0}(\dl+\dl^{k(1-\ve_0)})s^{-1}\int_{t_0}^s\tau^{-3/2}\|R^m\slashed\triangle\mu\|_{\tau,u}d\tau\\
&+\dl^{(2k-1)(1-\ve_0)-2}s^{-1}\sqrt{\int_0^u\tilde F_{1,m+2}(s,u')du'}.\no
\end{align}
Thus, it follows from the definition of $\tilde E$ in \eqref{EEF} and \eqref{ZmGT} that
\begin{align}\label{ZmE}
&\dl^l\|\bar Z^m\tilde E\|_{s,u}\no\\
\lesssim&\dl^{k(1-\ve_0)-1/2-\ve_0}s^{-1}+\dl^{k(1-\ve_0)-1}s^{-1}\big\{(1+\dl^{k(1-\ve_0)-1})s^{-\iota}\sqrt{\tilde E_{1,\leq m+2}}+\sqrt{\tilde E_{2,\leq m+2}}\big\}\no\\
&+\dl^{2k(1-\ve_0)-1-\ve_0}(\dl+\dl^{k(1-\ve_0)})s^{-1}\int_{t_0}^s\tau^{-3/2}\|R^m\slashed\triangle\mu\|_{\tau,u}d\tau\\
&+\dl^{(2k-1)(1-\ve_0)-2}s^{-1}\sqrt{\int_0^u\tilde F_{1,m+2}(s,u')du'}.\no
\end{align}
Replacing $R^m\slashed\triangle \mu$ by $\tilde F^m+R^m\tilde E$ in \eqref{ZmE} and using Gronwall's inequality yield
\begin{align}\label{ZmtE}
&\dl^l\|\bar Z^m\tilde E\|_{s,u}\no\\
\lesssim&\dl^{k(1-\ve_0)-1/2-\ve_0}s^{-1}+\dl^{k(1-\ve_0)-1}s^{-1}\big\{(1+\dl^{k(1-\ve_0)-1})\sqrt{\tilde E_{1,\leq m+2}}+\sqrt{\tilde E_{2,\leq m+2}}\big\}\no\\
&+\dl^{2k(1-\ve_0)-1-\ve_0}(\dl+\dl^{k(1-\ve_0)})s^{-1}\int_{t_0}^s\tau^{-3/2}\dl^l\|\tilde F^m\|_{\tau,u}d\tau\\
&+\dl^{(2k-1)(1-\ve_0)-2}s^{-1}\sqrt{\int_0^u\tilde F_{1,m+2}(s,u')du'}.\no
\end{align}
Set $F(s,u,\vartheta)=\varrho^2\tilde F^m(s,u,\vartheta)-\varrho_0^2\tilde F^m(t_0,u,\vartheta)$ in \eqref{Ff}.
Then applying \eqref{LbarF} yields
\begin{align}\label{tildeF}
&\delta^l\varrho^{3/2}\|\tilde F^m\|_{s,u}\no\\
\lesssim&\delta^{k(1-\varepsilon_0)-1/2-\ve_0}+\int_{t_0}^s\tau^{3/2}\delta^l\|\tilde e^m\|_{\tau,u}d\tau+\delta^{k(1-\varepsilon_0)-1}\int_{t_0}^s\tau^{1/2}\delta^l\|\slashed d\bar Z^m\textrm{tr}\chi\|_{\tau,u}d\tau.
\end{align}
The last term on the right hand side above can be estimated by \eqref{d} and \eqref{ZmtE}
as follows
\begin{align}\label{Y-23}
&\delta^{k(1-\varepsilon_0)-1}\int_{t_0}^s\tau^{1/2}\dl^l\|\slashed d\bar Z^m\textrm{tr}\chi\|_{\tau,u}d\tau\no\\
\lesssim&\delta^{2k(1-\varepsilon_0)-1/2-\ve_0}+\delta^{(2k-1)(1-\varepsilon_0)-1}\sqrt{\tilde E_{1,\leq m+2}}+\delta^{2k(1-\varepsilon_0)-1}\sqrt{\tilde E_{2,\leq m+2}}\\
+&\dl^{2k(1-\ve_0)-1}\Big\{\int_{t_0}^s\tau^{-1}(\dl+\dl^{k(1-\ve_0)}\ln\tau)\dl^l\|\tilde F^m\|_{\tau,u}d\tau+\dl^{-1}\sqrt{\int_0^u\tilde F_{1,m+2}(s,u')du'}\Big\}.\no
\end{align}

Next, one estimates $\tilde e^m$. By \eqref{bare}, one starts with $\bar Z^m\tilde e$. Note that $\tilde e$ in \eqref{te} contains
the following term
\begin{equation}\label{LGT}
\begin{split}
&\mathring L(\mu^{-1}G_{\mathring L\mathring L}^\g)T^2\varphi_\g\\
=&-\mu^{-2}(\mathring L\mu)G_{\mathring L\mathring L}^\g T^2\varphi_\g+2\mu^{-1}G_{i\beta}^\g(\underbrace{\mathring L\mathring L^i}_{\eqref{LL}})\mathring L^\beta(T^2\varphi_\g)+\mu^{-1}(\p_{\varphi_\nu}G_{\al\beta}^\g)(\mathring L\varphi_\nu)\mathring L^\al\mathring L^\beta(T^2\varphi_\g),
\end{split}
\end{equation}
where
\begin{equation*}
\begin{split}
&(\p_{\varphi_\nu}G_{\al\beta}^\g)(\mathring L\varphi_\nu)\mathring L^\al\mathring L^\beta(T^2\varphi_\g)\\
=&-k(k-1)m_{\al\al'}m_{\beta\beta'}g^{\al'\beta',\g\nu\g_3\dots\g_k}\varphi_{\g_3}\cdots\varphi_{\g_k}(\mathring L\varphi_\nu)\mathring L^\al\mathring L^\beta(T^2\varphi_\g)+f(\varphi,\mathring L^1,\mathring L^2)\varphi^{k-1}(\mathring L\varphi)T^2\varphi.
\end{split}
\end{equation*}
It follows from $\mathring L\varphi_\nu=\dl_\nu^0\mathring L^\al\mathring L\varphi_\al-\mathring L_\nu\tilde T^i\mathring L\varphi_i+g_{\nu j}\slashed d^Xx^j(\mathring L^\al\slashed d_X\varphi_\al)$, $T^2\varphi_\g=T(\mu\dl_\g^0\tilde T^i\mathring L\varphi_i-\mathring L_\g\tilde T^iT\varphi_i+\mu\tilde T^ig_{\g j}\slashed d^Xx^j\slashed d_X\varphi_i)$ and \eqref{null} that
\begin{equation}\label{pGLT}
(\p_{\varphi_\nu}G_{\al\beta}^\g)(\mathring L\varphi_\nu)\mathring L^\al\mathring L^\beta(T^2\varphi_\g)=f(\varphi, \check L^1,\check L^2,\f{x}\varrho,\check\varrho)\left(
\begin{array}{ccc}
\varphi^{k-1}(\mathring L\varphi)T^2\varphi\\
\varphi^{k-2}(\mathring L\varphi)T(\mu\tilde T^a\mathring L\varphi_a)\\
\varphi^{k-2}(\mathring L\varphi)T(\mathring L_\g\tilde T^a)T\varphi_a\\
\varphi^{k-2}(\mathring L\varphi)T(\mu\tilde T^a g_{\g j}\slashed d^Xx^j\slashed d_X\varphi_a)\\
\check L^{1}\varphi^{k-2}(\mathring L\varphi)T^2\varphi\\
\check L^{2}\varphi^{k-2}(\mathring L\varphi)T^2\varphi\\
\check\varrho\varphi^{k-2}(\mathring L\varphi)T^2\varphi\\
\varphi^{k-3}(\mathring L\phi)(\mathring L\varphi)T^2\varphi\\
\varphi^{k-3}(\slashed d^Xx)(\slashed d_X\phi)(\mathring L\varphi)T^2\varphi\\
\varphi^{k-2}(\mathring L^\al\mathring L\varphi_\al)T^2\varphi\\
\varphi^{k-2}\slashed d^Xx(\mathring L^\al\slashed d_X\varphi_\al)T^2\varphi
\end{array}
\right).
\end{equation}
Inserting \eqref{pGLT} into \eqref{LGT}, utilizing \eqref{LL} and \eqref{dL} to derive $\mathring L^\al\mathring L\varphi_\al=\mathring L^2\phi+f(\varphi,\mathring L^1,\mathring L^2)\varphi^{k}\left(
\begin{array}{ccc}\mathring L\varphi\\
\slashed d^X\varphi\slashed d_Xx\end{array}
\right)$ and $\mathring L^\al\slashed d_X\varphi_\al=\slashed d_X\mathring L\phi-\textrm{tr}\chi\slashed d_X\phi+f(\varphi,\mathring L^1,\mathring L^2)\varphi^{k}\left(
\begin{array}{ccc}\slashed d_Xx\mathring L\varphi\\
\slashed d_X\varphi\end{array}
\right)$, one then can get with the help of Propositions \ref{L2chi}, \ref{8.2}, Corollary \ref{Zphi}, \eqref{ZmGT} and \eqref{ZmtE} that
\begin{align}\label{ZmL}
&\dl^l\|\bar Z^m\big(\mathring L(\mu^{-1}G_{\mathring L\mathring L}^\g)T^2\varphi_\g\big)\|_{s,u}\no\\
\lesssim&\dl^{k(1-\ve_0)-\ve_0}s^{-3/2}\dl^{l_1}\|\bar Z^{m_1}(\mu^{-1}G_{\mathring L\mathring L}^\g T^2\varphi_\g)\|_{s,u}+\dl^{k(1-\ve_0)-1-\ve_0}s^{-5/2}\dl^{l_0}\|\bar Z^{m_0}\mu\|_{s,u}\no\\
&+\dl^{k(1-\ve_0)-1}s^{-2}\dl^l\|TZ^{m+1}\varphi\|_{s,u}+\dl^{k(1-\ve_0)-\ve_0}s^{-7/2}\dl^{l_1}\|\slashed{\mathcal L}_{\bar Z}^{m_1}\leftidx{^{(R)}}{\slashed\pi}_T\|_{s,u}\no\\
&+\dl^{k(1-\ve_0)-2}s^{-2}\dl^{l_0}\big(\|\bar Z^{m_0}\check L^1\|_{s,u}+\|\bar Z^{m_0}\check L^2\|_{s,u}+\|\bar Z^{m_0}\check\varrho\|_{s,u}+\|Z^{m_0}\varphi\|_{s,u}\big)\no\\
&+\dl^{(k-1)(1-\ve_0)-2}s^{-5/2}\dl^{l_0}\|Z^{m_0}\phi\|_{s,u}+\dl^{(k+1)(1-\ve_0)-2}s^{-7/2}\dl^{l_{-1}}\|\bar Z^{m_{-1}}x\|_{s,u}\no\\
&+\dl^{(k-1)(1-\ve_0)-2}s^{-1/2}\dl^{l_1}\big(\|\bar Z^{m_1}\mathring L^2\phi\|_{s,u}+\|\slashed d\bar Z^{m_1}\mathring L\phi\|_{s,u}\big)\no\\
&+\dl^{k(1-\ve_0)-1}s^{-3}\dl^{l_0}\|\slashed{\mathcal L}_{\bar Z}^{m_0}\slashed g\|_{s,u}+\dl^{k(1-\ve_0)-1}s^{-2}\dl^{l_1}\|\slashed{\mathcal L}_{\bar Z}^{m_1}\check\chi\|_{s,u}\no\\
\lesssim&\dl^{k(1-\ve_0)-1/2-\ve_0}s^{-2}+\dl^{(k-1)(1-\ve_0)-1}s^{-3/2-\iota}\sqrt{\tilde E_{1,\leq m+2}}+\dl^{k(1-\ve_0)-1}s^{-2}\sqrt{\tilde E_{2,\leq m+2}}\\
&+\dl^{k(1-\ve_0)-1}(\dl+\dl^{k(1-\ve_0)})s^{-2}\dl^l\|\tilde F^m\|_{s,u}+\dl^{(2k-1)(1-\ve_0)-2}s^{-2}\sqrt{\int_0^u\tilde F_{1,m+2}(s,u')du'}\no\\
&+\dl^{2k(1-\ve_0)-1-\ve_0}(\dl+\dl^{k(1-\ve_0)})s^{-2}\int_{t_0}^s\tau^{-3/2}\dl^l\|\tilde F^m\|_{\tau,u}d\tau.\no
\end{align}
Note that \eqref{te} and \eqref{H} imply that $\bar Z^m\tilde e$ contains the term
$\f12\mu^{-1}(G_{\mathring L\mathring L}^\gamma T\varphi_\gamma)\bar Z^m T\text{tr}\check\chi$, which equals $\f12\mu^{-1}(G_{\mathring L\mathring L}^\gamma T\varphi_\gamma)\bar Z^m\slashed\triangle\mu+\cdots$ by \eqref{Tchi'}.
Thus,
it follows from Propositions \ref{L2chi}, \ref{8.2}, \eqref{ZmGT}, \eqref{ZmtE} and \eqref{ZmL} that
\begin{align*}
&\dl^l\|\bar Z^m\tilde e\|_{s,u}\\
\lesssim&\dl^{k(1-\ve_0)-1-\ve_0}\big(s^{-3}\dl^{l_0}\|\bar Z^{m_0}\mu\|_{s,u}+s^{-3/2}\dl^{l_1}\|\slashed{\mathcal L}_{\bar Z}^{m_1}\check\chi\|_{s,u}\big)+s^{-1}\dl^{l_1}\|\bar Z^{m_1}\tilde E\|_{s,u}\\
&+\dl^{l_1}\|\bar Z^{m_1}\big(\mathring L(\mu^{-1}G_{\mathring L\mathring L}^\g)T^2\varphi_\g\big)\|_{s,u}+s^{-1}\dl^{l_1}\|\bar Z^{m_1}(\mu^{-1}G_{\mathring L\mathring L}^\g T^2\varphi_\g)\|_{s,u}\\
&+\dl^{k(1-\ve_0)-1}s^{-3}\dl^{l_1}\|\slashed{\mathcal L}_{\bar Z}^{m_1}\slashed g\|_{s,u}+\dl^{k(1-\ve_0)-1}s^{-2}\dl^{l_1}\|\slashed{\mathcal L}_{\bar Z}^{m_1}\leftidx{^{(T)}}{\slashed\pi}_{\mathring L}\|_{s,u}\\
&+\dl^{k(1-\ve_0)-1}(\dl^{k(1-\ve_0)-1}+1)s^{-3}(\dl^{l_0}\|\bar Z^{m_0}\check L^1\|_{s,u}+\dl^{l_0}\|\bar Z^{m_0}\check L^2\|_{s,u}+s^{-1}\dl^{l_{-1}}\|\bar Z^{m_{-1}}x\|_{s,u})\\
&+\dl^{2k(1-\ve_0)-\ve_0}s^{-7/2}\ln s\dl^{l_1}\|\slashed{\mathcal L}_{\bar Z}^{m_1}\leftidx{^{(T)}}{\slashed\pi}\|_{s,u}+\dl^{k(1-\ve_0)-\ve_0}s^{-3/2}\dl^{l_1}\|\bar Z^{m_1}\slashed\triangle\mu\|_{s,u}\\
&+\dl^{(k-1)(1-\ve_0)-1}s^{-3/2-\iota}\sqrt{\tilde E_{1,\leq m+2}}+\dl^{k(1-\ve_0)-1}s^{-2}\sqrt{\tilde E_{2,\leq m+2}}\\
&+\dl^{k(1-\ve_0)}s^{-5}(\dl^{k(1-\ve_0)-1}\ln s+1)\dl^{l_1}\|\slashed{\mathcal L}_{\bar Z}^{m_1}\leftidx{^{(R)}}{\slashed\pi}_T\|_{s,u}\\
\lesssim&\dl^{k(1-\ve_0)-1/2-\ve_0}s^{-2}+\dl^{(k-1)(1-\ve_0)-1}s^{-3/2-\iota}\sqrt{\tilde E_{1,\leq m+2}}+\dl^{k(1-\ve_0)-1}s^{-2}\sqrt{\tilde E_{2,\leq m+2}}\\
&+\dl^{k(1-\ve_0)-\ve_0}s^{-3/2}\dl^l\|\tilde F^m\|_{s,u}+\dl^{2k(1-\ve_0)-1-\ve_0}(\dl+\dl^{k(1-\ve_0)})s^{-2}
\int_{t_0}^s\tau^{-3/2}\dl^l\|\tilde F^m\|_{\tau,u}d\tau\\
&+\dl^{(2k-1)(1-\ve_0)-2}s^{-2}\sqrt{\int_0^u\tilde F_{1,m+2}(s,u')du'}.
\end{align*}
Therefore, it holds that
\begin{align}\label{barek}
&\delta^l\|\tilde e^m\|_{s,u}\no\\
\lesssim&\dl^{k(1-\ve_0)-1/2-\ve_0}s^{-2}+\dl^{(k-1)(1-\ve_0)-1}s^{-3/2-\iota}\sqrt{\tilde E_{1,\leq m+2}}\no+\dl^{k(1-\ve_0)-1}s^{-2}\sqrt{\tilde E_{2,\leq m+2}}\no\\
&+\dl^{k(1-\ve_0)-\ve_0}s^{-\f32}\dl^l\|\tilde F^m\|_{s,u}+\dl^{2k(1-\ve_0)-1-\ve_0}(\dl+\dl^{k(1-\ve_0)})s^{-2}\int_{t_0}^s\tau^{-3/2}\dl^l\|\tilde F^m\|_{\tau,u}d\tau\no\\
&+\dl^{(2k-1)(1-\ve_0)-2}s^{-2}\sqrt{\int_0^u\tilde F_{1,m+2}(s,u')du'}.
\end{align}
Substituting \eqref{Y-23} and \eqref{barek}  into \eqref{tildeF} and then applying Gronwall's inequality yield
\begin{equation*}
\begin{split}
\delta^l\|\tilde F^m\|_{s,u}\lesssim&\delta^{k(1-\varepsilon_0)-1/2-\ve_0}s^{-1}+\delta^{(k-1)(1-\varepsilon_0)-1}s^{-1/2-\iota}\sqrt{\tilde E_{1,\leq m+2}}\\
+\dl^{k(1-\ve_0)-1}&s^{-1}\sqrt{\tilde E_{2,\leq m+2}}+\dl^{(2k-1)(1-\ve_0)-2}s^{-1}\sqrt{\int_0^u\tilde F_{1,m+2}(s,u')du'}
\end{split}
\end{equation*}
and further
\begin{equation*}
\begin{split}
\delta^l\|\bar Z^m\slashed\triangle\mu\|_{s,u}\lesssim&\delta^{k(1-\varepsilon_0)-1/2-\ve_0}s^{-1}+\delta^{(k-1)(1-\varepsilon_0)-1}s^{-1/2-\iota}\sqrt{\tilde E_{1,\leq m+2}}\\
+\dl^{k(1-\ve_0)-1}&s^{-1}\sqrt{\tilde E_{2,\leq m+2}}+\dl^{(2k-1)(1-\ve_0)-2}s^{-1}\sqrt{\int_0^u\tilde F_{1,m+2}(s,u')du'}
\end{split}
\end{equation*}
with the help of \eqref{ZmtE}.

For the other cases containing at least one $\varrho\mathring L$ in $Z^m$, one can make use of the
commutators $[\varrho\mathring L,\bar Z]$ and $[\varrho\mathring L,\slashed\triangle]$ and subsequently
utilize the transport equation \eqref{lmu}
to obtain  the related $L^2$ norm estimates.
Therefore, we eventually arrive at
\begin{equation}\label{Zmu}
\begin{split}
\delta^l\|Z^m\slashed\triangle\mu\|_{s,u}\lesssim&\delta^{k(1-\varepsilon_0)-1/2-\ve_0}s^{-1}
+\delta^{(k-1)(1-\varepsilon_0)-1}s^{-1/2-\iota}\sqrt{\tilde E_{1,\leq m+2}}\\
&+\dl^{k(1-\ve_0)-1}s^{-1}\sqrt{\tilde E_{2,\leq m+2}}+\dl^{(2k-1)(1-\ve_0)-2}s^{-1}\sqrt{\int_0^u\tilde F_{1,m+2}(s,u')du'}.
\end{split}
\end{equation}
And hence, it follows from \eqref{dnnchi} and \eqref{Zmu} that
\begin{equation}\label{dnchi}
\begin{split}
&\delta^l\|\slashed dZ^m\textrm{tr}\chi\|_{s,u}
+\delta^l\|\slashed\nabla{\slashed{\mathcal L}}_{Z}^m\check{\chi}\|_{s,u}\\
\lesssim&\delta^{k(1-\varepsilon_0)+1/2-\ve_0}s^{-2}
+\delta^{(k-1)(1-\varepsilon_0)}s^{-3/2-\iota}\sqrt{\tilde E_{1,\leq m+2}}+\delta^{k(1-\varepsilon_0)} s^{-2}\sqrt{\tilde E_{2,\leq m+2}}\\
&+\dl^{k(1-\ve_0)-1}s^{-3/2-\iota}\sqrt{\int_0^u\tilde F_{1,m+2}(s,u')du'}.
\end{split}
\end{equation}

\section{Estimates for the error terms}\label{ert}
After all the preparations for the optimal $L^2$ estimates on the related quantities in Sections \ref{EE} and \ref{L2chimu}, we are ready to
handle the error terms $\delta\int_{D^{s, u}}|\Phi\cdot\mathring{\underline L}\Psi|$ and
$\int_{D^{s, u}}\varrho^{2\iota}|\Phi\cdot \mathring L\Psi|$ in \eqref{e}, and then complete the final energy estimates for $\vp$.
Although the main strategy is analogous to that in \cite{Ding4}, due to the slow time decay of the
solutions to the 2D wave equations
and the requirement on the optimal smallness exponent of short pulse data, we still give  all the details since many precise estimates
in Sections \ref{EE}-\ref{L2chimu} derived from the higher order null condition \eqref{null} will be applied (for instance,
one can see \eqref{D22}, \eqref{D22L} and Remark \ref{5.1} below for details).

For $\Psi=Z^{m+1}\vp_\g$ in \eqref{gel}, then the corresponding $\Phi$ is just $\Phi_\g^{m+1}$ which has been explicitly
given in \eqref{Phik}.
One can now deal with each term in \eqref{Phik} as follows.

\subsection{Treatment on $J_1^{m+1}$}\label{tot}

This subsection deals with the term $J_1^{m+1}$ in \eqref{Phik}.
First, expand $Z^{m+1}$ as $Z_{m+1}Z_{m}\cdots Z_2Z_1$ with $Z_i\in\{\varrho\mathring L,T,R\}$ and set $\varphi_\gamma^n=\begin{cases}Z_n\cdots Z_1\varphi_\gamma,\ &n\geq 1,\\\varphi_\gamma,\ &n=0.\end{cases}$
 \begin{enumerate}[(1)]
  \item Due to form of $\Phi_\gamma^{m+1}$, it is necessary to estimate the derivatives of $\mu\mathscr D^\al{\leftidx{^{(Z)}}C_\gamma^{n}}_{,\al}$ $(0\leq n\leq m)$. To this end, one treats first $\leftidx{^{(Z)}}D_{\gamma,1}^{n}$ and $\leftidx{^{(Z)}}D_{\gamma,3}^{n}$ in \eqref{muZC}, which do not contain the top order derivatives of $\varphi_\gamma$.
      In fact, substituting \eqref{Lpi}-\eqref{Rpi} in Appendix A into \eqref{D1} and \eqref{D3} yields directly
\begin{align}
\leftidx{^{(T)}}D_{\gamma,1}^{n}=&(T\mu)\mathring L^2\varphi_\gamma^{n}+\mu(\slashed d_X\mu+2\mu\zeta_X)\slashed d^X\mathring L\varphi_\gamma^{n}+\f12\textrm{tr}\leftidx{^{(T)}}{\slashed\pi}(\mathring L\mathring{\underline L}\varphi_\gamma^{n}+\f12\textrm{tr}\chi\mathring{\underline L}\varphi_\gamma^{n})\no\\
&+(\slashed d_X\mu+2\mu\zeta_X)\slashed d^X\mathring{\underline L}\varphi_\gamma^{n}-(T\mu)\slashed\triangle\varphi_\gamma^{n}
+\f12\mu\text{tr}\leftidx{^{(T)}}{\slashed\pi}\slashed\triangle\varphi_\gamma^{n},\label{T1}\\
\leftidx{^{(T)}}D_{\gamma,3}^{n}=&\big\{\textrm{tr}\chi T\mu+\f14(\mu\textrm{tr}\chi+\textrm{tr}\leftidx{^{(T)}}{\slashed \pi})\textrm{tr}\leftidx{^{(T)}}{\slashed \pi}-\f12|\slashed d\mu|^2-\mu\zeta_X(\slashed d^X\mu)\big\}\mathring L\varphi_\gamma^{n}\no\\
&+(\f12\mathring L\mu-\mu\textrm{tr}\chi)(\slashed d_X\mu+2\mu\zeta_X)\slashed d^X\varphi_\gamma^{n},\label{T3}\\
\leftidx{^{(\varrho\mathring L)}}D_{\gamma,1}^{n}=&(2-\mu+\varrho\mathring L\mu)\mathring L^2\varphi_\gamma^{n}-2\varrho(\slashed d_X\mu+2\mu\zeta_X)\slashed d^X\mathring L\varphi_\gamma^{n}+\varrho\textrm{tr}\chi(\mathring L\mathring{\underline L}\varphi_\gamma^{n}+\f12\textrm{tr}\chi\mathring{\underline L}\varphi_\gamma^{n})\no\\
&-(\varrho\mathring L\mu-\varrho\mu\textrm{tr}\check\chi)\slashed\triangle\varphi_\gamma^{n},\label{rL1}\\
\leftidx{^{(\varrho\mathring L)}}D_{\gamma,3}^{n}=&\textrm{tr}\chi\big\{2-\mu+\varrho\mathring L\mu+\f12\varrho\text{tr}\leftidx{^{(T)}}{\slashed\pi}+\f12\varrho\mu\textrm{tr}\chi\big\}\mathring L\varphi_\gamma^{n}+2\varrho\textrm{tr}\chi(2\mu\zeta^X+\slashed d^X\mu)\slashed d_X\varphi_\gamma^n,\label{rL3}\\
\leftidx{^{(R)}}D_{\gamma,1}^{n}=&(R\mu)\mathring L^2\varphi_\gamma^{n}-\leftidx{^{(R)}}{\slashed\pi}_{\mathring{\underline L}X}\slashed d^X\mathring L\varphi_\gamma^{n}+\f12\textrm{tr}\leftidx{^{(R)}}{\slashed\pi}(\mathring L\mathring{\underline L}\varphi_\gamma^{n}+\f12\textrm{tr}\chi\mathring{\underline L}\varphi_\gamma^{n})-\leftidx{^{(R)}}{\slashed\pi}_{\mathring LX}\slashed d^X\mathring{\underline L}\varphi_\gamma^{n}\no\\
&-(R\mu)\slashed\triangle\varphi_\gamma^{n}
+\f12\mu\text{tr}\leftidx{^{(R)}}{\slashed\pi}\slashed\triangle\varphi_\gamma^{n},\label{R1}\\
\leftidx{^{(R)}}D_{\gamma,3}^{n}=&\big\{\textrm{tr}\chi R\mu+\f14(\textrm{tr}\leftidx{^{(T)}}{\slashed\pi}+\mu\textrm{tr}\chi)\textrm{tr}\leftidx{^{(R)}}{\slashed\pi}+\f12\slashed d^X\mu\leftidx{^{(R)}}{\slashed\pi}_{\mathring LX}\big\}\mathring L\varphi_\gamma^{n}
+\big\{\textrm{tr}\chi\leftidx{^{(R)}}{\slashed\pi}_{TX}\no\\
&+\textrm{tr}\leftidx{^{(R)}}{\slashed\pi}(\mu\zeta_X+\f12\slashed d_X\mu)+(-\f12\mathring L\mu+\mu\text{tr}\chi+\f12\text{tr}\leftidx{^{(T)}}{\slashed\pi})\leftidx{^{(R)}}{\slashed\pi}_{\mathring LX}\big\}\slashed d^X\varphi_\gamma^{n}.\label{R3}
\end{align}

Note that it follows from \eqref{T1}, \eqref{rL1}, and \eqref{R1} that each of $\leftidx{^{(Z)}}D_{\gamma,1}^{n}$ contains the factor $\mathring L\mathring{\underline L}\varphi_\gamma^{n}+\f12\textrm{tr}\chi\mathring{\underline L}\varphi_\gamma^{n}$ which can be estimated by using \eqref{fequation} as explained after \eqref{D3}.
It follows easily from the expression of $J_1^{m+1}$ that at most
$(m-n)-$ order derivatives appear in $\leftidx{^{(Z)}}D_{\gamma,i}^{n}$ $(i=1,3)$.
 Thus, the $L^2$ norms of all terms in $J_1^{m+1}$ can be estimated by
the corresponding $L^\infty$ estimates in Subsection \ref{BA}
and the related $L^2$ estimates in Proposition \ref{L2chi}, which do not depend on the estimates of top order derivatives
in Section \ref{L2chimu}.
Therefore,

\begin{align}\label{D11}
&\delta^{2l+1}|\int_{D^{s, u}}\sum_{p=1}^{m}\big(Z_{m+1}+\leftidx{^{(Z_{m+1})}}\lambda\big)\dots\big(Z_{m+2-p}
+\leftidx{^{(Z_{m+2-p})}}\lambda\big)\leftidx{^{(Z_{m+1-p})}}D_{\gamma,1}^{m-p}\cdot\mathring{\underline L}\varphi_\gamma^{m+1}|\no\\
\lesssim&\delta^{2l+1}\int_{t_0}^s\sum_{p=1}^{m}\|\big(Z_{m+1}+\leftidx{^{(Z_{m+1})}}\lambda\big)\dots\big(Z_{m+2-p}
+\leftidx{^{(Z_{m+2-p})}}\lambda\big)\leftidx{^{(Z_{m+1-p})}}D_{\gamma,1}^{m-p}\|_{\tau,u}\no\\
&\qquad\qquad\cdot\|\mathring{\underline L}Z^{m+1}\varphi_\gamma\|_{\tau,u}d\tau\no\\
\lesssim&\delta^{6-6\varepsilon_0}+\int_{t_0}^s\tau^{-1-\iota}\tilde E_{1,\leq m+2}(\tau,u)d\tau
+\delta\int_{t_0}^s\tau^{-1-\iota}\tilde E_{2,\leq m+2}(\tau,u)d\tau+\dl^{-1}{\int_0^u\tilde F_{1,m+2}(s,u')du'}.
\end{align}
Similarly,
\begin{align}\label{D33}
&\delta^{2l+1}|\int_{D^{s, u}}\sum_{p=1}^{m}\big(Z_{m+1}+\leftidx{^{(Z_{m+1})}}\lambda\big)\dots\big(Z_{m+2-p}
+\leftidx{^{(Z_{m+2-p})}}\lambda\big)\leftidx{^{(Z_{m+1-p})}}D_{\gamma,3}^{m-p}\cdot\mathring{\underline L}Z^{m+1}\varphi_\gamma|\no\\
\lesssim&\delta^{4-2\varepsilon_0}+\dl^{2(k+1)(1-\ve_0 )}+\int_{t_0}^s\tau^{-2}\tilde E_{1,\leq m+2}(\tau,u)d\tau+\delta\int_{t_0}^s\tau^{-2}\ln \tau\tilde E_{2,\leq m+2}(\tau,u)d\tau\\
&+\dl^{-1}{\int_0^u\tilde F_{1,m+2}(s,u')du'}.\no
\end{align}

The terms in $\delta^{2l}\int_{D^{s, u}}\varrho^{2\iota}|\Phi\cdot \mathring L\Psi|$
related to the integrand factors $\leftidx{^{(Z)}}D_{\gamma,i}^{n}$ $(i=1,3)$  can also be estimated as follows
\begin{align}\label{D13}
&\delta^{2l}\int_{D^{s, u}}|\sum_{p=1}^{m}\varrho^{2\iota}\big(Z_{m+1}+\leftidx{^{(Z_{m+1})}}\lambda\big)\dots\big(Z_{m+2-p}
+\leftidx{^{(Z_{m+2-p})}}\lambda\big)\big(\leftidx{^{(Z_{m+1-p})}}D_{\gamma,1}^{m-p}\no\\
&\qquad\qquad\qquad+\leftidx{^{(Z_{m+1-p})}}D_{\gamma,3}^{m-p}\big)\mathring L\varphi_\gamma^{m+1}|\no\\
\lesssim&\delta^{2l+1}\int_{D^{s, u}}\varrho^{2\iota}\Big\{\sum_{p=1}^{m}\big(Z_{m+1}+\leftidx{^{(Z_{m+1})}}\lambda\big)\dots\big(Z_{m+2-p}
+\leftidx{^{(Z_{m+2-p})}}\lambda\big)\big(\leftidx{^{(Z_{m+1-p})}}D_{\gamma,1}^{m-p}\no\\
&\qquad+\leftidx{^{(Z_{m+1-p})}}D_{\gamma,3}^{m-p}\big)\Big\}^2+\delta^{2l-1}\int_{D^{s, u}}\varrho^{2\iota}|\mathring LZ^{m+1}\varphi_\gamma|^2\no\\
\lesssim&\delta^{6-6\varepsilon_0}+\dl^{4-2\ve_0}+\int_{t_0}^s\tau^{-2}\tilde E_{1,\leq m+2}(\tau,u)d\tau+\delta\int_{t_0}^s\tau^{-4+2\iota}\ln^4\tau\tilde E_{2,\leq m+2}(\tau,u)d\tau\\
&+\delta^{-1}\int_0^u\tilde F_{1,m+2}(s,u')du'.\no
\end{align}

\item  We now estimate the terms involving $\leftidx{^{(Z)}}D_{\gamma,2}^{n}$ $(0\leq n\leq m)$ in $J_1^{m+1}$. Note that in the special case $n=0$, $j$ equals $m$ in $J_1^{m+1}$, and the order of the top derivatives in $\leftidx{^{(Z)}}D_{\gamma,2}^{0}$ is $m$, so that $\leftidx{^{(Z)}}D_{\gamma,2}^{0}$ contains terms involving the $(m+1)^{\text{th}}$ order derivatives of the deformation tensor.
This prevents one from using Proposition \ref{L2chi} to estimate the $L^2$ norm of $\leftidx{^{(Z)}}D_{\gamma,2}^0$
directly since the $L^2$ norm of the $(m+1)^{\text{th}}$ order derivatives of the deformation tensor can be controlled only by $\tilde E_{1,\leq m+3}(s,u)$ and $\tilde E_{2,\leq m+3}(s,u)$ in the energy estimates
(note that $\tilde E_{1,\leq m+3}(s,u)$ and $\tilde E_{2,\leq m+3}(s,u)$ cannot be absorbed directly by the energies
$\tilde E_{1,\leq m+2}(s,u)$ and $\tilde E_{2,\leq m+2}(s,u)$ on the left hand side of the resulting energy inequality by \eqref{e}).
Thus, we will
examine the expression of $\leftidx{^{(Z)}}D_{\gamma,2}^n$ carefully and apply the estimates in Section \ref{L2chimu}
to handle the top order derivatives of $\textrm{tr}\chi$ and $\mu$.
Indeed, it follows from direct computations that
\begin{align}
\leftidx{^{(T)}}D_{\gamma,2}^{n}=&\big\{\mathring L T\mu+\f14\uwave{\mathring{\underline L}(\text{tr}\leftidx{^{(T)}}{\slashed\pi})}+\boxed{\slashed\nabla_X\big(\f12\mu\slashed d^X\mu}+\mu^2\zeta^X\big)\big\}\mathring L\varphi_\gamma^n+\big\{\f14\mathring L(\textrm{tr}\leftidx{^{(T)}}{\slashed\pi})\no\\
&+\f12\boxed{\slashed\nabla_X(\slashed d^X\mu}+2\mu\zeta^X)\big\}\mathring{\underline L}\varphi_\gamma^n+\big\{\f12\slashed{\mathcal L}_{\mathring L}(2\mu^2\zeta_X+\mu\slashed d_X\mu)-\uline{\slashed d_XT\mu}\no\\
&+\underline{\f12\slashed{\mathcal L}_{\mathring{\underline L}}(\slashed d_X\mu}+2\mu\zeta_X)+\underbrace{\f12\slashed d_X(\mu\text{tr}\leftidx{^{(T)}}{\slashed\pi})}\big\}\slashed d^X\varphi_\gamma^n,\label{T2}\\
\leftidx{^{(\varrho\mathring L)}}D_{\gamma,2}^{n}=&\big\{\mathring L\big(\varrho\mathring L\mu-\mu\big)+\uwave{\f12\mathring{\underline L}\big(\varrho\textrm{tr}\check\chi\big)}-\varrho\boxed{\slashed\nabla_X\big(\slashed d^X\mu}+2\mu\zeta^X\big)\big\}\mathring L\varphi_\gamma^n+\f12\mathring L\big(\varrho\textrm{tr}\check\chi\big)\mathring{\underline L}\varphi_\gamma^n\no\\
&-\big\{\slashed{\mathcal L}_{\mathring L}\big(\varrho\slashed d_X\mu+2\varrho\mu\zeta_X\big)+\slashed d_X\big(\mu+\varrho\mathring L\mu\big)-\varrho\underbrace{\slashed d_X(\mu\textrm{tr}\chi)}\big\}\slashed d^X\varphi_\gamma^n,\label{rL2}\\
\leftidx{^{(R)}}D_{\gamma,2}^n=&\big\{\mathring LR\mu
-\f12\underbrace{\slashed\nabla^X\leftidx{^{(R)}}{\slashed\pi}_{\mathring{\underline L}X}}+\f14\uwave{\mathring{\underline L}(\textrm{tr}\leftidx{^{(R)}}{\slashed\pi})}\big\}\mathring L\Psi_\g^n+\big\{\f14\mathring L(\textrm{tr}\leftidx{^{(R)}}{\slashed{\pi}})\no\\
&-\f12\underbrace{\slashed\nabla^X\leftidx{^{(R)}}{\slashed\pi}_{\mathring LX}}\big\}\mathring{\underline L}\Psi_\gamma^n+\big\{-\f12\slashed{\mathcal L}_{\mathring L}\leftidx{^{(R)}}{\slashed\pi}_{\mathring{\underline L}X}-\f12\uwave{\slashed{\mathcal L}_{\mathring{\underline L}}\leftidx{^{(R)}}{\slashed\pi}_{\mathring LX}}
-\boxed{\slashed d_XR\mu}\no\\
&+\f12\underbrace{\slashed d_X(\mu\textrm{tr}\leftidx{^{(R)}}{\slashed\pi})}\big\}\slashed d^X\Psi_\gamma^n.\label{R2}
\end{align}

It is emphasized that special attentions are needed to handle the terms with underlines, wavy lines, boxes, or braces
in the \eqref{T2}-\eqref{R2}. In $\leftidx{^{(T)}}D_{\gamma,2}^{n}$,
due to $\f12\slashed{\mathcal L}_{\mathring{\underline L}}\slashed d_X\mu=\slashed d_XT\mu+\f12\mu\slashed d_X\mathring L\mu$ by $\mathring{\underline L}=\mu\mathring L+2T$, the corresponding underline part becomes
\begin{equation}\label{Y-25}
-\slashed d_XT\mu+\f12\slashed{\mathcal L}_{\mathring{\underline L}}\slashed d_X\mu=\f12\mu\slashed d_X\mathring L\mu,
\end{equation}
which can be treated by using \eqref{lmu}. It follows from \eqref{Rpi} and \eqref{theta} that all terms with wavy lines in \eqref{T2}-\eqref{R2} contain $\slashed{\mathcal L}_{\mathring{\underline L}}\check\chi$, which can be written as $2\slashed\nabla^2\mu+\mu\slashed{\mathcal L}_{\mathring{L}}\check\chi+\cdots$ by \eqref{Tchi'}. Note that the terms in $\mu\slashed{\mathcal L}_{\mathring{L}}\check\chi+\cdots$
may be estimated by \eqref{Lchi'} and Proposition \ref{L2chi} while $L^2$ norm of the terms corresponding to $\slashed\nabla^2\mu$
may be treated by \eqref{Zmu}. Meanwhile,
one can use \eqref{Zmu} and \eqref{dnchi} to estimate those terms with boxes and braces respectively.

On the other hand, it is noticed that there are some terms whose factors are the derivatives of the deformation tensors
with respect to $\mathring L$, for example, $\f14{\mathring L}(\text{tr}\leftidx{^{(T)}}{\slashed\pi})\underline{\mathring L}\varphi_\gamma^n$ appears in \eqref{T2}. In fact, these terms are not ``bad" in the sense that the derivatives of $\mathring L$
for the involved deformation tensors are
 actually equipped with the ``good" quantities in terms of \eqref{lmu} and \eqref{Lchi'} after examining
 each term in \eqref{Lpi} and \eqref{Rpi} in Appendix A.

In summary, using \eqref{Zmu}-\eqref{dnchi}, one can eventually arrive at
\begin{align}\label{D22}
&\delta^{2l+1}|\int_{D^{s, u}}\sum_{p=1}^{m}\big(Z_{m+1}+\leftidx{^{(Z_{m+1})}}\lambda\big)\dots\big(Z_{m+2-p}
+\leftidx{^{(Z_{m+2-p})}}\lambda\big)\leftidx{^{(Z_{m+1-p})}}D_{\gamma,2}^{m-p}\cdot\mathring{\underline L}\varphi_\gamma^{m+1}|\no\\
\lesssim&\dl\int_{t_0}^s\Big(\dl^{1-\ve_0}\tau^{-1/2}\dl^l\|Z^m\slashed\triangle\mu\|_{\tau,u}+\dl^{-\ve_0}\tau^{1/2}\dl^l\|\slashed dZ^m\textrm{tr}\chi\|_{\tau,u}+\cdots\Big)\sqrt{E_{2,m+2}(\tau,u)}d\tau\no\\
\lesssim&\delta^{2(k+1)(1-\varepsilon_0)-2\ve_0}+\int_{t_0}^s\tau^{-1-\iota}\tilde E_{1,\leq m+2}(\tau,u)d\tau
+\delta\int_{t_0}^s\tau^{-1-\iota}\tilde E_{2,\leq m+2}(\tau,u)d\tau\\
&+\dl^{2k(1-\ve_0)-2\ve_0-1}\int_0^u\tilde F_{1,m+2}(s,u')du'\no
\end{align}
and
\begin{align}\label{D22L}
&\delta^{2l}\int_{D^{s, u}}\sum_{p=1}^{m}|\varrho^{2\iota}\big(Z_{m+1}+\leftidx{^{(Z_{m+1})}}\lambda\big)\dots\big(Z_{m+2-p}
+\leftidx{^{(Z_{m+2-p})}}\lambda\big)\leftidx{^{(Z_{m+1-p})}}D_{\gamma,2}^{m-p}\cdot\mathring L\varphi_\gamma^{m+1}|\no\\
\lesssim&\delta^{2(k+1)(1-\varepsilon_0)-2\ve_0}+\int_{t_0}^s\tau^{-2}\tilde E_{1,\leq m+2}(\tau,u)d\tau
+\delta\int_{t_0}^s\tau^{-3+2\iota}\ln^2\tau\tilde E_{2,\leq m+2}(\tau,u)d\tau\\
&+\delta^{-1}\int_0^u\tilde F_{1,m+2}(s,u')du'.\no
\end{align}
\end{enumerate}

\begin{remark}\label{5.1}
It should be pointed out that the main reason for getting the precise improvements of the $L^2$ norm of $Z^{m+1}\mu$ in Proposition \ref{8.2} is to
obtain the optimal estimates of $Z^m\slashed\triangle\mu$ and $\slashed dZ^{m}\textrm{tr}\chi$ in  \eqref{Zmu} and \eqref{dnchi}, respectively. These optimal estimates lead to the smallness factor $\dl^{2(k+1)(1-\ve_0)-2\ve_0}$ in the right hand sides of \eqref{D22} and \eqref{D22L}. Note that $\dl^{2(k+1)(1-\ve_0)-2\ve_0}<\delta^{2-2\ve_0}$ holds
due to $\ve_0<\ve_k^*$. Based on this, one can close the energy estimates in Section \ref{YY} (see \eqref{E} for details).
\end{remark}

\subsection{Treatments on $J_2^{m+1}$ and $\Phi_\g^1$}\label{l}
This subsection deals with $J_2^{m+1}$ and $\Phi_\g^1$ in \eqref{Phik}.
Note that $J_2^{m+1}$ and $\Phi_\g^1$ do not contain the top order derivatives of $\textrm{tr}{\chi}$ and $\slashed\triangle\mu$. Therefore, according to Proposition \ref{L2chi} and the expressions
of $\leftidx{^{(Z)}}D_{\g, j}^n$ in \eqref{T1}-\eqref{R3} and \eqref{T2}-\eqref{R2}, one can obtain
\begin{align}\label{ZL}
&\delta^{2l+1}\int_{D^{s,u}}|\sum_{p=1}^3\leftidx{^{(Z_{m+1})}}D_{\gamma,p}^m\cdot \mathring{\underline L}{\varphi}_\gamma^{m+1}|\no\\
\lesssim&\delta^{2(k+1)(1-\varepsilon_0)}+\int_{t_0}^s\tau^{-1-\iota}\ln^2\tau\tilde E_{1,\leq m+2}(\tau,u)d\tau
+\delta\int_{t_0}^s\tau^{-1-\iota}\tilde E_{2,\leq m+2}(\tau,u)d\tau\\
&+\delta^{-1}\int_0^u\tilde F_{1,m+2}(s,u')du'.\no
\end{align}
and
\begin{align}\label{ZrL}
&\delta^{2l}\int_{D^{s,u}}|\sum_{p=1}^3\varrho^{2\iota}\leftidx{^{(Z_{m+1})}}D_{\gamma,p}^m\cdot \mathring L{\varphi}_\gamma^{m+1}|\no\\
\lesssim&\delta^{2(k+1)(1-\varepsilon_0)}+\int_{t_0}^s\tau^{-2}\ln^2\tau\tilde E_{1,\leq m+2}(\tau,u)d\tau
+\delta\int_{t_0}^s\tau^{-4+2\iota}\ln^2\tau\tilde E_{2,\leq m+2}(\tau,u)d\tau\\
&+\delta^{-1}\int_0^u\tilde F_{1,m+2}(s,u')du'.\no
\end{align}

In addition, $\Phi_\g^0=\mu\Box_g\varphi_\g$ is given explicitly in \eqref{ge}.
Then it follows from this, Proposition \ref{L2chi} and \eqref{lamda} that
\begin{align}\label{Phi0}
&\delta^{2l+1}\int_{D^{s,u}}|(Z_{m+1}+\leftidx{^{(Z_{m+1})}}\lambda)\cdots(Z_{1}
+\leftidx{^{(Z_{1})}}\lambda)\Phi_\gamma^0\cdot\mathring{\underline L}\varphi_\gamma^{m+1}|\no\\
\lesssim&\delta^{2(k+1)(1-\varepsilon_0)}+\int_{t_0}^s\tau^{-1-\iota}\tilde E_{1,\leq m+2}(\tau,u)d\tau
+\delta\int_{t_0}^s\tau^{-1-\iota}\tilde E_{2,\leq m+2}(\tau,u)d\tau\\
&+\delta^{-1}\int_0^u\tilde F_{1,m+2}(s,u')du'\no
\end{align}
and
\begin{equation}\label{rPhi}
\begin{split}
&\delta^{2l}\int_{D^{s,u}}\varrho^{2\iota}|(Z_{m+1}+\leftidx{^{(Z_{m+1})}}\lambda)\cdots(Z_{1}
+\leftidx{^{(Z_{1})}}\lambda)\Phi_\gamma^0\cdot\mathring L\varphi_\gamma^{m+1}|\\
\lesssim&\delta^{2(k+1)(1-\varepsilon_0)}+\int_{t_0}^s\tau^{-2}\tilde E_{1,\leq m+2}(\tau,u)d\tau
+\delta\int_{t_0}^s\tau^{-4+2\iota}\tilde E_{2,\leq m+2}(\tau,u)d\tau\\
&+\delta^{-1}\int_0^u\tilde F_{1,m+2}(s,u')du'.
\end{split}
\end{equation}

\section{Global estimates in $A_{2\dl}$}\label{YY}

Following the idea of \cite{Ding4}, we are now ready to prove the global estimates on the smooth solution $\phi$ to the equation \eqref{quasi} with \eqref{id},  \eqref{Y-0} and \eqref{Y-0-a}-\eqref{g00} near $C_0$ when $\ve_0<\ve_k^*$. To this end, one has to estimate the solution in $\t C_{2\dl}$ as in \cite{Ding4}. Compared with \cite[Section 11]{Ding4}, to obtain the behavior of the solution for $\ve_0$ approaching $\ve_k^*$, one needs the more precise estimates on all related quantities.

By substituting \eqref{D11}-\eqref{D13} and \eqref{D22}-\eqref{rPhi} into \eqref{e}, one gets from
Gronwall's inequality and $0<\ve_0<\ve_k^*$ that
\begin{equation}\label{E}
\delta\tilde E_{2,\leq 2N-4}(s,u)+\delta F_{2,\leq 2N-4}(s,u)+\tilde E_{1,\leq 2N-4}(s,u)+F_{1,\leq 2N-4}(s,u)
\lesssim\delta^{2-2\varepsilon_0}.
\end{equation}
This, together with \eqref{Zm1p} and \eqref{Zmu}, yields that for $m\leq 2N-5$,
\begin{equation}\label{Ep}
\delta^l\|Z^{m}\phi\|_{s,u}\lesssim\delta^{5/2-\varepsilon_0}.
\end{equation}

Based on \eqref{E} and \eqref{Ep}, one can close the bootstrap
assumptions $(\star)$ in Subsection \ref{BA} by using the analogous Sobolev-type embedding formula which follows form Proposition 18.10 in \cite{J}.
\begin{lemma}
 For any function $f\in H^2(S_{s,u})$, under the assumptions $(\star)$, if $\delta>0$ is small, then
 \begin{equation}\label{et}
 \|f\|_{L^\infty(S_{s,u})}\lesssim\f{1}{s^{1/2}}\sum_{p\leq 1}\|R^p f\|_{L^2(S_{s, u})}.
 \end{equation}
\end{lemma}

It follows from \eqref{E}-\eqref{et} that for $0<\ve_0<\ve_k^*$ and $m\leq 2N-6$,
\begin{equation}\label{im}
\delta^l|Z^m\varphi_\gamma|\lesssim\f{\delta^l}{s^{1/2}} \sum_{p\leq 1}\|R^p Z^m\varphi_\gamma\|_{L^2(S_{s, u})}\overset{\eqref{SSi}}{\lesssim}\f{\delta^{1/2}}{s^{1/2}}\big(\sqrt{E_{1,\leq 2N-4}}+\sqrt{E_{2,\leq 2N-4}}\big)
\lesssim\delta^{1-\varepsilon_0} s^{-1/2};
\end{equation}
while for $m\leq 2N-7$,
\begin{equation}\label{imp}
\delta^l|Z^m\phi|\lesssim\delta^{1/2}s^{-1/2}\delta^l\sum_{p\leq 1}(\|\underline{\mathring L}R^p Z^m\phi\|_{s,u}
+\|{\mathring L}R^p Z^m\phi\|_{s,u})\lesssim\delta^{2-\varepsilon_0}s^{-1/2}.
\end{equation}
Replacing $(\star)$ by \eqref{im}-\eqref{imp} in Section \ref{BA} and following all procedures there yield the following
improved estimates over Proposition \ref{LTRh}: for $m\leq 2N-9$,
\begin{equation}\label{Z}
\begin{split}
&|\slashed{\mathcal{L}}_{ Z}^{m}\check{\chi}|\lesssim \delta^{k(1-\varepsilon_0)-l}s^{-2},\quad|\slashed{\mathcal{L}}_{Z}^{m+1}\slashed dx^j|\lesssim \delta^{-l},\quad |\slashed{\mathcal{L}}_{Z}^{m}\leftidx{^{(T)}}{\slashed\pi}_{\mathring L}|\lesssim \delta^{k(1-\varepsilon_0)-1-l}s^{-1},\\
&|Z^{m+1}\check L^j|+|\slashed{\mathcal{L}}_{Z}^{m}\leftidx{^{(R)}}{\slashed\pi}|+|\slashed{\mathcal{L}}_{Z}^{m}\leftidx{^{(R)}}{\slashed\pi}_{\mathring L}|+|Z^{m+1}\check\varrho|\lesssim \delta^{k(1-\varepsilon_0)-l}s^{-1}\ln s,\\
& |\slashed{\mathcal{L}}_{Z}^{m}R|+|Z^{m+1}\upsilon|\lesssim \delta^{k(1-\varepsilon_0)-l}\ln s,\quad|Z^{m+1}\mu|\lesssim \delta^{k(1-\varepsilon_0)-\varepsilon_0-l}s^{-(k-1)/2},\\
& |\slashed{\mathcal{L}}_{Z}^{m}\leftidx{^{(T)}}{\slashed\pi}|\lesssim \delta^{k(1-\varepsilon_0)-1-l}s^{-1}+s^{-l-1},\quad|\slashed{\mathcal{L}}_{Z}^{m}\leftidx{^{(R)}}{\slashed\pi}_{T}|\lesssim(\delta^{k(1-\ve_0)-l}
+\delta^{2k(1-\ve_0)-1-l})s^{-1}\ln s,
\end{split}
\end{equation}
where $l$ is the number of $T$ as before.

Since the estimates in \eqref{im} and \eqref{imp} are independent of the constant $M$ in $(\star)$,
the bootstrap assumptions $(\star)$ are proved. Therefore, the global estimates  on the solution
$\phi$ to \eqref{quasi} with \eqref{id}, \eqref{null} and \eqref{Y-0}
in the domain $D^{s,4\delta}$ (see Figure \ref{pic:p2} in Subsection \ref{p}) are established.

In addition, one can further update the estimates on $\check{\varrho}$ (see \eqref{rrho}), $\upsilon$ (see \eqref{omega}) and $g_{ij}\check L^i\o^j$, which will be used to obtain the more precise smallness orders and time decay rates for the $L^\infty$ norms of $\phi$ under the actions of $\Gamma\in\{L,\underline L,\O_{ij}\}$ in $A_{2\dl}$
and further be taken as the boundary values on $\t C_{2\dl}$ to study the global Goursat problem
of \eqref{quasi} in $B_{2\dl}$.

\begin{lemma}\label{orL}
	In $D^{s,4\dl}$, the quantities $\check L^i$ and $\check{\varrho}$ have the following estimates:
	\begin{equation}\label{rhoL}
    \delta^l|Z^m\check{\varrho}|
	+\delta^l|Z^m(g_{ij}\check L^i\o^j)|\lesssim\delta^{(k+1)(1-\ve_0)}s^{-1},\quad 	\delta^l|Z^m(g_{\al\beta}\o^\al\o^\beta)|\lesssim \delta^{(k+1)(1-\varepsilon_0)}s^{-3/2},
	\end{equation}
	where $l$ is the number of $T$ in $Z^m$ and $m\leq 2N-9$.
\end{lemma}
\begin{proof}
	On the initial time $t_0$,
	\begin{align}
	&\mu=\f{1}{\sqrt{(g^{0i}\o_i)^2+g^{ij}\o_i\o_j}},\label{m}\\
	&\tilde L^i=-g^{0i}\big(-g^{0j}\o_j+\sqrt{(g^{0j}\o_j)^2+g^{ab}\o_a\o_b}\big)+g^{ij}\o_j.\label{tL}
	\end{align}
	Due to $g^{\al\beta}=m^{\al\beta}+g^{\al\beta,\g_1\cdots\g_k}\vp_{\g_1}\cdots\vp_{\g_k}+O(\vp^{k+1})$, then it
follows from \eqref{errorv} that
	\begin{equation}\label{iL}
	\begin{split}
	\check L^i|_{t_0}
	=&\Big\{(-g^{0i,\g_1\cdots\g_k}-\f12g^{ab,\g_1\cdots\g_k}\o_a\o_b\o^i+g^{ij,\g_1\cdots\g_k}\o_j)\vp_{\g_1}\cdots\vp_{\g_k}
+f(\o,\vp)\vp^{k+1}\Big\}\Big|_{t_0}.
	\end{split}
	\end{equation}
In addition, by $\p_t=-\f12\o_0\underline L+\f12L$, $\p_i=-\f12\o_i\underline L+\f12\o_i L+\f1r\o^i_{\perp}\O$
	and \eqref{null}, one has
	\begin{equation}\label{gLo}
	\begin{split}
	g_{ij}\check L^i\o^j|_{t_0}&=\Big\{(-\f{\underline L\phi}2)^k( g^{0i,\g_1\cdots\g_k}\o_0\o_i+\f12g^{ij,\g_1\cdots\g_k}\o_i\o_j)\o_{\g_1}\cdots\o_{\g_k}+f(\varphi,\o)\left(
	\begin{array}{ccc}
	\vp^{k+1}\\
	\vp^{k-1}L\phi\\
	\vp^{k-1}\f1r\O\phi
	\end{array}
	\right)\Big\}\Big|_{t_0}\\
	&=f(\varphi,\o)\left(
	\begin{array}{ccc}
	\vp^{k+1}\\
	\vp^{k-1}L\phi\\
	\vp^{k-1}\f1r\O\phi
	\end{array}
	\right)\Bigg|_{t_0}.
	\end{split}
	\end{equation}
	On the other hand, it follows from \eqref{iL} that
	\begin{equation}\label{iu}
	\begin{split}
	\upsilon|_{t_0}=&\{-g_{0j}\epsilon_i^jx^i-g_{aj}\check L^a\epsilon_i^jx^i-(g_{aj}-m_{aj})\o^a\epsilon_i^jx^i\}|_{t_0}\\
	=&\epsilon_i^jx^i\Big\{(-g^{0j,\gamma_1,\cdots\g_k}+g^{aj,\g_1,\cdots\g_k}\o_a)\vp_{\g_1}\cdots\vp_{\g_k}-\check L^j+f(\varphi,\o)\vp^{k+1}\Big\}\Big|_{t_0}\\
	=&rf(\varphi,\o)\vp^{k+1}\Big|_{t_0}.
	\end{split}
	\end{equation}
	For any $\bar Z\in\{T=-\mu(g^{0i}+\mathring L^i)\p_i, R=\O-\upsilon\tilde T\}$, \eqref{gLo} and \eqref{iu} imply
	\begin{equation}\label{bZ}
	\delta^l|\bar Z^m(g_{ij}\check L^i\o^j)|_{t_0}\lesssim\delta^{(k+1)(1-\ve_0)}.
	\end{equation}
	Under the action of $\mathring L$, the term $\varrho g_{ij}\check L^i\o^j$ can be written as
	\begin{equation}\label{LrLo}
	\begin{split}
	\mathring L(\varrho g_{ij}\check L^i\o^j)=\big(G_{ij}^\gamma \mathring L\varphi_\gamma\big)\varrho\check L^i\o^j
	+\f \varrho rg_{ij}\mathring L(\varrho\check L^i)(\mathring L^j-\check L^j)+\f \varrho r g_{ij}\check L^i(\check L^j-\o^j\o_a\check L^a),
	\end{split}
	\end{equation}
	where
	\begin{equation}\label{gLL}
	g_{ij}\mathring L(\varrho\check L^i)\mathring L^j=
	-\f12\varrho G_{\mathring L\mathring L}^\gamma\mathring L\varphi_\gamma+f(\varphi,\mathring L^1,\mathring L^2)\varrho\vp^{2k-1}\left(
	\begin{array}{ccc}
	\mathring L\vp\\
	\slashed d_X\vp\slashed d^Xx
		\end{array}
	\right)
	\end{equation}
	follows from \eqref{LeL}, \eqref{LL}, and $g_{ij}(\slashed d^Xx^i)\mathring L^j=-g_{0j}(\slashed d^Xx^j)$.
	In addition, \eqref{null} implies
	\begin{equation}\label{GLLL}
	\begin{split}
	G_{\mathring L\mathring L}^\gamma\mathring L\varphi_\gamma=&
	-(\p_{\varphi_\gamma}g^{\kappa\lambda})\mathring L_\kappa\mathring L_\lambda\mathring L^{\nu}\p_\gamma\varphi_\nu\\
	=&-(\p_{\varphi_0}g^{\kappa\lambda})\mathring L_\kappa\mathring L_\lambda\big(\mathring L^2\phi-(\mathring L\mathring L^i)\varphi_i\big)+\big((\p_{\varphi_\gamma}g^{\kappa\lambda})\mathring L_\kappa\mathring L_\lambda\mathring L_{\gamma}\big)\tilde T^i\mathring L\varphi_i\\
	&-(\p_{\varphi_\gamma}g^{\kappa\lambda})\mathring L_\kappa\mathring L_\lambda g_{\gamma j}(\slashed d^Xx^j)\big(\slashed d_X\mathring L\phi-(\slashed d_X\mathring L^i)\varphi_i\big),
	\end{split}
	\end{equation}
	where $(\mathring L\mathring L^i)\vp_i=O(|\vp^k\mathring L\varphi|+|\vp^k\slashed d\varphi|)$ and $(\slashed d_X\mathring L^i)\varphi_i=\textrm{tr}\chi\slashed d_X\phi+O(|\vp^k\mathring L\varphi|+|\vp^k\slashed d\varphi|)$ due to \eqref{LL} and \eqref{dL} respectively, and   $(\p_{\varphi_\gamma}g^{\kappa\lambda})\mathring L_\kappa\mathring L_\lambda\mathring L_{\gamma}$ satisfies \eqref{gLLL}. Substituting \eqref{gLL}-\eqref{GLLL} into \eqref{LrLo} and using \eqref{im}-\eqref{Z} lead to
	\begin{equation}\label{ZrgLo}
	\dl^l |\bar Z^m\mathring L(\varrho g_{ij}\check L^i\o^j)|\lesssim\dl^{(k+1)(1-\ve_0)}s^{-3/2},\quad m\leq 2N-9.
	\end{equation}
	Therefore, for $m\leq 2N-9$, it follows from \eqref{Z} and \eqref{ZrgLo} that
	\begin{equation*}
	\begin{split}
	\delta^l|\mathring L\bar Z^m(\varrho g_{ij}\check L^i\o^j)|\lesssim&\delta^l\underbrace{\sum_{p_1+p_2=m-1}|\big(\slashed{\mathcal L}_{\bar Z}^{p_1}\leftidx{^{(\bar Z)}}{\slashed\pi}_{\mathring L}^X\big)\slashed d_X\bar Z^{p_2}(\varrho g_{ij}\check L^i\o^j)|}_{\text{vanish when}\ m=0}+\delta^l|\bar Z^m\mathring L(\varrho g_{ij}\check L^i\o^j)|\\
	\lesssim&\delta^{(k+1)(1-\ve_0)}s^{-3/2}.
	\end{split}
	\end{equation*}
	This, together with \eqref{bZ}, yields
	\begin{equation}\label{LbZ}
	\delta^l|\bar Z^m(g_{ij}\check L^i\o^j)|\lesssim\delta^{(k+1)(1-\ve_0)}s^{-1},\quad m\leq 2N-9.
	\end{equation}
	
	If there is at least one $\varrho\mathring L$ in $Z^m$ $(m\leq 2N-8)$, namely, $Z^m=Z^{p_1}(\varrho\mathring L)\bar Z^{p_2}$
	with $p_1+p_2=m-1$, then
	\begin{equation*}
	(\varrho\mathring L)\bar Z^{p_2}(g_{ij}\check L^i\o^j)=\underbrace{\sum_{q_1+q_2=p_2-1}\varrho\bar Z^{q_1}\leftidx{^{(\bar Z)}}{\slashed\pi}_{\mathring L}^X\cdot\slashed d_X\bar Z^{q_2}(g_{ij}\check L^i\o^j)}_{\text{vanish when}\ p_2=0}+\varrho\bar Z^{p_2}\big(\varrho^{-1}\mathring L(\varrho g_{ij}\check L^i\o^j)-\varrho^{-1}g_{ij}\check L^i\o^j\big).
	\end{equation*}
	With an induction argument on the number of $\varrho\mathring L$ in $Z^{p_1}$, it follows from the expression above, \eqref{Z} and \eqref{ZrgLo} that
	\begin{equation}\label{rLg}
	|Z^{p_1}(\varrho\mathring L)\bar Z^{p_2}(g_{ij}\check L^i\o^j)|\\
	\lesssim\delta^{(k+1)(1-\ve_0)-l}s^{-1}.
	\end{equation}
	Therefore, for any vector $Z\in\{\varrho\mathring L,T,R\}$, when $m\leq 2N-9$, by \eqref{LbZ} and \eqref{rLg}, one can get
	\begin{equation}\label{ZgLo}
	\begin{split}
	\delta^l|Z^m(g_{ij}\check L^i\o^j)|\lesssim\delta^{(k+1)(1-\ve_0)}s^{-1}.
	\end{split}
	\end{equation}

	In order to improve the estimate \eqref{cr} on $\check{\varrho}$, one needs to check the numerator in \eqref{rrho}.
	Due to $\vp_{\g_p}=\dl_{\g_p}^0\mathring L\phi-\mathring L_{\g_p}\tilde T^{i}\vp_i+g_{\g_pi}\slashed d_Xx^i\slashed d^X\phi$
and $g^{00}=-1$, then
	\begin{align}\label{1-go}
	&1-g_{ij}\o^i\o^j+2g_{ij}\check T^i\o^j=1-g_{ij}\o^i\o^j-2g_{0i}\o^i-2g_{ij}\check L^i\o^j\no\\
	=&g^{ij,\g_1\cdots\g_k}\o_i\o_j\vp_{\g_1}\cdots\vp_{\g_k}+2 g^{0i,\g_1\cdots\g_k}\o_0\o_i\vp_{\g_1}\cdots\vp_{\g_k}-2g_{ij}\check L^i\o^j+f(\varphi,\o)\vp^{k+1}\no\\
	=&(-1)^{k} g^{\al\beta,\g_1\cdots\g_k}\o_\al\o_\beta\mathring L_{\g_1}\cdots\mathring L_{\g_k}(\tilde T^{i_1}\vp_{i_1})\cdots(\tilde T^{i_k}\vp_{i_k})-2g_{ij}\check L^i\o^j+f(\varphi,\o)\left(
	\begin{array}{ccc}
	\vp^{k+1}\\
	\varphi^{k-1}\mathring L\phi\\
	\vp^{k-1}\slashed d^Xx\slashed d_X\phi
	\end{array}
	\right)\no\\
	=&-2g_{ij}\check L^i\o^j+f(\varphi,\mathring L^1,\mathring L^2,\f{x}\varrho,\check\varrho)\left(
	\begin{array}{ccc}
	\vp^{k+1}\\
	\varphi^{k-1}\mathring L\phi\\
	\vp^{k-1}\slashed d^Xx\slashed d_X\phi\\
	\check L^1\varphi^{k}\\
		\check L^2\varphi^{k}\\
	\check\varrho\vp^{k}
	\end{array}
	\right).
	\end{align}
	Substituting \eqref{1-go} into \eqref{rrho} and applying \eqref{ZgLo} yield that for $m\leq 2N-9$,
	\begin{equation*}
	\delta^l|Z^m\check\varrho|\lesssim\delta^{(k+1)(1-\ve_0)}s^{-1}.
	\end{equation*}	
	Since $g_{00}=-1+f(\vp)\vp^{k+1}$, then $g_{\al\beta}\o^\al\o^\beta=-1+2g_{0i}\o^i+g_{ij}\o^i\o^j+f(\vp,\o)\vp^{k+1}$.
This, together with the first identity in \eqref{1-go}, shows that
	\begin{equation*}
	g_{\al\beta}\o^\al\o^\beta
	=f(\varphi,\mathring L^1,\mathring L^2,\f{x}\varrho,\check\varrho)\left(
	\begin{array}{ccc}
	\vp^{k+1}\\
	\varphi^{k-1}\mathring L\phi\\
	\vp^{k-1}\slashed d^Xx\slashed d_X\phi\\
	\check L^1\varphi^{k}\\
	\check L^2\varphi^{k}\\
	\check\varrho\vp^{k}
	\end{array}
	\right),
	\end{equation*}
	which implies $\dl^l |Z^m(g_{\al\beta}\o^\al\o^\beta)|\lesssim\dl^{(k+1)(1-\ve_0)}s^{-3/2}$.
\end{proof}

\begin{remark}
	It follows from \eqref{rhoL} that $\check\varrho=\f{r}{\varrho}-1=O(\delta^{(k+1)(1-\ve_0)}s^{-1})$. This
shows that the distance between $C_0$ and $C_{4\delta}$ on the hypersurface $\Sigma_t$ is $4\delta
	+O(\delta^{(k+1)(1-\ve_0)})$. Thus, the characteristic
	surface $C_u$ $(0\leq u\leq 4\delta)$ is almost straight with the error $O(\delta^{(k+1)(1-\ve_0)})$ from
	the corresponding light conic surface due to $\dl^{(k+1)(1-\ve_0)}=o(\dl)$ for $\ve_0<\ve_k^*$.
\end{remark}

\begin{lemma}\label{n}
	In $D^{s,3\dl}$, for $m\leq 2N-8$, it holds that
	\begin{equation}\label{Zu}
	\delta^l|Z^m\upsilon|\lesssim\delta^{(k+1)(1-\ve_0)},
	\end{equation}
	where $l$ is the number of $T$ in $Z^m$.
\end{lemma}
\begin{proof}
	Write $g_{\al\beta}=m_{\al\beta}+g_{\al\beta}^{\g_1\cdots\g_k}\vp_{\g_1}\cdots\vp_{\g_k}+h_{\al\beta}(\vp)$ in \eqref{g}.
Then \eqref{omega} implies
	\begin{equation}\label{upsilon}
	\begin{split}
	\upsilon=&-g_{0j}\epsilon_i^jx^i-g_{aj}\check L^a\epsilon_i^jx^i-g_{aj}\epsilon_i^j\f{x^ax^i}{\varrho}\\
	=&\varrho\big\{-(g_{0j}^{\gamma_1\cdots\g_k}+g_{ja}^{\g_1\cdots\g_k}\f{x^a}\varrho)\epsilon_i^j\f{x^i}{\varrho}\vp_{\g_1}\cdots\vp_{\g_k}
-m_{ja}\check L^a\epsilon_i^j\f{x^i}{\varrho}\big\}+\varrho f(\varphi,\f{x}\varrho)\left(
	\begin{array}{ccc}
	\vp^{k+1}\\
	\check L^1\vp^k\\
\check L^2\vp^k
	\end{array}
	\right).
	\end{split}
	\end{equation}
	Let
	\begin{equation}\label{mE}
	\mathscr  F=-(g_{0j}^{\gamma_1\cdots\g_k}+g_{ja}^{\g_1\cdots\g_k}\f{x^a}\varrho)\epsilon_i^j\f{x^i}{\varrho}\vp_{\g_1}\cdots\vp_{\g_k}-m_{ja}\check L^a\epsilon_i^j\f{x^i}{\varrho}.
	\end{equation}
	Then
	\begin{align}\label{TA}
	T\mathscr F=&-k(g_{0j}^{\gamma_1\cdots\g_k}+g_{ja}^{\g_1\cdots\g_k}\f{x^a}\varrho)
\epsilon_i^j\f{x^i}{\varrho}(T\vp_{\g_1})\vp_{\g_2}\cdots\vp_{\g_k}-m_{ja}\epsilon_i^j\big((T\mathring L^a)\f{x^i}\varrho+\check L^aT(\f{x^i}\varrho)\big)\no\\
	&-\big(g_{0j}^{\gamma_1\cdots\g_k}T(\f{x^i}\varrho)+g_{ja}^{\g_1\cdots\g_k}T(\f{x^ax^i}{\varrho^2})\big)
\epsilon_i^j\vp_{\g_1}\cdots\vp_{\g_k}+m_{ja}T(\f{x^a}\varrho)\epsilon_i^j\f{x^i}\varrho\no\\
	=&-k(g_{0j}^{\gamma_1\cdots\g_k}+g_{ja}^{\g_1\cdots\g_k}\f{x^a}\varrho)\epsilon_i^j\f{x^i}{\varrho}(T\vp_{\g_1})
\vp_{\g_2}\cdots\vp_{\g_k}-m_{ja}\epsilon_i^j(T\mathring L^a)\f{x^i}\varrho\\
	&+\f1\varrho f(\vp,\check L^1,\check L^2,\mu,\f{x}\varrho)\left(\begin{array}{ccc}
	\vp^k\\
	\check L^1\\
	\check L^2\\
	\mu-1
	\end{array}
	\right),\no
	\end{align}
	where $T(\f{x^a}{\varrho})=-\f{\mu}{\varrho}g^{0a}-\f\mu\varrho\check L^a+\f{(1-\mu)x^a}{\varrho^2}$ has been used.
	Note that there is a factor $T\mathring L^a$ in \eqref{TA} containing the term $-G_{X\mathring L}^\gamma(T\varphi_\gamma)\slashed d^Xx^a$
due to \eqref{TL}. It follows from the definition of $G_{X\mathring L}^\gamma$ in \eqref{FG}
and \eqref{gab} that
	\begin{align}\label{GTd}
	&-G_{X\mathring L}^\gamma(T\varphi_\gamma)\slashed d^Xx^a=-G_{i\beta}^\gamma(\slashed d_Xx^i\cdot\slashed d^Xx^a)\mathring L^\beta(T\varphi_\gamma)\no\\
	=&-G_{a\beta}^\gamma\mathring L^\beta(T\varphi_\gamma)-G_{\mathring L\mathring L}^\g(T\varphi_\gamma)\tilde T^a
	+G_{0\beta}^\gamma\mathring L^\beta\tilde T^a(T\varphi_\gamma)+f(\varphi,\mathring L^i)\vp^{2k-1}T\varphi_\gamma\no\\
	=&-k(g_{a0}^{\g\g_2\cdots\g_k}+g_{aj}^{\g\g_2\cdots\g_k}\f{x^j}\varrho+g_{0\beta}^{\g\g_2\cdots\g_k}\mathring L^\beta\f{x^a}\varrho)\vp_{\g_2}\cdots\vp_{\g_k}(T\varphi_\g)-G_{\mathring L\mathring L}^\g(T\varphi_\gamma)\tilde T^a\\
	&+f(\varphi,\mathring L^1,\mathring L^2)\left(\begin{array}{ccc}
	\vp^k\\
	\vp^{k-1}\check L^1\\
	\vp^{k-1}\check L^2
	\end{array}
	\right)T\varphi.\no
	\end{align}
	Examining \eqref{TL} carefully and using \eqref{GTd}, one can get
	\begin{equation}\label{TLi}
	\begin{split}
	T\mathring L^a=&-k(g_{a0}^{\g\g_2\cdots\g_k}+g_{aj}^{\g\g_2\cdots\g_k}\f{x^j}\varrho+g_{0\beta}^{\g\g_2\cdots\g_k}\mathring L^\beta\f{x^a}\varrho)\vp_{\g_2}\cdots\vp_{\g_k}(T\varphi_\g)\\
	&+f(\varphi,\mathring L^1,\mathring L^1)\left(\begin{array}{ccc}
	\vp^k(T\varphi)\\
	\vp^{k-1}\check L^1(T\varphi)\\
		\vp^{k-1}\check L^2(T\varphi)\\
	G_{\mathring L\mathring L}^\g(T\varphi_\gamma)\\
	\slashed d^X\mu\slashed d_Xx\\
	\vp^{k-1}(\mu\mathring L\vp)\\
	\vp^{k-1}(\mu\slashed d^X\vp\slashed d_Xx)
	\end{array}
	\right).
		\end{split}
	\end{equation}
	For $m\leq 2N-8$, substitute \eqref{TLi} into \eqref{TA} to yield
	\begin{equation*}
	\delta^l|Z^mT\mathscr F|\lesssim\delta^{k(1-\ve_0)-\ve_0}s^{-1},
	\end{equation*}
	where \eqref{GT} and \eqref{gLLL} have been used.
	Thus,
	\begin{equation*}
	\delta^l|TZ^m\mathscr F|\leq\delta^l|[T,Z^m]\mathscr F|+\delta^l|Z^mT\mathscr F|\lesssim\delta^{k(1-\ve_0)-\ve_0}s^{-1},
	\end{equation*}
	which, together with \eqref{pal}, implies
	\begin{equation}\label{uZA}
	\delta^l|(\p_t-\p_r)Z^m\mathscr F|\leq\delta^l|TZ^m\mathscr F|+\dl^l|\mathring L Z^m\mathscr F|+\delta^l|\slashed dZ^m\mathscr F|\lesssim\delta^{k(1-\ve_0)-\ve_0}s^{-1}.
	\end{equation}
 Integrating \eqref{uZA} along integral curves of $\p_t-\p_r$ yields
	\begin{equation}\label{ZkA}
	\delta^l|Z^m\mathscr F|\lesssim\delta^{(k+1)(1-\ve_0)}s^{-1},\qquad m\leq 2N-8.
	\end{equation}
	It follows from \eqref{upsilon}, \eqref{ZkA} and \eqref{orL} that for $m\leq 2N-8$,
	\begin{equation*}
	\delta^l|Z^m\upsilon|\lesssim\delta^{(k+1)(1-\ve_0)}.
	\end{equation*}
\end{proof}

Now we are ready to obtain the more precise orders of smallness of $\dl$ and time decay rates for
the $L^\infty$ norms of $\phi$ under the actions of $\Gamma\in\{L,\underline L,\O\}$.
Recall the standard vectors $\{L,\underline L,\O\}$ defined in the end of Section \ref{in}.
It follows from \eqref{pal} and \eqref{R} that
\begin{equation}\label{sL}
\begin{split}
&L=\mathring L-\mu^{-1}\o^\al\mathring L_\al T+\o^\al g_{\al j}(\slashed d^Xx^j)X,\\
&\underline L=\mathring L+\mu^{-1}(\o^i\mathring L_i-\mathring L_0)T+(g_{0j}\slashed d^Xx^j-\o^i g_{ij}\slashed d^Xx^j)X,\\
&\Omega=R+\mu^{-1}\upsilon T.
\end{split}
\end{equation}

The identity
\[
\o^\al\mathring L_\al=g_{\al\beta}\o^\al\o^\beta+g_{0i}\check L^i+g_{ij}\o^i\check L^j+\check\varrho(g_{0i}\o^i+g_{ij}\o^i\o^j)
\]
and \eqref{rhoL} imply that for $m\leq 2N-9$,
\begin{equation}\label{oL}
\delta^l|Z^m(\o^\al\mathring L_\al)|\lesssim\delta^{(k+1)(1-\varepsilon_0)}s^{-1}
\end{equation}
holds in $D^{s,4\dl}$.
Therefore, collecting \eqref{Zu}, \eqref{sL}, \eqref{oL} with \eqref{im} and \eqref{imp} yields that in $D^{s,3\dl}$,
\begin{equation}\label{gv}
|\Gamma^m\varphi_\gamma|\lesssim\delta^{1-\varepsilon_0-l}s^{-1/2},
\qquad|\Gamma^m\phi|\lesssim\delta^{2-\varepsilon_0-l}s^{-1/2},\quad m\leq 2N-8,
\end{equation}
where $\Gamma\in\{(s+r)L,\underline L,\Omega\}$ and $l$ is the number of $\underline L$ in $\Gamma^m$.

For any point $P\in\tilde C_{2\delta}=\{(t,x):t\geq t_0,t-|x|=2\delta\}$, there is an integral curve of
$L$ across this point, and the corresponding initial point is denoted by $P_0(t_0,x_0)$ on $\Sigma_{t_0}$ with $|x_0|=1$.
From \eqref{local3-3}, $|\underline L^pL^q\O^c\phi(P_0)|\lesssim\delta^{2-\varepsilon_0}$ holds.
Next we derive an improved estimate on $\underline L^pL^q\O^c\phi$ at the point $P$.
It is emphasized here that the condition of $0<\ve_0<\ve_k^*$ plays a key role
in the following proposition.

\begin{proposition}\label{11.1}
	For any $(t,x)\in\tilde C_{2\delta}$, when $N$ is large enough, it holds that
	\begin{equation}\label{LLOp}
	|\underline L^pL^q\O^c\phi(t,x)|\lesssim\delta^{2-\varepsilon_0}t^{-1/2-q}.
	\end{equation}
\end{proposition}
\begin{proof}
According to \eqref{me} in Appendix B and \eqref{gv}, one has
\begin{equation}\label{Y-28}
| L(r^{1/2}\bar\p^m\O^c{\underline L}\phi)|\lesssim\delta^{2-\varepsilon_0-m}t^{-2}+\delta^{(k+1)(1-\varepsilon_0)-\varepsilon_0-m} t^{-3/2},
\end{equation}
where $\bar\p\in\{\p_t,\p_r\}$.
Integrating \eqref{Y-28} along integral curves of $L$ and using the estimate \eqref{local3-3} on initial data at $t_0$,
one then can get
that on $\tilde C_{2\delta}$,
\begin{equation}\label{poL}
|\bar\p^m\O^c{\underline L}\phi|\lesssim\delta^{2-\varepsilon_0-m}t^{-1/2}+\delta^{(k+1)(1-\varepsilon_0)-\varepsilon_0-m}t^{-1/2},\ m+c\leq 2N-10.
\end{equation}
Next we construct a sequence of functions inductively as follows:
\begin{align*}
a_0(\epsilon)&=2-\epsilon,\\
a_1(\epsilon)&=(k+1)(1-\epsilon)-\epsilon,\\
a_{n+1}(\epsilon)&=(k+2)a_{n}(\epsilon)-1,\qquad n\geq 1.
\end{align*}
Then for $n\geq 1$,
\begin{equation*}
a_{n+1}(\epsilon)=(k+2)^na_1(\epsilon)-\f{(k+2)^n-1}{k+1}.
\end{equation*}
Let $\epsilon(n)$ and $\bar\epsilon(n)$ be the solutions to $a_n(\epsilon)=a_0(\epsilon)$ and $ka_n(\epsilon)+1-\epsilon=a_{n+1}(\epsilon)$ respectively, so that
\begin{align*}
\epsilon(n)&=\f k{k+1}-\f1{(k+2)^n-1},\\
\bar\epsilon(n)&=\f k{k+1}-\f k{(k+1)\big(2(k+2)^n-1\big)}.
\end{align*}
Thus,
\begin{align*}
&\lim_{n\rightarrow\infty}\epsilon(n)=\lim_{n\rightarrow\infty}\bar\epsilon(n)=\f{k}{k+1},\\
&\epsilon(n)<\bar\epsilon(n)<\epsilon(n+1).
\end{align*}
For any fixed number $\ve_0\in (0,\ve_k^*)$, there exists an $n\in\mathbb N$ such that $\epsilon(n-1)<\ve_0\leq\epsilon(n)$
with the convention that $\epsilon(0)=0$.
\begin{enumerate}
	\item When $n=1$, i.e., $0<\ve_0\leq\epsilon(1)$, $(k+1)(1-\ve_0)-\ve_0\geq 2-\ve_0$ holds.
It follows from \eqref{poL} that on $\tilde C_{2\delta}$,
	\begin{equation*}
	|\bar\p^m\O^c{\underline L}\phi|\lesssim\delta^{2-\varepsilon_0-m}t^{-1/2},\ m+c\leq 2N-10.
	\end{equation*}
	\item When $n=2$ and $\ve_0\leq\bar\epsilon(1)$, one has $(k+1)(1-\ve_0)-\ve_0< 2-\ve_0$ and $2-\ve_0<1-\ve_0+ka_1(\ve_0)\leq a_2(\ve_0)$.
Then \eqref{poL} implies that on $\tilde C_{2\delta}$,
	\begin{equation}\label{pOuL}
	|\bar\p^m\O^c{\underline L}\phi|\lesssim\dl^{a_1(\ve_0)-m}t^{-1/2},\ m+c\leq 2N-10.
	\end{equation}
	Using \eqref{pOuL} and \eqref{me} in Appendix B again, then one has that for $m+c\leq 2N-11$,
	\begin{equation*}
	|L\big(r^{1/2}\bar\p^m\O^c{\underline L}\phi\big)|\lesssim\dl^{2-\ve_0-m}t^{-2}+\dl^{1-\ve_0+ka_1(\ve_0)-m}t^{-2}+\dl^{a_2(\ve_0)-m}t^{-3/2}\lesssim\dl^{2-\ve_0-m}t^{-3/2},
	\end{equation*}
	which implies that on $\tilde C_{2\delta}$,
	\begin{equation*}
	|\bar\p^m\O^c{\underline L}\phi|\lesssim\delta^{2-\varepsilon_0-m}t^{-1/2},\ m+c\leq 2N-11.
	\end{equation*}
	When $n=2$ and $\bar\epsilon(1)<\ve_0\leq\epsilon(2)$, one has $(k+1)(1-\ve_0)-\ve_0< 2-\ve_0$ and $2-\ve_0\leq a_2(\ve_0)<1-\ve_0+ka_1(\ve_0)$. Exactly similar to the case for $n=2$ and $\ve_0\leq\bar\epsilon(1)$, one has
that on $\tilde C_{2\delta}$,
		\begin{equation*}
	|\bar\p^m\O^c{\underline L}\phi|\lesssim\delta^{2-\varepsilon_0-m}t^{-1/2},\ m+c\leq 2N-11.
	\end{equation*}
	\item For general $n\geq 3$, we now focus on the case of $j\in\mathbb N\cap [1,n-1]$ and $\ve_0>\epsilon(j)$.
First, it follows from $\ve_0>\epsilon(1)$ and \eqref{pOuL} that on $\tilde C_{2\delta}$,
		\begin{equation*}
	|\bar\p^m\O^c{\underline L}\phi|\lesssim\dl^{a_1(\ve_0)-m}t^{-1/2},\ m+c\leq 2N-10.
	\end{equation*}	
Then, an induction argument on $j$ can be used to show that on $\tilde C_{2\delta}$,
\begin{equation}\label{aj}
|\bar\p^m\O^c{\underline L}\phi|\lesssim\dl^{a_j(\ve_0)-m}t^{-1/2}
\end{equation}
for $m+c\leq 2N-9-j$ ($1\leq j\leq n-1$). In fact, if \eqref{aj} holds true for all $j$ with $1\leq j\leq j_0\leq n-2$, then when $m+c\leq 2N-10-j_0$, it holds that
	$$|L(r^{1/2}\bar\p^m\O^c\underline L\phi)|\lesssim\dl^{2-\ve_0-m}t^{-2}+\dl^{1-\ve_0-m+ka_{j_0}(\ve_0)}t^{-2}+\dl^{a_{j_0+1}(\ve_0)-m}t^{-3/2}.$$
	It follows from $\ve_0>\epsilon(j_0+1)>\bar\epsilon(j_0)$ that $a_{j_0+1}(\ve_0)<2-\ve_0$ and $a_{j_0+1}(\ve_0)<1-\ve_0+ka_{j_0}(\ve_0)$
hold, and hence $|L(r^{1/2}\bar\p^m\O^c\underline L\phi)|\lesssim\dl^{a_{j_0+1}(\ve_0)-m}t^{-3/2}$, which implies \eqref{aj}.
Thus, on $\tilde C_{2\delta}$,
	\begin{equation}\label{bpO}
	|\bar\p^m\O^c{\underline L}\phi|\lesssim\dl^{a_{n-1}(\ve_0)-m}t^{-1/2},\ m+c\leq 2N-8-n.
	\end{equation}
	By \eqref{bpO} and \eqref{me} in Appendix B, we obtain that for $m+c\leq 2N-9-n$,
	\begin{equation}\label{2N-9-n}
	|L(r^{1/2}\bar\p^m\O^c\underline L\phi)|\lesssim\dl^{2-\ve_0-m}t^{-2}+\dl^{1-\ve_0+ka_{n-1}(\ve_0)-m}t^{-2}+\dl^{a_{n}(\ve_0)-m}t^{-3/2}.
	\end{equation}
	Since $\epsilon(n-1)<\ve_0\leq\epsilon(n)$, then $2-\ve_0\leq a_n(\ve_0)$ and $2-\ve_0<ka_{n-1}(\ve_0)+1-\ve_0$
hold. Therefore, integrating \eqref{2N-9-n} along integral curves of $L$ yields that on $\tilde C_{2\delta}$,
	\begin{equation*}
	|\bar\p^m\O^c{\underline L}\phi|\lesssim\dl^{2-\ve_0-m}t^{-1/2},\ m+c\leq 2N-9-n.
	\end{equation*}
\end{enumerate}
Consequently, for any $n\in\mathbb N$ and $m+c\leq 2N-9-n$, it holds that
\begin{align}
&|\bar\p^m\O^c{\underline L}\phi|_{\tilde C_{2\dl}}\lesssim\dl^{2-\ve_0-m}t^{-1/2},\label{bp}\\
&|L\bar\p^m\O^c{\underline L}\phi|_{\tilde C_{2\dl}}\lesssim\dl^{2-\ve_0-m}t^{-3/2}.\label{Lbp}
\end{align}
With an induction argument on the power of $\underline L$, one can show
that for $c+m+2p\leq 2N-7-n$ and on $\tilde C_{2\delta}$,
\begin{align}
&|\bar\p^m\O^c\underline L^p\phi|\lesssim\delta^{2-\varepsilon_0-m}t^{-1/2},\label{pOL}\\
&|L\bar\p^m\O^c\underline L^p\phi|\lesssim\delta^{2-\varepsilon_0-m}t^{-3/2}.\label{LpOL}
\end{align}

Based on \eqref{LpOL}, we now assume that for some positive integer $q_0$, when $1\leq q\leq q_0$ and $q-1+c+m+2p\leq 2N-7-n$,
there holds that on $\tilde C_{2\delta}$,
\begin{equation}\label{ass}
|L^q\bar\p^m\O^c\underline L^p\phi|\lesssim\delta^{2-\varepsilon_0-m}t^{-1/2-q}.
\end{equation}
For $q_0+c+m+2p\leq 2N-7-n$ and $p\geq 1$, by \eqref{gv}, \eqref{pOL} and the assumption \eqref{ass},
it follows from \eqref{me} in Appendix B  that on $\tilde C_{2\delta}$,
\begin{equation*}
\begin{split}
&\delta^m|L^{q_0+1}\bar\p^m\O^c\underline L^{p}\phi|=\delta^m| L^{q_0}\bar\p^m\O^c\underline L^{p-1}(L\underline L\phi)|\\
\lesssim&\delta^{2-\varepsilon_0}t^{-3/2-q_0}
+\sum_{c_1+m_1\leq c+m+1}\delta^{m_1}|L^{q_0+1}\bar\p^{m_1}\O^{c_1}\underline L^{\leq p-1}\phi|t^{-1}.
\end{split}
\end{equation*}
By reducing the power of $\underline L$ gradually, one can eventually obtain that on $\tilde C_{2\delta}$,
\begin{equation*}
\begin{split}
\delta^m|L^{q_0+1}\bar\p^m\O^c\underline L^{p}\phi|\lesssim&\delta^{2-\varepsilon_0}t^{-3/2-q_0}+\sum_{c_1+m_1\leq c+m+p}\delta^{m_1}|L^{q_0+1}\bar\p^{m_1}\O^{c_1}\phi|t^{-1}\\
\lesssim&\delta^{2-\varepsilon_0}t^{-3/2-q_0}.
\end{split}
\end{equation*}
This shows that \eqref{ass} holds for any positive $q_0$ by induction.

Then \eqref{LLOp} follows by setting $m=0$ in \eqref{pOL} and \eqref{ass} together with the second estimate in \eqref{gv}.
\end{proof}

\section{Global estimates in $B_{2\dl}$ and the proof of Theorem \ref{main}}\label{inside}

In this section, we will derive the global estimates on the solution $\phi$ to the equation \eqref{quasi} in $B_{2\dl}$.
Note that although the general procedure is exactly same as the last section in \cite{Ding3}, yet
due to the requirements of the precise orders of smallness of $\dl$ and the time decay rates  together
with the higher order null condition \eqref{null}, some details are still given here.

Set
$$
D_{t}:=\{(\bar t,x): \bar t-|x|\geq 2\delta, t_0\leq \bar t\leq t\},
$$
see Figure \ref{pic:p2} below. As explained in \cite{Ding3}, different from \cite{MPY} and \cite{Wang}, the solution
$\phi$ to \eqref{quasi} in $D_{t}$ may be large here due to its initial data on time $t_0$ (see Theorem \ref{Th2.1}).
By the way, when $\delta>0$ is small, we know that the $L^\infty$ norms of $\phi$ and its derivatives are small on the
boundary $\tilde C_{2\delta}$ of $D_{t}$ (see \eqref{LLOp}).

		\begin{figure}[htbp]
	\centering
	\includegraphics[scale=0.3]{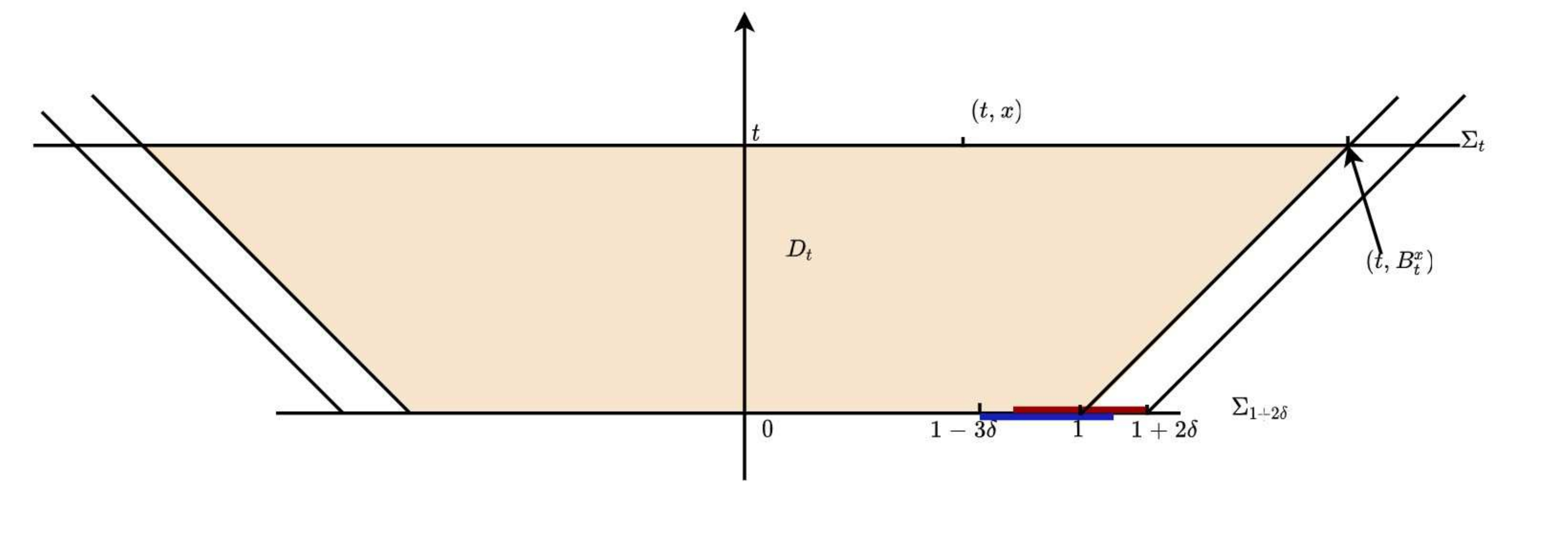}
	\caption{The domain $D_{t}$}\label{pic:p2}
\end{figure}

As in \cite{Ding3}, we use the following modified Klainerman-Sobolev Lemma (see \cite[Lemma ll.1]{Ding3})
to solve the Goursat problem here.

\begin{lemma}\label{KS}
	For any function $f(t,x)\in C^\infty(\mathbb R^{1+2})$, $t\geq 1$, $(t,x)\in D_{T}=
	\{(t,x): t-|x|\geq 2\delta, t_0\leq t\leq T\}$, the following inequalities hold:
	\begin{equation}\label{leq}
	\mid f(t,x)\mid\lesssim\sum_{i=0}^2 t^{-1}\delta^{(i-1)\mathtt{s}}\|\bar{\Gamma}^{i}f(t,\cdot)\|_{L^2(r\leq t/2)},\ |x|\leq\f14t,
	\end{equation}
	\begin{equation}\label{geq}
	\mid f(t,x)\mid\lesssim \mid f(t,B_t^x)\mid+\sum_{a\leq 1,|\beta|\leq 1}t^{-1/2}
	\|\Omega^a\p^\beta f(t,\cdot)\|_{L^2(t/4\leq r\leq t-2\delta)},\ |x|\geq\f14t,
	\end{equation}
	where $\bar{\Gamma}\in\{S,H_i,\Omega\}$, $(t,B_t^x)$ is the intersection point of
	the boundary $\tilde C_{2\delta}$ and the ray crossing $(t,x)$ which emanates from $(t,0)$,
	and $\mathtt s$ is the any nonnegative constant in \eqref{leq}.
\end{lemma}

The following inequality will also be needed (see \cite[Lemma 11.2]{Ding3}).

\begin{lemma}\label{L2}
	For $f(t,x)\in C^\infty(\mathbb R^{1+2})$ and $t\geq 1$, it holds that for $1\leq\bar t\leq t-2\delta$,
	\begin{equation}\label{L2L2}
	\|\f{f(t,\cdot)}{1+t-|\cdot|}\|_{L^2(\bar t\leq |x|\leq t-2\delta)}\lesssim t^{1/2}
	\|f(t,B_t^{\cdot})\|_{L^\infty(\bar t\leq |x|\leq t-2\delta)}+\|\p f(t,\cdot)\|_{L^2(\bar t\leq |x|\leq t-2\delta)}.
	\end{equation}
\end{lemma}

We intend to use the energy method to prove global estimates of the solution $\phi$ to \eqref{quasi} in $B_{2\dl}$.
Set
\begin{equation}\label{Linfty}
E_{p,l}(t)=\sum_{\tilde\Gamma\in\{\p,S,H_1,H_2\}}\Big\{\|\p\tilde\Gamma^p\Omega^l\phi(t,\cdot)\|_{L^2(\Sigma_t\cap D_{t})}^2+\sum_{i=1}^2\iint_{D_t}\f{|\tilde Z_i\tilde\Gamma^p\Omega^l\phi|^2(t',x)}{1+(t'-|x|)^{3/2}}dxdt'\Big\},
\end{equation}
where $\tilde Z_i=\omega_i\p_t+\p_i, i=1,2$ and $p+l\leq 4$. As in \cite{Ding3},
the introduction of the term $\sum_{i=1}^2\iint_{D_t}\f{|\tilde Z_i\tilde\Gamma^p\Omega^l\phi|^2(t',x)}{1+(t'-|x|)^{3/2}}dxdt'$
in \eqref{Linfty} is motivated by the works on global existence of small data solutions to the 3D nonlinear wave equations
satisfying the weak null condition in \cite{Lin2}.

Thanks to
the estimate \eqref{local3-2} on $\Sigma_{t_0}$, one makes the following bootstrap assumptions:

For $t\ge t_0$, there exists a uniform constant $M_0$ such that
\begin{equation}\label{EA}
\begin{split}
&E_{0,l}(t)\leq {M_0}^2\delta^{4-2\varepsilon_0},\qquad E_{1,l}(t)\leq {M_0}^2\delta^{\f {17}6-2\varepsilon_0},\qquad E_{2,l}(t)\leq {M_0}^2\delta^{1-2\varepsilon_0},\\
&E_{3,l}(t)\leq {M_0}^2\delta^{-1-2\varepsilon_0},\qquad E_{4,l}(t)\leq {M_0}^2\delta^{-3-2\varepsilon_0}.
\end{split}
\end{equation}

According to Lemma \ref{KS} and the assumptions in \eqref{EA}, similarly to the last section of \cite{Ding3}, one can establish
the following $L^\infty$ estimates.

\begin{proposition}\label{P4.1}
	Under the assumptions \eqref{EA}, when $\delta>0$ is small, it holds that
	\begin{equation}\label{pphi}
	\begin{split}
	&|\p\Omega^{\leq 2}\phi|\lesssim M_0\delta^{5/4-\varepsilon_0}t^{-1/2},\qquad\quad |\p\tilde\Gamma\Omega^{\leq 1}\phi|\lesssim M_0\delta^{11/24-\varepsilon_0}t^{-1/2},\\
	&|\p\tilde\Gamma^2\phi|\lesssim M_0\delta^{-1/2-\varepsilon_0}t^{-1/2}
	\end{split}
	\end{equation}
	and
		\begin{equation}\label{tZ}
	\begin{split}
	&|\t Z\Omega^{\leq 2}\phi|\lesssim M_0\delta^{5/4-\varepsilon_0}t^{-3/2}(1+t-r),\qquad\quad |\t Z\tilde\Gamma\Omega^{\leq 1}\phi|\lesssim M_0\delta^{11/24-\varepsilon_0}t^{-3/2}(1+t-r),\\
	&|\t Z\tilde\Gamma^2\phi|\lesssim M_0\delta^{-1/2-\varepsilon_0}t^{-3/2}(1+t-r),
	\end{split}
	\end{equation}
	where $\t Z\in\{\t Z_1,\t Z_2\}$.
\end{proposition}

\begin{proof}
	Following the proof of \cite[Proposition 11,1]{Ding3}, one can choose $s=\f34$ for $\t Z\O^{\leq 2}\phi$, $s=\f{23}{24}$
		for $\tilde Z\tilde\Gamma\O^{\leq 1}\phi$ and $s=1$ for $\tilde Z\tilde\G^2\phi$ in (11.11) of \cite{Ding3}
to obtain \eqref{pphi}-\eqref{tZ}, respectively. The details are omitted here.
\end{proof}

As in the proof of \cite[Corollary 11.1]{Ding3}, it follows from Proposition \ref{P4.1} that
\begin{corollary}\label{C4.1}
	Under the same assumptions in Proposition \ref{P4.1}, it holds that
	\begin{equation}
	\begin{split}
	&|\p^2\Omega^{\leq 1}\phi|\lesssim M_0\delta^{11/24-\varepsilon_0}t^{-1/2}(1+t-r)^{-1},\\
	&|\p^2\tilde\Gamma\phi|\lesssim M_0\delta^{-1/2-\varepsilon_0}t^{-1/2}(1+t-r)^{-1}
	\end{split}
	\end{equation}
	and
	\begin{equation}
	\begin{split}
	&|\tilde Z\p\Omega^{\leq 1}\phi|\lesssim M_0\delta^{11/24-\varepsilon_0}t^{-3/2},\\
	&|\tilde Z\p\tilde\Gamma\phi|\lesssim M_0\delta^{-1/2-\varepsilon_0}t^{-3/2}.
	\end{split}
	\end{equation}
\end{corollary}

We now derive the global weighted energy estimates on $\phi$ in $B_{2\dl}$. As in \cite{Ding3}, after choosing
the smooth ghost weight function $W=e^{2(1+t-r)^{-1/2}}$ introduced first in \cite{A},
then direct computations show
 \begin{equation}\label{Wvv}
 \begin{split}
 &W(\p_tv)g^{\al\beta}\p_{\al\beta}^2v\\
 =&\f12\p_t\big\{W\big((g^{00}(\p_t v)^2-g^{ij}\p_iv\p_jv)\big)\big\}+\p_i\big\{W\big((g^{0i}(\p_tv)^2
 +g^{ij}\p_tv\p_jv)\big)\big\}\\
 &+\big(-(\p_ig^{0i})(\p_tv)^2-(\p_ig^{ij})\p_tv\p_jv+\f12(\p_tg^{ij})\p_iv\p_jv\big)W\\
 &+\f{W}{(1+t-r)^{3/2}}\big(\f12g^{00}(\p_tv)^2-\f12g^{ij}(\p_iv)(\p_jv)-g^{0i}\o_i(\p_tv)^2-g^{ij}\o_i(\p_tv)(\p_jv)\big).
 \end{split}
 \end{equation}
Due to $\p_i=\t Z_i-\o_i\p_t$, then
 \begin{equation}\label{Dt2}
 \begin{split}
 &-(\p_ig^{0i})(\p_tv)^2-(\p_ig^{ij})\p_tv\p_jv+\f12(\p_tg^{ij})\p_iv\p_jv\\
 =&-k\big\{g^{0i,\g_1\cdots\g_k}(\p_i\p_{\g_1}\phi)(\p_tv)^2+g^{ij,\g_1\cdots\g_k}(\p_i\p_{\g_1}\phi)(\p_tv)(\p_jv)\\
 &-\f12g^{ij,\g_1\cdots\g_k}(\p_t\p_{\g_1}\phi)(\p_iv)(\p_jv)\big\}(\p_{\g_2}\phi)\cdots(\p_{\g_k}\phi)+f(\p\phi)(\p\phi)^k\p^2\phi(\p v)^2\\
 =&f(\o, \p\phi)\left(\begin{array}{ccc}
 \p^2\phi(\p\phi)^k(\p v)^2\\
 \t Z\p\phi(\p\phi)^{k-1}(\p v)^2\\
 \p^2\phi(\t Z\phi)(\p\phi)^{k-2}(\p v)^2\\
 \p^2\phi(\p\phi)^{k-1}(\t Zv)\p v
 \end{array}
 \right).
 \end{split}
 \end{equation}
 Note that integrating \eqref{Wvv} over $D_t$ and using the integration by parts give rise to a term on the lateral
 boundary $\tilde C_{2\delta}$ as
 \begin{equation}\label{lb}
 \begin{split}
 &\Big\{-\f12(g^{00}(\p_t v)^2-g^{ij}\p_iv\p_jv)+(g^{0i}(\p_tv)^2+g^{ij}\p_tv\p_jv)\o_i\Big\}W\\
 =&\Big\{\f12(Lv)^2+\f12(\f{1}{r}\O v)^2\Big\}W+f(\o, \p\phi)W\left(\begin{array}{ccc}
 (\p\phi)^{k+1}(\p v)^2\\
 \t Z\phi(\p\phi)^{k-1}(\p v)^2\\
 (\p\phi)^{k}(\t Zv)\p v
 \end{array}
 \right).
 \end{split}
 \end{equation}
This can be controlled on $\tilde C_{2\delta}$ by
\begin{equation}\label{C2}
\Big\{(Lv)^2+(\f{1}{r}\O v)^2+\delta^{k(2-\varepsilon_0)}t^{-3/2}(\p v)^2\Big\}W
\end{equation}
with the help of \eqref{LLOp} and $\tilde Z_i=\omega^iL+\f{\omega^i_\perp}{r}\O$, where one has neglected the
unimportant constant coefficients in \eqref{C2}.

Integrating \eqref{Wvv} on $D_t$ and
utilizing the estimates in Proposition \ref{P4.1} and Corollary \ref{C4.1},
 then one can get by \eqref{Dt2}-\eqref{C2} and Gronwall's inequality that for small $\delta>0$,
\begin{equation}\label{W}
\begin{split}
&\int_{\Sigma_t\cap D_t}W\big((\p_tv)^2+|\nabla v|^2\big)+\iint_{D_t}\sum_i\f{|\tilde Z_iv|^2}{(1+\tau-r)^{3/2}}W\\
\lesssim&\int_{\Sigma_{t_0}\cap D_t}W\big((\p_\tau v)^2+|\nabla v|^2\big)+\iint_{D_t}W| \p_\tau vg^{\al\beta}\p_{\al\beta}^2v|\\
&+\int_{\tilde C_{2\delta}\cap D_t}W\big\{(Lv)^2+(\f{1}{r}\O v)^2
+\delta^{k(2-\varepsilon_0)}\tau^{-3/2}|\p v|^2\big\}.
\end{split}
\end{equation}

To close the bootstrap assumptions \eqref{EA}, one sets $v=\tilde\Gamma^p\Omega^l\phi$ $(p+l\leq 4)$ in \eqref{W}.
 By \eqref{LLOp}, $(L\tilde\Gamma^p\Omega^l\phi)^2+(\f{1}{r}\O\tilde\Gamma^p\Omega^l\phi)^2
 +\delta^{k(2-\varepsilon_0)}t^{-3/2}|\p \tilde\Gamma^p\Omega^l\phi|^2\lesssim\delta^{4-2\varepsilon_0}t^{-5/2}$ holds
 on $\tilde C_{2\delta}$. Therefore,
\[
\int_{\tilde C_{2\delta}\cap D_t}W\big\{(Lv)^2+(\f{1}{r}\O v)^2+\delta^{k(2-\varepsilon_0)}\tau^{-3/2}|\p v|^2\big\}\lesssim\delta^{4-2\varepsilon_0}.
\]
In addition, on the initial hypersurface $\Sigma_{t_0}\cap D_t$,
one has $|\p\tilde\Gamma^p\Omega^l\phi|\lesssim\delta^{2-p-\varepsilon_0}$
for $0\leq p\leq 4-l$ by \eqref{local3-2}. Hence, \eqref{W} gives that
\begin{align}
&E_{0,l}(t)\lesssim\delta^{4-2\varepsilon_0}+\iint_{D_t}| (\p_\tau\Omega^l\phi)(g^{\al\beta}\p_{\al\beta}^2\Omega^l\phi)|,
\quad l\leq 4\label{E01},\\
&E_{p,l}(t)\lesssim\delta^{5-2p-2\varepsilon_0}+\iint_{D_t}| (\p_\tau\tilde\Gamma^p\Omega^l\phi)(g^{\al\beta}\p_{\al\beta}^2\tilde\Gamma^p\Omega^l\phi)|,\quad 1\leq k\leq 4-l.\label{E2-6}
\end{align}

It remains to estimate $\iint_{D_t}|(\p_\tau\tilde\Gamma^p\Omega^l\phi)(g^{\al\beta}\p_{\al\beta}^2\tilde\Gamma^p\Omega^l\phi)|$
in the right hand sides of \eqref{E01}-\eqref{E2-6} and further obtain the global estimates of $\phi$ in $B_{2\dl}$.

\begin{theorem}\label{T4.1}
	When $\delta>0$ is small, \eqref{EA} in $B_{2\dl}$ holds true.
\end{theorem}
\begin{proof}
	Acting the operator $\tilde\Gamma^k\Omega^l$ on \eqref{quasi} and commuting it with $g^{\al\beta}\p_{\al\beta}^2$ yield
	\begin{align*}
	|g^{\al\beta}\p_{\al\beta}^2&\tilde\Gamma^p\Omega^l\phi|\lesssim\sum_{\mbox{\tiny$\begin{array}{cc}p_0+\cdots+p_k\leq p,\\
			l_0+\cdots+l_k\leq l\\p_0+l_0<p+l\end{array}$}}\Big\{|\t Z\t\G^{p_1}\O^{l_1}\phi|\cdot|\p\t\G^{p_2}\O^{l_2}\phi|\cdots|\p\t\G^{p_k}\O^{l_k}\phi|\cdot|\p^2\tilde\Gamma^{p_0}\O^{l_0}\phi|\\
		&+|\p\t\G^{p_1}\O^{l_1}\phi|\cdots|\p\t\G^{p_k}\O^{l_k}\phi|\cdot|\t Z\p\t\G^{p_0}\O^{l_0}\phi|\Big\}+\sum_{\mbox{\tiny$\begin{array}{cc}q_0+q_1\leq p,\\j_0+j_1\leq l\\q_0+j_0<p+l\end{array}$}}|\t\G^{q_1}\O^{j_1}h^{\al\beta}|\cdot|\p^2\t\G^{q_0}\O^{j_0}\phi|,
	\end{align*}
		which can be treated case by case as follows.
\begin{enumerate}

\vskip 0.1 true cm
	\item {\bf The case of $p=0$ and $l\leq 4$}
\vskip 0.1 true cm
	\begin{enumerate}
		\item If $l_0\geq\max\{l_1,\cdots,l_k\}$, then $l_1,\cdots,l_k\leq 2$ and $l_0\leq 3$. Then by Proposition \ref{P4.1},
		\begin{align*}
		&|\t Z\O^{l_1}\phi|\cdot|\p\O^{l_2}\phi|\cdots|\p\O^{l_k}\phi|\cdot|\p^2\O^{l_0}\phi|\lesssim M_0^k\dl^{k(5/4-\ve_0)}t^{-2}(1+t-r)|\p^2\O^{l_0}\phi|,\\
		&|\p\O^{l_1}\phi|\cdots|\p\O^{l_k}\phi|\cdot|\t Z\p\O^{l_0}\phi|\lesssim M_0^k\dl^{k(5/4-\ve_0)}t^{-1}|\t Z\p\O^{l_0}\phi|.
		\end{align*}
		In addition, due to $(1+t-r)|\p^2\O^{l_0}\phi|\lesssim|\t\G\p\O^{l_0}\phi|+|\O\p\O^{l_0}\phi|$ and $|\t Z\p\O^{l_0}\phi|\lesssim t^{-1}|\t\G\p\O^{l_0}\phi|+t^{-1}|\O\p\O^{l_0}\phi|$, one has by the assumptions \eqref{EA} and $\ve_0\in(0,\ve_k^*)$ that
\begin{equation}\label{l1}
\begin{split}
&\iint_{D_t}|\p\Omega^l\phi|\Big\{|\t Z\t\G^{p_1}\O^{l_1}\phi|\cdot|\p\t\G^{p_2}\O^{l_2}\phi|\cdots|\p\t\G^{p_k}\O^{l_k}\phi|\cdot|\p^2\tilde\Gamma^{p_0}\O^{l_0}\phi|\\
&\qquad\qquad\qquad+|\p\t\G^{p_1}\O^{l_1}\phi|\cdots|\p\t\G^{p_k}\O^{l_k}\phi|\cdot|\t Z\p\t\G^{p_0}\O^{l_0}\phi|\Big\}\\
\lesssim&\int_{t_0}^t\tau^{-2}E_{0,\leq 4}(\tau)d\tau+(M_0\dl^{5/4-\ve_0})^{2k}\int_{t_0}^t\tau^{-2}E_{1,\leq 3}(\tau)d\tau\\
\lesssim&\int_{t_0}^t\tau^{-2}E_{0,\leq 4}(\tau)d\tau+\dl^{4-2\ve_0}.
\end{split}
\end{equation}

\vskip 0.2 true cm
		\item If $l_0<\max\{l_1,\cdots,l_k\}$, then $l_0\leq 1$. It follows from Corollary \ref{C4.1} that
		\begin{align*}
		&|\p^2\Omega^{l_0}\phi|\lesssim M_0\delta^{11/24-\varepsilon_0}t^{-1/2}(1+t-r)^{-1},\\
		&|\t Z\p\O^{l_0}\phi|\lesssim M_0\delta^{11/24-\varepsilon_0}t^{-3/2}.
		\end{align*}
		Similarly, with the help of \eqref{EA} one can get
		\begin{equation}\label{l2}
		\begin{split}
		&\iint_{D_t}|\p\Omega^l\phi|\Big\{|\t Z\t\G^{p_1}\O^{l_1}\phi|\cdot|\p\t\G^{p_2}\O^{l_2}\phi|\cdots|\p\t\G^{p_k}\O^{l_k}\phi|\cdot|\p^2\tilde\Gamma^{p_0}\O^{l_0}\phi|\\
		&\qquad\qquad\qquad+|\p\t\G^{p_1}\O^{l_1}\phi|\cdots|\p\t\G^{p_k}\O^{l_k}\phi|\cdot|\t Z\p\t\G^{p_0}\O^{l_0}\phi|\Big\}\\
		\lesssim&M_0^k\delta^{\f{5(k-1)}4+\f{11}{24}-k\varepsilon_0}\Big\{\iint_{D_t}\tau^{-1}|\p\Omega^l\phi|\cdot|\f{\t Z\O^{\leq l}\phi}{1+\tau-r}|+\iint_{D_t}\tau^{-2}|\p\O^{\leq l}\phi|^2\Big\}\\
		\lesssim&\int_{t_0}^t\tau^{-2}E_{0,\leq 4}(\tau)d\tau+\dl^{4-2\ve_0}.
		\end{split}
		\end{equation}
		\item If $j_0\geq j_1$, then $j_1\leq 2$ and $j_0\leq 3$. Similarly to \eqref{l1}, one can show
		\begin{equation}\label{j0}
		\begin{split}
		&\iint_{D_t}|\p\Omega^l\phi|\cdot|\O^{j_1}h^{\al\beta}|\cdot|\p^2\O^{j_0}\phi|\\
		\lesssim&(M_0\dl^{5/4-\ve_0})^{k+1}\int_{t_0}^t\tau^{-(k+1)/2}\sqrt{E_{0,\leq 4}(\tau)E_{1,\leq 3}(\tau)}d\tau\\
		\lesssim&\int_{t_0}^t\tau^{-3/2}E_{0,\leq 4}(\tau)d\tau+\dl^{4-2\ve_0}.
		\end{split}
		\end{equation}
		\item When $j_0<j_1$, then it holds that
		\begin{equation}\label{j1}
		\begin{split}
		&\iint_{D_t}|\p\Omega^l\phi|\cdot|\O^{j_1}h^{\al\beta}|\cdot|\p^2\O^{j_0}\phi|\\
		\lesssim&(M_0\dl^{5/4-\ve_0})^{k}(M_0\dl^{11/24-\ve_0})\int_{t_0}^t\tau^{-(k+1)/2}E_{0,\leq 4}(\tau)d\tau\lesssim\int_{t_0}^t\tau^{-3/2}E_{0,\leq 4}(\tau)d\tau.
		\end{split}
		\end{equation}
	\end{enumerate}
Inserting \eqref{l1} and \eqref{j1} into \eqref{E01} yields
\begin{equation}\label{E0}
E_{0,l}(t)\lesssim\delta^{4-2\varepsilon_0},\quad l\leq 4.
\end{equation}
\vskip 0.1 true cm

\item {\bf The case $p+l\leq 4$ and $p\geq 1$}

\vskip 0.1 true cm

For $p+l\leq 4$ $(p\geq 1)$, as in the above case, one can make use of Proposition \ref{P4.1} and Corollary \ref{C4.1} to estimate the
related integrals in \eqref{E2-6} for $\max\{p_1+l_1,\cdots,p_k+l_k\}\leq p_0+l_0$ and $q_1+j_1\leq q_0+j_0$ ($p_1+l_1,\cdots,p_k+l_k\leq2$, $p_0+l_0\leq3$, $q_1+j_1\leq 2$ and $q_0+j_0\leq 3$) and the rest terms ($p_0+l_0\leq 1$ or $q_0+j_0\leq 1$). In the whole process,
one can use repeatedly the bootstrap assumptions \ref{EA} and the restricted condition $\varepsilon_0\in(0,\ve_k^*)$ for $k\geq 2$.
\end{enumerate}

Consequently, by Grownwall's inequality, one can arrive at
\begin{equation}\label{Y-34}
\begin{split}
&E_{0,l}(t)\lesssim \delta^{4-2\varepsilon_0},\qquad E_{1,l}(t)\lesssim\delta^{17/6-2\varepsilon_0},\qquad E_{2,l}(t)\lesssim \delta^{1-2\varepsilon_0},\\
&E_{3,l}(t)\lesssim \delta^{-1-2\varepsilon_0},\qquad E_{4,l}(t)\lesssim \delta^{-3-2\varepsilon_0}.
\end{split}
\end{equation}
Since the constants in \eqref{Y-34} are all independent of $M_0$, the bootstrap
assumptions  \eqref{EA} are proved as long as \eqref{EA} holds for the time $t=t_0$
by the continuous induction argument.

\end{proof}

We are now ready to prove Theorem \ref{main}.
\begin{proof}
By Theorem \ref{Th2.1}, the local existence of the smooth solution $\phi$ to \eqref{quasi} with \eqref{id},
\eqref{Y-0} and
\eqref{Y-0-a}-\eqref{g00}  has been obtained.
On the other hand,  the global estimates of $\phi$ in $A_{2\dl}$
and $B_{2\dl}$  have been established in Section \ref{YY} and Section \ref{inside} respectively.
Then it follows from the local well-posedness of the smooth solution to \eqref{quasi} with \eqref{id} (see \cite{H} or \cite{J2})
and the continuous induction argument
that $\phi\in C^\infty([1,+\infty)\times\mathbb R^2)$ exists globally.
In addition, $|\p\phi|\lesssim\delta^{1-\varepsilon_0}t^{-1/2}$ follows from \eqref{Lle} in Appendix B, \eqref{imp} and
the first inequality in \eqref{pphi}. Meanwhile, $|\phi|\lesssim\delta^{1-\varepsilon_0}t^{1/2}$ holds true by integration
on $\p\phi$ from the conic surface $C_0$.
Thus Theorem \ref{main} is proved.
\end{proof}

\appendix
\setcounter{equation}{1}
\section{Computations on the deformation tensor and commutators relations}

Let $\leftidx{^{(V)}}{\slashed\pi}_{UX}:=\leftidx{^{(V)}}{\pi}_{UX}$ for $U\in\{\mathring L,\underline{\mathring L},X\}$.
As in \cite[Proposition 7.7]{J} for 3D case or in \cite[(2.43)-(2.45)]{Ding4} for 4D case,
it follows from direct computations that

\begin{enumerate}[(1)]
	\item for $V=T$,
	\begin{equation}\label{Lpi}
	\begin{split}
	&\leftidx{^{(T)}}\pi_{\mathring L\mathring L}=0,\quad \leftidx{^{(T)}}\pi_{T\tilde T}=2T\mu,\quad \leftidx{^{(T)}}\pi_{\mathring LT}=-T\mu,\quad
	\leftidx{^{(T)}}{\slashed\pi}_{TX}=0,\\
	&\leftidx{^{(T)}}{\slashed\pi}_{\mathring LX}=-2\mu\zeta_X-\slashed d_X\mu,\quad \leftidx{^{(T)}}{\slashed\pi}_{XX}=2\mu\sigma_{XX};
	\end{split}
	\end{equation}
	
	\item for $V=\mathring L$,
	\begin{equation}\label{uLpi}
	\begin{split}
	&\leftidx{^{(\mathring L)}}\pi_{\mathring L\mathring L}=0,\quad \leftidx{^{(\mathring L)}}\pi_{T\tilde T}=2\mathring L\mu,\quad \leftidx{^{(\mathring L)}}\pi_{\mathring LT}=-\mathring L\mu,\quad \leftidx{^{(\mathring L)}}{\slashed\pi}_{\mathring LX}=0,\\
	&\leftidx{^{(\mathring L)}}{\slashed\pi}_{TX}=2\mu\zeta_X+\slashed d_X\mu,\quad \leftidx{^{(\mathring L)}}{\slashed\pi}_{XX}=2\chi_{XX};
	\end{split}
	\end{equation}
	
	\item for $V=R$,
	\begin{equation}\label{Rpi}
	\begin{split}
	&\leftidx{^{(R)}}\pi_{\mathring L\mathring L}=0,\quad \leftidx{^{(R)}}\pi_{T\tilde T}=2R\mu,\quad \leftidx{^{(R)}}\pi_{\mathring LT}=-R\mu,\\
	&\leftidx{^{(R)}}{\slashed\pi}_{\mathring LX}=-{R}^X\check{\chi}_{XX}-\upsilon\{G_{\mathring L\tilde T}^\gamma\slashed d_X\varphi_\gamma+\f12G_{\tilde T\tilde T}^\gamma\slashed d_X\varphi_\gamma-G_{X\tilde T}^\gamma{\mathring L}\varphi_\gamma\}\\
	&\qquad\qquad\quad+\f12{R}^XG_{XX}^\gamma\mathring L\varphi_\gamma+\epsilon_i^jg_{ja}\check L^i\slashed d_Xx^a,\\
	&\leftidx{^{(R)}}{\slashed\pi}_{TX}=\mu{R}^X\check\chi_{XX}+\upsilon\slashed d_X\mu+\mu G_{X\mathring L}^\gamma R\varphi_{\gamma}-\f12\mu R^XG_{XX}^\gamma{\mathring L}\varphi_{\gamma}\\
	&\qquad\qquad\quad+\mu G_{X\tilde T}^\gamma R\varphi_{\gamma}+\upsilon\{G_{X\tilde T}^\gamma T\varphi_\gamma-\f12\mu G_{\tilde T\tilde T}^\gamma\slashed d_X\varphi_{\gamma}\}+\mu\epsilon_i^j g_{ja}\check T^i\slashed d_Xx^a,\\
	&\leftidx{^{(R)}}{\slashed\pi}_{XX}=2\upsilon\chi_{XX}+\upsilon\{2G_{X\mathring L}^\gamma\slashed d_X\varphi_{\gamma}
+2G_{X\tilde T}^\gamma\slashed d_X\varphi_{\gamma}-G_{XX}^\gamma\mathring L\varphi_{\gamma}\}\\
&\qquad\qquad\quad+G_{XX}^\gamma R\varphi_{\gamma}
+2\epsilon_i^j\check g_{ja}(\slashed d_Xx^a)\slashed d_Xx^i,
	\end{split}
	\end{equation}
where $\check g_{ia}:=g_{ia}-m_{ia}$, and the definitions of $\check T^i$, $\check L^i$ are given in \eqref{errorv}.	
\end{enumerate}

In addition, the following two results hold for commutators,
which are given in Lemma 4.10, 8.9 and 8.11 of \cite{J}.

\begin{lemma}\label{com}
	It holds that
	\begin{equation}\label{c}
	\begin{split}
	[\mathring L, R]&={\leftidx{^{(R)}}{\slashed\pi}_{\mathring L}}^XX,\quad
	[\mathring L, T]={\leftidx{^{(T)}}{\slashed\pi}_{\mathring L}}^XX,\quad
	[T,R]={\leftidx{^{(R)}}{\slashed\pi}_{T}}^XX.
	\end{split}
	\end{equation}
\end{lemma}

\begin{lemma}\label{commute}
For any vector field $Z\in\{\mathring L,T,R\}$,

(a) if $f$ is a smooth function, then
 \begin{align}
 \big([\slashed\nabla^2,\slashed{\mathcal L}_Z]f\big)_{XX}&
 =\f12{\slashed\nabla}_X\big(\textrm{tr}\leftidx{^{(Z)}}{\slashed\pi}\big)\slashed d_Xf,\label{nZf}\\
 [\slashed\triangle, Z]f&=\leftidx{^{(Z)}}{\slashed\pi}^{XX}\slashed\nabla_{X}^2f
 +\f12({\slashed\nabla}_X\leftidx{^{(Z)}}{\slashed\pi}^{XX})\slashed d_Xf;\label{LZf}
 \end{align}

(b) if $\xi$ is a one-form on $S_{s, u}$, then
\begin{equation}\label{nZx}
([\slashed\nabla_X,\slashed{\mathcal L}_Z]\xi)_X=\f12{\slashed\nabla}_X(\textrm{tr}\leftidx{^{(Z)}}{\slashed\pi})\xi_X;
\end{equation}

(c) if  $\xi$ is a $(0,2)$-type tensor on $S_{s, u}$, then
\begin{align}
([\slashed\nabla_X,\slashed{\mathcal L}_Z]\xi)_{XX}&={\slashed\nabla}_X\big(\textrm{tr}\leftidx{^{(Z)}}{\slashed\pi}\big)\xi_{XX},\label{nZxi}\\
([\slashed\nabla_X,\slashed{\mathcal L}_Z]\slashed\nabla\xi)_{XXX}&=\f32{\slashed\nabla}_X\big(\textrm{tr}\leftidx{^{(Z)}}{\slashed\pi}\big)\slashed\nabla_X\xi_{XX}.\label{nZxX}
\end{align}
\end{lemma}

\section{The proof  of Theorem \ref{Th2.1}}

In this Appendix, we utilize the energy method and the special structure of the initial data \eqref{id}
together with \eqref{Y-0} and
\eqref{Y-0-a}-\eqref{g00}
to prove Theorem \ref{Th2.1}.
Although the procedure is similar
to that in \cite[Theorem 2.1]{Ding3},
some details are required to be given due to the general form of the quasilinear wave
equation \eqref{quasi} with the higher order null condition structure here.

\begin{proof}
	 Denote by $Z_g$ any fixed vector field in $\{S, H_i, i=1,2\}$.
Then analogously to the proof of (2.5) in \cite{Ding3}, one can have that for $1\leq t\leq t_0$, $\nu\in \mathbb N_0^3$
	and $N_0\in\mathbb N_0$ with $N_0\ge 3$,
	\begin{equation}\label{ba}
	|\p^\kappa\O^pZ_g^\nu\phi|\leq\delta^{2-\epsilon_0-|\kappa|}\quad (|\kappa|+p+|\nu|\leq N_0,\quad |\nu|\leq 1).
	\end{equation}

In addition, by $\ds L=\f{S+\o^iH_i}{t+r}$, it can be shown that for $|\kappa|+p+m\leq 2N_0-3$ and $m\leq 1$,
\begin{equation}\label{Lle}
|L^m\p^\kappa\O^p\phi(t,x)|\lesssim\sum_{|\nu|\leq 1}|Z_g^\nu\p^\kappa\O^p\phi(t,x)|\lesssim\delta^{2-|\kappa|-\varepsilon_0}.
\end{equation}

		\begin{figure}[htbp]
	\centering
	\includegraphics[scale=0.23]{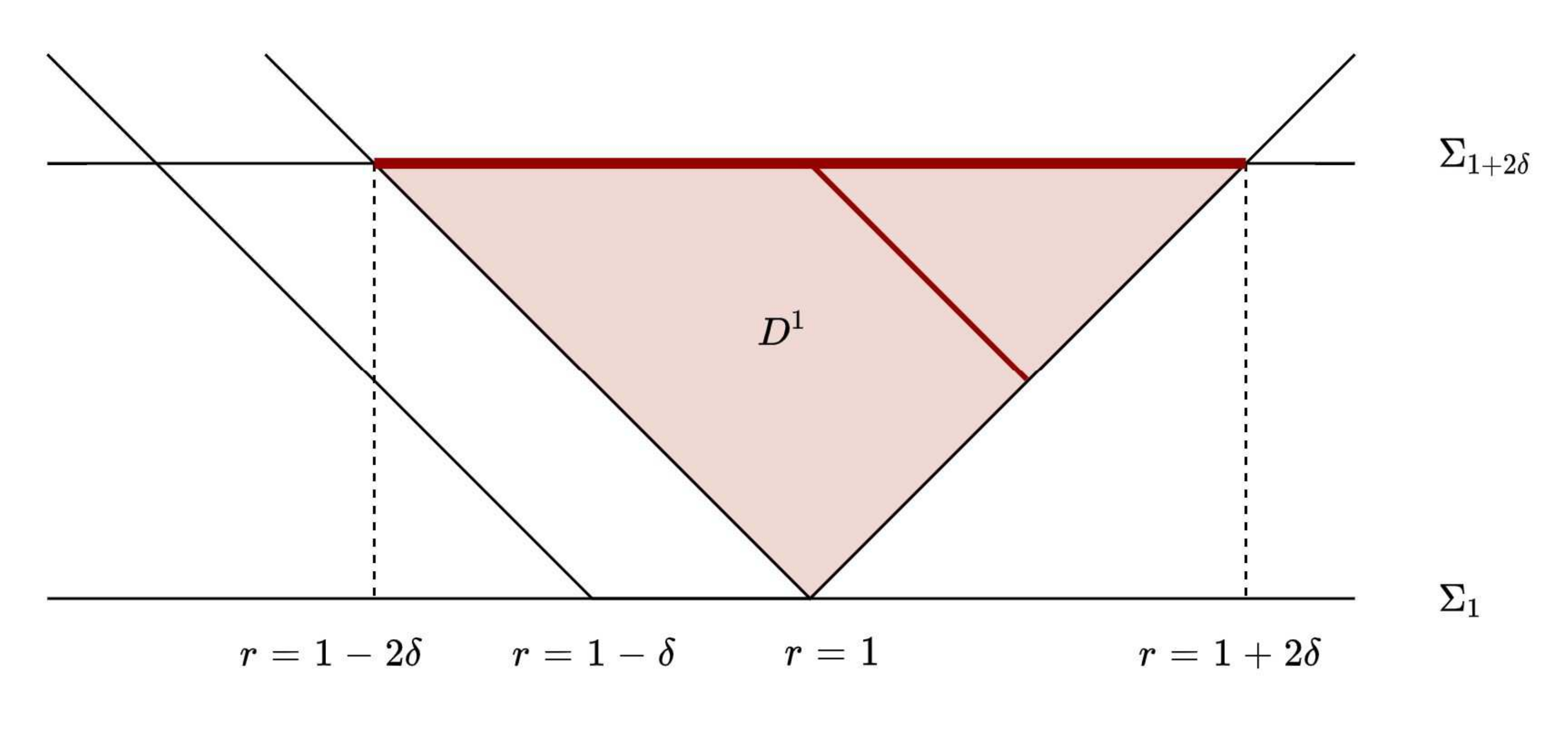}
	\caption{Space-time domain $D^1=\{(t,r): 1\le t\le t_0, 2-t\le r\le t\}$}\label{pic:p3}
\end{figure}

Now we start to improve the $L^\infty$ estimate of $\phi(t_0,x)$ on some special space domains. To this end,
one rewrites \eqref{quasi} as
\begin{equation}\label{me}
L\underline L\phi=\f{1}{2r}L\phi-\f{1}{2r}\underline L\phi+\f{1}{r^2}\Omega^2\phi+g^{\al\beta,
	\gamma_1,\cdots,\gamma_k}(\p_{\gamma_1}\phi)\cdots(\p_{\gamma_k}\phi)\p^2_{\al\beta}\phi+h^{\al\beta}(\p\phi)\p^2_{\al\beta}\phi.
\end{equation}

Applying the operator $L\bar\p^{\iota}\O^p$ to both sides of \eqref{me}, where $\bar\p\in\{\p_t,\p_r\}$ and $\iota\in\mathbb N_0^2$, then one can obtain an expression of $\underline LL^2\bar\p^{\iota}\O^p\phi$ by using $\underline LL^2\bar\p^{\iota}\O^p=L\bar\p^{\iota}\O^pL\underline L$
and a direct computation.
This, together with \eqref{Lle}, yields that for $|\iota|+p\leq 2N_0-6$,
\begin{equation}\label{YHC-1}
|\underline L(L^2\bar\p^{\iota}\O^p\phi)|\lesssim\delta^{1-\varepsilon_0-|\iota|}+\delta^{(k+1)(1-\ve_0)-1-|\iota|}.
\end{equation}
Due to the vanishing property of $\phi$ on $\{(t, x): t\ge 1, t=r\}$,
integrating \eqref{YHC-1} along integral curves of $\underline L$ yields that for $(t,r)\in D^1$
(see Figure \ref{pic:p3}) and for $|\iota|+p\leq 2N_0-6$,
\begin{equation}\label{D}
|L^2\bar\p^{\iota}\O^p\phi(t,x)|\lesssim\delta^{2-\varepsilon_0-|\iota|}+\delta^{(k+1)(1-\ve_0)-|\iota|}.
\end{equation}
Note that the null condition \eqref{null} implies
\begin{equation}\label{YHC-2}
\begin{split}
g^{\al\beta,\gamma_1\cdots\gamma_k}\p_{\gamma_1}\phi\cdots\p_{\gamma_k}\phi\p_{\al\beta}^2\phi=&f_1(\o)L\phi(\p\phi)^{k-1}\p^2\phi
+f_2(\o)r^{-1}\O\phi(\p\phi)^{k-1}\p^2\phi\\
&+f_3(\o)(\p\phi)^kL\p\phi+f_4(\o)r^{-1}(\p\phi)^k\O\p\phi,
\end{split}
\end{equation}
where  $f_i$ $(i=1,2,3,4)$ are some smooth functions.
Using the equation for $\underline LL^2\bar\p^{\iota}\O^p\phi$ again, together with \eqref{YHC-2} and
\eqref{D}, gives $|\underline L(L^2\bar\p^{\iota}\O^p\phi) |\lesssim\delta^{1-\varepsilon_0-|\iota|}$ for $0<\varepsilon_0<\ve_k^*$ and $|\iota|+p\leq2N_0-7$, which yields that for $(t,x)\in D^1$,
\begin{equation}\label{D-1}
|L^2\bar\p^{\iota}\O^p\phi(t,x)|\lesssim\delta^{2-\varepsilon_0-|\iota|},\quad|\iota|+p\leq2N_0-7.
\end{equation}
Then an induction argument and  \eqref{me} show that for $(t,x)\in D^1$,
 \begin{equation}\label{oi}
|LL^m\bar\p^\iota\O^p\phi(t,x)|\lesssim\delta^{2-|\iota|-\varepsilon_0},\quad 3m+|\iota|+p\leq 2N_0-4,
\end{equation}
which implies \eqref{local1-2} due to $\p_i=\o^i\p_r+\f{\o_\perp^i}r\O$.

\begin{figure}[htbp]
\centering
\includegraphics[scale=0.23]{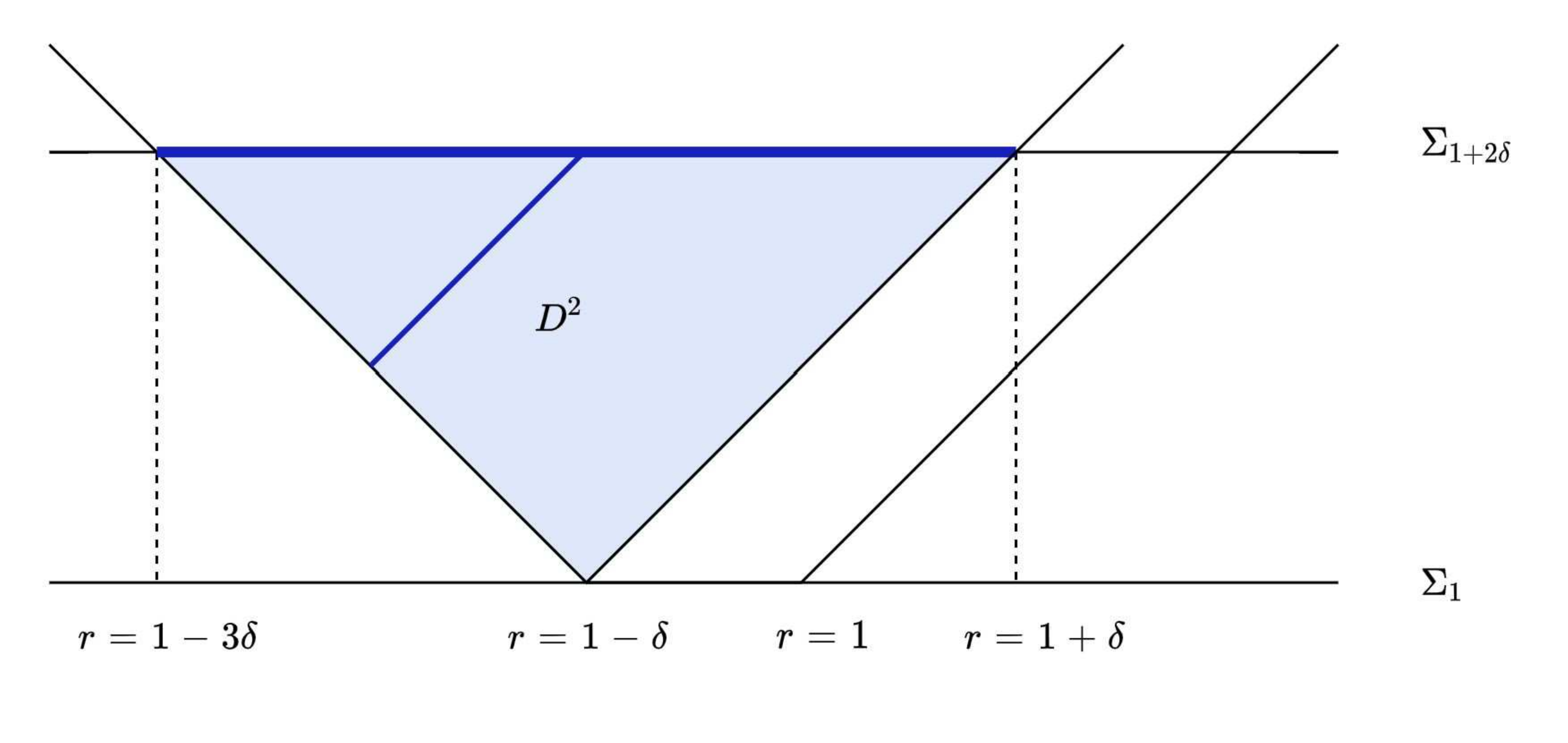}
\caption{Spatial domain for $1-3\delta\le r\le 1+\delta$ on $\Sigma_{1+2\dl}$}\label{pic:p4}
\end{figure}

Similarly, by the expression of $L\underline L^n\bar\p^\iota\O^p\phi$, integrating it
along integral curves of $L$ yields that for $r\in [1-3\delta,1+\delta]$ (see Figure \ref{pic:p4}),
\begin{equation}\label{ii}
|\underline L^n\bar\p^\iota\O^p\phi(t_0,x)|\lesssim\delta^{2-|\iota|-\varepsilon_0},\quad 2n+|\iota|+p\leq 2N_0-3,
\end{equation}
which implies \eqref{local2-2}.
In addition, it follows from \eqref{Lle} that
\begin{equation}\label{Op}
|\O^p\phi|_{D^2}\lesssim\dl^{3-\ve_0},\ p\leq 2N_0-4
\end{equation}
after integrating $L\O^k\phi$ along integral curves of $L$ in $D^2$.
Furthermore, using $\p_t=\f12(L+\underline L)$,
$\p_i=\f{\o^i}{2}(L-\underline L)+\f{\o^i_{\perp}}{r}\O$,
\eqref{ii}, \eqref{Op}, \eqref{me} and \eqref{Lle}, one concludes that for $|\kappa|+p\leq 2N_0-5$
and $r\in [1-3\delta, 1+\delta]$,
\begin{equation}\label{iid}
|\p^\kappa\O^p\phi(t_0,x)|\lesssim
\delta^{3-|\kappa|-\varepsilon_0},
\end{equation}
and then \eqref{local3-2} follows.

Next we prove the estimate \eqref{local3-3}. It is already shown that on the surface $\Sigma_{t_0}$, when $r\in[1-2\delta,1+\delta]$ and $2n+|\kappa|+p\leq 2N_0-3$, $|\underline L^{n}\p^\kappa\O^p\phi|\lesssim\delta^{2-\varepsilon_0-|\kappa|}$ holds. Then $|L\underline L^n\p^\kappa\O^p\phi|\lesssim\delta^{2-\varepsilon_0-|\kappa|}$ is derived by \eqref{me} and a direct computation.
Meanwhile, when $r\in[1-2\delta,1+\delta]$ and $3m+|\kappa|+p\leq 2N_0-4$,
one has $|L^{1+m}\p^\kappa\O^p\phi|\lesssim\delta^{2-\varepsilon_0-|\kappa|}$.
Furthermore, one can obtain
\begin{equation}\label{LLO}
|L^m\underline L^n\p^\kappa\O^p\phi|\lesssim\delta^{2-\varepsilon_0-|\kappa|}
\quad\text{for}\ 3m+2n+|\kappa|+p\leq 2N_0-3\ \textrm{and $n\geq 1$}.
\end{equation}
Indeed, by an induction argument, it suffices to prove \eqref{LLO} for $3(m+1)+2n+|\kappa|+p\leq 2N_0-3$ and $n\geq 1$.
It follows from \eqref{me} and the assumption \eqref{LLO} that
\begin{equation}\label{La}
\begin{split}
&|L^{m+1}\underline L^{n}\bar\p^\iota\O^p\phi|=|L^{m}\underline L^{n-1}\bar\p^\iota\O^p(L\underline L\phi)|\\
\lesssim&\delta^{2-\varepsilon_0-|\iota|}+\sum_{\tiny\begin{array}{c}|\iota_1|+p_1\leq|\iota|+p+1\\
|\iota_1|\leq|\iota|+1,p_1\leq p+1\end{array}}\delta^{|\iota_1|-|\iota|}|L^{m+1}\underline L^{\leq n-1}\bar\p^{\iota_1}\O^{p_1}\phi|.
\end{split}
\end{equation}
Applying the same argument as for \eqref{La} to treat $|L^{m+1}\underline L^{\leq n-1}\bar\p^{\iota_1}\O^{p_1}\phi|$,
one can reduce the number of $\underline L$ gradually to get
\begin{equation}\begin{split}
|L^{m+1}\underline L^{n}\bar\p^\iota\O^p\phi|&\lesssim\delta^{2-\varepsilon_0-|\iota|}+\sum_{\tiny\begin{array}{c}|\iota_1|
+p_1\leq|\iota|+p+n\\
	|\iota_1|\leq|\iota|+n,p_1\leq n+p\end{array}}\delta^{|\iota_1|-|\iota|}|L^{m+n}\bar\p^{\iota_1}\O^{p_1}\phi|\\
&\lesssim\delta^{2-\varepsilon_0-|\iota|},
\end{split}\end{equation}
which yields \eqref{LLO}. Then \eqref{local3-3} is  deduced through utilizing \eqref{LLO} with $\kappa=0$.
Therefore, the proof of Theorem \ref{Th2.1} is completed.
\end{proof}

\end{document}